%% file: riesz-DX7.tex
\documentclass[11pt]{amsart}

\usepackage{amsmath}
\usepackage{amssymb}
\usepackage{bbm}
\usepackage{pdfsync}
\usepackage{esint} 
\usepackage[breaklinks=true]{hyperref}
\usepackage[utf8]{inputenc}

\usepackage[T1]{fontenc}
\usepackage{mathabx}
\usepackage[textsize=tiny,disable]{todonotes}

\newcommand{\R}{{\mathbb R}}       
\newcommand{\N}{{\mathbb N}}

\newcommand{\Z}{{\mathbb Z}}       
\newcommand{\DD}{{\mathcal D}}

\newcommand{\HH}{{\mathcal H}}
\newcommand{\LL}{{\mathcal L}}
\newcommand{\PP}{{\mathcal P}}
\newcommand{\QQ}{{\mathcal Q}}

\newcommand{\cS}{{\mathcal S}}

\newcommand{\TT}{{\mathcal T}}

\newcommand{\RR}{{\mathcal R}}
\newcommand{\cM}{{\mathcal M}}
\newcommand{\cA}{{\mathcal A}}
\newcommand{\cF}{{\mathcal F}}
\newcommand{\Ch}{{\mathcal Ch}}
\newcommand{\EE}{{\mathcal E}}
\newcommand{\UU}{{\mathcal U}}

\newcommand{\cI}{{\mathcal I}}
\newcommand{\cJ}{{\mathcal J}}

\newcommand{\diam}{{\rm diam}}
\newcommand{\dist}{{\rm dist}}

\newcommand{\rf}[1]{{(\ref{#1})}}

\newcommand{\supp}{\operatorname{supp}}

\newcommand{\vphi}{{\varphi}}
\newcommand{\ve}{{\varepsilon}}
\newcommand{\vv}{{\vspace{2mm}}}
\newcommand{\vvv}{{\vspace{3mm}}}
\newcommand{\wt}[1]{{\widetilde{#1}}}
\newcommand{\wh}[1]{{\widehat{#1}}}
\newcommand{\wk}[1]{{\widecheck{#1}}}

\newcommand{\noi}{\noindent}
\newcommand{\rest}{{\lfloor}}

\newcommand{\sss}{{\mathsf {Stop}}}
\newcommand{\ttt}{{\mathsf {Top}}}
\newcommand{\treg}{{\mathsf{Treg}}}

\newcommand{\Gen}{{\mathsf {Gen}}}
\newcommand{\tree}{{\rm Tree}}

\newcommand{\pv}{\operatorname{pv}}

\newcommand{\GG}{{\mathsf G}}
\newcommand{\HE}{{\mathsf {HE}}}
\newcommand{\bad}{{\mathsf{Bad}}}

\newcommand{\HD}{{\mathsf{HD}}}
\newcommand{\hd}{{\mathsf{hd}}}
\newcommand{\GH}{{\mathsf{GH}}}
\newcommand{\Ot}{{\mathsf{Ot}}}

\newcommand{\LD}{{\mathsf{LD}}}
\newcommand{\BR}{{\mathsf {BR}}}

\newcommand{\MDW}{{\mathsf{MDW}}}
\newcommand{\Stop}{{\mathsf{Stop}}}

\newcommand{\Trc}{{\mathsf{Trc}}}
\newcommand{\sL}{{\mathsf{L}}}
\newcommand{\sD}{{\mathsf{D}}}

\newcommand{\sF}{{\mathsf{F}}}
\newcommand{\sG}{{\mathsf{G}}}
\newcommand{\sH}{{\mathsf{H}}}
\newcommand{\sM}{{\mathsf{M}}}

\newcommand{\Neg}{{\mathsf{Neg}}}

\newcommand{\End}{{\mathsf{End}}}
\newcommand{\Reg}{{\mathsf{Reg}}}

\def\Xint#1{\mathchoice
{\XXint\displaystyle\textstyle{#1}}%
{\XXint\textstyle\scriptstyle{#1}}%
{\XXint\scriptstyle\scriptscriptstyle{#1}}%
{\XXint\scriptscriptstyle\scriptscriptstyle{#1}}%
\!\int}
\def\XXint#1#2#3{{\setbox0=\hbox{$#1{#2#3}{\int}$ }
\vcenter{\hbox{$#2#3$ }}\kern-.58\wd0}}

\def\avint{\;\Xint-}

\usepackage{pgf,tikz} 

\definecolor{ffffff}{rgb}{1.0,1.0,1.0}
\definecolor{qqqqff}{rgb}{0.0,0.0,1.0}
\definecolor{ffqqqq}{rgb}{1.0,0.0,0.0}
\definecolor{zzzzqq}{rgb}{0.6,0.6,0.0}
\definecolor{marronet}{rgb}{0.6,0.2,0}
\definecolor{negre}{rgb}{0,0,0}
\definecolor{vermell}{rgb}{0.8,0.05,0.05}
\definecolor{blau}{rgb}{0.3,0.2,1.}
\definecolor{blauclar}{rgb}{0.,0.,1.}
\definecolor{grisfosc}{rgb}{0.25098039215686274,0.25098039215686274,0.25098039215686274}
\definecolor{verd}{rgb}{0.1,0.6,0.1}
\definecolor{taronja}{rgb}{0.9,0.6,0.05}
\definecolor{vermellclar}{rgb}{1.,0.,0.}
\definecolor{verdet}{rgb}{0,0.8,0.1}
\definecolor{blauverd}{rgb}{0,0.4,0.2}
\definecolor{grisclar}{rgb}{0.6274509803921569,0.6274509803921569,0.6274509803921569}

\newtheorem{theorem}{Theorem}[section]
\newtheorem{lemma}[theorem]{Lemma}
\newtheorem{mlemma}[theorem]{Main Lemma}

\newtheorem{claim}[theorem]{Claim}
\newtheorem*{claim*}{Claim}

\newtheorem*{theorem*}{Theorem}

\theoremstyle{definition}

\theoremstyle{remark}
\newtheorem{rem}[theorem]{\bf Remark}

\numberwithin{equation}{section}

\newcommand{\brem}{\begin{rem}}
\newcommand{\erem}{\end{rem}}

\begin{document}


\title[The measures with $L^2$-bounded Riesz transform]{The measures with $L^2$-bounded Riesz transform satisfying a subcritical Wolff-type energy condition}

\author[D. D\k{a}browski]{Damian D\k{a}browski}
\address{Damian D\k{a}browski\\ Departament de Matem\`atiques, Universitat Aut\`onoma de Barcelona, and Barcelona Graduate School of Mathematics (BGSMath)\newline\indent Edifici C Facultat de Ci\`encies, 08193 Bellaterra, Barcelona, Catalonia, Spain\newline\indent
Current address: P.O. Box 35 (MaD), 40014 University of Jyväskylä, Finland}

\email{damian.m.dabrowski ``at'' jyu.fi}

\author[X. Tolsa]{Xavier Tolsa}
\address{Xavier Tolsa
\\
ICREA, Passeig Llu\'{\i}s Companys 23 08010 Barcelona, Catalonia\\
 Departament de Matem\`atiques, and BGSMath
\\
Universitat Aut\`onoma de Barcelona
\\
08193 Bellaterra (Barcelona), Catalonia.
}

\email{xtolsa ``at'' mat.uab.cat}

\thanks{Both authors were supported by 2017-SGR-0395 (Catalonia) and MTM-2016-77635-P (MINECO, Spain). D.D. was supported by the Academy of Finland via the project \textit{Incidences on Fractals}, grant No. 321896. X.T. was supported by
 the European Research Council (ERC) under the European Union's Horizon 2020 research and innovation programme (grant agreement 101018680).}


\begin{abstract}
In this work we obtain a geometric characterization of the measures $\mu$ in $\R^{n+1}$ with polynomial upper growth of degree $n$ such that the $n$-dimensional Riesz transform $\RR\mu (x) = \int \frac{x-y}{|x-y|^{n+1}}\,d\mu(y)$ belongs to $L^2(\mu)$, under the assumption that 
$\mu$ satisfies the following Wolff energy estimate, for any ball $B\subset\R^{n+1}$:
$$\int_B \int_0^\infty \left(\frac{\mu(B(x,r))}{r^{n-\frac38}}\right)^2\,\frac{dr}r\,d\mu(x)\leq
M\,\bigg(\frac{\mu(2B)}{r(B)^{n-\frac38}}\bigg)^2\,\mu(2B)\quad .$$
More precisely, we show that $\mu$ satisfies the following estimate:
$$\|\RR\mu\|_{L^2(\mu)}^2 + \|\mu\|\approx \int\!\!\int_0^\infty \beta_{\mu,2}(x,r)^2\,\frac{\mu(B(x,r))}{r^n}\,\frac{dr}r\,d\mu(x) +
\|\mu\|,$$
where $\beta_{\mu,2}(x,r)^2 = \inf_L \frac1{r^n}\int_{B(x,r)} \left(\frac{\dist(y,L)}r\right)^2\,d\mu(y),$
with the infimum taken over all affine $n$-planes $L\subset\R^{n+1}$. In a companion paper which relies on the results obtained in this work it is shown that the same
result holds without the above assumption regarding the Wolff energy of $\mu$. This
result has important consequences for the Painlev\'e problem for Lipschitz harmonic functions.
\end{abstract}

\maketitle

\tableofcontents

\input{riesz-DX7-S1-1}

\input{riesz-DX7-S2-2}

\input{riesz-DX7-S3-1}

\input{riesz-DX7-S4-2}

\input{riesz-DX7-S5-2}

\input{riesz-DX7-S6-2}

\input{riesz-DX7-S7-2}

\input{riesz-DX7-S8-1}

\input{riesz-DX7-S9-2}

\appendix\section{A list of parameters}\label{app:param}
Here is the list of the most important constants and parameters that appear in the paper, in the order of appearance. We also point out the dependence of different constants on each other (usually we do not track dependence on dimension).
\begin{itemize}
	\item $C_0,\, A_0$ are the constants from the David-Mattila lattice, and they depend only on dimension, see Remark~\ref{rema00}. Moreover, $A_0$ is assumed big enough that Lemma~\ref{lem:43} and Lemma~\ref{lem:66} hold. Starting from Section \ref{sec:Pdoubling}, $A_0$ and $C_0$ are considered to be fixed constants, and all the subsequent estimates and parameters may depend on them. We do not track this dependence.
	\item $C_d$ is the constant from the definition of $\PP$-doubling cubes in Subsection~\ref{subsec:Pdoubling}, and we have $C_d =4A_0^n$.
	\item $M_0$ is the constant from the definition of the $\HE$ family. It is an arbitrarily large parameter chosen in Theorem~\ref{propomain} (the Main Theorem).
	\item $\Lambda$ is the $\HD$ constant, and it is of the form $\Lambda = A_0^{k_{\Lambda}n}$ for some large integer $k_{\Lambda}$. It depends on $M_0$, see the end of the proof of Lemma~\ref{lemneg3}, as well as Remark~\ref{rem9.12}.
	\item $\Lambda_*$ is the $\HD_*$ constant, and it is of the form $\Lambda_* = \Lambda^{1-\frac1N}$, where $N$ is a large dimensional constant fixed above Lemma~\ref{lemtoptop} (for example, $N=500n$ works).
	\item $\delta_0$ is the $\LD$ constant, and it is of the form $\delta_0 = \Lambda^{-N_0 - \frac{1}{2N}}$ for some large integer $N_0$ depending on dimension. See also Remark~\ref{rem9.12}.
	\item $B$ is the $\MDW$ constant, and it is of the form $B=\Lambda_*^{\frac{1}{100n}}.$
	\item $\ell_0>0$ is the parameter used in the definition of function $d_{R,\ell_0}$ at the beginning of Section~\ref{sec6.2*}. It is assumed to be small enough in Lemma~\ref{lemregtot*} (depending on Lemma~\ref{lemregot}), and in Lemma~\ref{lemalter*}. We have no control over how small $\ell_0$ is (it comes from a ``qualitative argument''), but none of the other constants depend on it.
	\item $\varepsilon_n>0$ is a small auxiliary parameter introduced above Lemma~\ref{lem:66}. In the present paper one can take $\varepsilon_n=1/15$, but for the application of some of the results from Section~\ref{sec6} in \cite{Tolsa-riesz}, we allow $\varepsilon_n$ to be an arbitrarily small parameter depending on dimension.
	\item $\gamma$ appears in the paper twice; in Lemma \ref{lemDMimproved} it is an auxiliary parameter with $\gamma\in (0,1)$, and it retains that meaning until the end of Section~\ref{sec:DMlatt}. Later on, in Section~\ref{sec99}, $\gamma>0$ is the constant from the definition of $\gamma$-nice trees, introduced below Remark~\ref{rem9.7}. It is chosen to be small enough depending on $\Lambda$ and $M_0$ in Remark~\ref{rem9.12}.
	\item $\varepsilon_Z>0$ is a small constant introduced in Lemma~\ref{lemregot}. It is fixed in Remark~\ref{rem9.12}, depending on $\Lambda$ and $M_0$.
	\item $K>0$ is the $\BR$ constant. It is chosen in the proof of Lemma~\ref{lemtop8}, and it is big enough depending on $\Lambda, \delta_0$ and $M_0$.
	\item $C_*>0$ is a constant from Lemma~\ref{lem:RegGamma}. It depends on $\Lambda$ and $\delta_0$.
\end{itemize}

\section{A list of cube families and related objects}\label{app:fam}
Below we list the most important families of cubes that appear in the paper, along with a link to the section where they were defined. We also list some other notions related to cubes (e.g., the approximating measures, or enlarged cubes).
\begin{itemize}
	\item $\DD_\mu$ is the family of David-Mattila cubes from Section~\ref{sec:DMlatt}.
	\item $\DD^\PP_\mu$ is the family of $\PP$-doubling cubes from Subsection~\ref{subsec:Pdoubling}.
	\item $\HE$ are the high energy cubes defined in Subsection~\ref{subsec:main thm}.
	\item $\HD(R),\ \LD(R),\ \Stop(R),\ \End(R),\ \tree(R),\ \ttt$ are defined in Subsection~\ref{sec3.3}.
	\item $\MDW,\ \HD_*(R),\ \bad(R),\ \sss_*(R)$ are defined in Subsection~\ref{secMDW}.
	\item $e_j(R),\ e(R),\ e'(R),\ e''(R),\ e^{(k)}(R)$ are various enlarged cubes, defined in Subsection~\ref{subsec:enlar}.
	\item $\sss(e_j(R))$ is defined in Subsection~\ref{subsec:enlar}.
	\item $\sss_*(e(R)),\ \sss_*(e'(R)),\ \HD_1(R),\ \HD_1(e(R)),\ \HD_1(e'(R)),\ \HD_2(e'(R)),$\\ $\sss_2(e'(R)),\ \TT_\sss(e'(R)),\ \Neg(e'(R)),\ \End(e'(R)),\ \TT(e'(R))$ are defined in Subsection~\ref{subsec:generalized}.
	\item $\Trc\subset\MDW$ is the family of tractable cubes (which are the roots of tractable trees). It is defined in Subsection~\ref{subsec:trc}.
	\item $\GH(R)$ is the family of good high density cubes, and it is defined in Lemma~\ref{lemalg1}.
	\item $\Gen(R)$ are the cubes generated by $R$, and $\Trc(R)= \Gen(R)\cap\Trc$. These families are constructed below Lemma~\ref{lemalg1}.
	\item $\sF_j,\ \sF_j^h,\ \sL_j,\ \sL_j^h,\ \sL$ are defined in Section~\ref{sec-layers}.
	\item $\Reg(e'(R))$ is the family of regularized stopping cubes, and $\TT_{\Reg}$ is the corresponding tree. They are defined in Subsection~\ref{sec6.2*}.
	\item $\eta$ is used to denote two different approximating measures; first, in Sections \ref{sec6}-\ref{sec99} it is a measure approximating $\mu$ at the level of $\Reg(e'(R))$. It is defined below Lemma~\ref{lem74}. Later, in Section \ref{sec9}, it is an AD-regular measure approximating $\mu$ at the level of $\wh\sss(R)$. It is defined above Lemma~\ref{lem9.2}.
	\item $\sH_j(e'(R))$ are auxiliary families defined in Subsection~\ref{subsec:H}.
	\item $\sH$ and $\sH'$ are defined in Subsection~\ref{subsec9.5}.
	\item $\nu$ is a smoother version of the approximating measure $\eta$, and it is defined in Subsection~\ref{subsec9.5}.
	\item $\TT_{\sH'}$ is defined above \eqref{eqdefPsi1}.
	\item $\Reg_{\Neg},\ \sD_\Neg,\ \sM_\Neg,\ \wt\End,\ \wt\TT,\ Z,\ \wt Z$ are defined in Subsection~\ref{sec7.1}.
	\item $\gamma$-nice trees are defined below Remark~\ref{rem9.7}.
	\item $\BR(R)$ is the family of cubes with big Riesz transform. It is introduced in Subsection~\ref{subsec:91}.
	\item $\wh\Ch(Q),\ \wh\Stop(R),\ \wh\tree_0(R)$ are defined at the beginning of Subsection~\ref{subsec:91}.
	\item $(i)_R,\ (ii)_R,\ (iii)_R$ are defined above Lemma~\ref{lem9.5*}.
	\item $\wh\End(R),\ \wh\tree(R),\ \wh\Stop_*(R),\ \wh\tree_*(R)$ are defined at the end of Subsection~\ref{subsec:91}.
	\item $\Reg_*(R)$ and $\treg_*(R)$ are constructed above Lemma~\ref{lem:reg prop}. 
	\item $\wh\ttt$ is defined above Lemma~\ref{lemtop8}.
\end{itemize}

\vvv

\end{document}

%% file: riesz-DX7-S1-1.tex
\section{Introduction}

Given a Radon measure $\mu$ in $\R^{n+1}$, its 
$n$-dimensional Riesz transform at $x\in\R^{n+1}$ is defined by
$$\RR\mu (x) = \int \frac{x-y}{|x-y|^{n+1}}\,d\mu(y),$$
whenever the integral makes sense. For $f\in L^1_{loc}(\mu)$,
we write $\RR_\mu f(x) = \RR(f\mu)(x)$.
Given $\ve>0$, the $\ve$-truncated Riesz transform of $\mu$ equals
$$\RR_\ve\mu (x) = \int_{|x-y|>\ve} \frac{x-y}{|x-y|^{n+1}}\,d\mu(y),$$
and we also write $\RR_{\mu,\ve}f(x) = \RR_\ve(f\mu)(x)$. We say that $\RR_\mu$ is bounded in $L^2(\mu)$ if
the operators $\RR_{\mu,\ve}$ are bounded on $L^2(\mu)$ uniformly in $\ve$, and then we denote
$$\|\RR_\mu\|_{L^2(\mu)\to L^2(\mu)} = \sup_{\ve>0} \|\RR_{\mu,\ve}\|_{L^2(\mu)\to L^2(\mu)}.$$
We also write
$$\RR_*\mu(x) = \sup_{\ve>0} |\RR_{\ve}\mu(x)|, \qquad  \pv\RR\mu(x) = \lim_{\ve>0} \RR_{\ve}\mu(x),$$
in case that the latter limit exists. Remark that, sometimes, abusing notation we will
write $\RR\mu$ instead of $\pv\RR\mu$.

This is the first of a series of two papers where we obtain a geometric characterization
of the measures $\mu$ in $\R^{n+1}$ such that 
the Riesz operator $\RR_\mu$ is bounded in $L^2(\mu)$. 
One of the main motivations for such characterization is the application to the Painlev\'e problem for Lipschitz harmonic functions (i.e., the geometric description of the removable singularities for
Lipschitz harmonic functions).  Also, there may be other
applications regarding 
questions on the approximation by Lipschitz or $C^1$-harmonic functions, as well as to
the study of harmonic measure, which is a field where the Riesz transform has played an important role in recent advances, such as \cite{AHM3TV} and \cite{AMTV}, for example.

To state our results in detail we need
to introduce additional notation.
For a ball $B$ with radius $r(B)$, we consider the density 
$$\theta_\mu(B)=\frac{\mu(B)}{r(B)^n}.$$
Also we define
$$\beta_{\mu,2}(B) = \left(\inf_L \frac1{r(B)^n}\int_{B} \left(\frac{\dist(y,L)}r\right)^2\,d\mu(y)
\right)^{1/2},$$
where the infimum is taken over all affine $n$-planes $L\subset\R^{n+1}$.
For $B=B(x,r)$ we may also write $\theta_\mu(x,r)$ and $\beta_{\mu,2}(x,r)$ instead of $\theta_\mu(B)$
and $\beta_{\mu,2}(B)$.

We consider the following Wolff type energy
\begin{equation}\label{eqEmu0}
\mathbb E(\mu) = \int\!\! \int_0^\infty \left(\frac{\mu(B(x,r))}{r^{n-\frac38}}\right)^2\,\frac{dr}r\,d\mu(x)
= \int\!\! \int_0^\infty r^{\frac3{4}}\,\theta_\mu(x,r)^2\,\frac{dr}{r}\,d\mu(x).
\end{equation}

In this paper we prove the following result.

\begin{theorem}\label{teomain}
	Let $\mu$ be a Radon measure in $\R^{n+1}$ satisfying the polynomial growth condition
	\begin{equation}\label{eqgrow00}
		\mu(B(x,r))\leq \theta_0\,r^n\quad \mbox{ for all $x\in\supp\mu$ and all $r>0$}.
	\end{equation}
	Suppose that there exists some
	constant $M$ such that
	\begin{equation}\label{eqwolff99}
		\mathbb E(\mu\rest_B) \leq M\,r(B)^{\frac3{4}}\,\theta_\mu(2B)^2\,\mu(2B)\quad \mbox{ for any ball $B\subset\R^{n+1}$}.
	\end{equation}
	Suppose also that $\RR_*\mu(x)<\infty$ $\mu$-a.e. Then
	\begin{equation}\label{eqteo*}
		\int \!\!\int_0^\infty \beta_{\mu,2}(x,r)^2\,\theta_\mu(x,r)\,
		\frac{dr}r\,d\mu(x) \leq C\,\big(\|\pv\RR\mu\|_{L^2(\mu)}^2 + \theta_0^2\,\|\mu\|\big),
	\end{equation}
	with $C$ depending on $M$.
\end{theorem}

Remark that, in this theorem, since $\mu$ has polynomial growth of degree $n$ and
$\RR_*\mu(x)<\infty$ $\mu$-a.e., then $\pv\RR\mu(x)$ exists
$\mu$-a.e.\ by \cite{NToV2}. 

A converse to the inequality \rf{eqteo*} also holds: if $\mu$ satisfies the growth condition
\rf{eqgrow00}, then
\begin{equation}\label{eqbetawolff'}
	\|\pv\RR\mu\|_{L^2(\mu)}^2\leq C\,\int\!\!\int_0^\infty \beta_{\mu,2}(x,r)^2\,\theta_\mu(x,r)\,\frac{dr}r\,d\mu(x) 
	+C\,\theta_0^2\,\|\mu\|,
\end{equation}
where $C$ is an absolute constant. This was proved in \cite{Azzam-Tolsa} in the case $n=1$, and in \cite{Girela} in full generality. 

From \rf{eqbetawolff'}, Theorem \ref{teomain}, and a direct application of the $T1$ theorem for non-doubling measures (\cite{NTrV1}, \cite{NTrV2}) we obtain the following.

\begin{theorem}\label{teomain2}
	Let $\mu$ be a Radon measure in $\R^{n+1}$ such that
	\begin{equation}\label{eqwolff99'}
		\mathbb E(\mu\rest_B) \leq M\,r(B)^{\frac3{4}}\,\theta_\mu(2B)^2\,\mu(2B)\quad \mbox{ for any ball $B\subset\R^{n+1}$}
	\end{equation}
	for some fixed constant $M$.
	Then $\RR_\mu$ is bounded in $L^2(\mu)$ if and
	only if it satisfies the polynomial growth condition
	\begin{equation}\label{eqgrow01}
		\mu(B(x,r))\leq C\,r^n\quad \mbox{ for all $x\in\supp\mu$ and all $r>0$}
	\end{equation}
	and
	\begin{equation}\label{eqbetawolff2}
		\int_B\int_0^{r(B)} \beta_{\mu,2}(x,r)^2\,\theta_\mu(x,r)\,\frac{dr}r\,d\mu(x)\leq C^2\,\mu(B)\quad\mbox{ for any ball
			$B\subset\R^{n+1}$.}
	\end{equation}
	Further, $\|\RR_\mu\|_{L^2(\mu)\to L^2(\mu)}$ is bounded above by some constant depending just on $C$ and $M$.
\end{theorem}

Some remarks are in order.
\begin{rem}\label{rem:exp}
	The same results (Theorems \ref{teomain} and \ref{teomain2}) are valid if one replaces the constants $3/8$ and $3/4$ in
	the definition of $\mathbb E(\mu)$ and in \rf{eqwolff99}, \rf{eqwolff99'} by $\alpha/2$ and  $\alpha$, with $\alpha\in(0,1)$. We have chosen $3/8$ and $3/4$ for simplicity.
	
On the other hand, in the case $\alpha=0$, Theorem \ref{teomain} and the ``if'' direction of
Theorem \ref{teomain2} hold trivially because
$$\int \!\!\int_0^\infty \beta_{\mu,2}(x,r)^2\,\theta_\mu(x,r)\,
		\frac{dr}r\,d\mu(x) \leq \int\!\! \int_0^\infty \left(\frac{\mu(B(x,r))}{r^{n}}\right)^2\,\frac{dr}r\,d\mu(x) =: \mathbb E_0(\mu),$$
so that the analog of \rf{eqwolff99}, namely
\begin{equation}\label{eqE000}
\mathbb E_0(\mu\rest_B) \leq M\,\theta_\mu(2B)^2\,\mu(2B)\quad \mbox{ for any ball $B\subset\R^{n+1}$},
\end{equation}
and the polynomial growth condition \rf{eqwolff99'} yield \rf{eqteo*}. The estimate \rf{eqbetawolff2} follows by the same argument. However, the condition \rf{eqE000} is much stronger than \rf{eqwolff99}.
In fact, this does not hold even in the case when $\mu$ is AD-regular (see also Remark \ref{rem333} below). 

One should think of $\mathbb E_0(\mu)$ as the critical Wolff energy in connection with the $L^2(\mu)$ boundedness of $\RR_\mu$, while the energy $\mathbb E(\mu)$ in \rf{eqEmu0} should be considered as a subcrtical energy. 
\end{rem}

\begin{rem}
	In a sense, the condition \rf{eqwolff99} ensures that the
	density of the balls $B(x,r)$ centered in $B$ does not grow too fast as the radius $r$ becomes
	smaller than $r(B)$. One can check easily that the condition
	\begin{equation}\label{eq:denscontr}
	\theta_{\mu}(x,r)\leq C\left(\frac{R}{r}\right)^{3/8}\theta_{\mu}(x,R)\quad\text{for $0<r<R$},
	\end{equation}
	implies \rf{eqwolff99} (with a suitable $M$). As noted in Remark~\ref{rem:exp}, the exponent $3/8$ could be replaced by any parameter strictly smaller than $1/2$, and the results of the paper would still hold for such measures. 
	
	We point out that for domains satisfying the so called capacity density condition (see \cite{AH}) the associated harmonic measure satisfies a similar condition, namely \eqref{eq:denscontr} with exponent $3/8$ replaced by $1-\delta$, for some small $\delta$. Since this exponent is larger than $1/2$, and increasing the exponent makes the condition \eqref{eq:denscontr} weaker, one cannot conclude that this class of measures is covered by our assumption \eqref{eqwolff99}.
\end{rem}

\begin{rem}\label{rem333}
	It is easy to check that the polynomial growth condition on $\mu$ implies that
	$$\mathbb E(\mu\rest_{B})\lesssim \theta_0^2\,r(B)^{3/4}\,\mu(2B).$$
	In particular, if $\mu$ is an AD-regular measure, i.e.\
	$$\mu(B(x,r))\approx r^n\quad \mbox{ for all $x\in\supp\mu$ and $0<r\leq\diam(\supp\mu)$},$$
	then \rf{eqwolff99} holds for a suitable $M$. Further, in this case the statement \rf{eqbetawolff2}
	is equivalent to saying that $\mu$ is uniformly rectifiable, by \cite{DS1}. So in the AD-regular case Theorem \ref{teomain2} reduces to
	the solution of the David-Semmes problem in \cite{NToV1}.
\end{rem}

In fact, we will prove a result more general than Theorem \ref{teomain} which does not
require the condition \rf{eqwolff99}, and instead asserts that, 
under the condition \rf{eqgrow00} and the assumption that $\RR_*\mu(x)<\infty$ $\mu$-a.e.,
\begin{equation}\label{eqhe53}
	\int \!\!\int_0^\infty \beta_{\mu,2}(x,r)^2\,\theta_\mu(x,r)\,
	\frac{dr}r\,d\mu(x) \leq C\,\big(\|\RR\mu\|_{L^2(\mu)}^2 + \theta_0^2\,\|\mu\|
	+ \sum_{Q\in\DD_\mu^\PP\cap\HE}\EE(4Q)\big),
\end{equation}
where $\DD_\mu^\PP\cap\HE$ is a family of ``$\PP$-doubling cubes'' $Q$ from a suitable lattice $\DD_\mu$
with a large Wolff type energy $\EE(4Q)$. See Theorem \ref{propomain} for more details.
In the companion paper \cite{Tolsa-riesz} it is shown that the last sum on right hand side of 
\rf{eqhe53} can be estimated in terms of $\|\pv\RR\mu\|_{L^2(\mu)}$. More precisely,
\begin{equation}\label{eqhe6d3}
	\sum_{Q\in\DD_\mu^\PP\cap\HE}\EE(4Q)\lesssim \|\pv\RR\mu\|_{L^2(\mu)}^2 + \theta_0^2\,\|\mu\|,
\end{equation}
so that combining the results of both papers, one gets:

\begin{theorem*}
	Let $\mu$ be a Radon measure in $\R^{n+1}$ satisfying the polynomial growth condition
	\begin{equation*}
		\mu(B(x,r))\leq \theta_0\,r^n\quad \mbox{ for all $x\in\supp\mu$ and all $r>0$}
	\end{equation*}
	and such that $\RR_*\mu(x)<\infty$ $\mu$-a.e.
	Then
	\begin{equation*}
		\int\!\!\int_0^\infty \beta_{\mu,2}(x,r)^2\,\theta_\mu(x,r)\,\frac{dr}r\,d\mu(x)\leq C\,(\big\|\pv\RR\mu\|_{L^2(\mu)}^2
		+\theta_0^2\,\|\mu\|\big),
	\end{equation*}
	where $C$ is an absolute constant.
\end{theorem*}

Combining also the estimate \rf{eqhe6d3} with Theorem \ref{teomain2}, it turns out that the assumption
\rf{eqwolff99'} can be eliminated in that theorem and then one gets a complete geometric characterization of the measures $\mu$ such that $\RR_\mu$ is bounded in $L^2(\mu)$.

\begin{theorem*}
	Let $\mu$ be a Radon measure in $\R^{n+1}$. Then $\RR_\mu$ is bounded in $L^2(\mu)$ if and
	only if it satisfies the polynomial growth condition
	$$
	\mu(B(x,r))\leq C\,r^n\quad \mbox{ for all $x\in\supp\mu$ and all $r>0$}
	$$
	and
	$$
	\int_B\int_0^{r(B)} \beta_{\mu,2}(x,r)^2\,\theta_\mu(x,r)\,\frac{dr}r\,d\mu(x)\leq C^2\,\mu(B)\quad\mbox{ for any ball
		$B\subset\R^{n+1}$.}
	$$
	Further, the optimal constant $C$ is comparable to $\|\RR_\mu\|_{L^2(\mu)\to L^2(\mu)}$.\end{theorem*}

The preceding result has some important applications. For example, it implies that the class 
of measures $\mu$ such that $\RR_\mu$ is bounded in $L^2(\mu)$ is invariant by bilipschitz maps. That is, given a bilipschitz $T:\R^n\to\R^n$, if $\RR_\mu$ is bounded in $L^2(\mu)$, then $\RR_{T\#\mu}$ is bounded in $L^2(T\#\mu)$, where 
$T\#\mu$ is the image measure of $\mu$ by $T$. As another corollary, one obtains a description of the removable singularities for Lipschitz harmonic functions in terms a metric-geometric potential
involving the $\beta_{\mu,2}$ coefficients, and one deduces that the class of sets which are removable
is invariant by bilipschitz mappings. 
These results can be considered the extension to higher dimensions of the results from 
\cite{Tolsa-bilip} in connection with analytic capacity.
See \cite{Tolsa-riesz} for more details.

Next we will describe the main ideas involved in the proof of Theorem \ref{teomain}, as well as the 
main difficulties and innovations.
The strategy consists in performing a corona decomposition of the dyadic lattice $\DD_\mu$ into trees of cubes where the density of the cubes does not oscillate too much. Then, roughly speaking,  in each tree the measure
$\mu$ behaves as if it were AD-regular, and from the $L^2(\mu)$ boundedness of $\RR_\mu$ and
\cite{NToV1}, one should expect that $\mu$ is close to some uniformly rectifiable measure at the locations and scales of the cubes in the tree, so that one can obtain a good packing condition
for the $\beta_{\mu,2}$ coefficients of the cubes in the tree.
For this strategy to work, we need to show that the roots of the trees where the density does
not oscillate too much satisfy a suitable packing condition. This is the content of the Main Lemma
\ref{mainlemma}, whose proof takes most of the paper (Sections \ref{sec4}-\ref{sec8}).

To prove the desired packing condition, we reduce matters to obtaining good lower estimates for 
the Haar coefficients of $\RR\mu$ for the cubes of some suitable trees
(the so-called tractable trees) where, in a sense the density of many cubes in some intermediate generations increases (with respect to the density of the root), and later in many stopping cubes the density decreases. These lower estimates are obtained by a variational argument applied to some measure $\eta$ that approximates $\mu$ at the locations and scales of the cubes in the tractable tree.
By that variational argument one obtains some lower bounds for $\|\RR\eta\|_{L^2(\eta)}$
that later are transferred to $\RR\mu$ (i.e., to the Haar coefficients of $\RR\mu$ 
for the cubes in the tree). The idea of applying a variational argument like this
originates from the work \cite{ENVo} by Eiderman, Nazarov, and Volberg and the reduction
to the tractable trees comes from the work \cite{Reguera-Tolsa} by Reguera and the second author of 
this paper. The article \cite{JNRT} includes an improved version of that variational argument.
Unlike the present paper, \cite{JNRT} makes an extensive use of compactness arguments, which
do not work so well in our situation, where the geometry plays a more important role.

The implementation of the variational argument and the transference of the estimates from the approximating measure $\eta$ to $\mu$ is more difficult in the present
paper than in \cite{Reguera-Tolsa} or in other related works such as \cite{JNRT}. Some of the
difficulties arise from the fact that, for technical reasons (essentially, we need that many cubes of 
the intermediate generations with high density are located far from the boundary of the root of the tree), we have to consider trees of ``enlarged cubes''. This causes an overlapping between different trees that
has to be quantified carefully (this is done in Section \ref{sec-layers}). 
On the other hand, the transference of the lower estimate for $\|\RR\eta\|_{L^2(\eta)}$ 
to the Haar coefficients of $\RR\mu$ for the cubes in the tree originates many error terms. Roughly speaking, in
order to be able to transfer that lower bound for $\|\RR\eta\|_{L^2(\eta)}$ to $\mu$ we need
the error terms to be smaller than the lower bound of $\|\RR\eta\|_{L^2(\eta)}$. Some of these
error terms are difficult to handle and we only can show that they are small under the condition \rf{eqwolff99}. If this condition is not assumed, then we can bound them in terms of the 
energies $\EE(4Q)$ that appear in \rf{eqhe6d3}. For this to work, we need 
an enhanced version of the dyadic lattice of David and Mattila that is obtained in Section
\ref{sec5}. This is an essential tool for our arguments.

As explained above, the last stage of the proof of Theorem \ref{teomain} consists of estimating the 
$\beta_{\mu,2}$ coefficients in each tree where the density does not oscillate too much.
This step, which requires a delicate approximation by an AD-regular measure which has its own interest,
is performed in Section~\ref{sec9}.

Throughout the proof a large number of parameters and families of cubes is defined. To help the reader navigate the paper and keep track of different objects, we list most of them in Appendices~\ref{app:param} and \ref{app:fam}.

\vv

In the whole paper we denote by $C$ or $c$ some constants that may depend on the dimension and perhaps other fixed parameters. Their value may change at different occurrences. On the contrary, constants with subscripts, like $C_0$, retain their values.
For $a,b\geq 0$, we write $a\lesssim b$ if there is $C>0$ such that $a\leq Cb$. We write $a\approx b$ to mean $a\lesssim b\lesssim a$. 

\vv


%% file: riesz-DX7-S2-2.tex
\section{The modified dyadic lattice of David and Mattila}\label{sec:DMlatt}\label{sec5}

Next we will introduce the dyadic lattice of cubes
with small boundaries of David-Mattila \cite{David-Mattila} associated with a Radon measure $\mu$. The properties of the lattice are summarized in the next lemma. Later on we will show how its construction
can be modified in order to obtain addtional properties relevant for our arguments.

\begin{lemma}[David, Mattila]
	\label{lemcubs}
	Let $\mu$ be a compactly supported Radon measure in $\R^{d}$.
	Consider two constants $C_0>1$ and $A_0>5000\,C_0$ and denote $E=\supp\mu$. 
	Then there exists a sequence of partitions of $E$ into
	Borel subsets $Q$, $Q\in \DD_{\mu,k}$, with the following properties:
	\begin{itemize}
		\item For each integer $k\geq0$, $E$ is the disjoint union of the ``cubes'' $Q$, $Q\in\DD_{\mu,k}$, and
		if $k<l$, $Q\in\DD_{\mu,l}$, and $R\in\DD_{\mu,k}$, then either $Q\cap R=\varnothing$ or else $Q\subset R$.
		\vv
		
		\item The general position of the cubes $Q$ can be described as follows. For each $k\geq0$ and each cube $Q\in\DD_{\mu,k}$, there is a ball $B(Q)=B(x_Q,r(Q))$ such that
		$$x_Q\in E, \qquad A_0^{-k}\leq r(Q)\leq C_0\,A_0^{-k},$$
		$$E\cap B(Q)\subset Q\subset E\cap 28\,B(Q)=E \cap B(x_Q,28r(Q)),$$
		and
		$$\mbox{the balls\, $5B(Q)$, $Q\in\DD_{\mu,k}$, are disjoint.}$$
		
		\vv
		\item The cubes $Q\in\DD_{\mu,k}$ have small boundaries. That is, for each $Q\in\DD_{\mu,k}$ and each
		integer $l\geq0$, set
		$$N_l^{ext}(Q)= \{x\in E\setminus Q:\,\dist(x,Q)< A_0^{-k-l}\},$$
		$$N_l^{int}(Q)= \{x\in Q:\,\dist(x,E\setminus Q)< A_0^{-k-l}\},$$
		and
		$$N_l(Q)= N_l^{ext}(Q) \cup N_l^{int}(Q).$$
		Then
		\begin{equation}\label{eqsmb2}
			\mu(N_l(Q))\leq (C^{-1}C_0^{-3d-1}A_0)^{-l}\,\mu(90B(Q)).
		\end{equation}
		\vv
		
		\item Denote by $\DD_{\mu,k}^{db}$ the family of cubes $Q\in\DD_{\mu,k}$ for which
		\begin{equation}\label{eqdob22}
			\mu(100B(Q))\leq C_0\,\mu(B(Q)).
		\end{equation}
		We have that $r(Q)=A_0^{-k}$ when $Q\in\DD_{\mu,k}\setminus \DD_{\mu,k}^{db}$
		and
		\begin{equation}\label{eqdob23}
			\mu(100B(Q))\leq C_0^{-l}\,\mu(100^{l+1}B(Q))\quad
			\mbox{for all $l\geq1$ with $100^l\leq C_0$ and $Q\in\DD_{\mu,k}\setminus \DD_{\mu,k}^{db}$.}
		\end{equation}
	\end{itemize}
\end{lemma}

\vv

\begin{rem}\label{rema00}
	The constants $C_0$ and $A_0$ are chosen so that
	$$A_0 = C_0^{C(d)},$$
	where $C(d)$ depends
	just on $d$ and $C_0$ is big enough.
\end{rem}

We use the notation $\DD_\mu=\bigcup_{k\geq0}\DD_{\mu,k}$. Observe that the families $\DD_{\mu,k}$ are only defined for $k\geq0$. So the diameter of the cubes from $\DD_\mu$ are uniformly
bounded from above.
We set
$\ell(Q)= 56\,C_0\,A_0^{-k}$ and we call it the side length of $Q$. Notice that 
$$C_0^{-1}\ell(Q)\leq \diam(28B(Q))\leq\ell(Q).$$
Observe that $r(Q)\approx\diam(Q)\approx\ell(Q)$.
Also we call $x_Q$ the center of $Q$, and the cube $Q'\in \DD_{\mu,k-1}$ such that $Q'\supset Q$ the parent of $Q$.
We denote the family of cubes from $\DD_{\mu,k+1}$ which are contained in $Q$ by $\Ch(Q)$, and we call their elements children or sons of $Q$.
We set
$B_Q=28 B(Q)=B(x_Q,28\,r(Q))$, so that 
$$E\cap \tfrac1{28}B_Q\subset Q\subset B_Q\subset B(x_Q,\ell(Q)/2).$$

For a given $\gamma\in(0,1)$, let $A_0$ be big enough so that the constant $C^{-1}C_0^{-3d-1}A_0$ in 
\rf{eqsmb2} satisfies 
$$C^{-1}C_0^{-3d-1}A_0>A_0^{\gamma}>10.$$
Then we deduce that, for all $0<\lambda\leq1$,
\begin{align}\label{eqfk490}
	\mu\bigl(\{x\in Q:\dist(x,E\setminus Q)\leq \lambda\,\ell(Q)\}\bigr) + 
	\mu\bigl(\bigl\{x\in 3.5B_Q\setminus Q:\dist&(x,Q)\leq \lambda\,\ell(Q)\}\bigr)\\
	&\leq_\gamma
	c\,\lambda^{\gamma}\,\mu(3.5B_Q).\nonumber
\end{align}

We denote
$\DD_\mu^{db}=\bigcup_{k\geq0}\DD_{\mu,k}^{db}$.
Note that, in particular, from \rf{eqdob22} it follows that
\begin{equation}\label{eqdob*}
	\mu(3B_{Q})\leq \mu(100B(Q))\leq C_0\,\mu(Q)\qquad\mbox{if $Q\in\DD_\mu^{db}.$}
\end{equation}
For this reason we will call the cubes from $\DD_\mu^{db}$ doubling. 
Given $Q\in\DD_\mu$, we denote by $\DD_\mu(Q)$
the family of cubes from $\DD_\mu$ which are contained in $Q$. Analogously,
we write $\DD_\mu^{db}(Q) = \DD^{db}_\mu\cap\DD(Q)$.

As shown in \cite[Lemma 5.28]{David-Mattila}, every cube $R\in\DD_\mu$ can be covered $\mu$-a.e.\
by a family of doubling cubes:
\vv

\begin{lemma}\label{lemcobdob}
	Let $R\in\DD_\mu$. Suppose that the constants $A_0$ and $C_0$ in Lemma \ref{lemcubs} are
	chosen as in Remark \ref{rema00}. Then there exists a family of
	doubling cubes $\{Q_i\}_{i\in I}\subset \DD_\mu^{db}$, with
	$Q_i\subset R$ for all $i$, such that their union covers $\mu$-almost all $R$.
\end{lemma}

The following result is proved in \cite[Lemma 5.31]{David-Mattila}.
\vv

\begin{lemma}\label{lemcad22}
	Let $R\in\DD_\mu$ and let $Q\subset R$ be a cube such that all the intermediate cubes $S$,
	$Q\subsetneq S\subsetneq R$ are non-doubling (i.e.\ belong to $\DD_\mu\setminus \DD_\mu^{db}$).
	Suppose that the constants $A_0$ and $C_0$ in Lemma \ref{lemcubs} are
	chosen as in Remark \ref{rema00}. 
	Then
	\begin{equation}\label{eqdk88}
		\mu(100B(Q))\leq A_0^{-10d(J(Q)-J(R)-1)}\mu(100B(R)).
	\end{equation}
\end{lemma}
\vv


From this lemma we deduce:

\vv
\begin{lemma}\label{lemcad23}
	Let $Q,R\in\DD_\mu$ be as in Lemma \ref{lemcad22}.
	Then
	$$\theta_\mu(100B(Q))\leq (C_0A_0)^{n+1}\,A_0^{-9d(J(Q)-J(R)-1)}\,\theta_\mu(100B(R))$$
	and
	$$\sum_{S\in\DD_\mu:Q\subset S\subset R}\theta_\mu(100B(S))\leq c\,\theta_\mu(100B(R)),$$
	with $c$ depending on $C_0$ and $A_0$.
\end{lemma}

For the easy proof, see
\cite[Lemma 4.4]{Tolsa-memo}, for example.

For $f\in L^2(\mu)$ and $Q\in\DD_\mu$ we define
\begin{equation}\label{eqdq1}
	\Delta_Q f=\sum_{S\in\Ch(Q)}m_{\mu,S}(f)\chi_S-m_{\mu,Q}(f)\chi_Q,
\end{equation}
where $m_{\mu,S}(f)$ stands for the average of $f$ on $S$ with respect to $\mu$.
Then we have the orthogonal expansion, for any cube $R\in\DD_\mu$,
$$\chi_{R} \bigl(f - m_{\mu,R}(f)\bigr) = \sum_{Q\in\DD_\mu(R)}\Delta_Q f,$$
in the $L^2(\mu)$-sense, so that
$$\|\chi_{R} \bigl(f - m_{\mu,R}(f)\|_{L^2(\mu)}^2 = \sum_{Q\in\DD_\mu(R)}\|\Delta_Q f\|_{L^2(\mu)}^2.$$

In this paper we will have to estimate terms such as $\|\RR(\chi_{Q}\mu)\|_{L^2(\mu\rest_{2B_Q\setminus Q})}$, 
which leads to deal with integrals of the form
$$\int_{2B_Q\setminus Q}\left(\int_Q \frac1{|x-y|^n}\,d\mu(y)\right)^2 d\mu(x).$$
Our next objective is to show that integrals such as this one can be estimated in terms of the Wolff type energy $\EE(2Q)$, to be defined soon.

We need some additional notation.\todo{the definition of $\lambda Q$ was above Lemma 2.8, but $2Q$ appears already in Lemma 2.6, so I moved the definition}
Given $Q\in\DD_\mu$ and $\lambda>1$, we denote by $\lambda Q$ the union of cubes $P$ from the same
generation as $Q$ such that $\dist(x_Q,P)\leq \lambda \,\ell(Q)$. Notice that
\begin{equation}\label{eqlambq12}
	\lambda Q\subset B(x_Q,(\lambda+\tfrac12)\ell(Q)).
\end{equation}
Also, we let
$$\DD_\mu(\lambda Q)=\{P\in\DD_\mu:P\subset \lambda Q,\,\ell(P)\leq \ell(Q)\},$$
and, for $k\geq0$,
$$\DD_{\mu,k}(\lambda Q) =\{P\in\DD_\mu:P\subset \lambda Q,\,\ell(P)=A_0^{-k} \ell(Q)\},\qquad
\DD_\mu^k(\lambda Q) = \bigcup_{j\geq k} \DD_{\mu,j}(\lambda Q).
$$
\vv


\begin{lemma}\label{lemDMimproved}
	Let $\mu$ be a compactly supported Radon measure in $\R^{d}$.
	Assume that $\mu$ has polynomial growth of degree $n$ and let $\gamma\in(0,1)$. The lattice $\DD_\mu$ from Lemma
	\ref{lemcubs} can be constructed so that the following holds for all
	all $Q\in\DD_{\mu}$:
	\begin{align*}
		\int_{2B_Q\setminus Q}\left(\int_Q \frac1{|x-y|^n}\,d\mu(y)\right)^2 d\mu(x) 
		+ &\int_{Q}\left(\int_{2B_Q\setminus Q} \frac1{|x-y|^n}\,d\mu(y)\right)^2 d\mu(x)\\
		&\leq C(\gamma)\sum_{P\in\DD_\mu: P\subset 2Q} \left(\frac{\ell(P)}{\ell(Q)}\right)^\gamma\theta_\mu(2B_P)^2\mu(P).
	\end{align*}
\end{lemma}

Remark that the polynomial growth assumption is just necessary to ensure that some of the integrals above are finite. In fact, the constant $C(\gamma)$ does not depend on the polynomial growth constant. 

To prove the lemma, we denote
\begin{equation}\label{eqdmuint}
	\wt \DD_\mu^{int}(Q) = \big\{P\in\DD_\mu(Q):2B_P\cap (\supp\mu\setminus Q)\neq \varnothing\big\}
\end{equation}
and
\begin{equation}\label{eqdmuext}
	\wt \DD_\mu^{ext}(Q) = \big\{P\in\DD_\mu:\ell(P)\leq \ell(Q),P\subset \R^{n+1}\setminus Q,\,2B_P\cap Q\neq \varnothing\big\}.
\end{equation}
Also,
\begin{equation}\label{eqdmutot}
	\wt \DD_\mu(Q) = \wt \DD_\mu^{int}(Q) \cup \wt \DD_\mu^{ext}(Q),
\end{equation}
and, for $k\geq0$,\todo{I removed "$P\subset \lambda Q$" from the definition of $\wt \DD_{\mu,k}$, I think this was an artefact}
$$\wt \DD_{\mu,k}(Q) = 
\{P\in\wt \DD_\mu:\ell(P)=  A_0^{-k}\ell(Q)\}.$$

We need some auxiliary results. The first one is the following.

\begin{lemma}\label{lemDMimproved2}
	Let $\mu$ be a compactly supported Radon measure in $\R^{d}$ and $Q\in\DD_\mu$. For any  $\alpha\in(0,1)$, we have
	\begin{align}\label{eqdosint}
		\int_{2B_Q\setminus Q}\left(\int_Q \frac1{|x-y|^n}\,d\mu(y)\right)^2d\mu(x) \,
		+ &\int_{Q}\left(\int_{2B_Q\setminus Q} \frac1{|x-y|^n}\,d\mu(y)\right)^2d\mu(x) \\
		& \lesssim_{\alpha,A_0}
		\sum_{P\in\wt \DD_\mu(Q)} \left(\frac{\ell(Q)}{\ell(P)}\right)^\alpha\theta_\mu(2B_P)^2
		\,\mu(P).\nonumber
	\end{align}
\end{lemma}

\begin{proof}
	Observe that, for $x\in 2B_Q\setminus Q$,
	\begin{align*}
		\int_Q\frac1{|x-y|^n}\,d\mu(y) & = \left(\int_{y\in Q:|x-y|\geq r(B_Q)/2} + \int_{y\in Q:|x-y|< r(B_Q)/2}\right)
		\frac1{|x-y|^n}\,d\mu(y)\\
		&\lesssim_{A_0} \frac{\mu(Q)}{r(B_Q)^n} + \sum_{P\in\wt \DD_\mu^{ext}(Q):x\in P} \theta_\mu(2B_P).
	\end{align*}
	Thus,
	\begin{multline*}
		\int_{2B_Q\setminus Q}\left(\int_Q \frac1{|x-y|^n}\,d\mu(y)\right)^2d\mu(x) \\
		\lesssim_{A_0} \left(\frac{\mu(Q)}{r(B_Q)^n}\right)^2\,\mu(2B_Q) +
		\int_{2B_Q\setminus Q}\bigg(\sum_{P\in\wt \DD_\mu^{ext}(Q):x\in P} \theta_\mu(2B_P)\bigg)^2d\mu(x).
	\end{multline*}
	By H\"older's inequality, for any $\alpha>0$,
	\begin{align*}
		\bigg(\sum_{P\in\wt \DD_\mu^{ext}(Q):x\in P} &\theta_\mu(2B_P)\bigg)^2\\
		& \leq \bigg(\sum_{P\in\wt \DD_\mu^{ext}(Q):x\in P} \left(\frac{\ell(Q)}{\ell(P)}\right)^\alpha\theta_\mu(2B_P)^2\bigg) \cdot \bigg(\sum_{P\in\wt \DD_\mu^{ext}(Q):x\in P} \left(\frac{\ell(P)}{\ell(Q)}\right)^{\alpha}\bigg).
	\end{align*}
	The last sum above is bounded above by
	$$\sum_{P\in\DD_\mu:x\in P,P\subset 2Q} \left(\frac{\ell(P)}{\ell(Q)}\right)^{\alpha}\leq C(\alpha).
	$$
	Therefore,
	\begin{align*}
		\int_{2B_Q\setminus Q}& \left(\int_Q \frac1{|x-y|^n}\,d\mu(y)\right)^2d\mu(x) \\
		& \lesssim_{\alpha,A_0} \frac{\mu(Q)^2\,\mu(2B_Q)}{r(B_Q)^{2n}}+
		\int_{2B_Q\setminus Q} \sum_{P\in\wt \DD_\mu^{ext}(Q):x\in P} \left(\frac{\ell(Q)}{\ell(P)}\right)^\alpha\theta_\mu(2B_P)^2\,d\mu(x)\\
		& \lesssim  \theta_\mu(2B_Q)^2\,\mu(Q)  +
		\sum_{P\in\wt \DD_\mu^{ext}(Q)} \left(\frac{\ell(Q)}{\ell(P)}\right)^\alpha\theta_\mu(2B_P)^2
		\,\mu(P).
	\end{align*}
	
	The estimate of the second integral on the left hand side of \rf{eqdosint} is analogous.
\end{proof}
\vv

\begin{lemma}\label{lemrecur5}
	Let $\mu$ be a compactly supported Radon measure in $\R^{d}$ and let $\gamma\in (0,1)$.
	Assume that $\mu$ has polynomial growth of degree $n$ and let $\gamma\in(0,1)$. The lattice $\DD_\mu$ from Lemma
	\ref{lemcubs} can be constructed so that the following holds for all
	all $Q\in\DD_{\mu}$:
	\begin{equation}\label{eqfir5}
		\sum_{S\in\wt \DD_{\mu,1}(Q)} \sum_{P\in\DD_\mu(2S)} \left(\frac{\ell(P)}{\ell(Q)}\right)^\gamma\theta_\mu(2B_P)^2\,\mu(P)\lesssim
		C_0^{6d+1}A_0^{-1}\!\!
		\sum_{P\in \DD_\mu^1(2Q)}\left(\frac{\ell(P)}{\ell(Q)}\right)^\gamma \theta_\mu(2B_P)^2\,\mu(P).
	\end{equation} 
\end{lemma}

\begin{proof}
	We will describe the relevant changes required in the arguments in \cite[Theorem 3.2]{David-Mattila} in order to get the estimate \rf{eqfir5}. We will
	use the same notation as in that theorem, with the exception of the constant $A$ in \cite[Theorem 3.2]{David-Mattila}, which here we denote by $A_0$. 
	
	Denote $E=\supp\mu$.
	For each generation $k\geq 0$, the starting point to construct $\DD_{\mu,k}$
	consists of choosing, for each $x\in E$, a suitable radius $r^k(x)$ such that
	\begin{equation}\label{eqrk1}
		A_0^{-k}\leq r^k(x)\leq C_0 A_0^{-k}
	\end{equation}
	depending on the doubling properties of the ball $B(x,r^k(x))$ (see \cite[(3.17)-(3.20)]{David-Mattila}).
	Next, one chooses two auxiliary radii $r_1^k(x)$ and $r_2^k(x)$ such that
	$$\frac{11}{10}\,r^k(x) < r_1^k(x)<\frac{12}{10}\,r^k(x),$$
	$$25\,r^k(x) < r_2^k(x)<26\,r^k(x),$$
	and such that the following small boundary conditions hold:
	\begin{equation}\label{eqthin1*}
		\mu\big(\big\{y\in\R^{d}\!: \dist(y,\partial B(x,r_1^k(x)))\leq \tau\,r^k(x)\big\}\big)
		\leq C\tau\,\mu\big(B(x,\tfrac{13}{10}r^k(x))\big)\quad
		\mbox{for $0<\tau < \tfrac1{10}$,}
	\end{equation}
	and
	\begin{equation}\label{eqthin2*}
		\mu\big(\big\{y\in\R^{d}\!: \dist(y,\partial B(x,r_2^k(x)))\leq \tau\,r^k(x)\big\}\big)
		\leq C\tau\,\mu\big(B(x,27r^k(x))\big)\quad
		\mbox{for $0<\tau < 1$.}
	\end{equation}
	
	At this point we will require the auxiliary radii $r_1^k(x)$ and $r_2^k(x)$ to be chosen so that an additional condition holds. Set $A(x,r,R) = B(x,R)\setminus B(x,r).$ 
Observe first that
\begin{align}\label{eqcalr45}
\int_{\tfrac{11}{10}r^k(x)}^{\tfrac{12}{10}r^k(x)}  &\sum_{j\geq0}  A_0^{-\gamma(j+1)} \int_{A(x,\,t-300C_0A_0^{-k-1},\,t+300C_0A_0^{-k-1})}
 \theta_\mu(y,112C_0A_0^{-k-j-1})^2\,d\mu(y)\,dt\\
& \leq\sum_{j\geq0}  A_0^{-\gamma(j+1)} \int_{B(x,\tfrac{12}{10}r^k(x)+300C_0A_0^{-k-1})}
 \theta_\mu(y,112C_0A_0^{-k-j-1})^2\,\LL^1(I_{x,y,k})\,d\mu(y),\nonumber
\end{align} 
where we applied Fubini and we denoted by $I_{x,y,k}$ the interval
\begin{align*}
I_{x,y,k} & = \{t\in\R:t-300C_0A_0^{-k-1}\leq |x-y|\leq t+300C_0A_0^{-k-1}\}\\
&= \big[|x-y|-300C_0A_0^{-k-1},|x-y|+300C_0A_0^{-k-1}\big].
\end{align*}
Obviously, its Lebesgue measure is $\LL^1(I_{x,y,k})= 600C_0A_0^{-k-1}$, and so the left hand side of \rf{eqcalr45}
is bounded above by
$$600 \,C_0\,A_0^{-k-1}\sum_{j\geq0} A_0^{-\gamma(j+1)} \!\! \int_{B(x,\tfrac{13}{10}r^k(x))}
 \theta_\mu(y,112C_0A_0^{-k-j-1})^2 \,d\mu(y).
$$  
	Thus, by Chebyshev, the set $U_{1}^k\subset\R$ of those $t\in [\tfrac{11}{10}r^k(x),\tfrac{12}{10}r^k(x)]$ such that
	\begin{align*}
		\sum_{j\geq0}  A_0^{-\gamma(j+1)} &\int_{A(x,t-300C_0A_0^{-k-1},t+300C_0A_0^{-k-1})}
		\theta_\mu(y,112C_0A_0^{-k-j-1})^2\,d\mu(y)\\
		&> \frac{10^5\,C_0\,A_0^{-k-1}}{r^k(x)}
		\sum_{j\geq0} A_0^{-\gamma(j+1)} 
		\int_{B(x,\tfrac{13}{10}r^k(x))}
		\theta_\mu(y,112C_0A_0^{-k-j-1})^2 \,d\mu(y)
	\end{align*}
	satisfies 
	$$|U_{1}^k|\leq \frac1{100}\,r^k(x).$$
	By a standard argument involving the  boundedness of the maximal Hardy-Littlewood operator 
	from $L^1(\R)$ to $L^{1,\infty}(\R)$, one can deduce that there exists some 
	$$r_1^k(x)\in [\tfrac{11}{10}r^k(x),\tfrac{12}{10}r^k(x)]\setminus
	U_{1}^k
	$$ 
	such that \rf{eqthin1*} holds. The fact that $r_1^k(x)\not\in\ U_{1}^k$
	ensures that
	\begin{align}\label{eqal848}
		\sum_{j\geq0}  A_0^{-\gamma(j+1)} &\int_{A(x,r_1^k(x)-300C_0A_0^{-k-1},r_1^k(x)+300C_0A_0^{-k-1})}
		\theta_\mu(y,112C_0A_0^{-k-j-1})^2\,d\mu(y)\\
		&\leq \frac{10^5\,C_0\,A_0^{-k-1}}{r^k(x)}
		\sum_{j\geq0} A_0^{-\gamma(j+1)} 
		\int_{B(x,\tfrac{13}{10}r^k(x))}
		\theta_\mu(y,112C_0A_0^{-k-j-1})^2 \,d\mu(y).
		\nonumber
	\end{align}
	An analogous argument shows that 
	$r^k_2(x)$ can be taken such that, besides \rf{eqthin2*}, the 
	preceding estimate also holds with $r_1^k(x)$ replaced
	by $r_2^k(x)$ and $B(x,\tfrac{13}{10}r^k(x))$ replaced by $B(x,27r^k(x))$.
	
	As in \cite[Theorem 3.2]{David-Mattila}, we denote $B_1^k(x) = B(x,r_1^k(x))$ and $B_2^k(x) = B(x,r_2^k(x))$, and by a Vitali type covering lemma we select a family of points $x\in I^k$
	such that the balls $\{B(x,5r^k(x))\}_{x\in I^k}$ are disjoint, while the balls $\{B(x,25r^k(x))\}_{x\in I^k}$ cover $E$. We also denote
	$$B_3^k(x) =B_2^k(x)\setminus \bigg(\bigcup_{y\in I^k\setminus\{x\}} B_1^k(y)\bigg).$$
	For $x\in I_k$, let $J(x)$ be the family of those
	$y\in I^k\setminus\{x\}$ such that $B_1^k(y)\cap B_2^k(x)\neq
	\varnothing$. As explained in \cite[Theorem 3.2]{David-Mattila}, using \rf{eqrk1} it is easy to check
	$\# J(x) \leq C C_0^d$.
	
	Next we consider an order in $I^k$ such that
	$$\mbox{$y<x$ in $I^k$ whenever $\mu(B(x,90r^k(x)))<\mu(B(y,90r^k(y)))$}$$
	and we define
	$$B_4^k(x) = B_3^k(x) \setminus \bigg(\bigcup_{y\in I^k:y<x} B_3^k(y)\bigg).$$
	Again, as explained in \cite[Theorem 3.2]{David-Mattila}, using \rf{eqrk1} it is easy to check
	that, for each $x\in I^k$, there are at most $C C_0^{n+1}$ sets $B_3^k(y)$ that intersect $B_3^k(x)$, with 
	$y\in I^k$.
	
	The family $\{B_4^k(x)\}_{x\in I^k}$ is a first approximation to $\{Q\}_{Q\in\DD_\mu^k}$.
	Indeed, by the arguments in \cite[Theorem 3.2]{David-Mattila}, for each $x\in I^k$ one constructs 
	a set $Q^k(x)\subset E$
	such that, denoting
	$\DD_{\mu,k} = \{Q^k(x)\}_{x\in I^k},$
	the properties stated in Lemma \ref{lemcubs} hold, with $r(Q^k(x))=r^k(x)$ and $B(Q^k(x))=B(x,r^k(x))$.
	In particular, 
	$$B(x,r^k(x))\cap E\subset Q^k(x) \subset B(x,28r^k(x)).$$
	Also, as shown in \cite[(3.61)]{David-Mattila}, it holds 
	\begin{equation}\label{eqdisb4}
		\dist(y,\partial B_4^k(x)) \leq 51C_0A_0^{-k-1}\quad \mbox{ for all $y\in N_1(Q(x))$.}
	\end{equation}
	
	For a cube $Q=Q^k(x)\in\DD_{\mu,k}$, we write $r(Q)=r^k(x)$, $B(Q)=B(x,r^k(x))$ and
	$B_i(Q) = B_i^k(x)$ for $i=1,\ldots,4$. 
	By an argument quite similar to the one used in \cite[Theorem 3.2]{David-Mattila} to prove \rf{eqdisb4}, we will show now that
	\begin{equation}\label{eqdisb4'}
		2S\subset \UU_{5A_0^{-1}\ell(Q)}(\partial B_4(Q))\quad \mbox{ for any $S\in\wt \DD_{\mu,1}(Q)$,}
	\end{equation}
	where $\UU_\ell(A)$ stands for the $\ell$-neighborhood of $A$.
	This will be needed below to prove 
	\rf{eqfir5}. The condition $S\in\wt \DD_{\mu,1}(Q)$ tells us that either $S\subset Q$ and 
	$2B_S\cap (E\setminus Q)\neq \varnothing$, or $S\subset E\setminus Q$ and 
	$2B_S\cap  Q\neq \varnothing$. Assume the first option (the arguments for the second one are analogous). So there exists some point $z\in E\setminus Q$ such that $|x_S-z|\leq 2r(B_S)=56r(S)$.
	Let $x\in I^k$ be such that $z\in Q^k(x)$. Then we have $\dist(z,B_4^k(x))\leq 50C_0A_0^{-k-1}$
	and also $\dist(x_S,B_4^k(x_Q))\leq 50C_0A_0^{-k-1}$, by 
	\cite[(3.50)]{David-Mattila} (see also the first paragraph after \cite[(3.61)]{David-Mattila}). 
	Since the sets $B_4^k(x_Q)$, $B_4^k(x)$ are disjoint, we deduce that
	$$\dist(x_S,\partial B_4^k(x))\leq 50C_0A_0^{-k-1} + 56\,r(S) \leq 106C_0A_0^{-k-1}< 2A_0^{-1}\ell(Q).
	$$
	Together with the fact that $2S\subset B\big(x_S,\tfrac52\ell(S)\big)$, this gives \rf{eqdisb4'}.
	
	Notice that, for each $j\geq0$,
	$$
	\sum_{P\in\DD_{\mu,j}(2S)} \left(\frac{\ell(P)}{\ell(Q)}\right)^\gamma\theta_\mu(2B_P)^2\,\mu(P)
	\lesssim C_0^{2n}A_0^{-\gamma(j+1)} \int_{2S} \theta_\mu(x,2A_0^{-j}\ell(S))^2\,d\mu(x).$$
	Then we obtain
	\begin{align}\label{eqhfk2}
		\sum_{S\in\wt \DD_{\mu,1}(Q)} \sum_{P\in\DD_\mu(2S)}& \left(\frac{\ell(P)}{\ell(Q)}\right)^\gamma\theta_\mu(2B_P)^2\,\mu(P) \\
		&\lesssim C_0^{2n}
		\sum_{S\in\wt \DD_{\mu,1}(Q)} \sum_{j\geq0} A_0^{-\gamma(j+1)} \int_{2S} \theta_\mu(x,2A_0^{-j}\ell(S))^2\,d\mu(x)\nonumber\\
		& \lesssim C_0^{2n+d}
		\sum_{j\geq0} A_0^{-\gamma(j+1)} \int_{\UU_{5A_0^{-1}\ell(Q)}(\partial B_4(Q))}
		\theta_\mu(x,2A_0^{-j-1}\ell(Q))^2\,d\mu(x).\nonumber
	\end{align}
	Denote by $\wt J(Q)$ the family of cubes $R\in\DD_{\mu,k}$ such that $B_2(R)\cap B_4(Q)\neq\varnothing$, so that, by the above construction we have
	$$\partial B_4(Q)\subset \bigcup_{R\in \wt J(Q)} (\partial B_1(R) 
	\cup \partial B_2(R)).$$
	Also, notice that $\#\wt J(Q)\leq C\,C_0^d$.
	From \rf{eqal848} we deduce that, for each $R\in\wt J(Q)$ and $i=1,2$,
	\begin{align}\label{eqrig56}
		\sum_{j\geq0} A_0^{-\gamma(j+1)} &\int_{\UU_{5A_0^{-1}\ell(R)}(\partial B_i(R))}
		\theta_\mu(x,2A_0^{-j-1}\ell(R))^2\,d\mu(x)\\
		& \leq C\,C_0A_0^{-1}
		\sum_{j\geq0} A_0^{-\gamma(j+1)} \int_{27B(R)}
		\theta_\mu(x,2A_0^{-j-1}\ell(R))^2\,d\mu(x)\nonumber \\
		& \leq C\,C_0A_0^{-1}
		\sum_{P\in \DD_\mu^{k+1}:P\cap 27B(R)\neq\varnothing}\left(\frac{\ell(P)}{\ell(Q)}\right)^\gamma\, \theta_\mu(x_P,3\ell(P))^2\,\mu(P),\nonumber
	\end{align}
	Notice that for $Q,R\in\DD_{\mu,k}$ as above, the condition $B_2(R)\cap B_4(Q)\neq\varnothing$ implies that $26B(Q) \cap 26B(R)\neq\varnothing$. Then, if
	$P\in \DD_\mu^{k+1}$ is such that $P\cap 27B(R)\neq\varnothing$, we derive
	$$\dist(x_Q,P)\leq |x_Q-x_R| + 27 r(R)
	\leq 26(r(Q) + r(R)) + 27 r(R) \leq 
	79C_0A_0^{-k} = \frac{79}{56}\,\ell(Q).$$
	Then, since $\ell(P)\leq A_0^{-1}\ell(Q)$, we infer that
	\begin{equation}\label{eqinc732}
		B(x_P,3\ell(P))\cap\supp\mu\subset 2Q.
	\end{equation}
	Also, we can write
	\begin{align*}
		\theta_\mu(x_P,3\ell(P))^2\,\mu(P) &\lesssim \frac{\mu(B(x_P,3\ell(P)))^3}{\ell(P)^{2n}}
		\lesssim \frac1{\ell(P)^{2n}}
		\bigg(\sum_{\substack{P'\in\DD_\mu:\ell(P')=\ell(P),\\
				P'\cap  B(x_P,3\ell(P))\neq\varnothing}}
		\mu(P')\bigg)^3\\
		& \lesssim C_0^{2d} \!\!\!\sum_{\substack{P'\in\DD_\mu:\ell(P')=\ell(P),\\
				P'\cap  B(x_P,3\ell(P))\neq\varnothing}}\!\frac{\mu(P')^3}{\ell(P')^{2n}}\lesssim
		C_0^{2d} \!\!\!\sum_{\substack{P'\in\DD_\mu:\ell(P')=\ell(P),\\
				P'\cap  B(x_P,3\ell(P))\neq\varnothing}}\!\!\theta_\mu(2B_{P'})^2\,\mu(P'),
	\end{align*}
	where we used the fact that the sums above are only over $CC_0^d$ terms at most.
	Together with \rf{eqinc732}, this implies that
	the right hand side of \rf{eqrig56} does not exceed
	\begin{multline*}
		C\,C_0^{2d+1}A_0^{-1} \sum_{P\in \DD_\mu^{k+1}:P\cap 27B(R)\neq\varnothing}\left(\frac{\ell(P)}{\ell(Q)}\right)^\gamma  \!\!\!\!\!\sum_{\substack{P'\in\DD_\mu(2Q):\ell(P')=\ell(P),\\
				P'\cap  B(x_P,3\ell(P))\neq\varnothing}}\!\!\!\theta_\mu(2B_{P'})^2\,\mu(P')
		\\ \leq 
		C\,C_0^{3d+1}A_0^{-1}
		\sum_{P'\in \DD_\mu^{1}(2Q)}\left(\frac{\ell(P')}{\ell(Q)}\right)^\gamma\, \theta_\mu(2B_{P'})^2\,\mu(P').
	\end{multline*}
	By this estimate and \rf{eqhfk2}, summing over all $R\in\wt J(Q)$, we get
	$$\sum_{S\in\wt \DD_{\mu,1}(Q)} \sum_{P\in\DD_\mu(2S)} \left(\frac{\ell(P)}{\ell(Q)}\right)^\gamma\theta_\mu(2B_P)^2\,\mu(P)\lesssim
	C_0^{6d+1}A_0^{-1}\!\!
	\sum_{P\in \DD_\mu^1(2Q)}\left(\frac{\ell(P)}{\ell(Q)}\right)^\gamma\, \theta_\mu(2B_P)^2\,\mu(P),
	$$
	as wished.
\end{proof}

\vv

By Lemma \ref{lemDMimproved2}, it is clear that to complete the proof of Lemma \ref{lemDMimproved} it suffices to show the following result.

\begin{lemma}\label{lemdmutot}
	Let $\mu$ be a compactly supported Radon measure in $\R^{d}$.
	Assume that $\mu$ has polynomial growth of degree $n$ and let $\gamma\in(0,1)$. The lattice $\DD_\mu$ from Lemma \ref{lemcubs} can be constructed so that the following holds for all
	all $Q\in\DD_{\mu}$:
	\begin{equation}\label{eqfhq29}
		\sum_{P\in\wt \DD_\mu(Q)} \left(\frac{\ell(Q)}{\ell(P)}\right)^{\frac{1-\gamma}2} \theta_\mu(2B_P)^2
		\,\mu(P) \lesssim_{A_0,\gamma}\sum_{P\in\DD_\mu: P\subset 2Q} \left(\frac{\ell(P)}{\ell(Q)}\right)^\gamma\theta_\mu(2B_P)^2\mu(P).
	\end{equation}
\end{lemma}

\begin{proof}
	To prove \rf{eqfhq29} notice that, for each $k\geq1$, by Lemma \ref{lemrecur5} we have
	\begin{align*}
		\sum_{R\in\wt \DD_{\mu,k}(Q)}
		\sum_{P\in \DD_{\mu}^1(2R)} &\left(\frac{\ell(P)}{\ell(R)}\right)^\gamma \theta_\mu(2B_P)^2 \,\mu(P)\\
		&\leq \sum_{R\in\wt \DD_{\mu,k-1}(Q)}
		\sum_{S\in\wt \DD_{\mu,1}(R)} 
		\sum_{P\in \DD_{\mu}(2S)}  \left(\frac{\ell(P)}{\ell(S)}\right)^\gamma\theta_\mu(2B_P)^2 \,\mu(P)\\
		& = A_0^{\gamma} \sum_{R\in\wt \DD_{\mu,k-1}(Q)}
		\sum_{S\in\wt \DD_{\mu,1}(R)} 
		\sum_{P\in \DD_{\mu}(2S)}  \left(\frac{\ell(P)}{\ell(R)}\right)^\gamma\theta_\mu(2B_P)^2 \,\mu(P)\\
		& \leq C C_0^{6d+1} A_0^{\gamma-1} \sum_{R\in\wt \DD_{\mu,k-1}(Q)}
		\sum_{P\in \DD_\mu^1(2R)}\left(\frac{\ell(P)}{\ell(R)}\right)^\gamma\, \theta_\mu(2B_P)^2\,\mu(P).
	\end{align*}
	Iterating the preceding estimate, we obtain
	\begin{align*}
		\sum_{P\in\wt \DD_{\mu,k+1}(Q)} 
		\theta_\mu(2B_P)^2 \,\mu(P) & \leq 
		\sum_{R\in\wt \DD_{\mu,k}(Q)}
		\sum_{P\in \DD_{\mu,1}(2R)} \theta_\mu(2B_P)^2 \,\mu(P)\\
		& \leq A_0^\gamma\sum_{R\in\wt \DD_{\mu,k}(Q)}
		\sum_{P\in \DD_{\mu}^1(2R)} \left(\frac{\ell(P)}{\ell(R)}\right)^\gamma \theta_\mu(2B_P)^2 \,\mu(P)\\
		& \leq A_0^\gamma\,(C C_0^{6d+1} A_0^{\gamma-1})^k \sum_{P\in \DD_\mu^1(2Q)}\left(\frac{\ell(P)}{\ell(Q)}\right)^\gamma\, \theta_\mu(2B_P)^2\,\mu(P).
	\end{align*}
	Therefore,
	\begin{align*}
		&\sum_{P\in\wt \DD_\mu(Q)} \left(\frac{\ell(Q)}{\ell(P)}\right)^{(1-\gamma)/2}\theta_\mu(2B_P)^2
		\,\mu(P)\\
		& \quad\leq \theta_\mu(2B_Q)^2\,\mu(Q) + \sum_{k\geq1} A_0^{(1-\gamma) k/2}\sum_{P\in\wt \DD_{\mu,k}(Q)} \theta_\mu(2B_P)^2
		\,\mu(P)\\
		& \quad\leq \theta_\mu(2B_Q)^2\,\mu(Q) + A_0^\gamma\sum_{k\geq1} A_0^{(1-\gamma) k/2}(C C_0^{6d+1} A_0^{\gamma-1})^{k-1}\!\!\!\!
		\sum_{P\in \DD_\mu^1(2Q)}\!\left(\frac{\ell(P)}{\ell(Q)}\right)^\gamma\, \theta_\mu(2B_P)^2\,\mu(P)\\
		& \quad\lesssim_{A_0,\gamma} \sum_{P\in \DD_\mu(2Q)}\left(\frac{\ell(P)}{\ell(Q)}\right)^\gamma\, \theta_\mu(2B_P)^2\,\mu(P),
	\end{align*}
	taking $A_0$ big enough for the last estimate. This yields \rf{eqfhq29}.
\end{proof}

\vv

%% file: riesz-DX7-S3-1.tex
\section{\texorpdfstring{$\PP$}{P}-doubling cubes, the Main Theorem, the corona decomposition, and the Main Lemma}\label{sec:Pdoubling}

In the rest of the paper we assume that $\mu$ is a compactly supported Radon measure in $\R^{n+1}$ with polynomial 
growth of degree $n$ and
that $\DD_\mu$ is a David-Mattila dyadic lattice satisfying the properties 
described in the preceding section, in particular, the ones in Lemmas \ref{lemcubs}
and \ref{lemDMimproved}, with $\gamma=9/10$.
By rescaling, we assume that $\DD_{\mu,k}$ is defined for all $k\geq k_0$, with $A_0^{-k_0}\approx
\diam(\supp\mu)$, and we also assume that there is a unique cube in $\DD_{\mu,k_0}$ which coincides
with the whole $\supp\mu$.
Further, from now on, we allow all the constants $C$ and all implicit constants in the relations
``$\lesssim$'', ``$\approx$", to depend on the parameters $C_0, A_0$ of the dyadic lattice of David-Mattila.

\subsection{\texorpdfstring{$\PP$}{P}-doubling cubes and the family \texorpdfstring{$\hd^k(Q)$}{hdk(Q)}}\label{subsec:Pdoubling}

We denote
$$\PP(Q) = \sum_{R\in\DD_\mu:R\supset Q} \frac{\ell(Q)}{\ell(R)^{n+1}} \,\mu(2B_R).$$
We say that a cube $Q$ is $\PP$-doubling if
$$\PP(Q) \leq C_d\,\frac{\mu(2B_Q)}{\ell(Q)^n},$$
for  $C_d =4A_0^n$. We write
$$Q\in\DD_\mu^\PP.$$
Notice that
$$\PP(Q) \approx_{C_0} \sum_{R\in\DD_\mu:R\supset Q} \frac{\ell(Q)}{\ell(R)} \,\theta_\mu(2B_R).$$
and thus saying that $Q$ is $\PP$-doubling implies that 
$$\sum_{R\in\DD_\mu:R\supset Q} \frac{\ell(Q)}{\ell(R)} \,\theta_\mu(2B_R)\leq C_d'\,\theta_\mu(2B_Q)$$
for some $C_d'$ depending on $C_d$. Conversely, the latter condition implies that $Q$ is $\PP$-doubling with another constant $C_d$ depending on $C_d'$.

From the properties of the David-Mattila lattice, we deduce the following.

\begin{lemma}\label{lempois00}
	Suppose that $C_0$ and $A_0$ are chosen suitably. If $Q$ is $\PP$-doubling, then $Q\in\DD_\mu^{db}$.
	Also, any cube $R\in\DD_\mu$ such that $R\cap 2Q\neq\varnothing$ and $\ell(R)=A_0\ell(Q)$ belongs to $\DD_\mu^{db}$.
\end{lemma}

\begin{proof}
	Let $Q\in\DD_\mu^\PP$. Regarding the fist statement of the lemma, if $Q\not\in\DD_\mu^{db}$, by \rf{eqdob23} we have
	$$\mu(2B_Q)\leq \mu(100B(Q))\leq C_0^{-l}\,\mu(100^{l+1}B(Q))\quad
	\mbox{for all $l\geq1$ with $100^l\leq C_0$}.$$
	In particular, if $Q'$ denotes the parent of $Q$,
	$$\mu(2B_Q)\leq C_0^{-l}\,\mu(2B_{Q'})\quad
	\mbox{for all $l\geq1$ with $100^l\leq C_0$}.$$
	So,
	\begin{equation}\label{eqigu8208}
		\mu(2B_Q)\leq C_0^{-c\log C_0}\,\mu(2B_{Q'})
	\end{equation}
	for some $c>0$. Using now that $Q$ is $\PP$-doubling, we get
	$$\mu(2B_Q)\leq C_0^{-c\log C_0}\,C_d\,\frac{\ell(Q')^{n+1}}{\ell(Q)^{n+1}} \mu(2B_Q)
	= 4C_0^{-c\log C_0}\,A_0^{2n+1} \mu(2B_Q)
	.$$
	Recall now that, as explained in Remark \ref{rema00}, we assume that $A_0=C_0^{C(n)}$, for some constant
	$C(n)$ depending just on $n$. Then, clearly, the preceding inequality fails if $C_0$ is big enough, which gives the desired contradiction.
	\vv
	
	To prove the second statement of the lemma, suppose that $Q\in\DD_{\mu,k}^\PP$ and let $R\in\DD_\mu$ be such that $R\cap 2Q\neq\varnothing$ and $\ell(R)=A_0\ell(Q)$.
	By the definition and the fact that $R\subset B_R$ we get
	$$|x_Q-x_R|\leq 3\ell(Q) + r(B_R).$$
	Since $$r(B_R)=28\,r(R) \geq 28A_0^{-k+1} = \frac12\,C_0^{-1}A_0\,\ell(Q)\geq 2500\,\ell(Q),$$
	we deduce that 
	$$2B_Q\subset 2B_R.$$
	If $R'$ denotes the parent of $R$ and $Q'''$ the great-grandparent of $Q$ (so that $\ell(Q''')=A_0^3\ell(Q)=A_0\ell(R')$), by an analogous argument we infer that
	$$2B_{R'}\subset 2B_{Q'''}.$$
	Then, using also that $Q$ is $\PP$-doubling, we obtain
	$$\mu(2B_{R'})\leq \mu(2B_{Q'''}) \leq C_d\,\left(\frac{\ell(Q''')}{\ell(Q)}\right)^{n+1}\!\mu(2B_Q)
	\leq C_d\,A_0^{3(n+1)}\,\mu(2B_R) = 4A_0^{4n+3}\,\mu(2B_R).$$
	If $R\not\in\DD_\mu^{db}$, arguing as in \rf{eqigu8208}, we infer that
	$$\mu(2B_R)\leq C_0^{-c\log C_0}\,\mu(2B_{R'}),$$
	which contradicts the previous statement if $C_0$ is big enough (recalling that $A_0=C_0^{C(n)}$).
\end{proof}

\vv
Notice that, by the preceding lemma, if $Q$ is $\PP$-doubling, then 
$$\sum_{R\in\DD_\mu:R\supset Q} \frac{\ell(Q)^{n+1}}{\ell(R)^{n+1}} \,\mu(2B_R) \lesssim_{C_d} \mu(Q).$$

For technical reasons that will be more evident below, it is appropriate to consider a discrete version of the density $\theta_\mu$. Given  $Q\in\DD_\mu$, we let
$$\Theta(Q) = A_0^{kn} \quad \mbox{ if\, $\dfrac{\mu(2B_Q)}{\ell(Q)^n}\in [A_0^{kn},A_0^{(k+1)n})$}.$$
Clearly, $\Theta(Q)\approx \theta_\mu(2B_Q)$.
Notice also that if $\Theta(Q) = A_0^k$ and $P$ is a son of $Q$, then
$$
\frac{\mu(2 B_P)}{\ell(P)^n} \leq \frac{\mu(2 B_Q)}{\ell(P)^n} = A_0^n\,\frac{\mu(2 B_Q)}{\ell(Q)^n}.
$$
Thus,
\begin{equation}\label{eqson1}
	\Theta(P)\leq A_0^n\,\Theta(Q)\quad \mbox{  for every son $P$ of $Q$.}
\end{equation}

Given $Q\in\DD_\mu$ and $k\geq1$, we denote by $\hd^k(Q)$ the family of maximal cubes $P\in\DD_\mu$ satisfying
\begin{equation}\label{a0tilde}
	\ell(P)<\ell(Q), \qquad \Theta(P)\geq  A_0^{kn}\Theta(Q).
\end{equation}

\vv
\begin{lemma}\label{lempdoubling}
	Let $Q\in\DD_\mu$ be $\PP$-doubling. Then, for $k\geq4$, every $P\in\hd^k(Q)\cap\DD_\mu(4Q)$ is also $\PP$-doubling
	and moreover $\Theta(P)=A_0^{kn}\Theta(Q)$.
\end{lemma}

Remark that this lemma implies that, under the assumptions in the lemma,
$$\Theta(P)\approx A_0^{kn}\,\Theta(Q)\quad\mbox{ for all $k\geq 1$.}$$

\begin{proof}
	First we show that $\Theta(P)=A_0^{kn}\Theta(Q)$. 
	The fact that $\Theta(P)\geq A_0^{kn}\,\Theta(Q)$ is clear. To see the converse inequality,
	denote by $\wh Q$ the parent of $Q$. Notice that any cube $S\subset 4Q$ with $\ell(S)=\ell(Q)$
	satisfies
	$$\frac{\mu(2B_S)}{\ell(S)^n}\leq \frac{\mu\big(2B_{\wh Q}\big)}{\ell(Q)^n} = A_0^{n}\frac{\mu(2B_{\wh Q})}{\ell(\wh Q)^n} \leq A_0^{n+1}\,\PP(Q) \leq C_d\,A_0^{n+1}\,\frac{\mu(2B_Q)}{\ell(Q)^n} < A_0^{3n}\,\frac{\mu(2B_Q)}{\ell(Q)^n}.$$
	Therefore,
	$$\Theta(S)\leq A_0^{3n}\,\Theta(Q).$$
	As a consequence, if $P\in\hd^k(Q)\cap\DD_\mu(4Q)$ with $k\geq4$, then its parent $\wh P$ satisfies
	$\Theta(\wh P)<A_0^{kn}\,\Theta(Q)$, which implies that $\Theta(P)\leq A_0^{kn}\,\Theta(Q)$.
	
	\vv
	To see that $P$ is $\PP$-doubling,
	we split
	\begin{equation}\label{eqvdx1}
		\PP(P) =\sum_{\substack{R\in\DD_\mu:R\supset P\\ \ell(R)\leq \ell(Q)}} \frac{\ell(P)}{\ell(R)^{n+1}} \,\mu(2B_R)+
		\sum_{\substack{R\in\DD_\mu:R\supset P\\ \ell(R)> \ell(Q)}}\frac{\ell(P)}{\ell(R)^{n+1}} \,\mu(2B_R).
	\end{equation}
	The cubes $R$ in the first sum on the right hand side satisfy
	$\Theta(R)\leq \Theta(P)$, by the definition of $\hd^k(Q)$. Thus,
	$$\sum_{\substack{R\in\DD_\mu:R\supset P\\ \ell(R)\leq \ell(Q)}} \frac{\ell(P)}{\ell(R)^{n+1}} \,\mu(2B_R)\leq 
	A_0^n 
	\sum_{\substack{R\in\DD_\mu:R\supset P\\ \ell(R)\leq \ell(Q)}}  \frac{\ell(P)}{\ell(R)} \,\Theta(R)\leq 
	2A_0^n\,\Theta(P)\leq 
	2A_0^n\,\frac{\mu(2B_P)}{\ell(P)^n}.$$
	Concerning the last sum in \rf{eqvdx1}, notice that the cubes $R$ in that sum satisfy $\ell(R)
	>\ell(Q)$.
	Using that $A_0\gg1$, it follows easily that $2B_R\subset 2B_{R'}$, where $R'$ is the cube containing $Q$ such that $\ell(R')=A_0\,\ell(R)$.
	Consequently, denoting by $\wh Q$ the parent of $Q$,
	\begin{align*}
		\sum_{\substack{R\in\DD_\mu:R\supset P\\ \ell(R)> \ell(Q)}}\frac{\ell(P)}{\ell(R)^{n+1}} \,\mu(2B_R)
		& \leq \sum_{R'\in\DD_\mu:R'\supset \wh Q} \frac{\ell(P)}{A_0^{-1}\ell(R')} \,\frac{\mu(2B_{R'})}{(A_0^{-1} \ell(R'))^n}\\
		& = A_0^{n+1} \frac{\ell(P)}{\ell(Q)} 
		\sum_{R'\in\DD_\mu:R'\supset \wh Q} \frac{\ell(Q)}{\ell(R')^{n+1}} \,\mu(2B_{R'})\\
		& \leq A_0^{n}\,\PP(Q) \leq A_0^{n}C_d\,\frac{\mu(2Q)}{\ell(Q)^n} \leq  A_0^{2n}C_d\,\Theta(Q)\\
		& 
		\leq \frac{A_0^{2n}C_d}{A_0^{4n}}\,\Theta(P)\leq \frac{C_d}{A_0^{2n}}\,\frac{\mu(2B_P)}{\ell(P)^n},
	\end{align*}
	where in the last two lines we took into account that $\ell(P)\leq A_0^{-1}\ell(Q)$ (because $P\in\hd^k(Q)$ for some $k\geq4$), that 
	$Q$ is $\PP$-doubling, and again that $P\in\hd^k(Q)$ for some $k\geq4$.
	
	From the estimates above, we infer that
	$$\PP(P) \leq \bigg(2A_0^n + \frac{C_d}{A_0^{2n}}\bigg) \,\frac{\mu(2B_P)}{\ell(P)^n} \leq C_d\,\frac{\mu(2B_P)}{\ell(P)^n},$$
	since $C_d= 4A_0^n$.
\end{proof}
\vv

\begin{lemma}\label{lemdobpp}
	Let $Q_0,Q_1,\ldots,Q_m$ be a family of cubes from $\DD_\mu$ such that $Q_j$ is a child of $Q_{j-1}$ for $1\leq j\leq 
	m$. Suppose that $Q_j$ is not $\PP$-doubling for $1\leq j\leq m$.
	Then
	\begin{equation}\label{eqcad35}
		\frac{\mu(2B_{Q_m})}{\ell(Q_m)^n}\leq A_0^{-m/2}\,\PP(Q_0).
	\end{equation}
	and
	\begin{equation}\label{eqcad35'}
		\PP(Q_m)\leq 2A_0^{-m/2}\,\PP(Q_0).
	\end{equation}
	
\end{lemma}

\begin{proof}
	Let us denote
	$\wt\Theta(R) = \frac{\mu(2B_R)}{\ell(R)^n},$ so that
	$$\PP(Q) = \sum_{R\in\DD_\mu:R\supset Q} \frac{\ell(Q)}{\ell(R)} \,\wt\Theta(R).$$
	For $1\leq j \leq m$,
	the fact that $Q_j$ is not $\PP$-doubling implies that
	\begin{equation}\label{eqsak33}
		\wt \Theta(Q_j) \leq \frac1{C_d}\,\PP(Q_j) = \frac1{C_d}\Biggl (\sum_{k=0}^{j-1} \frac{\ell(Q_j)}{\ell(Q_{j-k})}\,
		\wt\Theta(Q_{j-k})+ \frac{\ell(Q_j)}{\ell(Q_0)}\,\PP(Q_0)\Biggr).
	\end{equation}
	We will prove \rf{eqcad35} by induction on $j$. For $j=0$ this is in an immediate consequence of the
	definition of $\PP(Q_0)$. Suppose that \rf{eqcad35} holds for $0\leq h\leq j$, with $j\leq m-1$, and let us 
	consider the case $j+1$. From \rf{eqsak33} and the induction hypothesis we get
	\begin{align*}
		\wt\Theta(Q_{j+1}) & \leq  \frac1{C_d}\Biggl (\wt\Theta(Q_{j+1}) + \sum_{k=1}^j \frac{\ell(Q_{j+1})}{\ell(Q_{j+1-k})}\,
		\wt\Theta(Q_{j+1-k})+ \frac{\ell(Q_{j+1})}{\ell(Q_0)}\,\PP(Q_0)\Biggr)\\
		&= \frac1{C_d}\Biggl (\wt\Theta(Q_{j+1}) + \sum_{k=1}^j A_0^{-k}\,
		\wt\Theta(Q_{j+1-k})+ A_0^{-j-1}\,\PP(Q_0)\Biggr)\\
		&\leq \frac1{C_d}\Biggl (\wt\Theta(Q_{j+1}) + \sum_{k=1}^j A_0^{-k}\,A_0^{(-j-1+k)/2}\PP(Q_0)
		+ A_0^{-j-1}\,\PP(Q_0)\Biggr)
	\end{align*}
	Since 
	$$\sum_{k=1}^j A_0^{-k}\,A_0^{(-j-1+k)/2} = A_0^{(-j-1)/2}\sum_{k=1}^j A_0^{-k/2}\leq A_0^{-j/2},$$
	we obtain
	\begin{align*}
		\wt\Theta(Q_{j+1})  &\leq  \frac1{C_d}\bigl (\wt\Theta(Q_{j+1}) + A_0^{-j/2}\,\PP(Q_0)+ A_0^{-j-1}\,\PP(Q_0)\bigr)\\
		&\leq \frac1{C_d}\bigl (\wt\Theta(Q_{j+1}) + 2\,A_0^{-j/2}\,\PP(Q_0)\bigr) \\
	\end{align*}
	It is straightforward to check that this yields $\wt\Theta(Q_{j+1})\leq A_0^{-(j+1)/2}\,\PP(Q_0)$.
	
	The estimate \rf{eqcad35'} follows easily from \rf{eqcad35}:
	\begin{align*}
		\PP(Q_m) &= \sum_{k=0}^{m-1} \frac{\ell(Q_m)}{\ell(Q_{m-k})}\,
		\wt\Theta(Q_{m-k})+ \frac{\ell(Q_m)}{\ell(Q_0)}\,\PP(Q_0)\\
		&\leq \sum_{k=0}^{m-1} A_0^{-k}\,A_0^{-(m-k)/2}\,\PP(Q_0)+ A_0^{-m}\,\PP(Q_0)\\
		&\leq A^{-m/2}\sum_{k=0}^{m-1} A_0^{-k/2}\,\PP(Q_0)+ A_0^{-m}\,\PP(Q_0)\leq 2\,A_0^{-m/2}\,\PP(Q_0).
	\end{align*}
\end{proof}

\vv


\subsection{The Main Theorem}\label{subsec:main thm}

For a given $\lambda\geq1$ and $Q\in\DD_\mu$, we  consider the energy
$$\EE(\lambda Q) = \sum_{P\in\DD_\mu(\lambda Q)} \left(\frac{\ell(P)}{\ell(Q)}\right)^{3/4}\Theta(P)^2\,\mu(P),$$
Given a fixed constant $M_0\gg1$, we write $Q\in \HE$ (which stands for ``high energy'') if
$$\EE(4Q)\geq M_0^2\,\Theta(Q)^2\,\mu(Q).$$

\begin{theorem}[Main Theorem]\label{propomain}
	Let $\mu$ be a Radon measure in $\R^{n+1}$ such that
	$$
	\mu(B(x,r))\leq \theta_0\,r^n\quad \mbox{ for all $x\in\supp\mu$ and all $r>0$}.
	$$
	Then, for any choice of $M_0>1$,
	\begin{equation}\label{eqpropo*}
		\sum_{Q\in\DD_\mu} \beta_{\mu,2}(2B_Q)^2\,\Theta(Q)\,\mu(Q)\leq C\,\big(\|\RR\mu\|_{L^2(\mu)}^2 + \theta_0^2\,\|\mu\|
		+ \sum_{Q\in\DD_\mu^\PP\cap\HE}\EE(4Q)\big),
	\end{equation}
	with $C$ depending on $M_0$.
\end{theorem}

\vv
Notice that Theorem \ref{teomain} follows from the previous result. Indeed, the assumptions in Theorem \ref{teomain}
imply that there are no $\PP$-doubling cubes from $\HE$ if $M_0$ is big enough since,  for any $Q\in\DD_\mu^\PP$,
\begin{align*}
	\EE(4Q) &= \sum_{P\in\DD_\mu(4 Q)} \left(\frac{\ell(P)}{\ell(Q)}\right)^{3/4}\Theta(P)^2\,\mu(P)\\
	&
	\lesssim \int_{B(x_Q,20\ell(Q))} \int_0^\infty \left(\frac{\mu(B(x,r)\cap B(x_Q,20\ell(Q))}{r^n}\right)^2   \,\left(\frac r{\ell(Q)}\right)^{3/{4}}\,\frac{dr}r\,d\mu(x)\\
	& = \frac{\mathbb E(\mu\rest_{B(x_Q,20\ell(Q))})}{\ell(Q)^{3/4}} \lesssim M\,\theta_\mu(B(x_Q,40\ell(Q)))^2\,\mu(B(x_Q,40\ell(Q)))
	\approx M\,\Theta(Q)^2\,\mu(Q).
\end{align*}
So the last sum on the right hand side of \rf{eqpropo*} vanishes in this case.
Further, it is also easy to check that
$$\int \!\!\int_0^\infty \beta_{\mu,2}(x,r)^2\,\theta_\mu(B(x,r))\,
\frac{dr}r\,d\mu(x) \lesssim \sum_{Q\in\DD_\mu} \beta_{\mu,2}(2B_Q)^2\,\Theta(Q)\,\mu(Q).$$
Putting altogether, from \rf{eqpropo*} we get \rf{eqteo*}.

\vv

\subsection{The corona decomposition and the Main Lemma}\label{sec3.3}

In order to prove Theorem \ref{propomain} we have to use a suitable corona decomposition which splits
the lattice $\DD_\mu$ into appropriate trees. We need first to introduce some notation.

Given a big integer $k_\Lambda>10$ to be fixed below and $R\in\DD_\mu^\PP$, we denote 
$$\Lambda=A_0^{k_\Lambda\,n},\qquad \HD(R) = \hd^{k_\Lambda}(R).$$
Also, we consider a small constant $\delta_0 \in (0,\Lambda^{-C(n)})$ which will be chosen below too, with $C(n)>2$.
We let $\LD(R)$ be the family of cubes $Q\in\DD_{\mu}$
which
are maximal and satisfy
$$\ell(Q)< \ell(R)\quad \mbox{ and }\quad\PP(Q) \leq \delta_0\,\Theta(R).$$
We denote by $\sss(R)$ the family of maximal cubes from $\HD(R)\cup\LD(R)$ which are contained in $R$.
Also, we let $\End(R)$ be the family of maximal $\PP$-doubling cubes which are contained in some cube
from $\sss(R)$. Notice that, by Lemma \ref{lempdoubling}, the cubes from $\HD(R)\cap \DD_\mu(4R)$ are $\PP$-doubling, and thus
any cube from $\sss(R)\cap\HD(R)$ belongs to $\End(R)$. Finally, we let $\tree(R)$ denote the subfamily of the cubes from $\DD_\mu(R)$ which are not strictly contained in any cube
from $\End(R)$, and we say that $R$ is the root of the tree.

Next we define the family $\ttt$ inductively. 
We assume that $\supp\mu$ coincides with a cube $S_0$, and then
we set $\ttt_0=\{S_0\}$. Assuming $\ttt_k$ to be defined, we let
$$\ttt_{k+1} = \bigcup_{R\in\ttt_k} \End(R).$$
Then we let
$$\ttt =\bigcup_{k\geq0} \ttt_k.$$
Notice that we have 
$$\DD_\mu=\bigcup_{R\in\ttt} \tree(R).$$
Two trees $\tree(R)$, $\tree(R')$, with $R,R'\in\ttt$, $R\neq R'$ can only intersect if one
of the roots is and ending cube of the other, i.e., $R'\in\End(R)$ or $R\in\End(R')$.

\vv

The main step for the proof of Theorem \ref{propomain} consists of proving the following.

\begin{mlemma}\label{mainlemma}
	Let $\mu$ be a Radon measure in $\R^{n+1}$ such that
	$$
	\mu(B(x,r))\leq \theta_0\,r^n\quad \mbox{ for all $x\in\supp\mu$ and all $r>0$}.
	$$
	Then, for any choice of $M_0>1$,
	\begin{equation}\label{eqmainlemma*}
		\sum_{R\in\ttt} \Theta(R)^2\,\mu(R)\leq C\,\big(\|\RR\mu\|_{L^2(\mu)}^2 + \theta_0^2\,\|\mu\|
		+ \sum_{Q\in\DD_\mu^\PP\cap\HE}\EE(4Q)\big),
	\end{equation}
	with $C$ depending on $M_0$.
\end{mlemma}

The next Sections \ref{sec4}-\ref{sec8} are devoted to the proof of this lemma. Later, in Section \ref{sec9} we will complete the 
proof of Theorem \ref{propomain}.

\vv

%% file: riesz-DX7-S4-2.tex
\section{The cubes with moderate decrement of Wolff energy and the associated tractable trees} \label{sec4}

From now on we assume that we are under the assumptions of Main Lemma \ref{mainlemma}. Further, 
for a family $I\subset\DD_\mu$, we denote
$$\sigma(I) = \sum_{P\in I}\Theta(P)^2\,\mu(P).$$

\subsection{The family \texorpdfstring{$\MDW$}{MDW}}\label{secMDW}

For technical reasons, we need to define a kind of variant of the family $\HD(R)$. Given $R\in\DD_\mu^\PP$,
we denote
$$\Lambda_*= \Lambda^{1-\frac1N} = A_0^{k_\Lambda (1-\frac1N)n},\qquad
\HD_*(R) = \hd^{k_\Lambda (1-\frac1N)},$$
where $N$ is some big integer depending on $n$ that will be fixed below. Regarding the constant $\delta_0$ in the definition of $\LD(R)$ in Section \ref{sec3.3}, we will choose $\delta_0$ of the form
$$\delta_0 = \Lambda^{-N_0 - \frac1{2N}} = A_0^{-n\,k_\Lambda(N_0 + \frac1{2N})}$$
for some big integer $N_0$ depending on $n$, chosen later on.
Further we assume that $k_\Lambda$ is a multiple of $2N$, so that $k_\Lambda(N_0 + \frac1{2N})$ and $k_\Lambda (1-\frac1N)$ are integers.

We let $\bad(R)$ be the family of maximal cubes from $\LD(R)\cup\HD_*(R)$ (not necessarily  contained 
in $R$) and we denote
$$\sss_*(R)= \bad(R)\cap\DD_\mu(R).$$

We say that a cube $R\in\DD_{\mu}$ has moderate decrement of Wolff energy, and we write
$R\in\MDW$, if $R$ is $\PP$-doubling and
\begin{equation}\label{eq:MDWdef}
\sigma(\HD_*(R)\cap\sss_*(R))\geq B^{-1}\,\sigma(R),
\end{equation}
where
$$B= \Lambda_*^{\frac1{100n}}.$$

\vv
\begin{lemma}
For $R\in\ttt\setminus \MDW$, we have
\begin{equation}\label{eqmdw81}
\sigma(\End(R))\leq 2\Lambda^{2/N}B^{-1}\,\sigma(R).
\end{equation}
\end{lemma}

\begin{proof}
We have
$$\sigma(\End(R)) = \sigma(\HD(R)\cap\sss(R)) + \sum_{Q\in\LD(R)\cap\sss(R)}\sum_{P\in\End(R):P\subset Q}
\sigma(Q).$$
Clearly, since any cube from $\HD(R)$ is contained in some cube from $\HD_*(R)$, we infer that
\begin{align*}
\sigma(\HD(R)\cap\End(R)) & =\Lambda^2\,\Theta(R)^2\!\sum_{Q\in \HD(R)\cap\End(R)}\!\!\mu(Q)
\leq \Lambda^2\,\Theta(R)^2\!\sum_{Q\in \HD_*(R)\cap\sss_*(R)}\!\!\mu(Q) \\
& = \frac{\Lambda^2}{\Lambda_*^2}\,\sigma
(\HD_*(R)\cap\sss_*(R))\leq \Lambda^{2/N}B^{-1}\,\sigma(R).
\end{align*}
For $P\in\End(R)\setminus\HD(R)$, there exists some $Q$ such that $P\subset Q\in\LD(R)\cap\sss(R)$.  By \eqref{eqcad35} we have 
$$\Theta(P)\lesssim\PP(Q)\leq \delta_0\,\Theta(R),$$
and thus
$$
\sigma(\End(R)) \leq \Lambda^{2/N} B^{-1}\,\sigma(R) + C\delta_0^2\,\sigma(R) \leq 2\Lambda^{2/N} B^{-1}\,\sigma(R),
$$
since $\delta_0\ll B^{-1}$.
\end{proof}
\vv

We will take $N$ big enough so that 
$$2\Lambda^{2/N}\,B^{-1}\leq \frac12.$$

\begin{lemma}\label{lemtoptop}
We have
$$\sigma(\ttt)  \lesssim \sigma(\ttt\cap\MDW)+ \theta_0^2\,\|\mu\|.$$
\end{lemma}

\begin{proof}
For each $R\in \{S_0\}\cup (\ttt\cap\MDW)$, we denote
$I_0(R)=\{R\}$, and for $k\geq0$,
$$I_{k+1}(R) = \bigcup_{Q\in I_k(R)} \End(Q)\setminus \MDW.$$ 
In this way, we have
\begin{equation}\label{eqtttot}
\ttt = \bigcup_{R\in \{S_0\}\cup (\ttt\cap\MDW)} \;\bigcup_{k\geq0} I_k(R).
\end{equation}
Indeed, for each $Q\in \ttt$, let $R$ be the minimal cube from $\ttt\cap\MDW$ that contains $Q$, and in case this does not exists, let $R=S_0$. Then it follows that
$$Q\in \bigcup_{k\geq0} I_k(R).$$

Given $R\in\{S_0\}\cup (\ttt\cap\MDW)$, for each $k\geq 1$ and $Q\in I_k(R)$, 
since $I_k(R)\subset \ttt\setminus \MDW$, by \rf{eqmdw81} we have 
$$\sigma(\End(Q)\setminus \MDW)\leq 2 \Lambda^{2/N}B^{-1}\,\sigma(Q)\leq \frac12\,\sigma(Q).$$
Then we deduce that
$$\sigma(I_{k+1}(R)) = \sum_{Q\in I_k(R)} \sigma(\End(Q)\setminus \MDW) \leq \frac12 \sum_{Q\in I_k(R)}\sigma(Q)
= \frac12\, \sigma(I_k(R)).$$
So 
$$\sigma(I_{k}(R))\leq \frac1{2^k}\,\sigma(R)\quad \mbox{ for each $k\geq0$.}$$
Then, by \rf{eqtttot}, 
\begin{align*}
\sigma(\ttt) & = \sum_{R\in \{S_0\}\cup (\ttt\cap\MDW)} \;\sum_{k\geq0} \sigma(I_k(R))
\leq \sum_{R\in \{S_0\}\cup (\ttt\cap\MDW)}\;\sum_{k\geq0} \frac1{2^k}\,\sigma(R) \\
& \approx \sigma(S_0) + \sigma(\ttt\cap\MDW)
\lesssim \theta_0^2\,\|\mu\|+\sigma(\ttt\cap\MDW).
\end{align*}
\end{proof}

\vv


\subsection{The enlarged cubes}\label{subsec:enlar}

For $R\in\MDW$, the fact that the cubes from family $\HD_*(R)\cap\sss_*(R)$ may be located close to $\supp\mu\setminus R$ 
may cause problems when trying to obtain estimates involving the Riesz transform. For this reason we need to
introduce some ``enlarged cubes''.

Given $j\geq0$ and $R\in\DD_{\mu,k}$, 
we let 
$$e_j(R) = R \cup \bigcup Q,$$
where the last union runs over the cubes $Q\in\DD_{\mu,k+1}$ such that
\begin{equation}\label{eqxrq83}
\dist(x_R,Q)< \frac{\ell(R)}2 + 2j\ell(Q).
\end{equation}
We say that $e_j(R)$ is an enlarged cube.
Notice that, since $\diam(Q)\leq \ell(Q)$,
\begin{equation}\label{eqxrq84}
\supp\mu\cap B\big(x_R,\tfrac12\ell(R) + 2j\ell(Q)\big)\subset e_j(R) \subset B\big(x_R,\tfrac12\ell(R) + (2j+1)\ell(Q)\big).
\end{equation}
Also,  we have
\begin{equation}\label{eqqj8d}
e_j(R)\subset 2R\quad \mbox{ for $0\leq j\leq \frac34 A_0$,}
\end{equation}
since, for any $Q\in\DD_{\mu,k+1}$ satisfying \rf{eqxrq83}, its parent satisfies $\wh Q$
$$\dist(x_R,\wh Q)< \frac{\ell(R)}2 + 2j A_0^{-1}\ell(\wh Q) \leq 2\ell(R).$$

For $R\in\MDW$, we let
$$\sss_*(e_j(R)) = \bad(R) \cap\DD_\mu(e_j(R)),$$
where $\DD_\mu(e_j(R))$ stands for the family of cubes from $\DD_\mu$ which are contained in
$e_j(R)$ and have side length at most $\ell(R)$. Notice that we are not assuming $R\in\ttt$.

\vv
\begin{lemma}\label{lem:43}
For any $R\in\MDW$ there exists some $j$, with $10\leq j\leq A_0/4$ such that
\begin{equation}\label{eqsigmaj}
\sigma(\HD_*(R)\cap\sss_*(e_{j}(R))) \leq B^{1/4} \sigma(\HD_*(R)\cap\sss_*(e_{j-10}(R))),
\end{equation}
 assuming $A_0$ big enough, depending just on $n$. 
\end{lemma}

\begin{proof}
Given $R\in\MDW$, suppose that such $j$ does not exist. Let $j_0$ be the largest integer which is multiple of $10$ and smaller that $A_0/4$.
Then we get
\begin{align*}
\sigma(\HD_*(R)\cap\sss_*(e_{j_0}(R))) & \geq B^{\frac14}\sigma(\HD_*(R)\cap\sss_*(e_{j_0-10}(R)))\\
&\geq
\ldots \geq \big(B^{\frac14}\big)^{\frac{j_0}{10}-1}\sigma(\HD_*(R)\cap\sss_*(R)) \overset{\eqref{eq:MDWdef}}{\geq} B^{\frac{j_0}{40}-\frac54}\sigma(R).
\end{align*}
By \rf{eqqj8d}, we have $e_{j_0}(R)\subset 2R$ and thus
$$\sigma(\HD_*(R)\cap\sss_*(e_{j_0}(R)))= \sum_{Q\in\HD_*(R)\cap \sss_*(e_{j_0}(R))} \Lambda_*^2\Theta(R)^2\mu(Q)\leq 
\Lambda_*^2\Theta(R)^2\mu(2R).$$
Since $R$ is $\PP$-doubling (and in particular $R\in\DD_\mu^{db}$), denoting by $\wh R$ the parent of $R$,  we derive 
\begin{equation}\label{eqdoub*11}
\mu(2R)\leq \mu(2B_{\wh R}) \leq \frac{\ell(\wh R)^{n+1}}{\ell(R)}\,\PP(R) \leq C_d\,A_0^{n+1}
\mu(2B_R)\leq C_0\,C_d\,A_0^{n+1}
\mu(R).
\end{equation}
So we deduce that
$$B^{\frac{j_0}{40}-\frac54}\sigma(R)\leq C_0\,C_d\,A_0^{n+1}\,\Lambda_*^2\sigma(R),$$
or equivalently, recalling the choice of $B$ and $C_d$,
$$\Lambda_*^{\frac{1}{100n}\left(\frac{j_0}{40}-\frac54\right) -2} \leq 4 C_0\,A_0^{2n+1}.$$
Since $\Lambda_*\geq A_0^n$ and $j_0\approx A_0$, it is clear that this inequality is violated if $A_0$ is big enough, depending just on $n$.
\end{proof}
\vv

Given $R\in\MDW$,  let $j\geq 10$ be minimal such that \rf{eqsigmaj} holds. We denote
$h(R)=j-10$ and we write
$$e(R) = e_{h(R)}(R),\qquad e'(R) = e_{h(R)+1}(R), \quad e''(R)=e_{h(R)+2}(R), \quad e^{(k)}(R) = e_{h(R)+k}(R),$$
for $k\geq 1$. 
We let
\begin{align*}
B(e(R)) &= B\big(x_R,(\tfrac12 + 2A_0^{-1}h(R))\ell(R)\big),\\
B(e'(R)) & = B\big(x_R,(\tfrac12 + 2A_0^{-1}(h(R)+1))\ell(R)\big),\\
B(e''(R)) & = B\big(x_R,(\tfrac12 + 2A_0^{-1}(h(R)+2))\ell(R)\big),\\
B(e^{(k)}(R)) & = B\big(x_R,(\tfrac12 + 2A_0^{-1}(h(R)+k))\ell(R)\big).
\end{align*}
By construction (see \rf{eqxrq84}) we have 
$$B(e'(R))\cap\supp\mu\subset e'(R),$$
and analogously replacing $e'(R)$ by $e(R)$ or $e''(R)$.
Remark also that
$$e(R)\subset B(e'(R))\quad \text{ and }\quad \dist(e(R),\partial B(e'(R))) \geq A_0^{-1}\ell(R),$$
and, analogously,
$$e'(R)\subset B(e''(R))\quad \text{ and }\quad \dist(e'(R),\partial B(e''(R))) \geq A_0^{-1}\ell(R).$$

\vv

\begin{lemma}\label{lem-calcf}
For each $R\in\MDW$  we have
$$B(e''(R)) \subset (1+8A_0^{-1})\,\,B(e(R)) \subset B(e^{(6)}(R)),$$
and more generally, for $k\geq 2$ such that $h(R)+k-2\leq A_0/2$,
$$B(e^{(k)}(R)) \subset (1+8A_0^{-1})\,\,B(e^{(k-2)}(R)) \subset B(e^{(k+4)}(R)).$$
Also,
$$B(e^{(10)}(R))\subset B\big(x_R,\tfrac32 \ell(R)\big).$$
\end{lemma}

\begin{proof}
This follows from straightforward calculations. 
Indeed,
\begin{multline*}
r(B(e^{(k)}(R)))  = \frac{(\tfrac12 + 2A_0^{-1}(h(R)+k))\ell(R)}{(\tfrac12 + 2A_0^{-1}(h(R)+k-2))\ell(R)}
\,r(B(e^{(k-2)}(R)))\\
 = 1 + \frac{8A_0^{-1}}{1 + 4A_0^{-1}(h(R)+k-2)}\,r(B(e^{(k-2)}(R)))\leq (1+8A_0^{-1})\,r(B(e^{(k-2)}(R))).
\end{multline*}
Also, using that $h(R)+k-2\leq A_0/2$,
\begin{align*}
(1+8A_0^{-1})\,r(B(e^{(k-2)}(R))) & = (1+8A_0^{-1})\,\big(\tfrac12 + 2A_0^{-1}(h(R)+k-2)\big)\,\ell(R)\\
& \leq \big(\tfrac12 + 2A_0^{-1}(h(R)+k-2) + 4A_0^{-1} + 8A_0^{-1}
\big)\,\ell(R)\\
& =
r(B(e^{(k+4)}(R))).
\end{align*}

The last statement of the lemma follows from the fact that $h(R)+10\leq A_0/4<A_0/2$:
$$
B(e^{(10)}(R)) = B\big(x_R,(\tfrac12 + 2A_0^{-1}(h(R)+10))\ell(R)\big) \subset
B\big(x_R,(\tfrac12 + 2)\ell(R)\big) = B\big(x_R,\tfrac32 \ell(R)\big).$$
\end{proof}
\vv

\subsection{Generalized trees and negligible cubes}\label{subsec:generalized}

Next we need to define some families that can be considered as ``generalized trees''. First, we introduce some additional notation regarding the stopping cubes. For $R\in\DD_{\mu}^\PP$ we set
$$\HD_{1}(R) = \sss_*(R)\cap \HD_*(R).$$ 
Assume additionally that $R\in\MDW$. We write $\sss_*(e(R))=\sss_*(e_{h(R)}(R))$ and $\sss_*(e'(R))=\sss_*(e_{h(R)+1}(R))$. Furthermore,
$$\HD_1(e(R)) = \sss_*(e(R))\cap \HD_*(R),$$
and
$$\HD_{1}(e'(R)) = \sss_*(e'(R))\cap \HD_*(R).$$
We define $\HD_{1}(e^{(k)}(R))$ for $2\le k\le 10$ analogously. Also, we set 
$$\HD_2(e'(R)) = \bigcup_{Q\in \HD_1(e'(R))} (\sss_*(Q)\cap \HD_*(Q))$$
and
\begin{equation}\label{eqstop2}
\sss_2(e'(R)) = \big(\sss_*(e'(R)) \setminus \HD_1(e'(R))\big) \cup \bigcup_{Q\in \HD_1(e'(R))} \sss_*(Q).
\end{equation}
We let $\TT_\sss(e'(R))$ be the family of cubes made up of $R$ and all the cubes of the next generations which are contained in $e'(R)$ but are not 
strictly contained in any cube from $\sss_2(e'(R))$.

Observe that the defining property of $\MDW$ \eqref{eq:MDWdef} can now be rewritten as 
\begin{equation}\label{eq:MDWdef2}
	\sigma(R)\le B\, \sigma(\HD_1(R)).
\end{equation}
Moreover, by \eqref{eqsigmaj} and the definition of $e(R)$ we have
\begin{equation}\label{eq:sigmae'lesigmae}
\sigma(\HD_1(e^{(10)}(R)))\le B^{1/4}\sigma(\HD_1(e(R))).
\end{equation}

We define now the family of negligible cubes. We say that a cube $Q\in\TT_\sss(e'(R))$ is negligible for $\TT_\sss(e'(R))$, and we write $Q\in\Neg(e'(R))$ if
there does not exist any cube from $\TT_\sss(e'(R))$ that contains $Q$ and is $\PP$-doubling.  

\vv
\begin{lemma}\label{lemnegs}
Let $R\in\MDW$. If $Q\in\Neg(e'(R))$, then $Q\subset e'(R)\setminus R$, $Q$ is not contained in any cube from $\HD_1(e'(R))$, and
\begin{equation}\label{eqcostat}
\ell(Q) \gtrsim \delta_0^{2}\,\ell(R).
\end{equation}
\end{lemma}

\begin{proof}
Let $Q\in\Neg(e'(R))$. We have $Q\subset e'(R)\setminus R$ due to the fact that $R$ is $\PP$-doubling.
For the same reason, $Q$ is not contained in any cube from $\HD_1(e'(R))$.

To prove \rf{eqcostat}, assume that $\ell(Q)\leq A_0^{-2}\ell(R)$. Otherwise the inequality is immediate.
By Lemma \ref{lemdobpp}, since all the ancestors $Q_1,\ldots,Q_m$ of $Q$ that are contained in $e'(R)$ are not $\PP$-doubling, it follows that $Q_1$ (the parent of $Q$) satisfies
$$\PP(Q_1)\lesssim A_0^{-m/2-1}\,\PP(Q_m).$$
Because $Q_m\subset e'(R)\subset 2R$ and $\ell(Q_m)=A_0^{-1}\ell(R)$, it is easy to see that $\PP(Q_m)\lesssim \PP(R)\lesssim C_{d}\,\Theta(R)$, and so
\begin{equation*}
\PP(Q_1)\lesssim A_0^{-m/2}\,\Theta(R)\approx\left(\frac{\ell(Q)}{\ell(R)}\right)^{1/2}\,\Theta(R).
\end{equation*}
By the definition of $\LD(R)$, we know that $\PP(Q_1)\geq\delta_0\,\Theta(R)$, which together with 
the previous estimate yields \rf{eqcostat}.
\end{proof}
\vv

The cubes from $\sss_2(e'(R))$ need not be $\PP$-doubling, which is problematic for some of the estimates involving the Riesz transform localized around the
trees $\TT_\sss(e'(R))$ that will be required later. For this reason, we need to consider enlarged versions of them. For $R\in\MDW$, we let $\End(e'(R))$ be the family made up of the following cubes:
\begin{itemize}
\item the cubes from $\sss_2(e'(R))\cap \Neg(e'(R))$,
\item the cubes that are contained in any cube from $\sss_2(e'(R))\setminus \Neg(e'(R))$ which are $\PP$-doubling and, moreover, are maximal.
\end{itemize}
Notice that all the cubes from $\End(e'(R))$ are $\PP$-doubling, with the exception of the ones from $\Neg(e'(R))$.
We let $\TT(e'(R))$ be the family of cubes that are contained in $e'(R)$ and are not 
strictly contained in any cube from $\End(e'(R))$.

 \vv
 
\subsection{Tractable trees}\label{subsec:trc} 
Given $R\in\MDW$, we say that $\TT(e'(R))$ is tractable (or that $R$ is tractable) if 
$$\sigma(\HD_2(e'(R)))\leq B\,\sigma(\HD_1(e(R))).$$
In this case we write $R\in\Trc$.

 Our next objective consists in showing how we can associate a family of tractable trees to any $R\in\MDW\cap\ttt$, so that we can reduce the estimate of $\sigma(\ttt)$ to estimating the Haar coefficients of $\RR\mu$ from below on such family of tractable trees.
First we need the following lemma.

\begin{lemma}\label{lemalg1}
Let $R\in\MDW$ be such that $\TT(e'(R))$ is not tractable. Then there exists a family $\GH(R)\subset
\HD_1(e'(R))\cap\MDW$ satisfying:
\begin{itemize}
\item[(a)] The balls $B(e''(Q))$, with $Q\in\GH(R)$ are pairwise disjoint.
\item[(b)] For every $Q\in\GH(R)$, $\sigma(\HD_1(e(Q)))\geq \sigma(\HD_1(Q))\geq B^{1/2}\sigma(Q)$.
\item[(c)] $$B^{1/4} \sum_{Q\in\GH(R)} \sigma(\HD_1(e(Q))) \gtrsim \sigma(\HD_2(e'(R))).$$
\end{itemize}
\end{lemma}

The name ``$\GH$'' stands for ``good high (density)''. Remark that the property (c) and the fact that $R\not\in\Trc$ yield
\begin{equation}\label{eqiter582}
\sum_{Q\in\GH(R)} \sigma(\HD_1(e(Q))) \gtrsim B^{3/4}\,\sigma(\HD_1(e(R))),
\end{equation}
which is suitable for iteration.

\begin{proof}[Proof of Lemma \ref{lemalg1}]
Let $R\in\MDW$ be such that $\TT(e'(R))$ is not tractable. Notice first that
$$\sigma(\HD_1(e'(R)))\overset{\eqref{eq:sigmae'lesigmae}}{\leq} B^{1/4}
\sigma(\HD_1(e(R)))\leq B^{-3/4}\sigma(\HD_2(e'(R))) .$$
Let $I\subset \HD_1(e'(R))$ be the subfamily of the cubes $Q$ such that
$$\sigma(\HD_1(Q))< B^{1/2}\sigma(Q).$$
Then we have 
\begin{align*}
\sum_{Q\in I} \sigma(\HD_2(e'(R))\cap\DD_\mu(Q)) & \leq B^{1/2} \sum_{Q\in I} \sigma(Q)
\leq B^{1/2} \sigma(\HD_1(e'(R))) \\
&\leq \frac{B^{1/2}}{B^{3/4}}\,\sigma(\HD_2(e'(R)))\leq
 \frac12\,\sigma(\HD_2(e'(R))).
\end{align*}
Therefore,
\begin{align}\label{eqdj723}
\sum_{Q\in \HD_1(e'(R))\setminus I} \sigma(\HD_2(e'(R))\cap\DD_\mu(Q)) & = \sigma(\HD_2(e'(R))) -
\sum_{Q\in I} \sigma(\HD_2(e'(R))\cap\DD_\mu(Q))\\
& \geq \frac12\,\sigma(\HD_2(e'(R))).\nonumber
\end{align}

Next we will choose a family $J\subset \HD_1(e'(R))\setminus I$ satisfying
\begin{itemize}
\item[(i)] The balls $B(e''(Q))$, with $Q\in J$, are pairwise disjoint.
\item[(ii)] $$B^{1/4}\sum_{Q\in J}\sigma(\HD_1(e(Q)))
\gtrsim \sum_{Q\in \HD_1(e'(R))\setminus I} \sigma(\HD_2(e'(R))\cap\DD_\mu(Q)).$$
\end{itemize}
Then, choosing $\GH(R) = J$ we will be done. Indeed, the  property (a) in the statement of the lemma is the same as (i), and the property (b) is a consequence of the fact that
$J\subset I^c$ and the definition of $I$. This also implies that $\GH(R)\subset\MDW$. Finally, the property (c) follows from \rf{eqdj723} and (ii).

Let us see how $J$ can be constructed. By the covering Theorem 9.31 from \cite{Tolsa-llibre}, there
is a family $J_0\subset \HD_1(e'(R))\setminus I$ such that
\begin{itemize}
\item[1)] The balls $B(e''(Q))$, with $Q\in J_0$, have finite superposition, that is, 
$$\sum_{Q\in J_0}\chi_{B(e''(Q))}\leq C,$$
and
\item[2)] 
$$\bigcup_{Q\in \HD_1(e'(R))\setminus I} B(e''(Q)) \subset \bigcup_{Q\in J_0} (1+8A_0^{-1})\,B(e''(Q)),$$
\end{itemize}
Actually, in Theorem 9.31 from \cite{Tolsa-llibre} the result above is stated for a finite family of
balls. However, it is easy to check that the same arguments work as soon as the family $\HD_1(e'(R))\setminus I$ is countable and can be ordered so that $\HD_1(e'(R))\setminus I=\{Q_1,Q_2,\ldots\}$,
with $\ell(Q_1)\geq \ell(Q_2)\geq\ldots$. Further, one can check that the constant $C$ in 1)
does not exceed some absolute constant times $A_0^{n+1}$.

From the finite superposition property 1), by rather standard arguments which are analogous to the
ones in the proof of Besicovitch's covering theorem in \cite[Theorem 2.7]{Mattila-llibre}, say, 
one deduces that $J_0$ can be split into $m_0$ subfamilies $J_1,\ldots, J_{m_0}$ so that, for each $k$,  the balls $\{B(e''(Q)): Q\in J_k\}$  are pairwise disjoint, with $m_0\leq C(A_0)$.

Notice that the condition 2) and Lemma \ref{lem-calcf} applied to $Q$ ensure that
\begin{equation}\label{equni98-1}
\bigcup_{Q\in \HD_1(e'(R))\setminus I} Q\subset \bigcup_{Q\in \HD_1(e'(R))\setminus I} B(e''(Q) )\subset \bigcup_{Q\in J_0} (1+8A_0^{-1})\,B(e''(Q)) \subset \bigcup_{Q\in J_0} B(e^{(8)}(Q)).
\end{equation}
Next we choose $J:=J_k$ to be the family such that
$$\sum_{Q\in J_k}\sigma(\HD_1(e(Q)))$$
is maximal among $J_1,\ldots,J_{m_0}$, so that
\begin{align*}
\sum_{Q\in J}\sigma(\HD_1(e(Q))) & \geq \frac1{m_0}\,
\sum_{Q\in J_0}\sigma(\HD_1(e(Q)))\\
& \overset{\eqref{eq:sigmae'lesigmae}}{\geq} \frac{1}{m_0\,B^{1/4}} \sum_{Q\in J_0}\sigma(\HD_1(e^{(8)}(Q)))\\
& \overset{\rf{equni98-1}}{\ge} \frac{1}{m_0\,B^{1/4}} \sum_{Q\in \HD_1(e'(R))\setminus I} \sigma(\HD_1(Q))\\
& = \frac{1}{m_0\,B^{1/4}}\sum_{Q\in \HD_1(e'(R))\setminus I} \sigma(\HD_2(e'(R))\cap\DD_\mu(Q)).
\end{align*}
This proves (ii).
\end{proof}
\vv

Given $R\in\ttt\cap \MDW$, we will construct now a subfamily of cubes from $\MDW$ generated by $R$,
which we will denote $\Gen(R)$, by iterating the construction of Lemma \ref{lemalg1}.
The algorithm goes as follows.
Given $R\in\ttt\cap \MDW$, we denote 
$$\Gen_0(R) = \{R\}.$$
If $R\in\Trc$, we set $\Gen_1(R)=\varnothing$, and otherwise
$$\Gen_1(R) = \GH(R),$$
where $\GH(R)$ is defined in Lemma \ref{lemalg1}.
For $j\geq 2$, we set
$$\Gen_{j}(R) = \bigcup_{Q\in\Gen_{j-1}(R)\setminus \Trc} \GH(Q).$$
For $j\geq0$, we also set
$$\Trc_j(R) = \Gen_j(R)\cap\Trc,$$
and
$$\Gen(R) = \bigcup_{j\geq0}\Gen_j(R),\qquad\Trc(R) = \bigcup_{j\geq0}\Trc_j(R).$$
\vv

\begin{lemma}\label{eqtec74}
For $R\in\ttt\cap \MDW$, we have
\begin{equation}\label{eqtec741}
\bigcup_{Q\in\Trc(R)}Q\subset\bigcup_{Q\in\Gen(R)}Q \subset B(e''(R)).
\end{equation}
Also,
\begin{equation}\label{eqiter*44}
\sigma(\HD_1(e(R)))\leq \sum_{j\geq0} B^{-j/2}\sum_{Q\in\Trc_j(R)}\sigma(\HD_1(e(Q))).
\end{equation}
\end{lemma}

\begin{proof}

The first inclusion in \eqref{eqtec741} holds because $\Trc(R)\subset\Gen(R)$. So we only have to show the second inclusion.

By construction, for any $R'\in\MDW$, $\GH(R')\subset \HD_1(e'(R'))$, and thus any $Q\in \GH(R')$ is contained in $e'(R')$. This implies that 
$$|x_{R'}-x_Q|\leq r(B(e'(R'))) + \frac12\,\ell(Q)\leq  \Big(1+ 2A_0^{-1}+ \frac12\,A_0^{-1}\Big)\ell(R') \leq 1.1\,\ell(R').$$
Then, given $Q\in\Gen_j(R)$, $x\in Q$, and $0\leq k\leq j$, if we denote by $R_k$ the cube from $\Gen_k(R)$ such that $Q\in\Gen_{j-k}(R_k)$, we have
\begin{align*}
|x_R-x|& \leq |x_R- x_{R_1}|+ \sum_{k=1}^{j-1} |x_{R_k}- x_{R_{k+1}}| + |x_Q-x|\\
& \leq r(B(e'(R)))+ \frac12\,A_0^{-1}\,\ell(R) + \sum_{k=1}^{j-1}1.1\,A_0^{-k}\ell(R) + \frac12\,A_0^{-1}\,\ell(R)\\
& \leq r(B(e'(R)))+ 2\,A_0^{-1}\,\ell(R),
\end{align*}
which shows that $Q\subset B(e''(R))$.

To prove the second statement in the lemma, observe that, for $Q\in\Gen_{j-1}(R)\setminus \Trc$,
by \rf{eqiter582} applied to $Q$ we have
$$\sum_{P\in\GH(Q)} \sigma(\HD_1(e(P))) \geq c\,B^{3/4}\,\sigma(\HD_1(e(Q)))\geq 
B^{1/2}\,\sigma(\HD_1(e(Q))),
$$
assuming $\Lambda_*$, and thus $B$, big enough. Therefore,
\begin{align*}
\sum_{P\in\Gen_j(R)}\sigma(\HD_1(e(P))
& = 
\sum_{Q\in\Gen_{j-1}(R)\setminus \Trc} \,\sum_{P\in\GH(Q)}\sigma(\HD_1(e(P)))\\
& \geq 
B^{1/2}\sum_{Q\in\Gen_{j-1}(R)\setminus \Trc}\sigma(\HD_1(e(Q)))
\end{align*}
So,
$$\sum_{Q\in\Gen_{j-1}(R)}\sigma(\HD_1(e(Q)))\leq 
\sum_{Q\in\Trc_{j-1}(R)}\sigma(\HD_1(e(Q))) +
B^{-1/2}\sum_{P\in\Gen_j(R)}\sigma(\HD_1(e(P))).
$$
Iterating this estimate, and taking into account that, by the polynomial growth of $\mu$,
$\Gen_{j-1}(R)=\varnothing$ for some large $j$, we get \rf{eqiter*44}.
\end{proof}

\vv

%% file: riesz-DX7-S5-2.tex
\section{The layers \texorpdfstring{$\sF_j^h$ and $\sL_j^h$}{Fjh and Ljh}, and the tractable trees}\label{sec-layers}

We denote
$$\sF_j= \big\{R\in \ttt\cap \MDW:\Theta(R)=A_0^{nj}\big\},$$
so that 
$$\ttt\cap \MDW= \bigcup_{j\in\Z} \sF_j.$$
Next we split $\sF_j$ into layers $\sF_j^h$, $h\geq1$, which are defined as follows:
$\sF_j^1$ is the family of maximal cubes from $\sF_j$, and by induction
$\sF_j^h$ is the family of maximal cubes from $\sF_j\setminus \bigcup_{k=1}^{h-1} \sF_j^{h-1}$.
So we have the splitting
$$\ttt\cap \MDW= \bigcup_{j\in\Z}\,\bigcup_{h\geq1} \sF_j^h.$$

Our next objective is to choose a suitable subfamily $\sL_j^h\subset \sF_j^h$, for each $j,h$.
By the covering Theorem 9.31 from \cite{Tolsa-llibre}, there
is a family $J_0\subset \sF_j^h$ such that\footnote{Actually the property 1) is not stated in that theorem, however this can be obtained by preselecting a subfamily of maximal balls from $\sF_j^h$
with respect to inclusion and then applying the theorem to the maximal subfamily.}
\begin{itemize}
\item[1)] no ball $B(e^{(4)}(Q))$, with $Q\in J_0$, is contained in any other ball  
$B(e^{(4)}(Q'))$, with $Q'\in \sF_j^h$, $Q'\neq Q$,
\item[2)] the balls $B(e^{(4)}(Q))$, with $Q\in J_0$, have finite superposition, 
and
\item[3)] 
every ball $B(e^{(4)}(Q))$, with $Q\in\sF_j^h$, is contained in some ball 
$(1+8A_0^{-1})\,B(e^{(4)}(R))$, with $R\in J_0$. Consequently,
$$\bigcup_{Q\in\sF_j^h} B(e^{(4)}(Q)) \subset \bigcup_{R\in J_0} (1+8A_0^{-1})\,B(e^{(4)}(R)).$$
\end{itemize}

From the finite superposition property 2), as in the proof of Lemma \ref{lemalg1}, the family 
$J_0$ can be split into $m_0$ subfamilies $J_1,\ldots, J_{m_0}$ so that, for each $k$,  the balls $\{B(e^{(4)}(Q)): Q\in J_k\}$  are pairwise disjoint, with $m_0\leq C(A_0)$.
The condition 3) and Lemma \ref{lem-calcf} applied to $Q$ ensure that
\begin{equation}\label{equni98}
\bigcup_{Q\in\sF_j^h} Q\subset \bigcup_{Q\in\sF_j^h} B(e^{(4)}(Q) )\subset \bigcup_{Q\in J_0} (1+8A_0^{-1})\,B(e^{(4)}(Q)) \subset \bigcup_{Q\in J_0} B(e^{(10)}(Q)).
\end{equation}
Next we choose $\sL_j^h=J_k$ to be the family such that
$$\sum_{Q\in J_k}\sigma(\HD_1(e(Q)))$$
is maximal among $J_1,\ldots,J_{m_0}$, so that
\begin{align*}
\sum_{Q\in \sL_j^h}\sigma(\HD_1(e(Q))) & \geq \frac1{m_0}\,
\sum_{Q\in J_0}\sigma(\HD_1(e(Q)))\\
& \overset{\eqref{eq:sigmae'lesigmae}}{\geq} \frac{1}{m_0\,B^{1/4}} \sum_{Q\in J_0}\sigma(\HD_1(e^{(10)}(Q)))\\
& \overset{\rf{equni98}}{\geq} \frac{1}{m_0\,B^{1/4}} \sum_{Q\in \sF_j^h} \sigma(\HD_1(Q)).
\end{align*}

So we have:

\begin{lemma}\label{lemljh}
The family $\sL_j^h$ satisfies:
\begin{itemize}
\item[(i)] no ball $B(e^{(4)}(Q))$, with $Q\in \sL_j^h$, is contained in any other ball  
$B(e^{(4)}(Q'))$, with $Q'\in \sF_j^h$, $Q'\neq Q$,
\item[(ii)] the balls $B(e^{(4)}(Q))$, with $Q\in \sL_j^h$, are pairwise disjoint, 
and
\item[(iii)] 
 $$\sum_{Q\in \sF_j^h} \sigma(\HD_1(Q)) \lesssim B^{1/4} 
\sum_{Q\in \sL_j^h}\sigma(\HD_1(e(Q))).$$
\end{itemize}
\end{lemma}

\vv
We denote 
$$\sL_j= \bigcup_{h\geq 1}\sL_j^h,\qquad \sL= \bigcup_{j\in\Z}\sL_j =
\bigcup_{j\in\Z}\,\bigcup_{h\geq 1}\sL_j^h.$$
By the property (iii) in the lemma, we have
\begin{align}\label{eqover5}
\sum_{R\in \ttt\cap \MDW}\!\!\sigma(\HD_1(R)) & = \sum_{j\in\Z, \,h\geq0}\,
\sum_{R\in\sF_j^h} \sigma(\HD_1(R)) \\ 
& \lesssim B^{1/4} \sum_{j\in\Z, \,h\geq0}\,\sum_{R\in \sL_j^h}\sigma(\HD_1(e(R))) = B^{1/4}\sum_{R\in \sL}\sigma(\HD_1(e(R))).\nonumber
\end{align}

\vv

\begin{lemma}\label{lemsuper**9}
We have
$$\sigma(\ttt) \lesssim B^{5/4} \sum_{R\in \sL} \sum_{k\geq0} B^{-k/2}\sum_{Q\in\Trc_k(R)}\sigma(\HD_1(e(Q))) + \theta_0^2\,\|\mu\|.$$
\end{lemma}

\begin{proof}
This is an immediate consequence of Lemma \ref{lemtoptop} and our earlier estimates: \todo{the estimate $\sigma(R)\le B\sigma(\HD_1(R))$ was missing, so the constant $B^{1/4}$ had to be changed to $B^{5/4}$}
\begin{align*}
\sigma(\ttt)  & \lesssim \sigma(\ttt\cap\MDW)+ \theta_0^2\,\|\mu\|\\
& \overset{\eqref{eq:MDWdef2}}{\lesssim} B\sum_{R\in \ttt\cap\MDW}\sigma(\HD_1(R)) + \theta_0^2\,\|\mu\|\\
& \overset{\rf{eqover5}}{\lesssim} B^{5/4}\sum_{R\in \sL}\sigma(\HD_1(e(R))) + \theta_0^2\,\|\mu\|\\
& \overset{\eqref{eqiter*44}}{\lesssim} B^{5/4} \sum_{R\in \sL} \sum_{k\geq0} B^{-k/2}\sum_{Q\in\Trc_k(R)}\sigma(\HD_1(e(Q))) + \theta_0^2\,\|\mu\|.
\end{align*}
\end{proof}

\vv

To be able to apply later the preceding lemma, we need to get an estimate for $\#\sL(P,k)$, where $P\in\DD_\mu,\, k\ge 0$ and
\begin{equation*}
	\sL(P,k)= \big\{R\in\sL:\exists \,Q\in\Trc_k(R) \mbox{ such that } P\in\TT(e'(Q))\big\}.
\end{equation*}
For $j\in\Z$ set also 
$$\sL_j(P,k)= \big\{R\in\sL_j:\exists \,Q\in\Trc_k(R) \mbox{ such that } P\in\TT(e'(Q))\big\},$$
so that $\sL(P,k) = \bigcup_j \sL_j(P,k)$. The following important technical result is the main achievement in this section.

\begin{lemma}\label{lemimp9}
There exists some constant $C_1$ such that, for all $P\in\DD_\mu$ and all $k\geq0$,
$$\#\sL(P,k)\leq C_1\,\log\Lambda_*.$$
More precisely, for each $P\in\DD_\mu$ and $k\geq0$
\begin{equation}\label{eq:lemimp91}
	\#\{j\in\Z:\sL_j(P,k)\neq\varnothing\}\lesssim \log\Lambda_*,
\end{equation}
and for each $j\in\Z$, $P\in\DD_\mu$, $k\geq0$,
\begin{equation}\label{eqlj83}
	\#\sL_j(P,k) \leq C_2.
\end{equation}
\end{lemma}

\vv

We prove first \eqref{eq:lemimp91}.
\begin{proof}[Proof of \eqref{eq:lemimp91}]
	Let $\wt P_1$ be the smallest $\PP$-doubling cube containing $P$, and let $\wt P_2$ be be the smallest $\PP$-doubling cube strictly containing $\wt P_1$. Suppose that  $R\in\sL_j(P,k)$. There are two cases to consider.
	\vv
	
	\emph{Case 1.} There exists $Q\in\Trc_k(R)$ such that $P\in\TT(e'(Q))\setminus\Neg(e'(Q))$. We claim that in this case we have $\wt P_i\in \TT_\sss(e'(Q))$ for some $i\in\{1,2\}$. Indeed, if $P\in\TT_\sss(e'(Q))\setminus\Neg(e'(Q))$, then  necessarily $\wt P_1\in \TT_\sss(e'(Q))$, by the definition of
	the family $\Neg(e'(Q))$. If $P\notin\TT_\sss(e'(Q)),$ then either $P=\wt P_1\in\End(e'(Q))$, in which case $\wt P_2\in \TT_\sss(e'(Q))$, or $P\neq \wt P_1$ and we have $\wt P_1\in \TT_\sss(e'(Q))$, again by the definition of $\Neg(e'(Q))$.
	
	Choosing $i\in\{1,2\}$ such that $\wt P_i\in \TT_\sss(e'(Q))$ we see by the definition of $\TT_\sss(e'(Q))$ that
	$$\delta_0\,\Theta(Q)\lesssim \Theta(\wt P_i)\leq \Lambda_*^2\,\Theta(Q).$$
	Since $\Theta(Q)=\Lambda_*^k\Theta(R)$ (by the definition of $\Trc_k(R)$), the above is equivalent to
	\begin{equation*}
		\Lambda_*^{-2}\Theta(\wt P_i)\le \Lambda_*^k\Theta(R)\le C\delta_0^{-1}\Theta(\wt P_i)
	\end{equation*}
	We have $\Theta(R)=A_0^{nj}$ because $R\in \sL_j(P,k)$, and so it follows that 
	$$-C\log\Lambda_*\leq j + c\,k\log\Lambda_* - c'\log\Theta(\wt P_i)\leq C|\log\delta_0| = C'\log\Lambda_*.$$
	Recall that $k\ge 0$ is fixed, and $\Theta(\wt P_i)$ is equal to either $\Theta(\wt P_1)$ or $\Theta(\wt P_2)$, where both of these values depend only on $P$, which is fixed. Thus, there are at most $C''\log\Lambda_*$ integers $j$ such that there exists  $R\in\sL_j(P,k)$ and $Q\in\Trc_k(R)$ for which $P\in\TT(e'(Q))\setminus\Neg(e'(Q))$.
	
	\vv
	
	\emph{Case 2.}
	Suppose now that there exists $Q\in\Trc_k(R)$ such that $P\in\Neg(e'(Q))\subset\TT(e'(Q))$.
	In this case, by Lemma \ref{lemnegs}, $\ell(P) \gtrsim \delta_0^{-2}\,\ell(Q)$. Hence, there
	are at most $C\,|\log\delta_0|\approx \log\Lambda$ cubes $Q$ such that 
	$P\in\TT(e'(Q))\cap\Neg(e'(Q))$. 
	
	For each such cube we have $\Theta(Q)=\Lambda_*^k\Theta(R) = \Lambda_*^k A_0^{nj}$. Thus, for each cube $Q$ as above there is exactly one value of $j$ such that there may exist $R\in\sL_j$ with $Q\in\Trc_k(R)$. It follows that there are at most 
	$C'''\log\Lambda_*$ values of $j$ such that
	there exists $R\in\sL_j(P,k)$ and $Q\in\Trc_k(R)$ for which $P\in\TT(e'(Q))\cap\Neg(e'(Q))$.
	
	\vv
	Putting the estimates from both cases together we get that $\sL_j(P,k)$ is non-empty for at most $(C''+C''')\log\Lambda_*$ integers $j$.
\end{proof}

The proof of \eqref{eqlj83} is more involved. Its key ingredient is the following auxiliary result.

\begin{lemma}\label{lemtrucguai}
There exists some positive integer
$N_1$ depending on $n$ (with $N_1\leq C\,N_0 N$) such that the following holds.
For a given $\theta>0$, consider the interval
$$I_\theta = \big(\theta\,\Lambda^{-\frac1{4N}}\delta_0,\, \theta\,\Lambda^{\frac1{4N}}\Lambda_*\big).$$
Let $R_1,R_2,\ldots,R_{N_1}$ be cubes from $\ttt$ such that $R_{k+1}\in\End(R_{k})$ for $k\geq1$. 
Then at least one of the cubes $R_k$, with $1\leq k\le N_1$, satisfies
$$\Theta(R_k)\not \in I_\theta.$$
\end{lemma}

\begin{proof}
Recall that
$$\Lambda_*= \Lambda^{1-\frac1N}\quad\mbox{ and }\quad\delta_0 = \Lambda^{-N_0 - \frac1{2N}},$$
so that
$$I_\theta = \big(\theta\,\Lambda^{-N_0-\frac3{4N}},\, \theta\,\Lambda^{1-\frac3{4N}}\big).$$
Consider a sequence $R_1,R_2,\ldots,R_{N_1}$ of cubes from $\ttt$ such that $R_{k+1}\in\End(R_{k})$ for $k\geq1$. By the definition of $\End(R_k)$ there are only two possibilities: either $R_{k+1}\in\HD(R_k)$, or $R_{k+1}$ is a maximal $\PP$-doubling cube contained in some cube from $\LD(R_k)$. Note that in the former case we have $\Theta(R_{k+1})=\Lambda\Theta(R_k)$, and in the latter case we have $\Theta(R_{k+1})\le C\delta_0\Theta(R_k)$, by \eqref{eqcad35} and the definition of $\LD(R_k)$.

The key observation is the following: 
\begin{equation}\label{eqsist1}
\Theta(R_k)\leq \theta\,\Lambda^{\frac{-1}{3N}} \quad \Rightarrow \quad \mbox{either \;$\Theta(R_{k+1})\not\in I_\theta$ \;or\;
$\Theta(R_{k+1})=\Lambda\,\Theta(R_k)$.}
\end{equation}
This follows from the fact that, in the case $\Theta(R_{k+1})\in I_\theta$, we have 
$R_{k+1}\in\HD(R_k)$ because otherwise 
$$\Theta(R_{k+1})\leq C\delta_0\,\Theta(R_k) \leq C\,\theta\,\delta_0\,\Lambda^{\frac{-1}{3N}}
\leq \theta\,\delta_0\,\Lambda^{\frac{-1}{4N}}
.$$
Analogously, 
\begin{equation}\label{eqsist2}
\Theta(R_k)\geq \theta\,\Lambda^{\frac{-3}{4N}} \quad \Rightarrow \quad \mbox{either \;$\Theta(R_{k+1})\not\in I_\theta$ \;or\;
$\Theta(R_{k+1})\leq C\delta_0\Theta(R_k)$,}
\end{equation}
because, in the case $\Theta(R_{k+1})\in I_\theta$, we have 
$R_{k+1}\not\in\HD(R_k)$,
since otherwise
$$\Theta(R_{k+1}) = \Lambda\,\Theta(R_k) \geq \theta\,\Lambda^{1-\frac3{4N}}
.$$

To prove the lemma, suppose that $\Theta(R_1)\in I_\theta$. Otherwise we are done. 
By applying \rf{eqsist1} $N_0+1$ times, we deduce that either
$\Theta(R_{k})\not\in I_\theta$ for some $k\in(1,N_0+2]$, or there exists some $k_1\in[1,N_0+1]$ such that
$$\Theta(R_{k_1})\geq \theta\,\Lambda^{\frac{-1}{3N}}.$$
Then, from \rf{eqsist2}, we deduce that either 
$\Theta(R_{k_1 + 1})\not\in I_\theta$, or
\begin{equation}\label{eqk111}
\Theta(R_{k_1+1})\leq C\delta_0\,\Theta(R_{k_1})\leq C\,\delta_0\,\theta\,\Lambda^{1-\frac3{4N}} \leq \delta_0\,\theta\,\Lambda^{1-\frac1{3N}}.
\end{equation}

Now we have:

\begin{claim}
Let $k\geq1$ and $a\in (0,1)$. Suppose that
$$\Theta(R_k) \in (\delta_0\,\theta\,\Lambda^{-\frac1{3N}},\,\delta_0\,\theta\,\Lambda^{a}).$$
Then, either there exists some
$k_2\in [k+1,k+N_0+1]$ such that $\Theta(R_{k_2})\not\in I_\theta$, or 
$$\Theta(R_{k+N_0+1}) \in 
(\delta_0\,\theta\,\Lambda^{-\frac1{4N}},\,\delta_0\,\theta\,\Lambda^{a-\frac1{3N}}).$$
\end{claim}

In  case that $a\leq\frac1{12N}$, one should understand that the second alternative is not possible.

\begin{proof}
Suppose that the first alternative in the claim does not hold.
Then we deduce that
$$\Theta(R_{k+j}) = \Lambda^j\,\Theta(R_k)\quad \mbox{ for $j=1,\ldots,N_0$,}$$
because, for all $j=1,\ldots,N_0-1$,
$$\Theta(R_{k+j}) \leq \Lambda^j\,\Theta(R_k)
\leq \delta_0\,\theta\,\Lambda^{N_0-1+a} = \theta\,\Lambda^{\frac{-1}{2N}-1+a}\leq \theta\,\Lambda^{\frac{-1}{2N}},
$$
and then \rf{eqsist1} implies that $\Theta(R_{k+j+1})=\Lambda\,\Theta(R_{k+j})$.
So we infer that
$$\Theta(R_{k+N_0}) = \Lambda^{N_0}\,\Theta(R_k)\geq\delta_0\,\theta\,\Lambda^{-\frac1{4N} + N_0}
= \theta\,\Lambda^{-\frac3{4N}}.
$$
Then, by \rf{eqsist2} we have
\begin{align*}
\Theta(R_{k+N_0+1})& \leq C\delta_0\,\Theta(R_{k+N_0}) = C\delta_0\Lambda^{N_0}\,\Theta(R_k)\\
&
= C\Lambda^{\frac{-1}{2N}}\,\Theta(R_k)
\leq C\Lambda^{\frac{-1}{2N}}\,\delta_0\,\theta\,\Lambda^{a} \leq \delta_0\,\theta\,\Lambda^{a-\frac1{3N}}.
\end{align*}
\end{proof}
\vv

To complete the proof of the lemma observe that, by \rf{eqk111} and a repeated application of the preceding claim, we infer that
 there exists some $k\in [k_1+2,k_1+C N_0\,N]$ such that $\Theta(R_{k})\not\in I_\theta$, since after $CN$ iterations the second alternative in the lemma is not possible. This concludes the proof of the lemma.
\end{proof}


\vvv
We proceed with the proof of \rf{eqlj83} from Lemma \ref{lemimp9}, that is, the estimate $\#\sL_j(P,k) \leq C_2$.
\begin{proof}[Proof of \rf{eqlj83}]

For $h\ge 1$ set
\begin{equation*}
\sL_j^{h}(P,k):= \sL_j(P,k)\cap \sL_j^{h}.
\end{equation*}
Notice that each family $\sL_j^h(P,k)$ consists of a single cube, at most.
Indeed, we have
\begin{equation}\label{eq:ljh-inclusion}
R\in\sL_j(P,k)\quad\Rightarrow\quad P\subset B(e^{(3)}(R))
\end{equation}
because $P\in e'(Q)$ for some $Q\in\Trc_k(R)$ and $Q\subset B(e''(R))$ by Lemma \ref{eqtec74}. Thus, if $R,R'\in\sL_j^h(P,k)$,
then $B(e^{(3)}(R))\cap B(e^{(3)}(R'))\neq\varnothing$, which can only happen if $R=R'$ (by Lemma \ref{lemljh} (ii)).

Let $R_0$ be a cube in $\sL_j(P,k)$ with maximal side length, and let $h_0$ be such that $R_0\in \sL_j^{h_0}(P,k)$. We will show that $\sL_j^{h_1}(P,k)\neq\varnothing$ implies $h_0\le h_1\le h_0+C_2$. Together with the observation $\#\sL_j^h(P,k)\le 1$ this will conclude the proof of \rf{eqlj83}.

\begin{claim}
Let $R_1\in\sL_j(P,k)\setminus\{R_0\}$, and let $h_1$ be such that $R_1\in \sL_j^{h_1}(P,k)$. 
Then $h_1\geq h_0$.
\end{claim}

\begin{proof}
Suppose that $h_1< h_0$.
Let $R_0^{h_1}$ be the cube that contains $R_0$ and belongs to
$\sF_j^{h_1}$. Observe that
\begin{multline*}
P \overset{ \eqref{eq:ljh-inclusion}}{\subset}  B(e^{(3)}(R_0)) \cap B(e^{(3)}(R_1)) \subset B(x_{R_0},\tfrac32\ell(R_0)) \cap
B(x_{R_1},\tfrac32\ell(R_1))\\
 \subset B(x_{R_0^{h_1}},\tfrac12\ell(R^{h_1}_0) + \tfrac32\ell(R_0)) \cap
B(x_{R_1},\tfrac32\ell(R_1)).
\end{multline*}
So the two balls $B(x_{R_0^{h_1}},\tfrac12\ell(R^{h_1}_0) + \tfrac32\ell(R_0))$ and 
$B(x_{R_1},\tfrac32\ell(R_1))$ have non-empty intersection. Since $\ell(R_0^{h_1})\ge A_0\,\ell(R_0)\ge A_0\,\ell(R_1)$ (the last inequality follows by the choice of $R_0$), we deduce that
\begin{equation*}
B(e^{(4)}(R_1))\subset B(x_{R_1},\tfrac32\ell(R_1)) \subset B(x_{R_0^{h_1}},\tfrac12\ell(R^{h_1}_0) + 5\ell(R_0))
\subset B(e^{(4)}(R_0^{h_1})).
\end{equation*}
However, these inclusions contradict the property (i) of the family $\sL_j^{h_1}$ in Lemma 
\ref{lemljh} because $R_1\neq R_0^{h_1}$.
\end{proof}

\begin{claim}
Let $R_1\in\sL_j(P,k)\setminus\{R_0\}$, and let $h_1$ be such that $R_1\in \sL_j^{h_1}(P,k)$. 
Then 
\begin{equation}\label{eqclaimh1}
h_1\leq h_0+C.
\end{equation}
\end{claim}

\begin{proof}
Suppose that $h_1> h_0+1$. This implies that there are cubes
$\{R_1^h\}_{h_0+1\leq h \leq h_1-1}$, with $R_1^h\in\sF_j^h$ for each $h$, such that
$$R_1^{h_0+1}\supsetneq R_1^{h_0+2}\supsetneq\ldots \supsetneq R_1^{h_1-1}\supsetneq R_1^{h_1}=R_1.$$
Observe now that $\ell(R_1^{h_0+1})< \ell(R_0)$. Otherwise, there exists some cube
$R_1^{h_0}\in\sF_j^{h_0}$ that contains $R_1^{h_0+1}$ with 
$$\ell(R_1^{h_0})\geq A_0\,\ell(R_1^{h_0+1})\geq A_0\,\ell(R_0).$$
Since $P \subset  B(e^{(3)}(R_0)) \cap B(e^{(3)}(R_1))$, arguing as in the previous claim, we deduce that
$B(e^{(4)}(R_0))\subset B(e^{(4)}(R_1^{h_0}))$, which contradicts again the property (i) of the family $\sL_j^{h_0}$ in Lemma \ref{lemljh}, as above. So we have
$$\ell(R_1^h)\leq \ell(R_1^{h_0+1})< \ell(R_0)\quad \mbox{ for $h\geq h_0+1$.}$$

By the construction of $\Trc_k(R_0)$, there exists a sequence of cubes
$S_0=R_0, S_1, S_2, \ldots, S_k=Q$ such that 
$$ S_{i+1}\in \GH(S_i)\; \mbox{ for $i=0,\ldots,k-1$,}$$
and $P\in\TT(e'(S_{k}))$. In case that $P$ is contained in some $Q'\in\HD_1(e'(Q))=\HD_1(e'(S_k))$, we write $S_{k+1}=Q'$, and we let $\tilde k:=k+1$. Otherwise, we let $\tilde k:=k$. All in all, we have
\begin{equation}\label{eq:bla1}
	S_{i+1}\in\HD_1(e'(S_i))\quad\text{for $i=0,\dots,\tilde k-1$},
\end{equation}
and $S_{\tilde{k}+1}:=P\subset e'(S_{\tilde k})$ is not strictly contained in any cube from $\HD_1(e'(S_{\tilde k}))$.

Obviously we have  $\ell(S_{i+1})<\ell(S_i)$ for all $i$.
So, for each $h$ with $h_0+1\leq h\leq h_1$ there is some $i=i(h)$ such that 
\begin{equation}\label{eq0asd}
\ell(S_i)>\ell(R_1^h)\geq \ell(S_{i+1}),
\end{equation}
with $0\leq i \leq \tilde k$. 
We claim that either $i\lesssim 1$ or $i=\wt k$, with the implicit constant depending on $n$. 
Indeed, in the case $i<\wt k$, let $T\in\DD_\mu$ be such that $T\supset S_{i+1}$ and $\ell(T)=\ell(R_1^h)$.
Notice that, since $2R_1^h\cap 2T\neq \varnothing$ (because both $2R_1^h$ and $2T$ contain $P$) and 
$\ell(R_1^h)=\ell(T)$, we have
\begin{equation}\label{eq1asd}
\PP(T) \approx \PP(R_1^h)\approx \Theta(R_1^h)=\Theta(R_0),
\end{equation}
where in the last equality we used the definition of $\sL_j$.
On the other hand, 
\begin{equation}\label{eq2asd}
\PP(T)\geq \delta_0\,\Theta(S_i)
\end{equation} 
because otherwise 
$T$ is contained in some cube from $\LD(S_i)$, which would imply that $S_{i+1}$ does not belong to $\HD_1(e'(S_i))$.
Thus, from \rf{eq1asd} and \rf{eq2asd}
we derive that
$$\Theta(R_0)\gtrsim \delta_0\,\Theta(S_i) =  \delta_0\,\Lambda^i\,\Theta(R_0).$$
Hence  $\Lambda^i\lesssim\delta_0^{-1}$, which yields $i\lesssim_n 1$ if $i<\wt k$, as claimed.

The preceding discussion implies that, in order to prove \rf{eqclaimh1}, it suffices to show that, for each fixed 
$i=0,\ldots,\tilde k$, there are at most $C=C(n)$ cubes $R_1^h$ satisfying 
\rf{eq0asd} with this fixed $i$. 

\emph{Case $i<\tilde k$}. Assume first that $i<\tilde k$. Recall that $N_1$ is the constant given by Lemma \ref{lemtrucguai}, and suppose that there exist more than $N_1$ cubes $R_1^h$ satisfying 
\rf{eq0asd}. Since $\{R_1^h\}$ is a nested sequence of cubes, this is equivalent to saying that there exists some $s\in [h_0+1,\, h_1-N_1]$ such that
\begin{equation}\label{eqfam11}
\text{for $h\in[s,s+N_1]$ the cubes $R_1^h$ satisfy \rf{eq0asd}.}
\end{equation}
Taking $\theta=\Theta(S_{i})$, Lemma \ref{lemtrucguai} ensures
that there exists some cube $T\in\ttt$ such that $R_1^s\supset T\supset R_1^{s+N_1}$ which satisfies either
\begin{equation}\label{eqdisju9}
\Theta(T)\leq \Lambda^{-\frac1{4N}}\delta_0\,\Theta(S_i) \qquad\mbox{ or }\qquad 
\Theta(T)\geq \Lambda^{\frac1{4N}}\Lambda_*\,\Theta(S_i).
\end{equation}
Now, let $T'\in\DD_\mu$ be such that $S_{i+1}\subset T'\subset e'(S_{i})$ and $\ell(T')=\ell(T)$, where we use the fact that $\ell(S_i)>\ell(R_1^s)\ge \ell(T)\ge\ell(R_1^{s+N_1})\ge\ell(S_{i+1})$ and $S_{i+1}\subset e'(S_i)$. Notice that
\begin{equation}\label{eqdisju99}
\PP(T')\approx\PP(T)\approx \Theta(T),
\end{equation}
because $2T\cap 2T'\neq\varnothing$.

If the first option in \rf{eqdisju9} holds, we deduce that
$$\PP(T')\leq C \Theta(T)\leq C\Lambda^{-\frac1{4N}}\delta_0\,\Theta(S_i) < \delta_0\,\Theta(S_i).$$
This implies that 
  $T'$ is contained in some cube from $\LD(S_i)\cap\sss_*(e'(S_i))$, which ensures that 
  $S_{i+1}\not\in\HD_1(e'(S_i))$ (notice that we are using the fact that $i<\tilde k$), which is a contradiction with \eqref{eq:bla1}. Thus, $T$ must satisfy the second estimate of \eqref{eqdisju9}. But in this case \rf{eqdisju99} yields $\PP(T')> \Lambda_*\Theta(S_i)$, and so $T'$ is strictly contained in some cube from $\HD_1(e'(S_i))$. Hence, $S_{i+1}\not\in\HD_1(e'(S_i))$, which again gives a contradiction.
  In consequence, if $i<\tilde k$, then \rf{eqfam11} does not hold for any $s$.
 \vv

\emph{Case $i=\tilde k$}. Assume again that
\rf{eqfam11} holds for some $s\in [h_0+1,\, h_1-N_1]$, and let $s$ be the smallest possible such that \rf{eqfam11} holds. The same argument as above shows that the cube $T'$ from the preceding paragraph is contained in some cube $T''\in\sss_*(e'(S_i))$. Since we assumed that $s$ is minimal, $R_1^s\supset T\supset R_1^{s+N_1}$, and $\ell(T'')\ge\ell(T')=\ell(T)$, we get that there are at most $N_1$ cubes $R_1^h$ satisfying \rf{eq0asd} such that 
$\ell(S_i)>\ell(R_1^h)\geq \ell(T'')$.
We claim that there is also a bounded number of cubes $R_1^h$ such that 
\begin{equation}\label{eqlastclaim5}
\ell(T'')\geq \ell(R_1^h)\geq \ell(P).
\end{equation}
Indeed, by the definition of the family $\End(e'(R))$ and Lemma
 \ref{lemdobpp}, if we denote by $T_m$ the $m$-th descendant of $T''$ which contains $P$, for $m'\geq m\geq 0$,
 it follows that 
$$\PP(T_{m'})\leq A_0^{-|m-m'|/2}\,\PP(T_m)\leq  A_0^{-m/2}\,\PP(T'').$$
Suppose then that there are two cubes $R_1^h$, $R_1^{h'}$ such that  $\ell(R_1^{h'})\leq \ell(R_1^h)\leq \ell(T'')$.
Let $T_m$ and $T_{m'}$ be such that $\ell(R_1^h)=\ell(T_m)$ and $\ell(R_1^{h'})= \ell(T_{m'})$.
By arguments analogous to the ones in \rf{eqdisju99}, we derive that 
$$\PP(T_m)\approx \PP(R_1^h) \approx \Theta(R_0)\quad \text{ and }\quad
\PP(T_{m'})\approx \PP(R_1^{h'}) \approx \Theta(R_0),$$
where we also used the fact that $\Theta(R_1^{h})=\Theta(R_1^{h'})=\Theta(R_0)$.
On the other hand, since $|m-m'|\geq |h-h'|$, we have
$$\Theta(R_0)\approx\PP(T_{m'})\leq A_0^{-|h-h'|/2}\,\PP(T_m)\approx A_0^{-|h-h'|/2}\,\Theta(R_0),$$
which implies that $|h-h'|\lesssim1$. From this fact it follows that there is a bounded number of cubes $R_1^h$ satisfying \rf{eqlastclaim5}, as claimed. Putting all together, we get \rf{eqclaimh1}.
\end{proof}
This ends the proof of Lemma \ref{lemimp9}.
\end{proof}

\vv


%% file: riesz-DX7-S6-2.tex

\section{The Riesz transform on the tractable trees: the approximating measures \texorpdfstring{$\eta$}{eta}, \texorpdfstring{$\nu$}{nu}, and the variational argument}\label{sec6}

In this section,
for a given $R\in\MDW$ such that $\TT(e'(R))$ is tractable (i.e., $R\in\Trc$),
we will define a suitable measure $\eta$ that approximates $\mu$ at the level of the cubes from
$\TT(e'(R))$ and we will estimate $\|\RR\eta\|_{L^p(\eta)}$ from below. To this end,
we will apply a variational argument in $L^p$ by techniques inspired by 
\cite{Reguera-Tolsa} and \cite{JNRT}.  In the next section we will transfer these
estimates to $\RR\mu$.


\subsection{The suppressed Riesz transform and a Cotlar type inequality}\label{sec6.1}


Let $\Phi:\R^{n+1}\to[0,\infty)$ be a $1$-Lipschitz function.
Below we will need to work with the suppressed Riesz kernel
\begin{equation}\label{eqsuppressed}
K_\Phi(x,y) = \frac{x-y}{\bigl (|x-y|^2+\Phi(x)\Phi(y)\bigr)^{(n+1)/2}}
\end{equation}
and the associated operator 
$$\RR_\Phi\alpha(x) =\int K_\Phi(x,y)\,d\alpha(y),$$
where $\alpha$ is a signed measure in $\R^{n+1}$.
For a positive measure $\omega$ and $f\in L^1_{loc}(\omega)$, we write $\RR_{\Phi,\omega} f = \RR_\Phi (f\,\omega)$.
The kernel $K_\Phi$ (or a variant of this) appeared for the first
time in the work of Nazarov, Treil and Volberg in connection with Vitushkin's conjecture (see
\cite{Volberg}). 
This is a Calder\'on-Zygmund kernel which satisfies the properties:
\begin{equation}\label{eqkafi1}
|K_\Phi(x,y)|\lesssim \frac1{\big(|x-y| + \Phi(x) + \Phi(y)\big)^n}
\end{equation}
and
\begin{equation}\label{eqkafi2}
|\nabla_x K_\Phi(x,y)|+ |\nabla_y K_\Phi(x,y)|
\lesssim \frac1{\big(|x-y| + \Phi(x) + \Phi(y)\big)^{n+1}}
\end{equation}
for all $x,y\in\R^{n+1}$.

Also, if $\ve\approx\Phi(x)$, then we have
\begin{equation}
\label{e.compsup''}
\bigl|\RR_{\ve}\alpha(x) - \RR_{\Phi}\alpha(x)\bigr|\lesssim  \sup_{r> \Phi(x)}\frac{|\alpha|(B(x,r))}{r^n},
\end{equation}
with the implicit constant in the inequality depending on the implicit constant in the comparability $\ve\approx\Phi(x)$.
See Lemmas 5.4 and 5.5 in \cite{Tolsa-llibre}. 

The following result is an easy consequence of a $Tb$ theorem of Nazarov, Treil and Volberg.
See Chapter 5 of \cite{Tolsa-llibre}, for example. We will use this to prove \rf{eqacpsi}.

\begin{theorem}\label{teontv}
Let $\omega$ be a Radon measure in $\R^{n+1}$ and let $\Phi:\R^{n+1}\to[0,\infty)$ be a $1$-Lipschitz function. Suppose that
\begin{itemize}
\item[(a)] $\omega(B(x,r))\leq c_0\,r^n$ for all $r\geq \Phi(x)$, and
\item[(b)] $\sup_{\ve>\Phi(x)}|\RR_\ve\omega(x)|\leq c_1$.
\end{itemize}
Then $\RR_{\Phi,\omega}$ is bounded in $L^p(\omega)$, for $1<p<\infty$, with a bound on its norm depending only on $p$, $c_0$ and
$c_1$. In particular, $\RR_\omega$ is bounded in $L^p(\omega)$ on the set $\{x:\Phi(x)=0\}$.
\end{theorem}
\vv


We define the energy $W_\omega$ (with respect to $\omega$) of a set $F\subset \R^{n+1}$ as
$$W_\omega(F) = \iint_{F\times F} \frac1{\diam(F)\,|x-y|^{n-1}}\,d\omega(x)d\omega(y).$$
We say that a ball $B\subset \R^{n+1}$ is $(a,b)$-doubling $$\omega(aB)\leq b\,\omega(B).$$
We denote by $\cM_\omega f$ the usual centered maximal Hardy-Littlewood operator applied to $f$:
$$\cM_\omega f(x) = \sup_{r>0}\,\frac1{\omega(B(x,r))}\int_{B(x,r)}|f|\,d\omega,$$
and by $\cM_\omega^{(r_0,a,b)} f$ the version
$$\cM_\omega^{(r_0,a,b)} f(x) = \sup\frac1{\omega(B(x,r))}\int_{B(x,r)}|f|\,d\omega,$$
where the $\sup$ is taken over all radii $r>r_0$ such that the ball $B(x,r)$ is $(a,b)$-doubling.

\vv

\begin{lemma}\label{lemcotlar1}
Let $x\in R$, $r_0>0$, and $\theta_1>0$.
Suppose that for all $r\geq r_0$ 
$$\theta_\omega(x,r)\leq \theta_1$$
and 
$$W_\omega(B(x,r))\leq \theta_1\,\omega(B(x,r))\qquad \mbox{if $B(x,r)$ is $(16,128^{n+2})$-doubling.}$$
Then
\begin{equation}\label{eqcot99}
\sup_{\ve\geq r_0} |\RR_\ve\omega(x)|\lesssim \cM_\omega^{(r_0,16,128^{n+2})}(\RR\omega)(x) + \theta_1.
\end{equation}
\end{lemma}

\begin{proof}
For a given $\ve\geq r_0$, let $k\geq0$ be minimal so that, for $r=128^k\,\ve$, the ball $B(x,r)$ is $(128,128^{n+2})$-doubling (in particular, this implies that $B(x,8r)$ is $(16,128^{n+2})$-doubling). It is easy to see that such $k$ exists using the assumption $\theta_\omega(x,r)\leq \theta_1$.
By a standard estimate (see Lemma 2.20 from \cite{Tolsa-llibre}), it follows that
$$|\RR_\ve \omega(x)|\leq |\RR_{8r} \omega(x)| + C\,\frac{\omega(B(x,8r))}{(8r)^n} \leq |\RR_{8r} \omega(x)| + C\,\theta_1.$$
It is immediate to check that for any $x'\in B(x,2r)$,
\begin{equation}\label{eqss1}
|\RR_{8r}\omega(x) - \RR\chi_{B(x,4r)^c}\omega(x')|\lesssim \theta_1.
\end{equation}
Consider radial $C^1$ functions $\psi_1$ and $\psi_2$ such that
$$\chi_{B(x,4r)}\leq \psi_1\leq \chi_{B(x,8r)}\qquad \text{and} \qquad
\chi_{B(x,r)}\leq \psi_2\leq \chi_{B(x,2r)},$$
and $\text{Lip}(\psi_1)\le r,\ \text{Lip}(\psi_2)\le r.$ Given a function $f\in L^1_{loc}(\omega)$, denote by $m_{\psi_2\omega}f$ the $(\psi_2\,\omega)$-mean of $f$, i.e.,
$$m_{\psi_2\omega}f = \frac1{\|\psi_2\|_{L^1(\omega)}} \int f\,\psi_2\,d\omega.$$
Notice that
$$\bigl|m_{\psi_2\omega}(\RR\omega)\bigr|\leq \frac1{\omega(B(x,r))} \int_{B(x,2r)}
|\RR\omega|\,d\omega \approx\avint_{B(x,2r)}
|\RR\omega|\,d\omega 
\leq \cM_\omega^{(r_0,16,128^{n+2})}(\RR\omega)(x).$$
From \rf{eqss1} we deduce that
$$\big|\RR_{8r}\omega(x) - m_{\psi_2\omega}\bigl(\RR(\chi_{B(x,4r)^c}\omega)\bigr)\big| \leq m_{\psi_2\omega}\bigl(|\RR_{8r}\omega(x) - \RR(\chi_{B(x,4r)^c}\omega)|\bigr)
\lesssim\theta_1.$$
Then we have
\begin{align*}
|\RR_{8r}\omega(x)|&\lesssim\theta_1 + |m_{\psi_2\omega}(\RR(\chi_{B(x,4r)^c}\omega)|\\
& \lesssim\theta_1 + \bigl|m_{\psi_2\omega}\bigl(\RR(\chi_{B(x,4r)^c}\omega) - \RR((1-\psi_1)\omega)\bigr)\bigr| +\bigl|m_{\psi_2\omega}\bigl(\RR(\psi_1\omega)\bigr)\bigr|+ \bigl|m_{\psi_2\omega}(\RR\omega)\bigr|\\
& \lesssim\theta_1 +  \bigl|m_{\psi_2\omega}\bigl(\RR(\psi_1\omega)\bigr)\bigr|
+ \cM_\omega^{(r_0,16,128^{n+2})} (\RR\omega)(x).
\end{align*}
To estimate the middle term on the right hand side we use the antisymmetry of $\RR$:
\begin{align*}
\bigl|m_{\psi_2}\bigl(\RR(\psi_1\omega)\bigr)\bigr|
& =\frac1{\|\psi_2\|_{L^1(\omega)}} \left|\iint K(y-z)\,\psi_1(y)\,\psi_2(z)\,d\omega(y)\,d\omega(z)\right|\\
& =\frac1{2\|\psi_2\|_{L^1(\omega)}} \left|\iint K(y-z)\,\bigl(\psi_1(y)\,\psi_2(z) - \psi_1(z)\,\psi_2(y)\bigr)\,d\omega(y)\,d\omega(z)\right|\\
& \lesssim \frac1{\omega(B(x,r))} \iint_{B(x,8r)\times B(x,8r)} \!\frac1{r\,|y-z|^{n-1}}\,d\omega(y)\,d\omega(z)
 \approx\frac{W_\omega(B(x,8r))}{\omega(B(x,8r))} \lesssim \theta_1.
\end{align*}
\end{proof}

\vv

\vv
\begin{rem}\label{rem**}
In fact, the proof of the preceding lemma shows that, given any measure $\omega$, $x\in\R^{n+1}$ and
$\ve>0$,
\begin{equation}\label{eqcotlar*99}
|\RR_\ve\omega(x)|\lesssim \avint_{B(x,2\ve')} |\RR\omega|\,d\omega + \sup_{r\geq \ve} \frac{\omega(B(x,r))}{r^n} + \frac{W_\omega(B(x,8\ve'))}{\omega(B(x,8\ve'))},
\end{equation}
where $\ve'=2^{7k}\ve$, with $k\geq0$ minimal such that the ball $B(x,\ve')$ is $(128,128^{n+2})$-doubling. 
\end{rem}

\vv


\subsection{The family \texorpdfstring{$\Reg(e'(R))$}{Reg(e'(R))} and the approximating measure \texorpdfstring{$\eta$}{eta}}\label{sec6.2*}

In the remaining of this section we fix a cube $R\in\MDW$ such that $\TT(e'(R))$ is tractable.

We need to define some regularized family of ending cubes for $\TT(e'(R))$. 
First, let 
$$d_R(x) = \inf_{Q\in\TT(e'(R))}\big(\dist(x,Q) + \ell(Q)\big).$$
Notice that $d_R$ is a $1$-Lipschitz function. 
Given $0<\ell_0\ll\ell(R)$, we denote
\begin{equation}\label{eql00*23}
d_{R,\ell_0}(x) = \max\big(\ell_0,d_R(x)\big),
\end{equation}
which is also $1$-Lipschitz.
For each $x\in e'(R)$ we take the largest cube $Q_x\in\DD_\mu$ 
such that $x\in Q_x$ with
\begin{equation}\label{eqdefqx}
\ell(Q_x) \leq \frac1{60}\,\inf_{y\in Q_x} d_{R,\ell_0}(y).
\end{equation}
We consider the collection of the different cubes $Q_x$, $x\in e'(R)$, and we denote it by $\Reg(e'(R))$ (this stands for ``regularized cubes''). 

The constant $\ell_0$ is just an auxiliary parameter that prevents $\ell(Q_x)$ from vanishing.  Eventually $\ell_0$ will be taken extremely small. In particular, we assume $\ell_0$ small enough
so that 
\begin{equation}\label{eql00}
\mu\bigg(\bigcup_{Q\in\HD_1(e(R)):\ell(Q)\geq \ell_0} Q \bigg)\geq \frac12\,\mu\bigg(\bigcup_{Q\in\HD_1(e(R))} Q\bigg).
\end{equation}

We let $\TT_\Reg(e'(R))$ be the family of cubes made up of $R$ and all the cubes of the next generations which are contained in $e'(R)$ but are not 
strictly contained in any cube from $\Reg(e'(R))$.

\vv

\begin{lemma}\label{lem74}
The cubes from $\Reg(e'(R))$ are pairwise disjoint and satisfy the following properties:
\begin{itemize}
\item[(a)] If $P\in\Reg(e'(R))$ and $x\in B(x_{P},50\ell(P))$, then $10\,\ell(P)\leq d_{R,\ell_0}(x) \leq c\,\ell(P)$,
where $c$ is some constant depending only on $n$. 

\item[(b)] There exists some absolute constant $c>0$ such that if $P,\,P'\in\Reg(e'(R))$ satisfy $B(x_{P},50\ell(P))\cap B(x_{P'},50\ell(P'))
\neq\varnothing$, then
$$c^{-1}\ell(P)\leq \ell(P')\leq c\,\ell(P).$$
\item[(c)] For each $P\in \Reg(e'(R))$, there are at most $C_3$ cubes $P'\in\Reg(e'(R))$ such that
$$B(x_{P},50\ell(P))\cap B(x_{P'},50\ell(P'))
\neq\varnothing,$$
 where $C_3$ is some absolute constant.
 
\end{itemize}
\end{lemma}

The proof of this lemma is standard. See for example \cite[Lemma 6.6]{Tolsa-memo}.

\vvv

Next we define a measure $\eta$ which, in a sense, approximates $\mu\rest_{e'(R)}$ at the level of the family $\Reg(e'(R))$. 
We let
$$\eta = \sum_{P\in\Reg(e'(R))} \mu(P) \frac{\LL^{n+1}\rest_{\tfrac12B(P)}}{\LL^{n+1}(\tfrac12B(P))},
$$
where $\LL^{n+1}$ stands for the Lebesgue measure in $\R^{n+1}$.
With each $Q\in\TT_{\Reg}(e'(R))$ we associate another ``cube'' $Q^{(\eta)}$ defined as follows:
$$Q^{(\eta)}= \bigcup_{P\in\Reg(e'(R)):P\subset Q} \tfrac12B(P).$$
Further, we consider a lattice $\DD_\eta$ associated with the measure $\eta$ which is made up of the cubes
$Q^{(\eta)}$ with $Q\in \TT_{\Reg}(e'(R))$ and other cubes which are descendants of the cubes from $\Reg(e'(R))$.
We assume that $\DD_\eta$ satisfies the first two properties of Lemma \ref{lemcubs} with the same parameters $A_0$ and $C_0$ as $\DD_\mu$. It is straightforward to check that $\DD_\eta$ can be constructed in this way.
For $S\in\DD_\eta$ such that $S=Q^{(\eta)}$ with $Q\in\TT_{\Reg}(e'(R))$, we let $Q=S^{(\mu)}$. Further, we write
$\ell(S):=\ell(Q)$, $B_S:=B_Q$, and $\Theta(S):=\Theta(Q)$.

\vv



\subsection{The auxiliary family \texorpdfstring{$\sH$}{H}}\label{subsec:H}

Given $p\geq1$ and a family $I\subset\DD_\mu$, we denote
$$\sigma_p(I) = \sum_{P\in I}\Theta(P)^p\,\mu(P),$$
so that $\sigma(I)=\sigma_2(I)$. 
Recall that  
$$\HD_*(R) = \hd^{k_{\Lambda_*}}(R),$$
with $k_{\Lambda_*}=k_\Lambda(1-\frac1N)$,
and that
$$\HD_1(e(R)) = \sss_*(e(R))\cap \HD_*(R),\qquad  \HD_1(e'(R)) = \sss_*(e'(R))\cap \HD_*(R).
$$
For $R\in\MDW$ and $j\geq0$, denote 
$$\sH_j(e'(R))= \TT_\Reg(e'(R)) \cap \hd^{k_{\Lambda_*} + j}(R),$$
so that $\sH_0(e'(R))=\HD_1(e'(R))\cap\TT_\Reg(e'(R))$. 
\brem\label{rem:Hjempty}
Note that for $j> k_{\Lambda_*}+2$ we have $\sH_j(e'(R))=\varnothing$. Indeed, by the definition of $\Reg(e'(R))$ and $\TT_\Reg(e'(R))$, for each $Q\in \TT_\Reg(e'(R))$ there exists $P\in \TT(e'(R))$ such that $2B_Q\subset 2B_P$ and $\ell(Q)\le \ell(P)\le A_0^2\ell(Q)$. Thus,
\begin{equation*}
\frac{\mu(2B_Q)}{\ell(Q)^n}\le \frac{\mu(2B_P)}{\ell(Q)^n}\le A_0^{2n}\frac{\mu(2B_P)}{\ell(P)^n}.
\end{equation*}
Since for each $P\in \TT(e'(R))$ we have $\Theta(P)\le \Lambda_*^2\Theta(R)$, it follows that $\Theta(Q)\le A_0^{2n}\Lambda_*^2\Theta(R)$.
\erem


The fact that $\max_{j\geq0} \sigma_p(\sH_j(e'(R)))$ may be much larger than $\sigma_p(\HD_1(e'(R))$ may cause problems in some estimates. For this reason,
we need to introduce an auxiliary family $\sH$. We deal with this issue in this section.

Recall that, for $R\in\DD_{\mu,k}$,
$$e(R) = e_{h(R)}(R) \quad \mbox{ and }\quad e'(R) = e_{h(R)+1},$$
where
$$e_i(R) = R \cup \bigcup Q,$$
with the union running over the cubes $Q\in\DD_{\mu,k+1}$ such that
$$
\dist(x_R,Q)< \frac{\ell(R)}2 + 2i\ell(Q).
$$
For $j\geq0$, we set 
$$e_{i,j}(R) = \bigcup_{Q\in\DD_{\mu,k+1}:Q\subset e_i(R)} e_j(Q),$$
and we let $\sH_k(e_{i,j}(R))$ be the subfamily of the cubes from $\sH_k(e'(R))$ which are contained
in $e_{i,j}(R)$.

From now on, we let $\ve_n$ be some positive constant depending just on $n$. In the present paper, later on, we will simply
take $\ve_n=1/15$. However, for another application of the results from this section in \cite{Tolsa-riesz} it
is convenient to allow $\ve_n$ to depend on $n$.

\vv
\begin{lemma}\label{lem:66}
Let $p\in (1,2]$.
For any $R\in\MDW$ there exist some $j,k$, with $10\leq j\leq A_0/4$ and  $0\leq k\leq k_{\Lambda_*}+2$ such that
\begin{equation}\label{eqsigmah}
\sigma_p(\sH_{m}(e_{h(R),j+1}(R))) \leq \Lambda_*^{\ve_n} \,\sigma_p(\sH_{k}(e_{h(R),j}(R)))
\quad \mbox{for all $m\geq0$,}
\end{equation}
 assuming $A_0$ big enough (possibly depending on $n$). 
\end{lemma}

\begin{proof}
For each $j$, we denote by $0\le k_j\le k_{\Lambda_*}+2$ the integer such that
$$\sigma_p(\sH_{k_j}(e_{h(R),j}(R))) = \max_{0\leq k\leq k_{\Lambda_*}} \sigma_p(\sH_{k}(e_{h(R),j}(R))).$$
The lemma can be rephrased in the following way: there exists $10\leq j\leq A_0/4$ such that
\begin{equation*}
\sigma_p(\sH_{k_{j+1}}(e_{h(R),j+1}(R))) \leq \Lambda_*^{\ve_n} \,\sigma_p(\sH_{k_j}(e_{h(R),j}(R))).
\end{equation*}
We prove this by contradiction. Suppose the estimate above fails for all $10\leq j\leq A_0/4$, and let $j_0$ be the largest integer smaller than $A_0/4$.
Then we have
\begin{multline}\label{eq:651}
\sigma_p(\sH_{k_{j_0}}(e_{h(R),j_0}(R)))
  \geq \Lambda_*^{\ve_n}\,\sigma_p(\sH_{k_{j_0-1}}(e_{h(R),j_0-1}(R)))\\
\geq
\ldots\geq \Lambda_*^{\ve_n(j_0-10)}\sigma_p(\sH_{k_{10}}(e_{h(R),10}(R))) \geq \Lambda_*^{\ve_n(j_0-10)}\sigma_p(\sH_{0}(e_{h(R),10}(R)))\\
\geq  \Lambda_*^{\ve_n(j_0-10)}\sigma_p(\sH_{0}(e_{h(R)}(R))) \overset{\rf{eql00}}{\geq} \frac12\,\Lambda_*^{\ve_n(j_0-10)}
\sigma_p(\HD_1(e(R))),
\end{multline}

Concerning the left hand side of the inequality above, since $e_{h(R),j_0}(R)\subset 2R$ and $k_{j_0}\le k_{\Lambda_*}+2$, we have
$$
\sigma_p(\sH_{k_{j_0}}(e_{h(R),j_0}(R)))\leq A_0^{2np}
\Lambda_*^{2p}\,\Theta(R)^p\mu(2R).$$
Due to the fact that $R$ is $\PP$-doubling, as in \rf{eqdoub*11} we have
$\mu(2R)\leq C_0\,C_d\,A_0^{n+1}
\mu(R).$ Thus,
\begin{equation}\label{eq:652}
\sigma_p(\sH_{k_{j_0}}(e_{h(R),j_0}(R)))\leq C_0\,C_d\,A_0^{2np+n+1}
\Lambda_*^{2p}\sigma_p(R).
\end{equation}

Concerning the right hand side of \eqref{eq:651}, observe that, denoting $\Theta(\HD_1)=\Lambda_*\Theta(R)=\Theta(Q)$ for any $Q\in\HD_1(e(R))$, we have
$$\sigma_p(\HD_1(e(R))) = \Theta(\HD_1)^{p-2}\sigma(\HD_1(e(R)))\overset{\eqref{eq:MDWdef2}}{\ge} B^{-1}\Theta(\HD_1)^{p-2}\sigma(R) \geq \Lambda_*^{-1}\Lambda_*^{p-2}\sigma_p(R).
$$
We deduce from \eqref{eq:651}, \eqref{eq:652}, and the above, that
$$\frac12\,\Lambda_*^{\ve_n(j_0-10)}\Lambda_*^{p-3}
\sigma_p(R)\leq C_0\,C_d\,A_0^{2np+n+1}
\Lambda_*^{2p}\sigma_p(R).$$
Since $\Lambda_*\geq A_0^n$ and $j_0\approx A_0$, it is clear that this inequality is violated if $A_0$ is big 
enough, depending just on $n$.
\end{proof}
\vv


\subsection{Some technical lemmas}\label{subsec9.5}

Let $j(R),\,k(R)$ be such that $10\leq j(R)\leq A_0/4$,  $0\leq k(R)\leq k_{\Lambda_*}+2$ and
$$
\sigma_p(\sH_{m}(e_{h(R),j(R)+1}(R))) \leq \Lambda_*^{\ve_n} \,\sigma_p(\sH_{k(R)}(e_{h(R),j(R)}(R)))
\quad \mbox{for all $m\geq0$,}
$$
We denote
$$\cS_\mu = \bigcup_{Q\in\Reg:Q\subset e_{h(R),j(R)}(R)} Q,\qquad
\cS_\eta = \bigcup_{Q\in\Reg:Q\subset e_{h(R),j(R)}(R)} \tfrac12B(Q)$$
and
$$\cS_\mu' = \bigcup_{Q\in\Reg:Q\subset e_{h(R),j(R)+1}(R)} Q,\qquad
\cS_\eta' = \bigcup_{Q\in\Reg:Q\subset e_{h(R),j(R)+1}(R)} \tfrac12B(Q).$$
Notice that, by construction,
\begin{equation}\label{eqtec732}
\dist(\supp\mu \setminus e'(R),\cS_\mu')\geq cA_0^{-1} \ell(R),
\end{equation}
where $c>0$ is an absolute constant.

For $m=1,2,3,4$, denote by $V_m$ the $m A_0^{-3}\ell(R)$-neighborhood of $\cS_\eta$.
Let $\vphi_R$ be a $C^1$ function which equals $1$ in $V_3$, vanishes out of $V_4$, and such that $\|\vphi_R\|_\infty\leq 1$ and
$\|\nabla\vphi_R\|_\infty\leq 2A_0^3 \ell(R)^{-1}$.
Observe that, for $x\in \supp\mu\setminus e'(R)$, $\dist(x,V_4)\gtrsim \ell(R)$.
In fact, from \rf{eqtec732} one can derive that
\begin{equation}\label{eqtec733}
\dist(Q,\supp\mu\setminus e'(R)) \gtrsim \ell(R)\quad\mbox{ for all $Q\in\Reg(e'(R))$ such that
$B_Q\cap V_4\neq\varnothing$,}
\end{equation}
taking into account that $\ell(Q)\leq \frac{A_0^{-1}}{60}\,\ell(R)$ for every $Q\in\Reg(e'(R))$.

 We consider the measure
$$\nu = \vphi_R\,\eta$$
and the function
$$G(x) = 2A_0^3\int_{\cS_\eta'\setminus V_2} 
 \frac1{\ell(R)\,|x-y|^{n-1}}\,d\eta(y).$$
 Notice that
$$G(x)\lesssim \Theta(R)\quad\mbox{ for all $x\in V_1$.}$$

To shorten notation, we write
$$\sH= \sH_{k(R)}(e_{h(R),j(R)}(R)),\qquad \sH'= \sH_{k(R)}(e_{h(R),j(R)+1}(R))$$
and
$$H= \bigcup_{Q\in\sH} Q^{(\eta)},\qquad  H'= \bigcup_{Q\in\sH'}Q^{(\eta)}.$$
We also set
$$\Theta(\sH) = A_0^{(k_{\Lambda_*}+k(R))n}\,\Theta(R),$$
so that for any $Q\in\sH$ we have
$\Theta(Q) = \Theta(\sH).$

\vv
\begin{lemma}\label{lem*921}
Let $A\subset \R^{n+1}$ be the set of those $x\in\R^{n+1}$ which belong to some 
$(16,128^{n+2})$-doubling (with respect to $\nu$) ball $B\subset\R^{n+1}$ such that
$$W_\nu(B)\geq M\,\Theta(\sH)\,\nu(B)\quad \mbox{ and } \quad
\theta_\eta( \gamma B)\leq c_2\,\Theta(\sH) \quad\mbox{for all $\gamma\geq1$.}$$
Then, for $M\geq1$ big enough, 
$$\nu(A) \lesssim \frac{\Lambda_*^{\ve_n}}{M} \,\nu(H).$$
\end{lemma}

\begin{proof}
Observe first that, for any ball $B\subset\R^{n+1}$ with $r(B)\in[A_0^{-k-1},A_0^{-k}]$,
\begin{align*}
W_\nu(B) &\lesssim
\theta_\nu(B)\,\nu(B) \\
&\quad+ \sum_{j\geq k+1} \sum_{Q\in\DD_{\eta,j}:Q\subset 2B}\int_{x\in Q}\int_{y:
A_0^{-j-2}<|x-y|\leq A_0^{-j-1}}
\frac1{r(B)\,|x-y|^{n-1}}\, d\nu(x)\,d\nu(y)\\
& \lesssim \sum_{Q\in\DD_{\eta}:Q\subset 2B} \frac{\ell(Q)}{r(B)}\,\theta_\nu(2B_Q)\,\nu(Q).
\end{align*}

To prove the lemma, we apply Vitali's $5r$-covering lemma to get a family of $(16,128^{n+2})$-doubling balls $B_i$, $i\in I$,
which satisfy the following:
\begin{itemize}
\item the balls $2B_i$, $i\in I$, are pairwise disjoint,
\item $A\subset \bigcup_{i\in I} 10B_i$,
\item 
$W_\nu(B_i)\geq M\,\Theta(\sH)\,\nu(B_i)$ and $\theta_\eta(\gamma B_i)\leq c_2\,\Theta(\sH)$ and  for all $i\in I$ and $\gamma\geq1$.
\end{itemize}
Then we deduce
\begin{align}\label{eqplug6}
\nu(A) &\leq \sum_{i\in I} \nu(10B_i) \lesssim \sum_{i\in I} \nu(B_i)
\leq \frac1{M\,\Theta(\sH)} \sum_{i\in I} W_\nu(B_i)\\
& \lesssim \frac1{M\,\Theta(\sH)} \sum_{i\in I}\sum_{Q\in\DD_{\eta}:Q\subset 2B_i} \frac{\ell(Q)}{r(B_i)}\,\theta_\nu(2B_Q)\,\nu(Q).\nonumber
\end{align}
Now we take into account that all the cubes $Q$ which are not contained in any cube $P$ with $P^{(\mu)}\in\sH'$
satisfy $\theta_\nu(2B_Q)\leq \theta_\eta(2B_Q)\lesssim \Theta(\sH)$. Then, for each $i\in I$,
\begin{align*}
\sum_{Q\in\DD_{\eta}:Q\subset 2B_i} \frac{\ell(Q)}{r(B_i)}\,\theta_\nu(2B_Q)\,\nu(Q) & \leq
\sum_{P^{(\mu)}\in \sH'}\,
\sum_{Q\in\DD_{\eta}:Q\subset 2B_i\cap P} \frac{\ell(Q)}{r(B_i)}\,\theta_\eta(2B_Q)\,\eta(Q)\\
&\quad + \Theta(\sH)\sum_{Q\in\DD_{\eta}:Q\subset 2B_i} \frac{\ell(Q)}{r(B_i)}\, \nu(Q)\\
\end{align*}
Observe that the last term on the right hand side is bounded above by $\Theta(\sH)\nu(2B_i)\approx
\Theta(\sH)\nu(B_i)$. 
So plugging the previous estimate into \rf{eqplug6}, we get
\begin{align*}
\nu(A) \lesssim \frac{1}{M\,\Theta(\sH)} \sum_{i\in I}
\sum_{P^{(\mu)}\in \sH'}\,
\sum_{Q\in\DD_{\eta}:Q\subset 2B_i\cap P} \frac{\ell(Q)}{r(B_i)}\,\theta_\eta(2B_Q)\,\eta(Q)
+ \frac{1}{M} \sum_{i\in I} \nu(B_i).
\end{align*}
Since $B_i\subset A$ for each $i$, the last term is at most $\frac12\nu(A)$ for $M$ big enough. Thus,
\begin{equation}\label{eqnua12}
\nu(A) \lesssim  \frac{1}{M\,\Theta(\sH)} \sum_{i\in I}
\sum_{P^{(\mu)}\in \sH'}\,
\sum_{Q\in\DD_{\eta}:Q\subset 2B_i\cap P} \frac{\ell(Q)}{r(B_i)}\,\theta_\eta(2B_Q)\,\eta(Q).
\end{equation}

To estimate the term on the right hand side above, we denote by $\sF_k$ the family of cubes from
$\DD_\eta$ which are contained in some cube $Q^{(\eta)}$ with $Q\in\sH_{k}':=\sH_{k}(e_{h(R),j(R)+1}(R))$ and are not contained
in any cube $P^{(\eta)}$ with $P\in\sH_{k+1}':=\sH_{k+1}(e_{h(R),j(R)+1}(R))$. Notice that
$$\theta_\eta(2B_Q)\lesssim \Theta(\sH_{k+1})\approx \Theta(\sH_{k}),$$
where $\Theta(\sH_{k})=\Theta(Q')$ for $Q'\in\sH_{k}'$ (this does not depend on the specific cube $Q'$). Then, for each $i\in I$, we have
\begin{align}\label{eqalj4}
\sum_{P^{(\mu)}\in \sH'}
\sum_{Q\in\DD_{\eta}:Q\subset 2B_i\cap P} \frac{\ell(Q)}{r(B_i)}\,\theta_\eta(2B_Q)\eta(Q) &=\!\!
\sum_{k\geq k(R)}\sum_{P^{(\mu)}\in \sH'_k}
\sum_{Q\in\sF_k:Q\subset 2B_i\cap P} \frac{\ell(Q)}{r(B_i)}\,\theta_\eta(2B_Q)\eta(Q)\\
& \lesssim 
\sum_{k\geq k(R)}\sum_{P^{(\mu)}\in \sH'_k} \Theta(\sH_k)
\sum_{Q\in\sF_k:Q\subset 2B_i\cap P} \frac{\ell(Q)}{r(B_i)}\,\eta(Q).\nonumber
\end{align}
We claim now that for $Q$ in the last sum, we have
\begin{equation}\label{eqcond*492}
\ell(Q)\lesssim A_0^{-(k-k(R))}\,r(B_i).
\end{equation}
To check this, let $P(Q,k)\in \sH_k'$ be such that $Q\subset P(Q,k)$. From the condition 
\begin{equation}\label{eqcond492}
\theta_\eta( \gamma B_i)\leq c_2\,\Theta(\sH) \quad\mbox{for all $\gamma\geq1$}
\end{equation}
and the fact that $P(Q,k)\cap 2B_i\neq\varnothing$ (because $Q\subset 2B_i$) we infer that 
$r(B_i)\geq \ell(P(Q,k))$ for $k$ big enough. Otherwise we would find a ball $\gamma B_i$ containing $P(Q,k)$ with radius comparable to $\ell(P(Q,k))$, so that
$$\theta_\eta(\gamma B_i) \geq c\,\Theta(P(Q,k)) >c_2\,\Theta(\sH)$$
for $k$ big enough (depending on $c_2$), contradicting \rf{eqcond492}. So we have $P(Q,k)\subset 6B_i$ and thus
$$c_2\Theta(\sH)\geq \theta_\eta(6B_i)\gtrsim \frac{\ell(P(Q,k))^n}{r(B_i)^n}\,\Theta(P(Q,k))= 
\frac{\ell(P(Q,k))^n}{r(B_i)^n}\,A_0^{n(k-k(R))}\,\Theta(\sH).$$
This gives 
$$\ell(Q)\leq \ell(P(Q,k))\lesssim A_0^{-(k-k(R))}\,r(B_i)$$
and proves \rf{eqcond*492} for $k$ big enough, and thus for all $k$ if we choose the implicit constant in \rf{eqcond*492} suitably.

From \rf{eqcond*492} we deduce that the right hand side of \rf{eqalj4}
does not exceed
\begin{multline*}
\sum_{k\geq k(R)}A_0^{-(k-k(R))/2} \Theta(\sH_k)\sum_{P^{(\mu)}\in \sH'_k} 
\sum_{Q\subset 2B_i\cap P} \left(\frac{\ell(Q)}{r(B_i)}\right)^{1/2}\,\eta(Q)\\
\lesssim \sum_{k\geq k(R)}A_0^{-(k-k(R))/2} \,\Theta(\sH_k)\sum_{P^{(\mu)}\in \sH'_k} 
\eta(2B_i\cap P).
\end{multline*}
Plugging this estimate into \rf{eqnua12} and recalling that the balls $2B_i$ are disjoint, we get
\begin{align*}
\nu(A) & \lesssim \frac{1}{M\,\Theta(\sH)} \sum_{i\in I}
\sum_{k\geq k(R)}A_0^{-(k-k(R))/2} \sum_{P^{(\mu)}\in \sH'_k} \Theta(\sH_k)\, \eta(2B_i\cap P)\\
& \leq
\frac{1}{M\,\Theta(\sH)}
\sum_{k\geq k(R)}A_0^{-(k-k(R))/2} \sum_{P^{(\mu)}\in \sH'_k} \Theta(\sH_k) \,\eta(P).
\end{align*}
By H\"older's inequality, for each $k\geq k(R)$ we have
$$\sum_{P^{(\mu)}\in \sH'_k} \Theta(\sH_k) \eta(P) \leq \sigma_p(\sH'_k)^{1/p}\,\eta(H')^{1/p'}.
$$
We can estimate the right hand side using \rf{eqsigmah}:
\begin{align*}
\sigma_p(\sH'_k)&\le \Lambda_*^{\ve_n}\,\sigma_p(\sH),\\ \eta(H')&=\frac{\sigma_p(\sH')}{\Theta(\sH)^p} \le \Lambda_*^{\ve_n}\,\frac{\sigma_p(\sH)}{\Theta(\sH)^p} = \Lambda_*^{\ve_n}\,\eta(H).
\end{align*}
Thus,
$$\nu(A) \lesssim \frac{\Lambda_*^{\ve_n}}{M\,\Theta(\sH)} \sum_{k\geq k(R)}A_0^{-(k-k(R))/2} \,\sigma_p(\sH)^{1/p}\,\eta(H)^{1/p'}
\approx
\frac{\Lambda_*^{\ve_n}}{M}\,\eta(H)\approx \frac{\Lambda_*^{\ve_n}}{M}\,\nu(H).$$
\end{proof}
\vv

Let $\TT_{\sH'}$ denote the family of all cubes from $\DD_\mu$ 
 made up of $R$ and all the cubes of the next generations which are contained in $e_{h(R),j(R)+1}(R)$ but are not 
strictly contained in any cube from from $\sH'$. We consider the function
\begin{equation}\label{eqdefPsi1}
\Psi(x) = \inf_{Q\in \TT_{\sH'}} \big(\ell(Q) +\dist(x,Q)\big).
\end{equation}
Notice that $\Psi$ is a $1$-Lipschitz function. We also have the following result, which will be proven
by quite standard arguments.

\begin{lemma}\label{lem6.77}
For all $x\in\R^{n+1}$, we have
\begin{equation}\label{ecreixnupsi}
\sup_{r\geq \Psi(x)} \frac{\nu(B(x,r))}{r^n} \leq \sup_{r\geq \Psi(x)} \frac{\eta(B(x,r))}{r^n} \lesssim \Theta(\sH)\quad\mbox{ for all $x\in \R^{n+1}$.}
\end{equation}
\end{lemma}

\begin{proof}
The first inequality in \rf{ecreixnupsi} is trivial. Concerning the second one,
in the case $r>\ell(R)/10$ we just use that
$$\eta(B(x,r))\leq \mu(e'(R))\lesssim\mu(R)\lesssim \Theta(R)\,\ell(R)^n\lesssim \Theta(\sH)\,r^n.$$
So we may assume that $\Psi(x)<r\leq \ell(R)/10$.
By the definition of $\Psi(x)$, there exists $Q\in\TT_{\sH'}$ such that
$$\ell(Q) + \dist(x,Q)\leq r.$$
Therefore, $B_Q\subset B(x,4r)$ and so there exists an ancestor $Q'\supset Q$ which belongs to $\TT_{\sH'}$ such that $B(x,r)\subset  2B_{Q'}$, with $\ell(Q')\approx r$. 
Then, $\Theta(Q')\leq \Theta(\sH)$ and
$$\eta(B(x,r))\leq \sum_{P\in\Reg(e'(R)):P\cap 2B_{Q'}\neq\varnothing} \eta(P^{(\eta)}) = \sum_{P\in\Reg(e'(R)):P\cap 2B_{Q'}\neq\varnothing} \mu(P).
$$
Observe now that if $P\in\Reg(e'(R))$ satisfies $P\cap 2B_{Q'}\neq\varnothing$, then $\ell(P)\lesssim \ell(Q')$ (this is an easy consequence of Lemma \ref{lem74}(b) and the fact that $Q'$ contains some cube from $\Reg(e'(R))$). So we deduce that
$P\subset CB_{Q'}$. Hence,
$$\eta(B(x,r))\leq \mu(CB_{Q'}) \lesssim \Theta(\sH)\,\ell(Q')^n \approx \Theta(\sH)\,r^n.$$
\end{proof}

\vv

In the next subsection we are going to use a variational argument to show that for some constant $c_3$, depending at most\footnote{The constant $c_3$
	will be chosen of the form $c_3=A_0^{C(n)}$, and it will not depend on $\Lambda_*$, $\ve_n$, or other
	parameters of the construction.} on $n$ and $A_0$, we have
$$\int \big|(|\RR\nu(x)| - G(x)- c_3\,\Theta(\sH))_+\big|^p\,d\nu(x) \gtrsim \Lambda_*^{-p'\ve_n}\sigma_p(\sH).$$
See Lemma \ref{lemvar} for details. We prove this estimate by contradiction, so that in particular we will assume that
\begin{equation}\label{eqassu8}
\int \big|(|\RR\nu(x)| - G(x)- c_3\,\Theta(\sH))_+\big|^p\,d\nu(x) \leq \sigma_p(\sH).
\end{equation}
Below we prove a few consequences of \eqref{eqassu8} that will be useful in the proof of Lemma \ref{lemvar}.

We denote
$$\RR_{*,\Psi}\nu(x) = \sup_{\ve>\Psi(x)} \big|\RR_\ve\nu(x)\big|.$$

\begin{lemma}\label{lemrieszpetit}
Suppose that \eqref{eqassu8} holds. Then,
\begin{equation}\label{eqassu9}
\nu\big(\big\{x\in \cS_\eta:\,\RR_{*,\Psi}\nu(x) > M\,\Theta(\sH)\big\}\big) \leq C_4\, 
\frac{\Lambda_*^{\ve_n}}{M} \,\nu(H).
\end{equation}
\end{lemma}

\begin{proof}
Recall that we denote by $V_1$ the $A_0^{-3}\ell(R)$-neighborhood of $\cS_\eta$
and that
$$G(x)\lesssim \Theta(R)\quad\mbox{ for all $x\in V_1$.}$$ 
Then the assumption in the lemma implies that
\begin{equation}\label{eqsig84}
\int_{V_1} \big|\RR\nu(x)|^p\,d\nu(x) \lesssim \sigma_p(\sH).
\end{equation}

By Remark \ref{rem**}, for any $x\in \cS_\eta$ and $\ve>\Psi(x)$,
$$|\RR_\ve\nu(x)|\lesssim \avint_{B(x,2\ve')} |\RR\nu|\,d\nu + \sup_{r\geq \ve} \frac{\nu(B(x,r))}{r^n} + \frac{W_\nu(B(x,8\ve'))}{\nu(B(x,8\ve'))},$$
where $\ve'=2^{7k}\ve$, with $k\geq0$ minimal such that the ball $B(x,\ve')$ is $(128,128^{n+2})$-doubling. 
In the case $\ve'\geq \frac12 A_0^{-3}\,\ell(R)$, by standard arguments,
$$
|\RR_\ve \nu(x)|\leq |\RR_{\ve'} \nu(x)| + C\,\frac{\nu(B(x,\ve'))}{(\ve')^n} \leq C\,
\Theta(R)< M\,\Theta(\sH),$$
for $\Lambda_*$ (or $M$) big enough.

In the case $\ve'< \frac12 A_0^{-3}\ell(R)$, we have $B(x,2\ve')\subset V_1$ and thus
$$|\RR_\ve\nu(x)|\lesssim \cM_\nu^{(\Psi(x),16,128^{n+2})}(\chi_{V_1}\RR\nu)(x) 
+ \sup_{r\geq \Psi(x)} \frac{\nu(B(x,r))}{r^n} + \sup_{r\in D(x):r\geq \Psi(x)} \\
 \frac{W_\nu(B(x,r))}{\nu(B(x,r))},
$$
where $D(x)$ denotes the set of radii $r>0$ such that $B(x,r)$ is $(16,128^{n+2})$-doubling.
Also, as shown in \rf{ecreixnupsi},
$$\sup_{r\geq \Psi(x)} \frac{\nu(B(x,r))}{r^n} \lesssim \Theta(\sH).$$

We deduce that in any case (i.e., for any $\ve'$), assuming $M$ larger than some absolute constant,
\begin{align*}
\nu\big(\big\{x\in \cS_\eta:\,\RR_{*,\Psi}\nu(x) &> M\,\Theta(\sH)\big\}\big) \\& \leq
\nu\Big(\Big\{x\in \cS_\eta:\,\cM_\nu^{(\Psi(x),16,128^{n+2})}(\chi_{V_1}\RR\nu)(x) > \frac{M\,\Theta(\sH)}3\Big\}\Big)\\
& \!\!+
\nu\Big(\Big\{x\in \cS_\eta:\!\!\!\sup_{r\in D(x):r\geq \Psi(x)} \!
 \frac{W_\nu(B(x,r))}{\nu(B(x,r))} > \frac{M\,\Theta(\sH)}3\Big\}\Big)\\
 & =: T_1 + T_2.
\end{align*}

To deal with the term $T_1$, we use the weak $L^p(\nu)$ boundedness of $\cM_\nu^{(\Psi(x),16,128^{n+2})}$ and \rf{eqsig84} to obtain
$$T_1\lesssim \frac1{(M\Theta(\sH))^p} \int_{V_1}|\RR\nu|^p\,d\nu
\lesssim \frac1{M^p}\,\eta(H).
$$
Concerning $T_2$, by Lemma \ref{lem*921},
$$T_2 \lesssim \frac{\Lambda_*^{\ve_n}}{M} \,\nu(H).$$
So we have
$$\nu\big(\big\{x\in \cS_\eta:\,\RR_{*,\Psi}\nu(x) > M\,\Theta(\sH)\big\}\big)\lesssim 
\frac1{M^p}\,\eta(H) + \frac{\Lambda_*^{\ve_n}}{M} \,\nu(H)\lesssim \frac{\Lambda_*^{\ve_n}}{M} \,\nu(H).$$
\end{proof}
\vv

For $\gamma>1$, we let
$$Z(\gamma) = \big\{x\in \cS_\eta': \RR_{*,\Psi}\nu(x) > \gamma \Lambda_*^{\ve_n}\,\Theta(\sH)\big\}.$$
Notice that, under the assumption  \rf{eqassu8}, by Lemma \ref{lemrieszpetit}
 (with $M= \gamma \Lambda_*^{\ve_n}$), we have
\begin{equation}\label{eqzgam1}
\nu(Z(\gamma)\cap \cS_\eta) \leq C_4\gamma^{-1}\,\nu(H).
\end{equation}
For $x\in Z(\gamma)$, let
\begin{equation}\label{eqe3194}
e(x,\gamma) = \sup\{\ve: \ve>\Psi(x),\,|\RR_\ve\nu(x)|>\gamma \Lambda_*^{\ve_n}\,\Theta(\sH)\}
\end{equation}
We define the exceptional set $Z'(\gamma)$ as
$$Z'(\gamma) = \bigcup_{x\in Z(\gamma)} B(x,e(x,\gamma)).$$

The next lemma shows that $\nu(Z'(\gamma)\cap \cS_\eta)$ is small if $\gamma$ is taken big enough.

\vv
\begin{lemma} \label{tamt}
If $y\in Z'(\gamma)$, then $\RR_{*,\Psi}\nu(y) > (\gamma \Lambda_*^{\ve_n}- c_4)\Theta(\sH)$,
for some $c_4>0$.
Thus, under the assumption  \rf{eqassu8},
if $\gamma\geq 2c_4$, then
\begin{equation} \label{wert1}
\nu(Z'(\gamma)\cap \cS_\eta) \leq \frac{2C_4}{\gamma}\,\nu(H).
\end{equation}
\end{lemma}

\begin{proof}
The arguments are similar to the ones in Lemma 5.2 from
\cite{Tolsa-llibre}. However, for the reader's convenience we will explain here the basic details.

By standard arguments (which just use the fact that the Riesz kernel is a Calder\'on-Zygmund kernel),
for all $y\in B(x,e(x,\gamma))$, with $x\in Z(\gamma)$,
 we have
$$|\RR_{e(x,\gamma)}\nu(x) -\RR_{2e(x,\gamma)}\nu(y)|\lesssim \sup_{r\geq e(x,\gamma)} \frac{\nu(B(x,r))}{r^n} \lesssim \Theta(\sH),$$
taking into account that $e(x,\gamma)\geq\Psi(x)$ and recalling \rf{ecreixnupsi} for the last estimate.
So we have
$$|\RR_{2e(x,\gamma)}\nu(y)| \geq |\RR_{e(x,\gamma)}\nu(x)|- c_4\,\Theta(\sH) \geq  \gamma\,\Lambda_*^{\ve_n}\,\Theta(\sH) -  c_4\,\Theta(\sH).$$
Observe now that
$$\Psi(y)\leq \Psi(x) + |x-y|< 2e(x,\gamma).$$ 
So 
$$\RR_{*,\Psi}\nu(y) \geq |\RR_{2e(x,\gamma)}\nu(y)|> (\gamma \Lambda_*^{\ve_n}- c_4)\Theta(\sH),$$
which proves the first statement of the lemma.

In particular, taking $\gamma\geq 2c_4$, from the last estimate we derive 
$$\RR_{*,\Psi}\nu(y) \geq \frac\gamma2\,\Lambda_*^{\ve_n}\,\Theta(\sH),$$
and so $y\in Z(\gamma/2)$, which shows that $Z'(\gamma)\subset Z(\gamma/2)$. Thus,
by \rf{eqzgam1},
$$\nu(Z'(\gamma)\cap \cS_\eta) \leq \nu(Z(\gamma/2)\cap \cS_\eta) 
\leq 2C_4\gamma^{-1}\,\nu(H).$$
\end{proof}
\vv

We choose $\gamma=\max(1,2c_4,4C_4)$, so that, under the assumption  \rf{eqassu8},
\begin{equation}\label{eqnugam*}
\nu(Z'(\gamma)\cap \cS_\eta) \leq \frac12\,\nu(H).
\end{equation}
Also we define
$$\Phi(x) = \max(\Psi(x),\dist(x,\R^{n+1}\setminus Z'(\gamma)).$$
Notice that $\Phi$ is a $1$-Lipschitz function which coincides with $\Psi(x)$ away from $Z'(\gamma)$ and
that 
\begin{equation}\label{eqphigam}
\Phi(x)\geq e(x,\gamma)\quad \mbox{ for all $x\in Z(\gamma)$.}
\end{equation}

\vv

\begin{lemma}\label{lemrieszphi4}
The suppressed operator $\RR_\Phi$ is bounded in $L^p(\nu)$, with
$\|\RR_\Phi\|_{L^p(\nu)\to L^p(\nu)}\lesssim \Lambda_*^{\ve_n}\Theta(\sH).$
\end{lemma}

\begin{proof}
This is an immediate consequence of Theorem \ref{teontv} and the construction of $\Phi$ above.
Indeed, by \rf{ecreixnupsi},
$$\nu(B(x,r)) \lesssim \Theta(\sH)\,r^n \quad\mbox{ for all $r\geq\Phi(x)$.}$$
Also, by \rf{eqphigam},
$$\sup_{\ve>\Phi(x)}|\RR_\ve\nu(x)|\leq \sup_{\ve>e(x,\gamma)}|\RR_\ve\nu(x)|
\leq \gamma \Lambda_*^{\ve_n}\,\Theta(\sH)
\quad \mbox{ for all $x\in Z(\gamma)$.}$$
On the other hand, by the definition of $Z(\gamma)$ we also have
$$\sup_{\ve>\Phi(x)}|\RR_\ve\nu(x)|\leq \sup_{\ve>\Psi(x)}|\RR_\ve\nu(x)|\leq
\gamma \Lambda_*^{\ve_n}\,\Theta(\sH) \quad \mbox{ for $x\in \cS_\eta'\setminus Z(\gamma)$.} 
$$
So we can apply Theorem \ref{teontv}, taking 
$\omega= (C\Lambda_*^{\ve_n}\Theta(\sH))^{-1}\,\nu$ with an appropriate absolute constant $C$, and
then the lemma follows.
\end{proof}

\vv

\begin{lemma}\label{lemH0}
Under the assumption \rf{eqassu8},
there exists a subset $H_0\subset H$ such that:
\begin{itemize}
\item[(a)] $\eta(H_0)\geq \frac12\,\eta(H)$,
\item[(b)] $\RR_{\Psi,\eta}$ is bounded from $L^p(\eta\rest_{H_0})$ to $L^p(\eta\rest_{\cS'_\eta})$, with 
$\|\RR_\Psi\|_{L^p(\eta\rest_{H_0})\to L^p(\eta\rest_{\cS'_\eta})}\lesssim \Lambda_*^{\ve_n}\Theta(\sH).$
\end{itemize}
\end{lemma}

\begin{proof}
We let 
$$H_0= H\setminus Z'(\gamma),$$
so that, by \rf{eqnugam*},
$$\eta(H_0) = \nu(H_0) \geq \nu(H) - \nu(Z'(\gamma)\cap \cS_\eta) \geq \frac12 \,\nu(H) = \frac12 \,\eta(H).$$

To prove (b), take $f\in L^p(\eta\rest_{H_0})$ and observe that for $x\in \cS'_\eta$, by \rf{e.compsup''},
$$|\RR_{\Psi} (f\eta)(x)| = |\RR_{\Psi} (f\nu)(x)|\leq |\RR_{\Psi(x)} (f\nu)(x)| + 
\cM_{\Psi,n}(f\nu)(x),$$
where $\RR_{\Psi(x)}$ is the $\Psi(x)$-truncated Riesz transform and
\begin{equation}\label{eq:maximaloppsi}
\cM_{\Psi,n} \alpha(x) = \sup_{r>\Psi(x)}\frac{|\alpha|(B(x,r))}{r^n}
\end{equation}
for any signed Radon measure $\alpha$. Taking into account that $\Phi(x)\geq\Psi(x)$, we can split
$$\RR_{\Psi(x)} (f\nu)(x) = \RR_{\Phi(x)} (f\nu)(x) + \int_{y\in H_0:\Psi(x)<|x-y|\leq \Phi(x)} 
\frac{x-y}{|x-y|^{n+1}}\,f(y)\,d\nu(y).$$
To estimate the last integral, notice that for $y\in H_0$, $\Psi(y) = \Phi(y)$, and then the condition $\Psi(x)<|x-y|$ implies that
$$\Phi(x) \leq \Phi(y) + |x-y| = \Psi(y) + |x-y| \leq \Psi(x) + 2|x-y|<3|x-y|.$$
So the last integral above is bounded by
$$\int_{\Phi(x)/3<|x-y|\leq \Phi(x)} \frac1{|x-y|^n}\,|f(y)|\,d\nu(y)\lesssim \cM_{\Psi,n}(f\nu)(x).$$
From the preceding estimate and \rf{e.compsup''} we derive that
$$|\RR_{\Psi} (f\eta)(x)| \leq |\RR_{\Phi(x)} (f\nu)(x)| + C\,\cM_{\Psi,n}(f\nu)(x) \leq
|\RR_{\Phi} (f\nu)(x)| + C\,\cM_{\Psi,n}(f\nu)(x).$$

From the last inequality and Lemma \ref{lemrieszphi4}, taking into account that $\eta$ coincides with $\nu$ on $V_1$, we deduce that 
$$\|\RR_\Psi\|_{L^p(\eta\rest_{H_0})\to L^p(\eta\rest_{V_1})}\lesssim \Lambda_*^{\ve_n}\Theta(\sH) 
+ \|\cM_{\Psi,n}\|_{L^p(\eta)\to L^p(\eta)}.$$
From the growth condition \rf{ecreixnupsi} and standard covering lemmas, it follows easily that
$$\|\cM_{\Psi,n}\|_{L^p(\eta)\to L^p(\eta)}\lesssim\Theta(\sH).$$
On the other hand, using the fact that $\dist(H_0,\R^{n+1}\setminus V_1)\gtrsim \ell(R)$, it is immediate (by Schur's criterion, for example) to check that also
$$\|\RR_\Psi\|_{L^p(\eta\rest_{H_0})\to L^p(\eta\rest_{\cS_\eta'\setminus V_1})}\lesssim 
\Theta(R)\leq \Theta(\sH).$$
\end{proof}

\vv


\subsection{The variational argument}

\begin{lemma}\label{lemvar}
There is a constant $c_3>0$, depending at most on $n$ and $A_0$, such that for any $p\in (1,2]$, if $\Lambda_*$ is big enough, we have
$$\int \big|(|\RR\nu(x)| - G(x)- c_3\,\Theta(\sH))_+\big|^p\,d\nu(x) \gtrsim \Lambda_*^{-p'\ve_n}\sigma_p(\sH),$$
where $p'=p/(p-1)$ and the implicit constant depends on $p$.
\end{lemma}

\begin{proof}
Suppose that, for a very small $0<\lambda<1$ to be fixed below,
\begin{equation}\label{eqsupp1}
\int \bigl|\bigl(|\RR\nu| - G - c_3\,\Theta(\sH)\bigr)_+\bigr|^p\,d\nu\leq \lambda\,\sigma_p(\sH).
\end{equation}


Let $H_0\subset H$ be the set found in Lemma \ref{lemH0}.
 Consider the measures of the form
$\tau=a\,\nu$, with $a\in L^\infty(\nu)$, $a\geq 0$, and let $F$ be the functional
\begin{align*}\label{e.functional}
F(\tau) & = 
\int \bigl|\bigl(|\RR\tau| - G- c_3\,\Theta(\sH)\bigr)_+\bigr|^p\,d\tau
 + \lambda\,\|a\|_{L^\infty(\nu)} \,\sigma_p(\sH) + \lambda \,\frac{\nu(H_0)}{\tau(H_0)}\,\sigma_p(\sH).\nonumber
\end{align*}
Let
$$m= \inf F(\tau),$$ where the infimum is taken over all the measures $\tau=a\,\nu$, with $a\in L^\infty(\nu)$. We call such measures admissible. 
Since $\nu$ is admissible we infer that
\begin{equation}
\label{e.admis}
m \leq F(\nu) \leq 3\lambda\,\sigma_p(\sH).
\end{equation}
So the infimum $m$ is attained over the collection of
measures $\tau=a\nu$ such that $\|a\|_{L^\infty(\nu)}\leq 3$ and $
\tau(H_0)\geq\frac13\,\nu(H_0)$. In particular, by taking a weak $*$ limit in $L^\infty(\nu)$, this guaranties
the existence of a minimizer.

Let $\tau$ be an admissible measure such that 
\begin{equation}\label{e.admis2}
	m=F(\tau)\leq 3\lambda\,\sigma_p(\sH).
\end{equation}
To simplify notation, we denote
$$f(x) = G + c_3\,\Theta(\sH).$$
We claim that
\begin{equation}\label{eqclaim*}
\bigl|\bigl(|\RR\tau(x)| - f(x)\bigr)_+|^p + 
p \,\RR_\tau^*\Bigl[\bigl|\bigl(|\RR\tau| - f\bigr)_+|^{p-1} |\RR\tau|^{-1} \RR\tau \Bigr](x)
 \lesssim \lambda\,\Theta(\sH)^p \quad \text{\!\!in $\supp\tau$,}
\end{equation}
where for a vector field $U$, we wrote
$$\RR_\tau^*U =\RR^*(U\tau) = -\sum_{i=1}^{n+1}\RR_i(U_i\,\tau).$$
Assume \rf{eqclaim*} for the moment. Since the function on the left hand side is subharmonic
in $\R^{n+1}\setminus\supp\tau$ 
and continuous in $\R^{n+1}$ (recall that $\tau$ and $\eta$ are absolutely continuous with respect to Lebesgue measure, with bounded densities), we deduce that the same estimate holds in the whole $\R^{n+1}$.

Next we need to construct an auxiliary function $\vphi$. First we claim that 
there exists a subfamily of cubes $\sH^0\subset \sH$ such that
\begin{itemize}
\item[(i)] the balls $3B_Q\equiv3B_{Q^{(\mu)}}$, with $Q^{(\mu)}\in\sH^0$, are disjoint,
\item[(ii)] $\frac1{12}\nu(Q)\leq \tau(Q\cap H_0)\leq 3 \nu(Q)$ for all $Q^{(\mu)}\in\sH^0$, and
\item[(iii)] $\sum_{Q^{(\mu)}\in\sH^0}\tau(Q\cap H_0)\approx \nu(H)$.
\end{itemize}
The existence of the family $\sH^0$ follows easily from the fact that
 $\tau(H_0)\geq\frac13\nu(H_0)\geq \frac16\,\nu(H)$ and $\tau=a\,\nu$ with $\|a\|_{L^\infty(\nu)}\leq 3$. Indeed,
 notice that the family $I$ of cubes $Q^{(\mu)}\in \sH$ such that 
$\tau(Q\cap H_0)\leq \frac1{12}\nu(Q)$ satisfies
$$\sum_{Q^{(\mu)}\in I}\tau(Q\cap H_0)\leq \frac1{12}\sum_{Q^{(\mu)}\in I}\nu(Q)\leq \frac1{12}\,\nu(H)
\leq \frac12\,\tau(H_0).$$
So
$$\sum_{Q^{(\mu)}\in \sH\setminus I} \tau(Q)\geq \sum_{Q^{(\mu)}\in \sH\setminus I} \tau(Q\cap H_0)\geq \frac12\,\tau(H_0).$$
By Vitali's 5r-covering lemma, there exists a subfamily $ \sH^0\subset \sH\setminus I$ 
such that the balls $\{3B_{Q}\}_{Q^{(\mu)} \in\sH^0}$ are disjoint and the balls $\{15B_{Q}\}_{Q^{(\mu)}\in\sH^0}$ cover
$\bigcup_{Q^{(\mu)}\in \sH\setminus I}Q$. From the fact that the cubes from $\sH^0$ are $\PP$-doubling and the properties of the family $\Reg(e'(R))$ in Lemma \ref{lem74}, we have
$\nu(Q) \approx \nu(15B_Q)$ if $Q^{(\mu)}\in\sH^0$, and thus
\begin{align*}
\sum_{Q^{(\mu)}\in\sH^0} \tau(Q\cap H_0) & \geq \frac1{12}\sum_{Q^{(\mu)}\in\sH^0} \nu(Q) \approx
\sum_{Q^{(\mu)}\in\sH^0} \nu(15B_Q) \\
& \geq \sum_{Q^{(\mu)}\in \sH\setminus I} \nu(Q)\geq \frac13\,
\sum_{Q^{(\mu)}\in \sH\setminus I} \tau(Q)\geq \frac16\,\tau(H_0)\geq \frac1{36}\,\nu(H).
\end{align*}
The converse estimate $\sum_{Q^{(\mu)}\in\sH^0}\tau(Q\cap H_0)\lesssim \nu(H)$ follows trivially from  $\|a\|_{L^\infty(\nu)}\leq3$.

We consider the function
\begin{equation}\label{eqvphi99}
\vphi = \sum_{Q^{(\mu)}\in\sH^0} \Theta(Q)\,\vphi_Q,
\end{equation}
where $\chi_{B_Q}\leq \vphi_Q\leq \chi_{1.1B_Q}$, $\|\nabla\vphi_Q\|_\infty\lesssim\ell(Q)^{-1}$. 
Remark that
$$\RR^*(\nabla\vphi\,\LL^{n+1}) = \tilde c\,\vphi,$$
where $\tilde c$ is some non-zero absolute constant.
Notice also that
\begin{equation}\label{eqnormpsi}
\|\nabla\vphi\|_1\lesssim \sum_{Q^{(\mu)}\in\sH^0} \Theta(Q)\,\ell(Q)^n\approx \tau(H_0).
\end{equation}
Multiplying the left hand side of \rf{eqclaim*} by $|\nabla\vphi|$, integrating with respect to the Lebesgue measure $\LL^{n+1}$, taking into account that the estimate in
\rf{eqclaim*} holds on the whole $\R^{n+1}$, and using \rf{eqnormpsi}, we get
\begin{multline}\label{eqpss3}
\int \bigl|\bigl(|\RR\tau| - f\bigr)_+|^p\,|\nabla\vphi|\,d\LL^{n+1} + p\! \int \RR_\tau^*\Bigl[\bigl|\bigl(|\RR\tau| - f\bigr)_+|^{p-1} |\RR\tau|^{-1} \RR\tau \Bigr]\,
|\nabla\vphi|\,d\LL^{n+1}\\
\lesssim \lambda\,\sigma_p(\sH).
\end{multline}

Next we estimate the second integral on the left hand side of the preceding inequality, which we denote by $I$.  For that purpose we will also need the estimate 
\begin{equation}\label{eqacpsi}
\int |\RR(|\nabla \vphi|\,d\LL^{n+1})|^p\,d\tau\lesssim \Lambda_*^{p\ve_n}\sigma_p(\sH),
\end{equation}
which we defer to Lemma \ref{lemtech79}.
Using this inequality we get
\begin{multline*}
|I| = \left|\int 
\bigl|\bigl(|\RR\tau| - f\bigr)_+|^{p-1} |\RR\tau|^{-1} \RR\tau \cdot
\RR(|\nabla \vphi|\,\LL^{n+1})\,d\tau\right| \\
\leq \left(\int \bigl|\bigl(|\RR\tau| - f\bigr)_+|^{p} \,d\tau\right)^{1/p'}
\left(\int |\RR(|\nabla \vphi|\,\LL^{n+1})|^p\,d\tau\right)^{1/p}\\
\le (F(\tau))^{1/p'}\left(\int |\RR(|\nabla \vphi|\,\LL^{n+1})|^p\,d\tau\right)^{1/p}\\
 \overset{\rf{e.admis2},\eqref{eqacpsi}}{\lesssim} \bigl (\lambda\,\sigma_p(\sH)\bigr)^{1/p'} \bigl (\Lambda_*^{p\ve_n}\sigma_p(\sH)\bigr)^{1/p} = \lambda^{1/p'}\,\Lambda_*^{\ve_n}\sigma_p(\sH).
\end{multline*}
From \rf{eqpss3} and the preceding estimate we deduce
that
$$\int \bigl|\bigl(|\RR\tau| - f\bigr)_+|^p\,|\nabla\vphi|\,d\LL^{n+1}\lesssim  \lambda^{1/p'}\,\Lambda_*^{\ve_n}
\sigma_p(\sH).
$$

Notice now that, for all $x\in \supp\vphi$,
$$G(x) = 2A_0^3\int_{\cS_\eta'\setminus V_2} 
 \frac1{\ell(R)\,|x-y|^{n-1}}\,d\eta(y) \lesssim \frac{\eta(\cS_\eta')}{\ell(R)\,\dist(\cS_\eta'\setminus V_2, \cS_\eta)^{n-1}} \lesssim \Theta(R),$$
and so
$$|f(x)|\leq C\, \Theta(R) + c_3\,\Theta(\sH)\leq 2c_3\,\Theta(\sH) \qquad\mbox{for all $x\in \supp\vphi$,}$$ 
taking into account that $\Lambda_*\gg1$.
So we have
\begin{align}\label{eqcont4}
\int |\RR\tau|^p\,|\nabla\vphi|\,d\LL^{n+1} &\lesssim \int \bigl|\bigl(|\RR\tau| - f\bigr)_+|^p\,|\nabla\vphi|\,d\LL^{n+1}
+ \int |f|^p\,|\nabla\vphi|\,d\LL^{n+1}\\
& \lesssim 
  \bigl(\lambda^{1/p'}\Lambda_*^{\ve_n} + c_3^p\bigr)\,\sigma_p(\sH).\nonumber
\end{align}

To get a contradiction, note that by the construction of $\vphi$ and by the properties of $\tau$, we have
\begin{multline}\label{eqprev}
\left|\int \RR\tau\cdot\nabla\vphi\,d\LL^{n+1}\right| = \left|\int \RR^*(\nabla\vphi\,\LL^{n+1})\,d\tau \right|=  
|\tilde c| \int \vphi\,d\tau\\
\ge |\tilde c|\,
\Theta(\sH) \sum_{Q^{(\mu)}\in\sH^0} \tau(B_Q) \gtrsim |\tilde c|\,
\Theta(\sH) \tau(H_0)\geq \frac{|\tilde c|}6\,
 \Theta(\sH) \nu(H).
 \end{multline}
However, by H\"older's inequality, \eqref{eqnormpsi}, and \rf{eqcont4}, we have
\begin{align*}
\left|\int \RR\tau\cdot\nabla\vphi\,d\LL^{n+1}\right| & \leq \left(\int |\RR\tau|^p \,|\nabla\vphi|\,d\LL^{n+1} \right)^{1/p}
\left(\int |\nabla\vphi|\,d\LL^{n+1}\right)^{1/p'} \\
&\lesssim \bigl(\lambda^{1/(p\,p')} \Lambda_*^{\ve_n/p}+ c_3\bigr)
\, \Theta(\sH)\nu(H),
\end{align*}
which contradicts \rf{eqprev} if $c_3$ is chosen small enough and 
$$\lambda\leq c\,\Lambda_*^{-p'\ve_n},$$
with $c$ small enough.
\end{proof}

\vv

\begin{proof}[\bf Proof of \rf{eqclaim*}]
Recall that $\tau=a\,\nu$ is a minimizer for $F(\cdot)$ that in particular satisfies
$F(\tau)\leq 3\lambda\,\sigma_p(\sH),$ by \rf{e.admis2}.
We have to show that, for all $x\in\supp\tau$,
$$
\bigl|\bigl(|\RR\tau(x)| - f(x)\bigr)_+|^p + 
p \,\RR_\tau^*\Bigl[\bigl|\bigl(|\RR\tau| - f\bigr)_+|^{p-1} |\RR\tau|^{-1} \RR\tau \Bigr](x)
 \lesssim \lambda\,\Theta(\sH)^p.
$$

Let $x_0\in\supp\tau$ 
and let $B=B(x_0,\rho)$, with $\rho>0$ small. For $0<t<1$ we consider the competing measure $\tau_t =
a_t\,\tau$, where $a_t$ is defined as follows:
$$a_t = (1-t\,\chi_B)a.$$
Notice that, for each $0<t<1$, $a_t$ is a non-negative function such that $$\|a_t\|_{L^\infty(\nu)}\leq \|a\|_{L^\infty(\nu)}.$$
Taking the above into account we deduce that
\begin{align*}
F(\tau_t) & = 
\int \bigl|\bigl(|\RR\tau_t| - f\bigr)_+\bigr|^p\,d\tau_t
 + \lambda\,\|a_t\|_{L^\infty(\nu)} \,\sigma_p(\sH) + \lambda \,\frac{\nu(H_0)}{\tau_t(H_0)}\,\sigma_p(\sH)\\
& \leq
\int \bigl|\bigl(|\RR\tau_t| - f\bigr)_+\bigr|^p\,d\tau_t
 + \lambda\,\|a\|_{L^\infty(\nu)} \,\sigma_p(\sH) + \lambda \,\frac{\nu(H_0)}{\tau_t(H_0)}\,\sigma_p(\sH)
\\& =:\wt F(\tau_t).
\end{align*}
Since $\wt F(\tau) = F(\tau) \leq F(\tau_t)\leq \wt F(\tau_t)$ for $t\geq 0$, we infer that 
\begin{equation}\label{eqvar49}
\frac1{\tau(B)}\,\frac{d}{dt}\,\wt F(\tau_t)\biggr|_{t=0+} \geq 0.
\end{equation}
Using that
$$
\frac d{dt}|\RR\tau_t(x)|\biggr|_{t=0} = \frac1{|\RR\tau(x)|} \RR\tau(x)\cdot \RR (-\chi_B\tau )(x)
,$$
an easy computation gives 
\begin{align*}
\frac{d}{dt}\,\wt F(\tau_t)\biggr|_{t=0} & = -\int_B \bigl|\bigl(|\RR\tau| - f\bigr)_+|^p \,d\tau\\
& \quad +
p \int \bigl|\bigl(|\RR\tau| - f\bigr)_+|^{p-1}\, \frac1{|\RR\tau|} \RR\tau\cdot \RR (-\chi_B\tau)\,
d\tau \\
&\quad +  
\lambda \,\frac{\nu(H_0)}{\tau(H_0)^2}\,\sigma_p(\sH)\,\tau(B).
\end{align*}
Recalling that $\tau(H_0)\geq\frac13\,\nu(H_0)\ge \frac16\,\nu(H)$, from \rf{eqvar49} and the preceding calculation we derive
\begin{multline}\label{eqmult82}
 \frac1{\tau(B)} 
\int_B \bigl|\bigl(|\RR\tau| - f\bigr)_+|^p \,d\tau 
 +
 \frac p{\tau(B)} \int \bigl|\bigl(|\RR\tau| - f\bigr)_+|^{p-1}\, \frac1{|\RR\tau|} \RR\tau\cdot \RR (\chi_B\tau )
\,
d\tau
 \\ 
\leq 
3\lambda \,\frac{\sigma_p(\sH)}{\tau(H_0)} =  3\lambda\, \Theta(\sH)^p \,\frac{\nu(H)}{\tau(H_0)} \le 18 \lambda \, \Theta(\sH)^p.
\end{multline}

 We rewrite the left hand side of \rf{eqmult82} as
$$ 
 \frac 1{\tau(B)} \int_B \bigl|\bigl(|\RR\tau| - f\bigr)_+|^p \,d\tau 
 +
 \frac p{\tau(B)} \int_B \RR_\tau^*\Bigl[\bigl|\bigl(|\RR\tau| - f\bigr)_+|^{p-1}\, |\RR\tau|^{-1} \RR\tau \Bigr]
\,
d\tau
.$$
Taking into account that the functions in the integrands are continuous on $\supp(\tau)$, 
letting the radius $\rho$ of $B$ tend to $0$, it turns out that the above expression converges to
$$\bigl|\bigl(|\RR\tau(x_0)| - f(x_0)\bigr)_+|^p 
 +
p\, \RR_\tau^*\Bigl[\bigl|\bigl(|\RR\tau| - f\bigr)_+|^{p-1}\, |\RR\tau|^{-1} \RR\tau \Bigr](x_0).$$
The desired estimate \rf{eqclaim*} follows from the above and \rf{eqmult82}.
\end{proof}
\vv

In order to complete the proof of Lemma \ref{lemvar} we need the following technical result.

\begin{lemma}\label{lemtech79}
Suppose that \rf{eqsupp1} holds with $\lambda\leq1$ and let $\vphi$ be as in \rf{eqvphi99}.
Then,
$$
\int |\RR(|\nabla \vphi|\,d\LL^{n+1})|^p\,d\nu\lesssim \Lambda_*^{p\ve_n}\sigma_p(\sH),
$$
\end{lemma}

\begin{proof}
Recall that 
$$\vphi = \sum_{Q^{(\mu)}\in\sH^0} \Theta(Q)\,\vphi_Q,$$
with $\chi_{B_Q}\leq \vphi_Q\leq \chi_{1.1B_Q}$, $\|\nabla\vphi_Q\|_\infty\lesssim\ell(Q)^{-1}$.
We consider the function
$$g= \sum_{Q^{(\mu)}\in\sH^0} g_Q,$$
where $g_Q = c_Q\,\chi_{Q\cap H_0}$, with $c_Q=\Theta(Q)\int|\nabla\vphi_Q|\,d\LL^{n+1}\,\eta(Q\cap H_0)^{-1}$.
Observe that
\begin{equation}\label{eq:gQintegral}
\int g_Q\,d\eta = \Theta(Q) \int|\nabla\vphi_Q|\,d\LL^{n+1}
\end{equation}
and that
$$0\leq c_Q \lesssim \frac{\Theta(Q) \,\ell(Q)^n}{\eta(Q\cap H_0)}\approx \frac{\mu(Q^{(\mu)})}{\eta(Q)} =1.$$

The first step of our arguments consists in comparing $\RR(|\nabla\vphi|\,d\LL^{n+1})(x)$ to 
$\RR_{\Psi(x)}(g\,\eta)(x)$, with $\Psi$ given by \rf{eqdefPsi1}.
 We will prove that, for each $Q^{(\mu)}\in \sH^0$,
 \begin{align}\label{eqtech4}
|\RR(\Theta(Q)|\nabla\vphi_Q|&\,\LL^{n+1})(x) - \RR_{\Psi(x)}(g_Q\,\eta)(x)|\\
& \lesssim \frac{\Theta(Q)\,\ell(Q)^{n+1}}{\dist(x,Q)^{n+1}+
\ell(Q)^{n+1}} + \chi_{2B_Q}(x)\, |\RR_{\Psi(x)}(g_Q\,\eta)(x)|\nonumber \\
&\quad+ \int_{c\Psi(x)\leq |x-y|\leq\Psi(x)} \frac{g_Q(y)}{|x-y|^{n}}\,d\eta(y),\nonumber
\end{align}
for some fixed $c>0$.
 The arguments to prove this estimate are quite standard.
 
 Suppose first that $x\in 2B_Q$. In this case, we have
 \begin{align*}
|\RR(\Theta(Q)|\nabla\vphi_Q|\,\LL^{n+1})(x)| &\lesssim \Theta(Q)\int_{1.1B_Q} \frac{1}{\ell(Q)\,|x-y|^n}\,d\LL^{n+1}(y) \lesssim \Theta(Q),
\end{align*}
which yields
$$|\RR(\Theta(Q)|\nabla\vphi_Q|\,\LL^{n+1})(x) - \RR_{\Psi(x)}(g_Q\,\eta)(x)|
 \lesssim \Theta(Q) + |\RR_{\Psi(x)}(g_Q\,\eta)(x)|,$$
 and shows that \rf{eqtech4} holds in this situation.
 
In the case $x\not\in 2B_Q$ we write
\begin{align*}
\RR(\Theta(Q)|&\nabla\vphi_Q|\,\LL^{n+1})(x) - \RR_{\Psi(x)}(g_Q\,\eta)(x)\\  &=
\RR\big(\Theta(Q)|\nabla\vphi_Q|\,\LL^{n+1} - g_Q\,\eta\big)(x) +
\int_{|x-y|\leq\Psi(x)} \frac{x-y}{|x-y|^{n+1}} g_Q(y)\,d\eta(y)\\
& \overset{\eqref{eq:gQintegral}}{=} 
\int \left(\frac{x-y}{|x-y|^{n+1}} - \frac{x-x_Q}{|x-x_Q|^{n+1}}\right) \bigl[\Theta(Q)|\nabla\vphi_Q(y)|\,d\LL^{n+1}(y) - g_Q(y)\,d\eta(y)\bigr]\\
& \quad\quad\quad + \int_{|x-y|\leq\Psi(x)} \frac{x-y}{|x-y|^{n+1}} g_Q(y)\,d\eta(y)\\
& =: I_1(x)+ I_2(x).
\end{align*}
Concerning the term $I_1(x)$, recalling that $\supp\vphi_Q \cup\supp g_Q\subset 1.1\overline{ B_Q}$,
we obtain
\begin{align*}
|I_1(x)|
&\lesssim \int \frac{\ell(Q)}{|x-x_Q|^{n+1}} \bigl[\Theta(Q)|\nabla\vphi_Q(y)|\,d\LL^{n+1}(y) + g_Q(y)\,d\eta(y)\bigr]
\lesssim \frac{\ell(Q)}{|x-x_Q|^{n+1}}\,\eta(Q).
\end{align*}
Regarding $I_2(x)$, we write
$$
I_2(x)\leq \int_{y\in Q:|x-y|\leq\Psi(x)} \frac1{|x-y|^{n}} \,g_Q(y)\,d\eta(y).$$
Notice that for $y\in Q$, since $x\notin 2B_Q$, we have
$$C|x-y|\geq \ell(Q)\overset{\eqref{eqdefPsi1}}{\ge}\Psi(y) \geq \Psi(x) - |x-y|.$$
Thus, $|x-y|\geq c\,\Psi(x)$, and so
$$I_2(x)\leq \int_{c\Psi(x)\leq |x-y|\leq\Psi(x)} \frac1{|x-y|^{n}}\, g_Q(y)\,d\eta(y).
$$
Gathering the estimates for $I_1(x)$ and $I_2(x)$ we see that \rf{eqtech4} also holds in this case.

From \rf{eqtech4} we infer that
\begin{align*}
|\RR(|\nabla\vphi|\,\LL^{n+1})&(x) - \RR_{\Psi(x)}(g\,\eta)(x)|\\
& \lesssim \sum_{Q^{(\mu)}\in \sH^0}\frac{\Theta(Q)\,\ell(Q)^{n+1}}{\dist(x,Q)^{n+1}+
\ell(Q)^{n+1}} + 
\sum_{Q^{(\mu)}\in \sH^0} \chi_{2B_Q}(x)\, |\RR_{\Psi(x)}(g_Q\,\eta)(x)|\\
&\quad + \int_{c\Psi(x)\leq |x-y|\leq\Psi(x)} \frac1{|x-y|^{n}}\, g(y)\,d\eta(y)\\
& = S_1(x) + S_2(x) + S_3(x).
\end{align*}
We split
\begin{equation}\label{eqs123}
\int |\RR(|\nabla\vphi|\,\LL^{n+1})(x) - \RR_{\Psi(x)}(g\,\eta)(x)|^p\,d\eta(x)
\lesssim \sum_{i=1}^3 \int |S_i(x)|^p\,d\eta(x).
\end{equation}

We estimate the first summand by duality. Consider a function $h\in L^{p'}(\eta)$. Then, 
\begin{equation}\label{eqmul429**}
\int S_1(x)\,h(x)\,d\eta(x)= 
  \sum_{Q^{(\mu)}\in \sH^0} \eta(Q) \int\frac{\ell(Q)}{\dist(x,Q)^{n+1}+
\ell(Q)^{n+1}}\,h(x)\,d\eta(x).
\end{equation}
For each $Q^{(\mu)}\in\sH^0$,
 using the fact that $\eta(\lambda Q)\leq C\,\Theta(\sH)\,\ell(\lambda Q)^n$ for every $\lambda\geq 1$,
 we obtain 
$$\int  \frac{\ell(Q)}{\dist(x,Q)^{n+1}+
\ell(Q)^{n+1}}\,h(x)\,d\eta(x)\lesssim C\,\Theta(\sH)\,\inf_{y\in Q} \cM_\eta h(y).$$
Therefore, the right side of \rf{eqmul429**} does not exceed
\begin{align*}
C\Theta(\sH)\sum_{Q^{(\mu)}\in \sH^0} \eta(Q)\,\inf_{y\in Q} \cM_\eta h(y) & \lesssim
\Theta(\sH) \int_{H}\cM_\eta h(y)\,d\eta(y) \\ 
& \le \Theta(\sH)\,\eta(H)^{1/p}\,\|\cM_\eta h\|_{L^{p'}(\eta)}
\lesssim \Theta(\sH)\,\eta(H)^{1/p}\,\|h\|_{L^{p'}(\eta)}.
\end{align*}
So we deduce that
$$\int |S_1(x)|^p\,d\eta(x)\lesssim \sigma_p(\sH).
$$

Regarding the summand in \rf{eqs123} involving $S_2$, since the balls $2B_Q$ are disjoint, we have
$$\int |S_2(x)|^p\,d\eta(x) = 
\sum_{Q^{(\mu)}\in \sH^0} \int_{2B_Q} |\RR_{\Psi(x)}(g_Q\,\eta)(x)|^p\,d\eta(x) 
\leq \sum_{Q^{(\mu)}\in \sH^0}\| \RR_{\Psi(\cdot)}(g_Q\,\eta)\|_{L^p(\eta)}^p,$$
where $\RR_{\Psi(\cdot)}(g_Q\,\eta)(x) = \RR_{\Psi(x)}(g_Q\,\eta)(x)$.
Finally, to estimate the last summand in \rf{eqs123}, we take into account that 
$|S_3(x)|\lesssim \cM_{\Psi,n}(g\,\eta)(x),$ where $\cM_{\Psi,n}(g\,\eta)$ is the maximal operator from \eqref{eq:maximaloppsi}.
Hence,
$$\int |S_3(x)|^p\,d\eta(x)  \lesssim \int | \cM_{\Psi,n}(g\,\eta)|^p\,d\eta \lesssim \Theta(\sH)^p\,
\|g\|_{L^p(\eta)}^p \lesssim \sigma_p(\sH).$$

Gathering the estimates obtained for $S_1$, $S_2$, $S_3$, by \rf{eqs123} we get
$$\|\RR(|\nabla\vphi|\,\LL^{n+1})\|_{L^p(\eta)}^p \lesssim \sigma_p(\sH) +\sum_{Q^{(\mu)}\in \sH^0}
\|\RR_{\Psi(\cdot)}(g_Q\,\eta)\|_{L^p(\eta)}^p + \|\RR_{\Psi(\cdot)}(g\eta)\|_{L^p(\eta)}^p.$$
From \rf{e.compsup''}, we deduce that
\begin{align*}
\|\RR(|\nabla\vphi|\,\LL^{n+1})\|_{L^p(\eta)}^p & \lesssim \sigma_p(\sH) +\sum_{Q^{(\mu)}\in \sH^0}
\|\RR_{\Psi}(g_Q\,\eta)\|_{L^p(\eta)}^p + \|\RR_{\Psi}(g\eta)\|_{L^p(\eta)}^p\\
&\quad + \sum_{Q^{(\mu)}\in \sH^0}
\|\cM_{\Psi,n}(g_Q\,\eta)\|_{L^p(\eta)}^p  + \|\cM_{\Psi,n}(g\,\eta)\|_{L^p(\eta)}^p.
\end{align*}
Using now that, by Lemma \ref{lemH0}, 
 $\RR_{\Psi,\eta}$ is bounded from $L^p(\eta\rest_{H_0})$ to $L^p(\eta\rest_{\cS'_\eta})$ with 
$$\|\RR_\Psi\|_{L^p(\eta\rest_{H_0})\to L^p(\eta\rest_{\cS'_\eta})}\lesssim \Lambda_*^{\ve_n}\Theta(\sH),$$
and that the same happens for $\cM_{\Psi,n}$, we get
$$\|\RR(|\nabla\vphi|\,\LL^{n+1})\|_{L^p(\eta)}^p \lesssim \sigma_p(\sH) + \Lambda_*^{p\ve_n}\Theta(\sH)^p\,\|g\|_{L^p(\eta)}^p \lesssim \Lambda_*^{p\ve_n}\sigma_p(\sH).$$
\end{proof}

\vv

\subsection{Lower estimates for \texorpdfstring{$\RR\eta$}{R\_eta}}

\begin{lemma}\label{lemrieszeta}
Let $R\in\MDW$ be such that $R\in\Trc$ and let $V_4$ and $\eta$ be as in Section \ref{subsec9.5}.
Assume $\Lambda_*>0$ big enough and let $c_3$ be as in Lemma \ref{lemvar}.
Then we have
$$\int_{V_4} \big|(|\RR\eta(x)| - \frac{c_3}2\,\Theta(\HD_1))_+\big|^p\,d\eta(x) \gtrsim \Lambda_*^{-p'\ve_n}\sigma_p(\HD_1(e(R))),$$
for any $p\in (1,\infty)$, with the implicit constant depending on $p$.
\end{lemma}

\begin{proof}
For all $x\in \supp\eta$, using that $\|\nabla\vphi_R\|_\infty\leq 2A_0^{-3}\ell(R)^{-1}$, we obtain
\begin{align*}
\big|\vphi_R(x)\,\RR\eta(x) - \RR\nu(x)\big| &= \big|\vphi_R(x)\,\RR\eta(x) - \RR(\vphi_R\eta)(x)\big|\\
& = \left|\int \frac{(\vphi_R(x) - \vphi_R(y))\,(x-y)}{|x-y|^{n+1}}\,d\eta(y)\right|\\
& \leq 2A_0^3\int_{y\in \cS_\eta\setminus V_2} \frac1{\ell(R)\,|x-y|^{n-1}}\,d\eta(y)\\
&\quad +
2A_0^3\int_{y\in V_2:\vphi_R(y)\neq \vphi_R(x)} \frac1{\ell(R)\,|x-y|^{n-1}}\,d\eta(y).
\end{align*}
Recall that $\vphi_R$ equals $1$ on $V_3$ and vanishes out of $V_4$.
Then it is clear that the last integral on the right hand side above vanishes if $x\in V_3$ and
it does not exceed $C\,\Theta(R)$ if $x\in V_3^c$.
So in any case
$$\big|\vphi_R(x)\,\RR\eta(x) - \RR\nu(x)\big|\leq G(x) + C_5\,\Theta(R).$$

From the preceding estimate we infer that
\begin{align*}
|\vphi_R(x)\,\RR\eta(x)| - \frac{c_3}2\,\Theta(\HD_1) & \geq 
|\RR\nu(x)| - G(x) - C_5\,\Theta(R)
- \frac{c_3}2\,\Theta(\HD_1)\\
& \geq |\RR\nu(x)| - G(x) 
- c_3\,\Theta(\HD_1).
\end{align*}
Therefore, since $|(\;\cdot\;)_+|^p$ is non-decreasing,
\begin{align*}
\int_{V_4} \big|(|\RR\eta| - \frac{c_3}2\,\Theta(\HD_1))_+\big|^p\,d\eta & 
\geq \int \big|(|\vphi_R\,\RR\eta| - \frac{c_3}2\,\Theta(\HD_1))_+\big|^p\,d\eta\\
& \geq \int  \big|(|\RR\nu| - G 
- c_3\,\Theta(\HD_1))_+\big|^p\,d\eta\\
& \geq \int  \big|(|\RR\nu| - G 
- c_3\,\Theta(\sH))_+\big|^p\,d\nu.
\end{align*}
By Lemma \ref{lemvar}, 
$$\int  \big|(|\RR\nu| - G 
- c_3\,\Theta(\sH))_+\big|^p\,d\nu\geq c \Lambda_*^{-p'\ve_n}\sigma_p(\sH)\geq c\Lambda_*^{-p'\ve_n}\sigma_p(\HD_1(e(R))),$$
and so the lemma follows.
\end{proof}
\vv

%% file: riesz-DX7-S7-2.tex

\section{The Riesz transform on the tractable trees: transference estimates}\label{sec99}

 In all this section we assume that $R\in\MDW$ is such that $R\in\Trc$, i.e., $\TT(e'(R))$ is tractable, and we consider the measure $\eta$ constructed in the previous section.
Essentially, our objective is to transfer the lower estimate we obtained for $\RR\eta$ in Lemma \ref{lemrieszeta} to 
the Haar coefficients of
$\RR\mu$ associated with cubes close to $\TT(e'(R))$.

\subsection{The operators \texorpdfstring{$\RR_{\TT_\Reg}$, $\RR_{\wt \TT}$,  and $\Delta_{\wt \TT}\RR$}{R\_{TReg}, R\_{T},  and Delta\_{T}R}}\label{sec7.1}

To simplify notation, in this section we will write
$$\End=\End(e'(R)), \quad \;\Reg = \Reg(e'(R)),\; \quad \;\TT = \TT(e'(R)),\quad\,\text{and}\quad\; \TT_\Reg = \TT_\Reg(e'(R)).$$
We need to consider an enlarged version of the generalized tree $\TT$, due to some technical difficulties that arise because the cubes from $\Neg:= \Neg(e'(R))\cap\End$ are not $\PP$-doubling.
To this end, denote by $\Reg_\Neg$ the family of the cubes from $\Reg$ which are contained in some cube from $\Neg$.
Let $\sD_\Neg$ be the subfamily of the cubes $P\in\Reg_\Neg$ for which there exists some
$\PP$-doubling cube $S\in\TT_\Reg$ that contains $P$. By the definition of $\Neg$, such cube $S$ is contained in the cube from $Q\in\Neg$ such that $P\subset Q$. 
We also denote by $\sM_\Neg$ the family of maximal $\PP$-doubling cubes which belong to $\TT_\Reg$ and are contained in some
cube from $\Neg$, so that, in particular, any cube from $\sD_\Neg$ is contained in some cube from $\sM_\Neg$.

We define
$$\wt\End = (\End \setminus \Neg) \cup \sM_\Neg,$$
and we let $\wt \TT=\wt \TT(e'(R))$ be the family of cubes which belong to $\TT_\Reg$ and are not strictly contained in any cube from $\wt\End$.
Further, we write
\begin{equation}\label{eqdef*f}
Z = Z(e'(R)) = e'(R)\setminus \bigcup_{Q\in\End} Q\quad \text{ and }\quad\wt Z = \wt Z(e'(R)) = e'(R)\setminus \bigcup_{Q\in\wt\End} Q
\end{equation}
Then we denote
$$\RR_{\TT_\Reg}\mu(x) = \sum_{Q\in\Reg} \chi_Q(x)\,\RR(\chi_{2R\setminus 2Q}\mu)(x),$$
$$\RR_{\wt \TT}\mu(x) = \sum_{Q\in\wt\End} \chi_Q(x)\,\RR(\chi_{2R\setminus 2Q}\mu)(x),$$
and
$$\Delta_{\wt\TT}\RR\mu(x) = \sum_{Q\in\wt\End} \chi_Q(x)\,\big(m_{\mu,Q}(\RR\mu) - m_{\mu,2R}(\RR\mu)\big)
+ \chi_{Z}(x) \big(\RR\mu(x) -  m_{\mu,2R}(\RR\mu)\big)
.$$
Here $\RR\mu(x)$ is understood in the principal value sense (which exists $\mu$-a.e.\ because we assume that
$\mu$ has polynomial growth and that $\RR_*\mu<\infty$ $\mu$-a.e.

Remark that the cubes from $\Reg$ have the advantage over the cubes from $\wt \End$ that their size changes smoothly so that, for example, neighboring cubes have comparable side lengths. However, they need not be $\PP$-doubling or doubling, unlike the cubes from $\wt \End$. 

We define
$$\QQ_\Reg(Q) = \sum_{P\in\Reg} \frac{\ell(P)}{D(P,Q)^{n+1}}\,\mu(P),$$
where 
$$D(P,Q) = \ell(P) + \dist(P,Q) + \ell(Q).$$
The coefficient $\QQ_\Reg(Q)$ will be used to bound some ``error terms'' in our transference 
arguments. We will see later how they can be estimated in terms of the coefficients $\PP(Q)$ by duality. Notice that, unlike $\PP(Q)$, the coefficient $\QQ_\Reg(Q)$ depends on the family
$\Reg$.

\vv

\begin{lemma}\label{lemaprox1}
For any $Q\in\Reg$ such that $(Q\cup \frac12 B(Q))\cap V_4\neq\varnothing$ and $x\in Q$, $y\in \frac12B(Q)$,
$$\big|\RR_{\TT_\Reg}\mu(x) - \RR\eta(y)\big| \lesssim \Theta(R) + \PP(Q) + \QQ_\Reg(Q).$$
\end{lemma}

\begin{proof}
By the triangle inequality, for $x$, $y$ and $Q$ as in the lemma, we have
\begin{align*}
\big|\RR_{\TT_\Reg}\mu(x) - \RR\eta(y)\big| & \leq \big|\RR_\mu\chi_{2R\setminus e'(R)}(x)\big| 
+ \big|\RR_\mu\chi_{e'(R)\setminus 2Q}(x) - \RR_\mu\chi_{e'(R)\setminus 2Q}(x_Q)\big|\\
& \quad 
+ \big|\RR_\mu\chi_{2Q\setminus Q}(x_Q)\big| 
+\big|\RR_\mu\chi_{e'(R)\setminus Q}(x_Q) - \RR_\eta \chi_{(\frac12B(Q))^c}(x_Q)\big|\\
&\quad + \big|\RR_\eta \chi_{(\frac12B(Q))^c}(x_Q) - \RR_\eta \chi_{(\frac12B(Q))^c}(y)\big| + \big|\RR_\eta \chi_{\frac12B(Q)}(y)\big|\\
& = I_1+ \ldots + I_6.
\end{align*}
To estimate $I_1$, notice that, by \rf{eqtec732}, $\dist(Q,\supp\mu\setminus e'(R))\gtrsim\ell(R)$.
Then it follows that
$$I_1\lesssim \frac{\mu(2R)}{\ell(R)^n} \lesssim \Theta(R).$$
Arguing similarly one also gets 
\begin{equation*}
	I_3\lesssim \frac{\mu(2Q)}{\ell(Q)^n}\lesssim \PP(Q).
\end{equation*}
By standard arguments involving continuity of the Riesz kernel, we also deduce that
$$I_2\lesssim\PP(Q),\quad I_5\lesssim\PP(Q).$$
Concerning the term $I_6$, using that $\eta$ is a constant multiple of $\LL^{n+1}$ on 
$\frac12B(Q)$, we get
$$I_6\leq \int_{\frac12B(Q)}\frac1{|x-y|^n}\,d\eta(y)\lesssim \frac{\eta(\tfrac12B(Q))}{\ell(Q)^n}
\lesssim \PP(Q).$$

Finally we deal with the term $I_4$. To this end, recall that $\mu(P) = \eta(\frac12B(P))$ for all 
$P\in \Reg$, and so
\begin{align*}
I_4 & \leq \sum_{P\in\Reg:P\neq Q} \left| \int K(x_Q-z)\,d\big(\eta\rest_{\frac12B(P)} - \mu\rest_P\big)(z)\right|\\
& \leq \sum_{P\in\Reg:P\neq Q}  \int |K(x_Q-z)-K(x_Q-x_P)|\,d\big(\eta\rest_{\frac12B(P)} + \mu\rest_P\big)(z)\\
& \lesssim \sum_{P\in\Reg:P\neq Q} \frac{\ell(P)}{|x_Q-x_P|^{n+1}}\,\mu(P)\approx \sum_{P\in\Reg:P\neq Q} \frac{\ell(P)}{D(Q,P)^{n+1}}\,\mu(P) \leq \QQ_\Reg(Q). 
\end{align*}
For the estimates in the last line we took into account the properties of the family $\Reg$ described 
in Lemma \ref{lem74}.
Gathering the estimates above, the lemma follows.
\end{proof}

\vv
\begin{lemma}\label{lemaprox2}
For any $Q\in\wt\End$ and $x,y\in Q$,
$$\big|\RR_{\wt\TT}\mu(x) - \Delta_{\wt\TT}\RR\mu(y)\big| \lesssim \PP(R) + \left(\frac{\EE(4R)}{\mu(R)}\right)^{1/2} + \PP(Q) +  \left(\frac{\EE(2Q)}{\mu(Q)}\right)^{1/2}.$$
\end{lemma}

\begin{proof}
For $Q\in\wt\End$ and $x,y\in Q$, we have
\begin{align}\label{eqal842}
\RR_{\wt\TT}\mu(x) - \Delta_{\wt\TT}\RR\mu(y) & = \RR(\chi_{2R\setminus 2Q}\mu)(x) - 
\big(m_{\mu,Q}(\RR\mu) - m_{\mu,2R}(\RR\mu)\big)\\
& = \big(\RR(\chi_{2R\setminus 2Q}\mu)(x) - m_{\mu,Q}(\RR\chi_{2R\setminus 2Q}\mu) \big) - m_{\mu,Q}(\RR(\chi_{2Q}\mu)))\nonumber\\
&\quad -
\big(m_{\mu,Q}(\RR(\chi_{2R^c}\mu))- m_{\mu,2R}(\RR\mu)\big).\nonumber
\end{align}
To estimate the first term on the right hand side, notice that for all $x,z\in Q$, by standard arguments we have
$$\big|\RR(\chi_{2R\setminus 2Q}\mu)(x) - \RR(\chi_{2R\setminus 2Q}\mu)(z)\big|\lesssim \PP(Q).$$
Averaging over $z\in Q$, we get
$$\big|\RR(\chi_{2R\setminus 2Q}\mu)(x) - m_{\mu,Q}(\RR\chi_{2R\setminus 2Q}\mu) \big|
\lesssim \PP(Q).$$
To estimate $m_{\mu,Q}(\RR(\chi_{2Q}\mu))$ we use the fact that, by the antisymmetry of the Riesz kernel, $m_{\mu,Q}(\RR(\chi_{Q}\mu))=0$, and thus
$$\big|m_{\mu,Q}(\RR(\chi_{2Q}\mu))\big| = \big|m_{\mu,Q}(\RR(\chi_{2Q\setminus Q}\mu))\big| 
\leq \avint_{z\in Q}\int_{\xi\in 2Q\setminus Q}\frac1{|z-\xi|^n}\,d\mu(\xi)d\mu(z).$$
By Cauchy-Schwarz and Lemma \ref{lemDMimproved}, we have
\begin{multline}\label{eqboundar3}
\avint_{z\in Q}\int_{\xi\in 2Q\setminus Q}\frac1{|z-\xi|^n}\,d\mu(\xi)d\mu(z)\\
= \avint_{z\in Q}\int_{\xi\in 2Q\setminus 2B_Q}\frac1{|z-\xi|^n}\,d\mu(\xi)d\mu(z) + \avint_{z\in Q}\int_{\xi\in 2B_Q\setminus Q}\frac1{|z-\xi|^n}\,d\mu(\xi)d\mu(z)\\
\leq
  \PP(Q)+\bigg(\avint_{z\in Q}\bigg(\int_{\xi\in 2B_Q\setminus Q}\frac1{|z-\xi|^n}\,d\mu(\xi)\bigg)^2d\mu(z)\bigg)^{1/2}
   \lesssim \PP(Q) + \left(\frac{\EE(2Q)}{\mu(Q)}\right)^{1/2}.
\end{multline}

Finally we deal with the last term on the right hand side of \rf{eqal842}. We split it:
\begin{align*}
\big|m_{\mu,Q}(\RR(\chi_{2R^c}\mu))- m_{\mu,2R}(\RR\mu)\big|
& \leq \big|m_{\mu,Q}(\RR(\chi_{3R\setminus 2R}\mu))| + \big|m_{\mu,2R}(\RR(\chi_{3R}\mu))\big|\\
&\quad +
\big|m_{\mu,Q}(\RR(\chi_{3R^c}\mu))- m_{\mu,2R}(\RR(\chi_{3R^c}\mu))\big| \\
& = J_1 + J_2 + J_3.
\end{align*}
Clearly,
$$J_1\lesssim \frac{\mu(3R)}{\ell(R)^n}\lesssim \PP(R).$$
Also, by the antisymmetry of the Riesz kernel and arguing as in \rf{eqboundar3},
\begin{align*}
J_2=& \big|m_{\mu,2R}(\RR(\chi_{3R}\mu))\big| = \big|m_{\mu,2R}(\RR(\chi_{3R\setminus 2R}\mu))\big| 
\\
& \leq  \avint_{z\in 2R}\int_{\xi\in 3R\setminus 2R}\frac1{|z-\xi|^n}\,d\mu(\xi)d\mu(z)
\lesssim \PP(R)+\left(\frac{\EE(4R)}{\mu(R)}\right)^{1/2}.
\end{align*}
Regarding $J_3$, by standard methods, for all $z,z'\in 2Q$, 
$$\big|\RR(\chi_{3R^c}\mu)(z) - \RR(\chi_{3R^c}\mu)(z')\big| \lesssim \PP(R).$$
Hence averaging for $z\in Q$, $z'\in 2R$, we obtain
$$J_3\lesssim \PP(R).$$
\end{proof}

\vv

\begin{lemma}\label{lemaprox3}
Let $1<p\leq2$.
For any $Q\in\wt\End$ such that $\ell_0\leq \ell(Q)$,
$$\int_Q \big|\RR_{\wt\TT}\mu - \RR_{\TT_\Reg}\mu\big|^p\,d\mu \lesssim \EE(2Q)^{\frac p2} \,\mu(Q)^{1-\frac p2}.$$
\end{lemma}

\begin{proof}
The lemma follows from H\"older's inequality and the estimate
\begin{equation}\label{eqf932}
\int_Q \big|\RR_{\wt\TT}\mu - \RR_{\TT_\Reg}\mu\big|^2\,d\mu \lesssim \EE(2Q),
\end{equation}
which we prove below.

Notice first that the cube $Q$ as above is covered by the cubes $P\in\DD_{\mu}(Q)\cap\Reg$. Indeed, if $Q\in \cM_\Neg$, then $Q\in\TT_{\Reg}$ and this is trivially true. On the other hand, if $Q\in\wt\End\setminus\cM_\Neg\subset\End$, then we have $Q\in\TT$, and the condition $\ell_0\leq \ell(Q)$ implies that $d_{R,\ell_0}(x)\leq\ell(Q)$. Hence, 
$Q$ is covered by the cubes $P\in\DD_\mu(Q)\cap\Reg$.

Observe now that, for $Q\in\wt\End$, $P\in\Reg$ such that $P\subset Q$, and $x\in P$,
$$\RR_{\TT_\Reg}\mu(x) - \RR_{\wt\TT}\mu(x) = \RR(\chi_{2R\setminus 2P}\mu)(x) -
\RR(\chi_{2R\setminus 2Q}\mu)(x) = \RR(\chi_{2Q\setminus 2P}\mu)(x).$$
Thus,
$$\big|\RR_{\TT_\Reg}\mu(x) - \RR_{\wt\TT}\mu(x)\big| \lesssim \sum_{S\in\DD_\mu:P\subset S\subset Q}
\theta_\mu(2B_S).$$
Notice now that, if $\wh Q$ denotes the cube from $\End$ that contains $Q$ (which coincides with
$Q$ when $P\not\in\Reg_\Neg$), we have
$$d_R(x) \gtrsim \dist(P,\supp\mu\setminus \wh Q) \geq \dist(P,\supp\mu\setminus Q) 
\quad \mbox{ for all $x\in P$.}$$
So by Lemma \ref{lem74} we also have
\begin{equation}\label{eqreg71}
\ell(P) \gtrsim \dist(P,\supp\mu\setminus Q)
\end{equation}
Thus,
$$\sum_{S\in\DD_\mu:P\subset S\subset Q}
\theta_\mu(2B_S) \lesssim \sum_{S\in\wt\DD_\mu^{int}(Q):S\supset P}
\theta_\mu(2B_S),$$
with $\wt\DD_\mu^{int}(Q)$ defined in \rf{eqdmuint}.

So we have
\begin{align*}
\int_Q \big|\RR_{\wt\TT}\mu - \RR_{\TT_\Reg}\mu\big|^2\,d\mu &=\sum_{P\in\Reg:P\subset Q}
\int_P \big|\RR_{\wt\TT}\mu - \RR_{\TT_\Reg}\mu\big|^2\,d\mu\\
& \lesssim \sum_{P\in\Reg:P\subset Q}
\bigg(\sum_{S\in\wt\DD_\mu^{int}(Q):S\supset P}
\theta_\mu(2B_S)\bigg)^2\,\mu(P).
\end{align*}
By H\"older's inequality, for any $\alpha>0$,
\begin{align*}
\bigg(\sum_{S\in\wt\DD_\mu^{int}(Q):S\supset P} &
\theta_\mu(2B_S)\bigg)^2\\
&\leq \bigg(\sum_{S\in\wt\DD_\mu^{int}(Q):S\supset P}\left(\frac{\ell(Q)}{\ell(S)}\right)^\alpha
\theta_\mu(2B_S)^2\bigg) \bigg(\sum_{S\in\wt\DD_\mu(Q):S\supset P}\left(\frac{\ell(S)}{\ell(Q)}\right)^{\alpha}\bigg)\\
& \lesssim_\alpha \sum_{S\in\wt\DD_\mu^{int}(Q):S\supset P}\left(\frac{\ell(Q)}{\ell(S)}\right)^\alpha
\theta_\mu(2B_S)^2.
\end{align*}
Therefore,
\begin{align*}
\int_Q \big|\RR_{\wt\TT}\mu - \RR_{\TT_\Reg}\mu\big|^2\,d\mu & \lesssim_\alpha 
\sum_{P\in\Reg:P\subset Q}\,\sum_{S\in\wt\DD_\mu^{int}(Q):S\supset P}\left(\frac{\ell(Q)}{\ell(S)}\right)^\alpha
\theta_\mu(2B_S)^2\mu(P)\\
& =
\sum_{S\in\wt\DD_\mu^{int}(Q)}\left(\frac{\ell(Q)}{\ell(S)}\right)^\alpha
\theta_\mu(2B_S)^2 \sum_{P\in\Reg:P\subset S}\mu(P)\\
& \le \sum_{S\in\wt\DD_\mu^{int}(Q)}\left(\frac{\ell(Q)}{\ell(S)}\right)^\alpha
\theta_\mu(2B_S)^2\,\mu(S).
\end{align*}
By Lemma \ref{lemdmutot}, for $\alpha$ small enough, the right hand side is bounded above by 
$C\EE(2Q)$, so that \rf{eqf932} holds.
\end{proof}

\vv


\subsection{Estimates for the \texorpdfstring{$\PP$ and $\QQ_\Reg$}{P and Q\_Reg} coefficients}

We will transfer the lower estimate obtained for the $L^p(\eta)$ norm of $\RR\eta$ in Lemma
\ref{lemrieszeta} to $\RR_{\TT}\mu$, $\RR_{\TT_\Reg}\mu$, and $\Delta_{\TT}\RR\mu$ by means of
Lemmas \ref{lemaprox1}, \ref{lemaprox2}, and \ref{lemaprox3}. To this end, we will need careful 
estimates for the $\PP$ and $\QQ_\Reg$ coefficients of cubes from $\wt\End$ and  $\Reg$. This is the
task we will perform in this section.
\vv


Given $R\in\MDW$, recall that
$\HD_1(e'(R))=\HD_*(R)\cap\sss_*(e'(R))$. 
To shorten notation, we will write $\HD_1=\HD_1(e'(R))$ in this section.
Notice also that, by \rf{eqstop2}, we have
\begin{equation}\label{eqsplit71}
\wt\End =  \LD_1 \cup \LD_2  \cup \HD_2\cup \sM_\Neg,
\end{equation}
where we introduced the following notations:
\begin{itemize}
\item $\LD_1$ is the subfamily of $\wt\End$ of those maximal $\PP$-doubling cubes which are contained both in $e'(R)$ and in some cube from $\LD(R)\cap\sss_1(e'(R))\setminus \Neg$.

\item $\LD_2$ is the subfamily of $\wt\End$ of those maximal $\PP$-doubling cubes which are contained in some cube
$Q\in \LD(Q')\cap \sss_*(Q')\setminus \Neg$ for some $Q'\in \HD_1$.
\item $\HD_2= \bigcup_{Q'\in\HD_1}(\HD_*(Q')\cap\sss_*(Q'))$.
\end{itemize}
Remark that the splitting in \rf{eqsplit71} is disjoint. Indeed, notice that, by the definition
of $\sM_\Neg$, the cubes from $\HD_2$ do not belong to $\sM_\Neg$, since they are strictly contained in some cube from $\HD_1$, which is $\PP$-doubling, in particular.

For $i=1,2$, we also denote by $\Reg_{\LD_i}$ the subfamily of the cubes from $\Reg$ which are contained in some cube from $\LD_i$, and we define $\Reg_{\HD_2}$, $\Reg_\Neg$, and
$\Reg_{\sM_\Neg}$ analogously.\footnote{Notice that $\sD_\Neg=\Reg_{\sM_\Neg}$.}
We let $\Reg_\Ot$ be the ``other'' cubes from $\Reg$: the ones which are not contained in any cube from $\End$ (which, in particular, have side length comparable to $\ell_0$ and are contained in $\TT$).
Also, we let $\Reg_{\HE}$ be the subfamily of the cubes from $\Reg$ which are contained in some cube from $\wt\End\cap\HE$. Notice that we have the splitting
\begin{equation}\label{eqsplitreg0}
\Reg = \Reg_{\LD_1} \cup  \Reg_{\LD_2}\cup \Reg_{\HD_2}\cup \Reg_{\Neg}\cup \Reg_{\Ot}.
\end{equation}
The families above may intersect the family $\Reg_{\HE}$.

Given a family $I\in\DD_\mu$ and $1<p\leq2$, we denote
$$\Sigma_p^\PP(I) = \sum_{Q\in I} \PP(Q)^p\,\mu(Q),\qquad \Sigma_p^\QQ(I) = \sum_{Q\in I} \QQ_\Reg(Q)^p\,\mu(Q).$$
We also write $\Sigma^\PP(I)= \Sigma_2^\PP(I)$, $\Sigma^\QQ(I)= \Sigma_2^\QQ(I)$.

\vv

\begin{lemma}\label{lemenereg}
For any $Q\in\wt\End$, 
$$\Sigma^\PP(\Reg\cap\DD_\mu(Q)) \lesssim \PP(Q)^2\,\mu(Q) + \EE(2Q).$$
\end{lemma}

\begin{proof}
 For all $S\in\Reg\cap \DD_\mu(Q)$, by H\"older's inequality, we have
\begin{align*}
\PP(S)^2 &\lesssim \bigg(\sum_{P:S\subset P\subset Q}\frac{\ell(S)}{\ell(P)}\,\Theta(P)\bigg)^2
+ \bigg(\frac{\ell(S)}{\ell(Q)}\,\PP(Q)\bigg)^2\\
& \leq \bigg(\sum_{P:S\subset P\subset Q}\frac{\ell(S)}{\ell(P)}\,\Theta(P)^2\bigg)
\bigg(\sum_{P:S\subset P\subset Q}\frac{\ell(S)}{\ell(P)}\bigg) +
\PP(Q)^2\\
& \lesssim \sum_{P:S\subset P\subset Q}\frac{\ell(S)}{\ell(P)}\,\Theta(P)^2+ 
\PP(Q)^2.
\end{align*}
We deduce that
\begin{align*}
\Sigma^\PP(\Reg\cap\DD_\mu(Q)) &=
\sum_{S\in\Reg\cap \DD_\mu(Q)} \PP(S)^2\,\mu(S)\\
& \lesssim
\sum_{S\in\Reg\cap \DD_\mu(Q)} 
\sum_{P:S\subset P\subset Q}\frac{\ell(S)}{\ell(P)}\,\Theta(P)^2\,\mu(S) +
\sum_{S\in\Reg\cap \DD_\mu(Q)} \PP(Q)^2\,\mu(S).
\end{align*}
Clearly, the last sum is bounded by $\PP(Q)^2\,\mu(Q)$. Concerning the first term,
arguing as in \rf{eqreg71}, we deduce that the cubes $P$ in the sum belong to 
$\wt\DD_\mu^{int}(Q)$. Thus, by Fubini,
\begin{align*}
\Sigma^\PP(\Reg\cap\DD_\mu(Q)) &\lesssim 
\sum_{P\in\wt \DD_\mu^{int}(Q)} \Theta(P)^2 \sum_{S\in\Reg:S\subset P}\frac{\ell(S)}{\ell(P)}\,\mu(S)
+ \PP(Q)^2\,\mu(Q)\\
& \leq \sum_{P\in\wt \DD_\mu^{int}(Q)}\! \Theta(P)^2 \,\mu(P) + \PP(Q)^2\,\mu(Q)\\ 
&\lesssim \EE(2Q)
+ \PP(Q)^2\,\mu(Q).
\end{align*}
\end{proof}


\vv
\begin{lemma}\label{lempoiss} 
We have:
\begin{itemize}
\item[(i)] If $Q\in\LD_1$, then $\PP(Q)\leq \delta_0\,\Theta(R)$.
\item[(ii)] If $Q\in\LD_2$, then $\PP(Q)\lesssim \delta_0\,\Lambda_*\,\Theta(R)$.
\vspace{1mm}

\item[(iii)] If $Q\in \HD_2$, then $\PP(Q)\lesssim \Lambda_*^2\,\Theta(R)$.

\vspace{1mm}

\item[(iv)] If $Q\in\Neg\cup\sM_\Neg$, then $\PP(Q)\lesssim \left(\frac{\ell(Q)}{\ell(R)}\right)^{1/3}\,\Theta(R)$.
\end{itemize}
\end{lemma}

\begin{proof}
The statements (i),  (ii), and (iii) follow from the stopping conditions used to define $\LD(\,\cdot\,)$
and $\HD_*(\,\cdot\,)$ and Lemma \ref{lemdobpp}.

Regarding the property (iv), by the definitions of $\Neg$ and $\sM_\Neg$ and Lemma \ref{lemdobpp}, for all
$S\in\DD_\mu$ such that $Q\subset S\subset R$, we have
$$\Theta(S)\lesssim \left(\frac{\ell(S)}{\ell(R)}\right)^{1/2}\,\PP(R)\lesssim \left(\frac{\ell(S)}{\ell(R)}\right)^{1/2}\,\Theta(R).$$
Thus,
\begin{align*}
\PP(Q) &\approx \frac{\ell(Q)}{\ell(R)}\,\PP(R) + \sum_{S:Q\subset S\subset R}\frac{\ell(Q)}{\ell(S)}\,\Theta(S)\\
& \lesssim \frac{\ell(Q)}{\ell(R)}\,\Theta(R) + \sum_{S:Q\subset S\subset R}\frac{\ell(Q)}{\ell(S)}
\,\left(\frac{\ell(S)}{\ell(R)}\right)^{1/2}\,\Theta(R)\\
& \lesssim \frac{\ell(Q)}{\ell(R)}\,\Theta(R) + \sum_{S:Q\subset S\subset R}\,\left(\frac{\ell(Q)}{\ell(R)}\right)^{1/2}\,\Theta(R)\\
& \lesssim \frac{\ell(Q)}{\ell(R)}\,\Theta(R) + \log\left(2+\frac{\ell(R)}{\ell(Q)}\right)\,\left(\frac{\ell(Q)}{\ell(R)}\right)^{1/2}\,\Theta(R) \lesssim \left(\frac{\ell(Q)}{\ell(R)}\right)^{1/3}\,\Theta(R).
\end{align*}

\end{proof}
\vv

\begin{rem}\label{rem9.7}
Since $\delta_0 = \Lambda^{-N_0 - \frac1{2N}}\le \Lambda_*^{-N_0 - \frac1{2N}}$, it follows that $\delta_0\,\Lambda_*
\leq \delta_0^{1/2}$, so that by the preceding lemma
 $$\PP(Q)\lesssim \delta_0^{1/2}\,\Theta(R)\quad \mbox{ for all $Q\in\LD_1\cup\LD_2$.}$$
\end{rem}

\vv
Given $Q\in\DD_\mu$, we write $Q\sim \TT$ if there exists some
$Q'\in\TT$ such that 
\begin{equation}\label{defsim0}
A_0^{-2}\ell(Q)\leq \ell(Q')\leq A_0^2\ell(Q)\quad \text{ and }\quad 20Q'\cap20Q\neq\varnothing.
\end{equation}
Given $\gamma\in(0,1)$, we say that the tree $\TT$ is $\gamma$-nice if
$$\sum_{Q\in\HE:Q\sim\TT} \EE(4Q)\leq\gamma \,\sigma(\HD_1).$$


\vv
\begin{lemma}\label{lemregmolt}
Let $R\in\MDW$ and suppose that $\TT=\TT(e'(R))$ is tractable and $\gamma$-nice.
Then


\begin{equation}\label{eqlem*1}
\Sigma^\PP(\Reg_\HE) \lesssim \sum_{Q\in\wt \End\cap \HE}\EE(4Q)\leq
\sum_{Q\in\HE:Q\sim \TT}\EE(4Q)
\leq \gamma\,\sigma(\HD_1),
\end{equation}

\begin{equation}\label{eqlem*2}
\Sigma^\PP(\Reg_{\LD_1}\cup \Reg_{\LD_2}) \lesssim \sum_{Q\in
\LD_1\cup \LD_2} \EE(4Q)\lesssim 
\big(B\,M_0^2\,\delta_0 + \gamma\big)\,\sigma(\HD_1),
\end{equation}

\begin{equation}\label{eqlem*2.5}
 \sum_{Q\in\HD_2} \EE(4Q)\lesssim \big(B\,M_0^2 + \gamma\big)\,\sigma(\HD_1).
\end{equation}

\noi Also, for $1<p\leq2$,

\begin{equation}\label{eqlem*3}
\Sigma_p^\PP(\Reg_\HE)\lesssim  \gamma^{\frac 12}\,B\,\Lambda_*^2\,\sigma_p(\HD_1),
\end{equation}

\begin{equation}\label{eqlem*4}
\Sigma_p^\PP(\Reg_{\LD_1}\cup\Reg_{\LD_2})\lesssim B\,\Lambda_*^2
\big(M_0^2\,\delta_0 + \gamma\big)\,\sigma_p(\HD_1),
\end{equation}

\begin{equation}\label{eqlem*5}
\Sigma_p^\PP(\Reg_{\HD_2})\lesssim \Sigma_p^\PP(\HD_2) \approx \sigma_p(\HD_2) \lesssim B\,\Lambda_*^{p-2}\,\sigma_p(\HD_1).
\end{equation}
\end{lemma}
\vv

\begin{proof}
To get \rf{eqlem*1} note that all the cubes in $\wt\End$ are $\PP$-doubling. Thus, by Lemma \ref{lemenereg} and the definition of $\HE$
\begin{multline*}
\Sigma^\PP(\Reg_\HE) \lesssim \Sigma^\PP(\wt\End\cap\HE) + \sum_{Q\in\wt \End\cap \HE}\EE(4Q)\\
 \lesssim \sum_{Q\in\wt \End\cap \HE}\Theta(Q)^2\mu(Q) + \sum_{Q\in\wt \End\cap \HE}\EE(4Q) \approx \sum_{Q\in\wt \End\cap \HE}\EE(4Q).
\end{multline*}
Observe that for all $Q\in\TT_{\Reg}$ (in particular, for all $Q\in\wt\End$) we have $Q\sim \TT$. Together with the fact that $\TT$ is $\gamma$-nice, \rf{eqlem*1} follows.

To see \rf{eqlem*3}, we first use H\"older's inequality together with \rf{eqlem*1}:
$$
\Sigma_p^\PP(\Reg_\HE)\leq \Sigma^\PP(\Reg_\HE)^{\frac p2}\,\mu(e'(R))^{1-\frac p2}\lesssim \gamma^{\frac 12}\,\sigma(\HD_1)^{\frac p2}
\mu(R)^{1-\frac p2}.
$$
Observe now that, writing $HD_i=\bigcup_{Q\in\HD_i}Q$,
\begin{equation}\label{eqmuhd1}
\mu(HD_1) = \frac1{\Lambda_*^2\,\Theta(R)^2}\,\sigma(\HD_1) \geq \frac1{B\,\Lambda_*^2\,\Theta(R)^2}\,\sigma(R) = \frac1{B\,\Lambda_*^2}\,\mu(R).
\end{equation}
Therefore,
$$
\Sigma_p^\PP(\Reg_\HE)\lesssim \gamma^{\frac 12}\,B\,\Lambda_*^2\,\sigma(\HD_1)^{\frac p2}\,\mu(HD_1)^{1-\frac p2}
= \gamma^{\frac 12}\,B\,\Lambda_*^2\,\sigma_p(\HD_1),$$
which proves \rf{eqlem*3}.

To show \rf{eqlem*2} observe first that, by Lemma \ref{lempoiss} and Remark \ref{rem9.7},
$$\sigma(\LD_1\cup\LD_2) \lesssim \delta_0\,\Theta(R)^2\,\mu(R)\lesssim \delta_0\,B\,\sigma(\HD_1).$$
Then, by Lemma \ref{lemenereg} again, 
\begin{multline*}
\Sigma^\PP(\Reg_{\LD_1}\cup \Reg_{\LD_2}) \lesssim \Sigma^\PP(\LD_1\cup\LD_2)+\sum_{Q\in
\LD_1\cup \LD_2} \EE(4Q)\\
 \leq \sigma(\LD_1\cup\LD_2)+ \sum_{Q\in
(\LD_1\cup \LD_2)\setminus\HE} \!\!\EE(4Q) + \sum_{Q\in\HE:Q\sim\TT} \EE(4Q)\\
\lesssim \sigma(\LD_1\cup\LD_2) + M_0^2\,\sigma\big((\LD_1\cup \LD_2)\setminus\HE\big) + \gamma\,\sigma(\HD_1)
 \\
 \lesssim
\big(B\,M_0^2\,\delta_0 + \gamma\big)\,\sigma(\HD_1).
\end{multline*}
So, by H\"older's inequality and \rf{eqmuhd1},
\begin{align*}
\Sigma_p^\PP(\Reg_{\LD_1}\cup \Reg_{\LD_2})& \leq \Sigma^\PP(\Reg_{\LD_1}\cup \Reg_{\LD_2})^{\frac p2}\,\mu(e'(R))^{1-\frac p2}\\
& \lesssim \big(B\,M_0^2\,\delta_0 + \gamma\big)^{\frac 12}\,\sigma(\HD_1)^{\frac p2}\,\mu(R)^{1-\frac p2}\\
&\lesssim
\big(B\,M_0^2\,\delta_0 + \gamma\big)^{\frac p2} (B\Lambda_*^2)^{1-\frac p2}
\,\sigma_p(\HD_1)\\
& \leq \big(B\,\Lambda_*^2\,M_0^2\,\delta_0 + B\,\Lambda_*^2\gamma\big)\,\sigma_p(\HD_1),
\end{align*}
which yields \rf{eqlem*4}.

To prove \rf{eqlem*2.5}, we just write
\begin{align*}
 \sum_{Q\in\HD_2} \EE(4Q) &\leq  \sum_{Q\in\HD_2\setminus \HE} \EE(4Q) +  \sum_{Q\in \HE:Q\sim\TT} \EE(4Q)\\
&\leq M_0^2\,\sigma(\HD_2) + \gamma\,\sigma(\HD_1)\leq M_0^2\,B\,\sigma(\HD_1) + \gamma\,\sigma(\HD_1).
\end{align*}
Finally, regarding
 \rf{eqlem*5}, recall that $\Theta(Q) \leq\Lambda_*^2\,\Theta(R)$ for all $Q\in\TT$.
This implies that also $\PP(Q) \lesssim\Lambda_*^2\,\Theta(R)$ for all $Q\in\TT$, by Lemma \ref{lemdobpp}. Arguing as in Remark \ref{rem:Hjempty} one gets
\begin{equation}\label{eq:Regdens}
\PP(Q) \lesssim\Lambda_*^2\,\Theta(R)\quad \text{for all $Q\in\Reg$.}
\end{equation}
Consequently,
$$\Sigma_p^\PP(\Reg_{\HD_2})\lesssim \Lambda_*^{2p}\,\Theta(R)^p\,\mu(HD_2) = \sigma_p(\HD_2) \approx
\Sigma_p^\PP(\HD_2).$$
On the other hand, since $R\in\Trc$,
\begin{align*}
\sigma_p(\HD_2) & = \Theta(\HD_2)^{p-2}\,\sigma(\HD_2)\leq B\,\Theta(\HD_2)^{p-2}\,\sigma(\HD_1)
\\ &= B\,\frac{\Theta(\HD_2)^{p-2}}{\Theta(\HD_1)^{p-2}}\,\sigma_p(\HD_1) = B\,\Lambda_*^{p-2}\,\sigma_p(\HD_1),
 \end{align*}
 which completes the proof of \rf{eqlem*5}.
\end{proof}

\vv

\begin{lemma}\label{lemneg3}
Let $R\in\MDW$ and suppose that $\TT=\TT(e'(R))$ is tractable and $\gamma$-nice.
 Then
\begin{equation}\label{eqlemneg01}
\Sigma^\PP(\sM_\Neg)\lesssim\Sigma^\PP(\Neg) \lesssim \big(\delta_0\,B\,\Lambda_*^6 + B\,\Lambda_*^{-4}\big)\,\sigma(\HD_1),
\end{equation}
and
\begin{equation}\label{eqlemneg02}
\Sigma^\PP(\Reg_\Neg) + \sum_{S\in\sM_\Neg} \EE(4S)\lesssim \big(\delta_0\,B\,M_0^2\,\Lambda_*^6 + B\,\Lambda_*^{-4}\,M_0^2+ \gamma\big)\,\sigma(\HD_1)
\end{equation}
Also, for $1<p\leq 2$,
\begin{equation}\label{eqlemneg03}
\Sigma_p^\PP(\Reg_\Neg) \lesssim \big(\delta_0^{\frac12}\,B\,M_0^2\,\Lambda_*^6 + B\,\Lambda_*^{-1}\,M_0+ B\Lambda_*^2\gamma^{\frac12}\big)
\,\sigma_p(\HD_1).
\end{equation}
\end{lemma}

\begin{proof}
Recall that $\Neg = \Neg(e'(R))\cap\End$. The first inequality in \eqref{eqlemneg01} follows from Lemma \ref{eqlemneg01} and the definition of $\sM_\Neg$. 

By Lemma \ref{lemnegs}, the cubes from $\Neg$ belong to $\LD(R)$. Recall also they are not $\PP$-doubling and that $\PP(Q)\lesssim \left(\frac{\ell(Q)}{\ell(R)}\right)^{1/3}\,\Theta(R)$ for all $Q\in\Neg$.
To estimate $\Sigma^\PP(\Neg)$, we split $\Neg$ into two subfamilies $I$ and $J$ so that the 
cubes from $I$ have side length at least $\Lambda_*^{-6}\ell(R)$, opposite to the ones from $J$. We have
\begin{align*}
\Sigma^\PP(I) &= \sum_{Q\in I}\PP(Q)^2\,\mu(Q) \lesssim \Theta(R)^2\sum_{Q\in I}\mu(Q)\\
& \lesssim \Theta(R)^2\sum_{Q\in I}\Theta(Q)\,\ell(Q)^n \lesssim \delta_0\Theta(R)^3\sum_{Q\in I}\ell(Q)^n.
\end{align*}
Using now that the balls $\frac12B(Q)$, with $Q\in I$, are disjoint and that $\ell(Q)\geq \Lambda_*^{-6}\ell(R)$, we get
$$\sum_{Q\in I}\ell(Q)^n \leq \sum_{Q\in I}\frac{\Lambda_*^6\,\ell(Q)^{n+1} }{\ell(R)}\lesssim
\frac{\Lambda_*^6\,\ell(R)^{n+1} }{\ell(R)}=\Lambda_*^6\,\ell(R)^n.$$
Therefore,
$$\Sigma^\PP(I)\lesssim \delta_0\,\Lambda_*^6\,\Theta(R)^3\,\ell(R)^n\approx  \delta_0\,\Lambda_*^6\,\sigma(R)\lesssim \delta_0\,B\,\Lambda_*^6\,\sigma(\HD_1).$$
In connection with the family $J$, we have
\begin{align*}
\Sigma^\PP(J) &= \sum_{Q\in J}\PP(Q)^2\,\mu(Q)\lesssim \Theta(R)^2\sum_{Q\in J}\left(\frac{\ell(Q)}{\ell(R)}\right)^{2/3}\,\mu(Q)\\
&\leq \Lambda_*^{-12/3}\,\Theta(R)^2\sum_{Q\in J}\mu(Q) \leq \Lambda_*^{-4}\,\sigma(R)\leq B\,\Lambda_*^{-4}\,\sigma(\HD_1).
\end{align*}
Hence,
$$\Sigma^\PP(\Neg)\lesssim \big(\delta_0\,B\,\Lambda_*^6 + B\,\Lambda_*^{-4}\big)\,\sigma(\HD_1).$$

\vv
To deal with \rf{eqlemneg02}, we split
$$\Sigma^\PP(\Reg_\Neg) = \Sigma^\PP(\Reg_{\sM_\Neg}) + \Sigma^\PP(\Reg_\Neg\setminus \Reg_{\sM_\Neg}).$$
Notice that, for all $P\in\Reg_\Neg\setminus \Reg_{\sM_\Neg}$ such that $P\subset Q\in\Neg$, by Lemma \ref{lemdobpp}, 
$\PP(P)\lesssim\PP(Q)$. Then it follows that
$$\Sigma^\PP(\Reg_\Neg\setminus \Reg_{\sM_\Neg})\lesssim \Sigma^\PP(\Neg).$$

Next we estimate $\Sigma^\PP(\Reg_{\sM_\Neg})$. By Lemma \ref{lemenereg}, we have 
\begin{align}\label{eq83e}
\Sigma^\PP(\Reg_{\sM_\Neg})  & =\sum_{S\in\sM_\Neg} \Sigma^\PP(\Reg_{\sM_\Neg}\cap\DD_\mu(S))
 \lesssim \sum_{S\in\sM_\Neg} \EE(4S) + \Sigma^\PP(\sM_\Neg)\\
 & \leq  \sum_{S\in\sM_\Neg\setminus\HE} \EE(4S) + \sum_{S\in\HE\cap\wt \End} \EE(4S) + \Sigma^\PP(\sM_\Neg).\nonumber
\end{align}
By the definition of $\HE$ and the fact that $\PP(S)\lesssim\PP(Q)$ whenever $S\subset Q\in\Neg$,
we derive
$$\sum_{S\in\sM_\Neg\setminus\HE} \EE(4S) + \Sigma^\PP(\sM_\Neg)\lesssim M_0^2\sum_{S\in\sM_\Neg\setminus\HE} \sigma(S) + \Sigma^\PP(\sM_\Neg)\lesssim
M_0^2\,\Sigma^\PP(\Neg).$$
On the other hand, by \rf{eqlem*1},
$$\sum_{S\in\HE\cap\wt \End} \EE(4S) \leq\gamma\,\sigma(\HD_1).
$$
 Therefore,
$$\Sigma^\PP(\Reg_{\sM_\Neg})+ \sum_{S\in\sM_\Neg} \EE(4S)\lesssim \sum_{S\in\sM_\Neg} \EE(4S) \lesssim M_0^2\,\Sigma^\PP(\Neg) + \gamma\,\sigma(\HD_1).$$
Gathering the estimates above, we deduce
\begin{align*}
\Sigma^\PP(\Reg_\Neg)&\lesssim M_0^2\,\Sigma^\PP(\Neg) + \gamma\,\sigma(\HD_1)\\
& \lesssim
\big(\delta_0\,B\,M_0^2\,\Lambda_*^6 + B\,\Lambda_*^{-4}\,M_0^2+ \gamma\big)\,\sigma(\HD_1).
\end{align*}

Finally, to prove \rf{eqlemneg03} we apply H\"older's inequality and \rf{eqmuhd1}:
\begin{align*}
\Sigma_p^\PP(\Reg_\Neg)& \leq \Sigma^\PP(\Reg_\Neg)^{\frac p2}\,\mu(e'(R))^{1-\frac p2}\\
& \lesssim \big(\delta_0\,B\,M_0^2\,\Lambda_*^6 + B\,\Lambda_*^{-4}\,M_0^2+ \gamma\big)^{\frac p2}\,\sigma(\HD_1)^{\frac p2}\,\mu(R)^{1-\frac p2}\\
&\lesssim 
\big(\delta_0\,B\,M_0^2\,\Lambda_*^6 + B\,\Lambda_*^{-4}\,M_0^2+ \gamma\big)^{\frac p2} (B\Lambda_*^2)^{1-\frac p2}
\,\sigma_p(\HD_1)\\
 & \lesssim
 \big(\delta_0^{\frac12}\,B\,M_0^2\,\Lambda_*^6 + B\,\Lambda_*^{-p}\,M_0^p+ B\,\Lambda_*^2\,\gamma^{\frac12}\big)
\,\sigma_p(\HD_1).
\end{align*} 
Assuming $\Lambda_*$ big enough, we have $\Lambda_*^{-1}\,M_0<1$. Then, it is clear that $\Lambda_*^{-p}\,M_0^p\leq \Lambda_*^{-1}\,M_0$, and so 
\rf{eqlemneg03} follows.
\end{proof}
\vv

Next lemma deals with $\Sigma_p^\PP(\Reg_\Ot)$. Below we write $\Reg_\Ot(\ell_0)$ to recall the dependence of the family $\Reg_\Ot$ on the parameter
$\ell_0$ in \rf{eql00*23}.

\begin{lemma}\label{lemregot}
For $1\leq p\leq2$, we have
\begin{equation}\label{eqgaga34}
\limsup_{\ell_0\to0} \Sigma_p^\PP(\Reg_\Ot(\ell_0)) \lesssim \Lambda_*^{2p}\,\Theta(R)^p\,\mu(Z).
\end{equation}
Consequently, if $\mu(Z)\leq \ve_Z\,\mu(R)$, then
\begin{equation}\label{eqgaga35}
\limsup_{\ell_0\to0} \Sigma_p^\PP(\Reg_\Ot(\ell_0)) \lesssim B\,\Lambda_*^4\,\ve_Z\,\sigma_p(\HD_1).
\end{equation}
\end{lemma}
\begin{proof}
Notice that if $x\in Q\in\End$ with $\ell(Q)\geq \ell_0$, then 
$$d_{R,\ell_0}(x)\leq \max(\ell_0,\ell(Q)) = \ell(Q)$$
(recall that $d_{R,\ell_0}$ is defined in \rf{eql00*23}), and thus $x$ is contained in some cube $Q'\in\Reg$ with $\ell(Q')\leq \ell(Q)$, by the definition of the family $\Reg$. So $Q'\subset Q$ and then $Q'\in\Reg\setminus \Reg_\Ot$. Therefore,
\begin{equation}\label{eqregotin}
\bigcup_{P\in\Reg_\Ot} P \subset e'(R) \setminus \bigcup_{Q\in\End:\ell(Q)>\ell_0} Q,
\end{equation}
and thus
\begin{equation}\label{eqlimot62}
\limsup_{\ell_0\to0} \mu\bigg(\bigcup_{P\in\Reg_\Ot(\ell_0)} P\bigg) \leq \mu\bigg(e'(R) \setminus \bigcup_{Q\in\End} Q\bigg) = \mu(Z).
\end{equation}
To complete the proof of \rf{eqgaga34} it just remains to notice that by \eqref{eq:Regdens} $\PP(P)\lesssim \Lambda_*^2\,\Theta(R)$ for all $P\in\Reg$, and so
$$\Sigma_p^\PP(\Reg_\Ot(\ell_0)) \lesssim \Lambda_*^{2p}\,\Theta(R)^p \,\mu\bigg(\bigcup_{P\in\Reg_\Ot(\ell_0)} P\bigg).$$

Regarding the second statement of the lemma recall that, as in \rf{eqmuhd1},
$\mu(HD_1) \geq  \frac1{B\,\Lambda_*^2}\,\mu(R)$, which implies that
$$\mu(Z) \leq B\,\Lambda_*^2\,\ve_Z\,\mu(HD_1).$$
Plugging this estimate into \rf{eqgaga34} and taking into account that $\Theta(\HD_1)=\Lambda_*\Theta(R)$, we get 
$$\limsup_{\ell_0\to0} \Sigma_p^\PP(\Reg_\Ot(\ell_0)) \lesssim B\,\Lambda_*^{p+2}\,\ve_Z\,\Theta(\HD_1)^p\,\mu(HD_1) \leq B\,\Lambda_*^4\,\ve_Z\,\sigma_p(\HD_1).$$
\end{proof}

\vv
\begin{rem}\label{remregot}
From \rf{eqregotin}, it is immediate to check that for all $x\in P\in\Reg_\Ot$, we have $d_R(x)\lesssim\ell_0$, and so $d_{R,\ell_0}(x)\approx \ell_0$, which implies that
$$\ell(P)\approx \ell_0.$$
\end{rem}

\vv

\begin{lemma}\label{lemregtot*}
Let $R\in\MDW$ and suppose that $\TT=\TT(e'(R))$ is tractable and $\gamma$-nice.
For $1\leq p\leq2$,  if $\mu(Z)\leq \ve_Z\,\mu(R)$ and $\ell_0$ is small enough, then
\begin{equation}\label{eqlemtot*}
\Sigma_p^\PP(\Reg) \lesssim \big(B\,M_0^2\,\Lambda_*^6 \delta_0^{\frac12}+ B\,\Lambda_*^{-1}M_0+ B\,\Lambda_*^2\,\gamma^{\frac12} + B\,\Lambda_*^4\,\ve_Z + B\,\Lambda_*^{p-2}\big)
\,\sigma_p(\HD_1).
\end{equation}
\end{lemma}

\begin{proof}
This follows by gathering the estimates obtained in Lemmas \ref{lemregmolt},
\ref{lemneg3}, and \ref{lemregot}.
\end{proof}

\vv

\begin{rem}\label{rem9.12}
Recall that $\delta_0 = \Lambda^{-N_0 - \frac1{2N}}\le \Lambda_*^{-N_0 - \frac1{2N}}$, and also that $\Lambda_* = A_0^{k_{\Lambda} (1-1/N) n}$ and $B = \Lambda_*^{\frac{1}{100n}}$. Assuming that $N_0>20$, that $k_{\Lambda}$ is big enough (depending on $M_0$), and that $\gamma$ and $\ve_Z$ are small enough (depending on $k_\Lambda$ and $M_0$), we have
$$B\,M_0^2\,\Lambda_*^6 \delta_0^{\frac12}+  B\,\Lambda_*^2\,\gamma^{\frac12} + B\, M_0^2\, \Lambda_*^4\,
\ve_Z^{\frac12} \leq \Lambda_*^{-1}.$$
In this way, for $\ell_0$ small enough, under the assumptions of Lemma \ref{lemregtot*} we have
$$\Sigma_p^\PP(\Reg) \lesssim \big(\Lambda_*^{-1}+ B\,\Lambda_*^{-1}M_0+ B\,\Lambda_*^{p-2}\big)
\,\sigma_p(\HD_1).
$$

Assuming again $k_{\Lambda}$ big enough (depending on $M_0$), and taking $p\in(1,3/2]$, we have
\begin{equation}\label{eqregtot45}
\Sigma_p^\PP(\Reg) \lesssim  B\,\Lambda_*^{-1/2}M_0
\,\sigma_p(\HD_1)\leq \Lambda_*^{-1/4}\,\sigma_p(\HD_1).
\end{equation}
\end{rem}

\vv

Next lemma shows how to estimate the $\QQ_\Reg$ coefficients in terms of the $\PP$ coefficients.
\vv

\begin{lemma}\label{lemregpq}
For all $p\in(1,\infty)$,
$$\Sigma_p^\QQ(\Reg)\lesssim \Sigma_p^\PP(\Reg).$$
\end{lemma}


\begin{proof}
By duality,
\begin{equation}\label{eqdual912}
\Sigma_p^\QQ(\Reg)^{1/p} = \bigg(\sum_{Q\in\Reg}\QQ_\Reg(Q)^p\,\mu(Q)\bigg)^{1/p}
= \sup \sum_{Q\in\Reg} \QQ_\Reg(Q)\,g_Q\,\mu(Q),
\end{equation}
where the supremum is taken over all sequences $g=\{g_Q\}_{Q\in\Reg}$ such that 
\begin{equation}\label{eqdual912b}
\sum_{Q\in\Reg} |g_Q|^{p'}\,\mu(Q)\leq1.
\end{equation}
We will identify the sequence $g$ with the function 
$$\wt g= \sum_{Q\in\Reg}g_Q\,\chi_Q,$$
so that the sum in \rf{eqdual912b} equals $\|\wt g \|_{L^{p'}(\mu)}^{p'}$.
 By the definition of $\QQ_\Reg$ and Fubini we have
\begin{align}\label{eqdjkq44}
\sum_{Q\in\Reg} \QQ_\Reg(Q)\,g_Q\,\mu(Q) &= \sum_{Q\in\Reg} \sum_{P\in\Reg} \frac{\ell(P)}{D(P,Q)^{n+1}}\,\mu(P)\,g_Q\,\mu(Q)\\
& = \sum_{P\in\Reg} \bigg( \sum_{Q\in\Reg} \frac{\ell(P)}{D(P,Q)^{n+1}}\,g_Q\,\mu(Q)\bigg)\,\mu(P).
\nonumber
\end{align}
For each $P\in\Reg$, we have
\begin{align}\label{eqalg539}
\sum_{Q\in\Reg} \frac{\ell(P)}{D(P,Q)^{n+1}}\,|g_Q|\,\mu(Q) & = \sum_{j\geq0}\;
\sum_{Q\in\Reg:2^j\ell(P)\leq D(P,Q)\leq 2^{j+1}\ell(P)} \frac{\ell(P)}{D(P,Q)^{n+1}}\,|g_Q|\,\mu(Q)\\
& \leq \sum_{j\geq0}\;
\sum_{Q\in\Reg:D(P,Q)\leq 2^{j+1}\ell(P)} \frac{2^{-j}}{(2^j\ell(P))^{n}}\,|g_Q|\,\mu(Q).\nonumber
\end{align}

Observe now that the condition
$$\ell(Q) + \dist(P,Q)\leq D(P,Q)\leq 2^{j+1}\ell(P),$$
implies that 
$$Q\subset B(x_Q,\ell(Q))\subset B(x_P,2^{j+3}\ell(P)).$$
From \rf{eqalg539} and this fact, we infer that
\begin{align*}
\sum_{Q\in\Reg} \frac{\ell(P)}{D(P,Q)^{n+1}}\,|g_Q|\,\mu(Q) &
\leq  \sum_{j\geq0}\;
\sum_{Q\in\Reg:Q\subset B(x_P,C2^j\ell(P))} \frac{2^{-j}}{(2^j\ell(P))^{n}}\,|g_Q|\,\mu(Q)\\
& \leq\sum_{j\geq0}
 \frac{2^{-j}}{(2^j\ell(P))^{n}}\,\int_{B(x_P,C2^j\ell(P))}|\wt g|\,d\mu
\end{align*}
Notice now that, for all $x\in P$,
\begin{align*}
\int_{B(x_P,C2^j\ell(P))}|\wt g|\,d\mu & \leq \int_{B(x,C'2^j\ell(P))}|\wt g|\,d\mu\\
&\leq \mu(B(x,C'2^j\ell(P)))\,\cM_\mu \wt g(x)\leq \mu(B(x_P,C''2^j\ell(P)))\,\cM_\mu \wt g(x),
\end{align*}
where $\cM_\mu$ is the centered Hardy-Littlewood maximal operator. Thus,
\begin{align*}
\sum_{Q\in\Reg} \frac{\ell(P)}{D(P,Q)^{n+1}}\,|g_Q|\,\mu(Q)
& \leq\sum_{j\geq0}
 \frac{2^{-j}\mu(B(x_P,C''2^j\ell(P)))}{(2^j\ell(P))^{n}}\,\cM_\mu \wt g(x)\\
 &\approx
 \sum_{k\geq0}
 2^{-k}\theta_\mu(2^kB_P)\,\cM_\mu \wt g(x) \approx \PP(P)\,\,\cM_\mu \wt g(x).
 \end{align*}
Plugging this estimate into \rf{eqdjkq44}, and taking the infimum for $x\in P$, we get
\begin{align*}
\sum_{Q\in\Reg} \QQ_\Reg(Q)\,|g_Q|\,\mu(Q) & \lesssim \sum_{P\in\Reg} \PP(P)\,\inf_{x\in P}\cM_\mu \wt g(x)\,\mu(P)\\
& \leq \bigg(\sum_{P\in\Reg} \PP(P)^p\,\mu(P)\bigg)^{1/p}
\bigg(\sum_{P\in\Reg} \int_P |\cM_\mu \wt g|^{p'}\,d\mu\bigg)^{1/p'}\\
& = \Sigma_p^\PP(\Reg)^{1/p}\,\|\cM_\mu\wt g\|_{L^{p'}(\mu\rest_{e'(R)})}\\
&\lesssim
\Sigma_p^\PP(\Reg)^{1/p}\,\|\wt g\|_{L^{p'}(\mu)}\leq \Sigma_p^\PP(\Reg)^{1/p},
\end{align*}
which concludes the proof of the lemma, by \rf{eqdual912}.
\end{proof}

\vv


\subsection{Transference of the lower estimates for \texorpdfstring{$\RR\eta$}{R eta} to \texorpdfstring{$\Delta_{\wt \TT}\RR\mu$}{Delta\_T R mu}}

We denote by $\wt V_4$ the union of the balls $\frac12B(Q)$, with $Q\in\Reg$, that intersect $V_4$.

\vv

\begin{lemma}\label{lemalter*}
Let $R\in\MDW$ and suppose that $\TT=\TT(e'(R))$ is tractable and $\gamma$-nice, with $\gamma$ small enough.
Suppose that $\mu(Z)\leq \ve_Z\,\mu(R)$, and take $\ve_Z,\gamma,M_0,\Lambda_*,B$ and $\ell_0$ as in Remark \ref{rem9.12}. 
Then 
$$\|\Delta_{\wt \TT} \RR\mu\|_{L^2(\mu)}^2\geq \Lambda_*^{-2}\,\sigma(\HD_1).$$
\end{lemma}

\begin{proof}
Recall that in Lemma \ref{lemrieszeta} we showed that  
\begin{equation}\label{eqI0038}
I_0:=\int_{V_4} \big|(|\RR\eta(x)| - \frac{c_3}2\,\Theta(\HD_1))_+\big|^p\,d\eta(x)
 \gtrsim \Lambda_*^{-p'\ve_n}\sigma_p(\HD_1),
\end{equation}
for any $p\in (1,\infty)$, with $c_3$ as in Lemma \ref{lemvar}.
The appropriate values of $\ve_n$ and $p$ will be chosen at the end of the proof.

By Lemma \ref{lemaprox1}, for all $Q\in\Reg$ such that $\frac12B(Q)\subset\wt V_4$, all $x\in \frac12B(Q)$, and all $y\in Q$,
$$|\RR\eta(x)| \leq |\RR_{\TT_\Reg}\mu(y)| + C\Theta(R) + C\PP(Q) + C\QQ_\Reg(Q).$$
Thus, for $\Lambda_*$ big enough, since $\Theta(R)=\Lambda_*^{-1}\Theta(\HD_1)<\frac{c_3}4 \Theta(\HD_1)$,
$$(|\RR\eta(x)| - \frac{c_3}2 \,\Theta(\HD_1))_+\leq
(|\RR_{\TT_\Reg}\mu(y)| - \frac{c_3}4\, \Theta(\HD_1))_+  + C\PP(Q) + C\QQ_\Reg(Q).$$
Therefore,
$$
I_0  \lesssim \sum_{Q\in\Reg} \int_Q \big|(|\RR_{\TT_\Reg}\mu(y)| - 
\frac{c_3}4 \Theta(\HD_1))_+\big|^{p}\,d\mu(y) + \sum_{Q\in\Reg} (\PP(Q)^{p} + \QQ_\Reg(Q)^{p})\,\mu(Q).$$
By \rf{eqregtot45} and Lemma \ref{lemregpq}, we have
\begin{align*}
\sum_{Q\in\Reg} (\PP(Q)^{p} + \QQ_\Reg(Q)^{p})\,\mu(Q) & = 
\Sigma_p^\PP(\Reg) +\Sigma_p^\QQ(\Reg) \lesssim \Lambda_*^{-1/4}\,\sigma_p(\HD_1).
\end{align*}

Taking into account that any cube from $\Reg\setminus \Reg_{\Ot}$ 
is contained in some cube $S\in\End$, we derive
\begin{align}\label{eqI0**}
I_0  & \lesssim \sum_{S\in\End} \int_S \big|(|\RR_{\TT_\Reg}\mu(y)| - 
\frac{c_3}4 \Theta(\HD_1))_+\big|^{p}\,d\mu(y) \\
&\quad
+ \sum_{Q\in\Reg_{\Ot}} \int_Q \big|(|\RR_{\TT_\Reg}\mu(y)| - 
\frac{c_3}4 \Theta(\HD_1))_+\big|^{p}\,d\mu(y) + \Lambda_*^{-1/4}\,\sigma_p(\HD_1)\nonumber\\
& =:I_{\End} + I_\Ot + \Lambda_*^{-1/4}\,\sigma_p(\HD_1).\nonumber
\end{align}

\vv
\noi {\bf Estimate of $I_\End$.}
Recall that
$$\wt\End =  \LD_1 \cup \LD_2 \cup \HD_2  \cup\sM_\Neg.$$
We split
\begin{align}\label{eqfkzx23}
I_{\End} &= \sum_{S\in\wt\End} \int_S \big|(|\RR_{\TT_\Reg}\mu(y)| - 
\frac{c_3}4 \Theta(\HD_1))_+\big|^{p}\,d\mu(y)\\
&\quad + \sum_{Q\in\Reg_\Neg\setminus \Reg_{\sM_\Neg}} \int_Q \big|(|\RR_{\TT_\Reg}\mu(y)| - 
\frac{c_3}4 \Theta(\HD_1))_+\big|^{p}\,d\mu(y).\nonumber
\end{align}
We claim that the second sum on the right hand side vanishes. Indeed, by \eqref{eqcad35}, given $Q\in\Reg_\Neg\setminus \Reg_{\sM_\Neg}$, all the cubes $S$ such that $Q\subset S\subset R$ satisfy
$$\Theta(S)\lesssim \left(\frac{\ell(S)}{\ell(R)}\right)^{1/2}\,\Theta(R),$$
and so, for any $x\in Q$,
\begin{align*}
|\RR_{\TT_\Reg}\mu(x)| &\lesssim \sum_{S:Q\subset S\subset R} \Theta(S)  \lesssim \sum_{S:Q\subset S\subset R}\left(\frac{\ell(S)}{\ell(R)}\right)^{1/2}\,\Theta(R) \lesssim \Theta(R).
\end{align*}
Hence, for $\Lambda_*$ big enough, $(|\RR_{\TT_\Reg}\mu(x)| - \frac{c_3}4 \Theta(\HD_1))_+=0$, which proves our claim.

To estimate the first sum on right hand side of \rf{eqfkzx23} notice that, for each $S\in\wt\End$, by the triangle inequality and the fact that $(\;\cdot\;)_+$ is a $1$-Lipschitz function, we have 
\begin{align*}
\int_S \big|(|\RR_{\TT_\Reg}\mu| - 
\frac{c_3}4 \Theta(\HD_1))_+\big|^{p}\,d\mu &\lesssim 
\int_S \big|(|\RR_{\wt \TT}\mu| - 
\frac{c_3}4 \Theta(\HD_1))_+\big|^{p}\,d\mu \\
&\quad+ \int_S\big|\RR_{\wt\TT}\mu - \RR_{\TT_\Reg}\mu\big|^{p}\,d\mu
\end{align*}
By Lemma \ref{lemaprox3}, the last integral does not exceed 
$C\EE(2S)^{\frac{p}2} \,\mu(S)^{1-\frac{p}2}$, and thus we deduce that
\begin{equation}\label{eqIa1}
I_\End\lesssim
\sum_{S\in\wt\End}\int_S \big|(|\RR_{\wt\TT}\mu| - 
\frac{c_3}4 \Theta(\HD_1))_+\big|^{p}\,d\mu + \sum_{S\in\wt\End}\EE(2S)^{\frac{p}2} \,\mu(S)^{1-\frac {p}2}.
\end{equation}

Next we apply Lemma \ref{lemaprox2}, which ensures that
for any $S\in\wt\End$ and all $x\in S$,
\begin{equation}\label{eqIa2}
|\RR_{\wt\TT}\mu(x)| \leq |\Delta_{\wt\TT}\RR\mu(x)| + C\left(\frac{\EE(4R)}{\mu(R)}\right)^{1/2} +  C\left(\frac{\EE(2S)}{\mu(S)}\right)^{1/2} + C\PP(R) + C\PP(S).
\end{equation}
Note that since $R$ and $S$ are $\PP$-doubling, $\PP(R)$ is bounded from above by the second term, and $\PP(S)$ is bounded by the third term.

In  case that $R\not\in\HE$,
\begin{equation}\label{eqEr4}
C\left(\frac{\EE(4R)}{\mu(R)}\right)^{1/2} \leq C\,M_0\,\Theta(R)\leq \frac{c_3}{20} \Theta(\HD_1),
\end{equation}
When $R\in\HE$, 
since $\TT$ is $\gamma$-nice, for $\gamma$ small enough we have
\begin{equation}\label{eqEr5}
\frac{\EE(4R)}{\mu(R)} \leq \frac1{\mu(R)}\sum_{Q\in\HE:Q\sim\TT} \EE(4Q)\leq\frac\gamma{\mu(R)} \,\sigma(\HD_1)
\leq \gamma\,\Theta(\HD_1)^2\ll \left(\frac{c_3}{20} \Theta(\HD_1)\right)^2.
\end{equation}
So in any case we deduce
$$(|\RR_{\wt\TT}\mu(x)| - \frac{c_3}4 \Theta(\HD_1))_+ \leq 
|\Delta_{\wt\TT}\RR\mu(x)| + C\left(\frac{\EE(2S)}{\mu(S)}\right)^{1/2}.$$
Plugging this estimate into \rf{eqIa1} and applying H\"older's inequality, we get
\begin{align}\label{eqIa99}
I_\End &\lesssim
\sum_{S\in\wt\End}\int_S |\Delta_{\wt\TT}\RR\mu|^{p}\,d\mu + \sum_{S\in\wt\End}\EE(2S)^{\frac {p}2} \,\mu(S)^{1-\frac {p}2}\\
& \leq \|\Delta_{\wt\TT}\RR\mu\|_{L^2(\mu)}^{p}\,
\mu(R)^{1-\frac{p}2} + \sum_{S\in\wt\End}\EE(2S)^{\frac {p}2} \,\mu(S)^{1-\frac {p}2}.\nonumber
\end{align}

Regarding the second term in \rf{eqIa99}, by H\"older's inequality again,
\begin{align*}
\sum_{S\in\wt\End}\EE(2S)^{\frac{p}2} \,\mu(S)^{1-\frac {p}2} &= 
\sum_{S\in\LD_1\cup \LD_2\cup\sM_\Neg}
\EE(2S)^{\frac{p}2} \,\mu(S)^{1-\frac {p}2} + \sum_{S\in\HD_2}
\EE(2S)^{\frac{p}2} \,\mu(S)^{1-\frac {p}2}\\
& \leq \bigg(\sum_{S\in\LD_1\cup \LD_2\cup\sM_\Neg}\EE(2S)\bigg)^{\frac{p}2} \,\bigg(\sum_{S\in\wt\End} \mu(S)\bigg)^{1-\frac {p}2} \\
&\quad + 
\bigg(\sum_{S\in\HD_2}\EE(2S)\bigg)^{\frac{p}2} \,\bigg(\sum_{S\in\HD_2} \mu(S)\bigg)^{1-\frac {p}2} 
.
\end{align*}
We estimate the first summand on the right hand side using \rf{eqlem*2}, \rf{eqlemneg02}, \rf{eqmuhd1},
and the choice of the constants $\Lambda_*$, $M_0$, $B$, and $\gamma$ in Remark \ref{rem9.12}:
\begin{align*}
\bigg(\sum_{S\in\LD_1\cup \LD_2\cup\sM_\Neg}\EE(2S)&\bigg)^{\frac{p}2} \,\bigg(\sum_{S\in\wt\End} \mu(S)\bigg)^{1-\frac {p}2} \\
& \leq \bigg(
\sum_{S\in
\LD_1\cup \LD_2} \EE(2S)  + \sum_{S\in\sM_\Neg} \EE(2S)\bigg)^{\frac{p}2}
\mu(R)^{1-\frac p2}
 \\
& \lesssim \big(B\,M_0^2\,\Lambda_*^6\,\delta_0 + B\,M_0^2\,\Lambda_*^{-4}
+\gamma\big)^{\frac{p}2}\,\sigma(\HD_1)^{\frac{p}2}\,\big(B\Lambda_*^2\,\mu(HD_1)\big)^{1-\frac p2}\\
& \lesssim \big(B\,M_0^2\,\Lambda_*^2\,\delta_0 + B\,M_0^2\,\Lambda_*^{-4}
+\gamma\big)^{\frac{p}2}\,\big(B\Lambda_*^2\big)^{1-\frac p2}\sigma_p(\HD_1)\\
& \lesssim
\big(B\,M_0^2\,\Lambda_*^2\,\delta_0^{\frac12}  + BM_0^2\Lambda_*^{-1}
+B\Lambda_*^2\gamma^{\frac12}\big) \sigma_p(\HD_1)\leq \Lambda_*^{\frac{-1}2}\,\sigma_p(\HD_1).
\end{align*}

On the other hand, by \rf{eqlem*2.5} and using that $\mu(HD_2)\leq B\,\Lambda_*^{-2}\mu(HD_1)$ (by the definition of the tractable trees), we get
\begin{align*}
\bigg(\sum_{S\in\HD_2}\EE(2S)\bigg)^{\frac{p}2} \,\bigg(\sum_{S\in\HD_2} \mu(S)\bigg)^{1-\frac {p}2} 
&\lesssim
\big(B\,M_0^2 + \gamma\big)^{\frac p2}\,\sigma(\HD_1)^{\frac p2}\,\mu(HD_2)^{1-\frac p2}\\
&\leq
\big(B\,M_0^2 + \gamma\big)^{\frac p2}\,(B\,\Lambda_*^{-2})^{1-\frac p2}\,\sigma_p(\HD_1)\\
& \lesssim \big(BM_0^p\Lambda_*^{p-2} + \gamma^{\frac12}\big)\,\sigma_p(\HD_1).
\end{align*}
Assuming $p\leq3/2$, the right hand side is at most $\Lambda_*^{\frac{-1}4}\,\sigma_p(\HD_1)$,

From the last estimates and \rf{eqIa99}, we infer that
$$I_\End \lesssim
 \|\Delta_{\wt\TT}\RR\mu\|_{L^2(\mu)}^{p}\,
\mu(R)^{1-\frac{p}2} + \Lambda_*^{\frac{-1}4}\,\sigma_p(\HD_1).$$

\vv
\noi {\bf Estimate of $I_\Ot$.} 
By Remark \rf{remregot}, every $Q\in\Reg_\Ot$ satisfies $\ell(Q)\approx\ell_0$. On the other hand, by Lemma \ref{lemnegs}, the cubes $P\in\Neg(e'(R))$, satisfy $\ell(P) \gtrsim \delta_0^{2}\,\ell(R)$.
Thus, assuming $\ell_0\ll\delta_0^2$, we have $Q\not\in\Neg(e'(R))$.

To estimate $I_{\Ot}$,
denote by $\sM_\Ot$ the family of maximal $\PP$-doubling cubes which are contained in some cube from $\Reg_{\Ot}$
and let
$$N_\Ot = \bigcup_{Q\in \Reg_{\Ot}} Q\setminus \bigcup_{P\in \sM_{\Ot}} P.$$
We claim that 
\begin{equation}\label{eqinclu827}
\sM_\Ot\subset \TT\quad \text{ and }\quad N_\Ot\subset Z.
\end{equation}
To check this, for a given $P\in\sM_\Ot$ with $P\subset Q\in\Reg_{\Ot}$, suppose there exists $S\in\End$ such that $S\supset P$. We have $S\subsetneq Q$ because $Q\in\Reg_{\Ot}$ implies that $Q\not\subset S$. Since $Q\in\Reg_{\Ot}\setminus \Neg(e'(R))$, we  have
$S\not\in \Neg$. 
As $S$ is $\PP$-doubling, we deduce that $P\supset S$, by the maximality of $P$ as $\PP$-doubling cube contained in $Q$. An analogous argument shows that $N_\Ot\subset Z$.

By H\"older's inequality and \rf{eqlimot62}, for $\ell_0$ small enough we have
\begin{align*}
I_\Ot &\leq \bigg(\sum_{Q\in\Reg_{\Ot}} \int_Q \big|(|\RR_{\TT_\Reg}\mu(x)| - 
\frac{c_3}4 \Theta(\HD_1))_+\big|^{2}\,d\mu(x)\bigg)^{\frac p2} \bigg(\sum_{Q\in\Reg_{\Ot}}\mu(Q)\bigg)^{1-\frac p2}\\
& \leq
\bigg(\sum_{P\in\sM_{\Ot}} \int_P \big|\RR_{\TT_\Reg}\mu\big|^{2}\,d\mu + \int_{N_\Ot} \big|\RR_{\TT_\Reg}\mu\big|^{2}\,d\mu
\bigg)^{\frac p2} \bigg(\mu(Z) + o(\ell_0)\bigg)^{1-\frac p2},
\end{align*}
with $o(\ell_0)\to 0$ as $\ell_0\to0$. 

Denote
$$\RR_{\sM_\Ot}\mu(x) = \sum_{P\in\sM_\Ot} \chi_P(x)\,\RR(\chi_{2R\setminus 2P}\mu)(x)$$
and
$$\Delta_{\sM_\Ot}\RR\mu(x) =
\sum_{P\in\sM_\Ot} \chi_P(x)\,\big(m_{\mu,P}(\RR\mu) - m_{\mu,2R}(\RR\mu)\big)
+ \chi_{Z}(x) \big(\RR\mu(x) -  m_{\mu,2R}(\RR\mu)\big)
.$$
Notice that, for $x\in P\in\sM_\Ot$ and $Q\in \Reg_\Ot\setminus \Neg(e'(R))$ such that $Q\supset P$, since there are no $\PP$-doubling cubes $P'$ such that $P\subsetneq P'\subset Q$,
$$\big|\RR_{\sM_\Ot}\mu(x) - \RR_{\TT_\Reg}\mu(x)\big| = |
\RR(\chi_{2 Q\setminus 2P}\mu)(x)|\lesssim \sum_{P: P\subset P'\subset Q} \Theta(P') \overset{\eqref{eqcad35}}{\lesssim} \PP(Q)\overset{\eqref{eq:Regdens}}{\lesssim}
\Lambda_*\,\Theta(\HD_1).$$
Almost the same argument shows also that, for $x\in N_\Ot$,
$$\big|\RR(\chi_{2R}\mu)(x) - \RR_{\TT_\Reg}\mu(x)\big| \lesssim
\Lambda_*\,\Theta(\HD_1).$$
Then, for $\ell_0$ small enough, we deduce that
$$I_\Ot \lesssim \bigg(\sum_{P\in\sM_{\Ot}} \int_P \big|\RR_{\sM_\Ot}\mu\big|^{2}\,d\mu + \int_Z \big|\RR(\chi_{2R}\mu)\big|^{2}\,d\mu + \Lambda_*^2\,\Theta(\HD_1)^2\,\mu(R)
\bigg)^{\frac p2} \big(\ve_Z\,\mu(R)\big)^{1-\frac p2}.$$

Almost the same arguments as in Lemma \ref{lemaprox2} show that for $x\in P\in\sM_\Ot$,
$$\big|\RR_{\sM_\Ot}\mu(x) - \Delta_{\sM_\Ot}\RR\mu(x)\big| \lesssim \PP(R) + \left(\frac{\EE(4R)}{\mu(R)}\right)^{1/2} + \PP(P) +  \left(\frac{\EE(2P)}{\mu(P)}\right)^{1/2}$$
and that, for $x\in Z$,
$$\big|\RR(\chi_{2R}\mu)(x) - \Delta_{\sM_\Ot}\RR\mu(x)\big| \lesssim \PP(R) + \left(\frac{\EE(4R)}{\mu(R)}\right)^{1/2}.$$
Therefore, by \rf{eqEr4} and \rf{eqEr5} and the fact that $\PP(P)\lesssim\Lambda_*\,\Theta(\HD_1)$ for 
$P\in\sM_\Ot$, we deduce
$$I_\Ot \lesssim \bigg(\int |\Delta_{\sM_\Ot}\RR\mu|^2\,d\mu +
\sum_{P\in\sM_{\Ot}} \EE(2P) + \Lambda_*^2\,\Theta(\HD_1)^2\,\mu(R)
\bigg)^{\frac p2} \big(\ve_Z\,\mu(R)\big)^{1-\frac p2}.$$ 

By the orthogonality of the functions $\Delta_Q\RR\mu$, $Q\in\DD_\mu$ and \rf{eqinclu827}, it is clear that
$$\int |\Delta_{\sM_\Ot}\RR\mu|^2\,d\mu\leq \|\Delta_{\wt\TT}\RR\mu\|_{L^2(\mu)}^2.$$
On the other hand, since the tree $\TT$ is $\gamma$-nice and $\sM_\Ot\subset\TT$,
\begin{align*}
\sum_{P\in\sM_{\Ot}} \EE(2P) & \leq \sum_{P\in\sM_{\Ot}\setminus\HE} \EE(2P) + \sum_{P\in \TT\cap\HE} \EE(2P)
\leq M_0^2\,\sigma(\sM_{\Ot}) + \gamma\,\sigma(\HD_1)\\
&
\leq M_0^2\Lambda_*^2\,\Theta(\HD_1)^2\,\mu(R)+ \gamma\,\sigma(\HD_1)\lesssim M_0^2\Lambda_*^2\,\Theta(\HD_1)^2\,\mu(R).
\end{align*}
Thus, using also \rf{eqmuhd1},
\begin{align*}
I_\Ot & \lesssim \|\Delta_{\wt\TT}\RR\mu\|_{L^2(\mu)}^p \,\mu(R)^{1-\frac p2}+ \ve_Z^{1-\frac p2}\,M_0^p\,\Lambda_*^p\,\Theta(\HD_1)^p\,\mu(R)\\
&
\lesssim \|\Delta_{\wt\TT}\RR\mu\|_{L^2(\mu)}^p \,\mu(R)^{1-\frac p2}+ \ve_Z^{1-\frac p2}\,M_0^p\,B\,\Lambda_*^{2+p}\,\sigma_p(\HD_1)\\
&\lesssim \|\Delta_{\wt\TT}\RR\mu\|_{L^2(\mu)}^p\, \mu(R)^{1-\frac p2}+ \ve_Z^{\frac 12}\,M_0^2\,B\,\Lambda_*^{4}\,\sigma_p(\HD_1).
\end{align*}
Remark that, by the choice of $\ve_Z$ in Remark \ref{rem9.12}, we have $\ve_Z^{\frac 12}\,M_0^2\,B\,\Lambda_*^{4}\leq \Lambda_*^{-1}$.
\vv

From \rf{eqI0038}, \rf{eqI0**}, and the estimates for $I_\End$ and $I_\Ot$, with $p=3/2$, we derive
$$\Lambda_*^{-3\ve_n}\sigma_{3/2}(\HD_1)\lesssim I_0 \lesssim
\|\Delta_{\wt\TT}\RR\mu\|_{L^2(\mu)}^{3/2} \mu(R)^{1/4}+ \Lambda_*^{{-1}/4}\,\sigma_{3/2}(\HD_1).$$
Thus, taking
$$\ve_n=\frac1{15},$$
we get
$$\Lambda_*^{-1/5}\sigma_{3/2}(\HD_1) \lesssim
\|\Delta_{\wt\TT}\RR\mu\|_{L^2(\mu)}^{3/2} \mu(R)^{1/4}.$$
Hence,
\begin{align*}
\|\Delta_{\wt\TT}\RR\mu\|_{L^2(\mu)}^2 &\gtrsim \Big(\Lambda_*^{-1/5}\sigma_{3/2}(\HD_1)\,\mu(R)^{-1/4}\Big)^{4/3}\\
& \overset{\eqref{eqmuhd1}}{\gtrsim} \Big(\Lambda_*^{-1/5}\sigma_{3/2}(\HD_1)\,\big(B\Lambda_*^2\,\mu(HD_1)\big)^{-1/4}\Big)^{4/3}
\gtrsim \Lambda_*^{-2}\,\sigma(\HD_1),
\end{align*}
which proves the lemma.
\end{proof}

\vv

%% file: riesz-DX7-S8-1.tex
\section{The proof of Main Lemma \ref{mainlemma}}\label{sec8}

We have to show that
$$\sigma(\ttt)\leq C\,\big(\|\RR\mu\|_{L^2(\mu)}^2 + \theta_0^2\,\|\mu\|
+ \sum_{Q\in\DD_\mu^\PP\cap\HE}\EE(4Q)\big).$$
Recall that by Lemma \ref{lemsuper**9}, we have
$$\sigma(\ttt) \lesssim B^{5/4} \sum_{R\in \sL} \sum_{k\geq0} B^{-k/2}\sum_{Q\in\Trc_k(R)}\sigma(\HD_1(e(Q))) + \theta_0^2\,\|\mu\|.$$
By Lemma \ref{lemalter*} we know that, for each cube $Q$ in the last sum, one of the following alternatives holds:
\begin{itemize}
\item 
$\TT(e'(Q))$ is not $\gamma$-nice, so that
$$\sum_{P\in\HE:P\sim\TT(e'(Q))} \EE(4P) >\gamma \,\sigma(\HD_1(e(Q))),$$ or
\vv

\item $\mu(Z(Q))> \ve_Z\,\mu(Q)$,\; where $Z(Q)$ is the set $Z$ appearing in \eqref{eqdef*f} (replacing $R$ by $Q$ there), which implies that
$$\sigma(\HD_1(e(Q)))\leq \theta_0^2\,\mu(e(Q)) \lesssim \ve_Z^{-1}\,\theta_0^2\,\mu(Z(Q)),$$
or
\vv

\item
$\|\Delta_{\wt \TT(e'(Q))} \RR\mu\|_{L^2(\mu)}^2\geq \Lambda_*^{-2}\,\sigma(\HD_1(e(Q))).$
\end{itemize}

So the following holds in any case:
$$\sigma(\HD_1(e(Q)))\lesssim \Lambda_*^2\,\|\Delta_{\wt \TT(e'(Q))} \RR\mu\|_{L^2(\mu)}^2+ \gamma^{-1}\!\!\sum_{P\in\HE:P\sim\TT(e'(Q))} \!\!\EE(4P) + \ve_Z^{-1}\, \theta_0^2\,\mu(Z(Q)).$$
Consequently,
\begin{align}\label{eqsumtot93}
\sigma(\ttt) & \lesssim  B^{5/4}\,\Lambda_*^2
\sum_{R\in\sL}\,  \sum_{k\geq0} B^{-k/2}\sum_{Q\in\Trc_k(R)}
\|\Delta_{\wt \TT(e'(Q))} \RR\mu\|_{L^2(\mu)}^2\\
&\quad+ B^{5/4}\,
\gamma^{-1}
\sum_{R\in\sL}\,  \sum_{k\geq0} B^{-k/2}\sum_{Q\in\Trc_k(R)}
\sum_{P\in\HE:P\sim\TT(e'(Q))} \!\!\EE(4P)\nonumber\\
&\quad+ \,B^{5/4}\,
\ve_Z^{-1}\, \theta_0^2\,
\sum_{R\in\sL}\,  \sum_{k\geq0} B^{-k/2}\sum_{Q\in\Trc_k(R)}
\mu(Z(Q))\nonumber\\
&=: T_1 + T_2 + T_3.\nonumber
\end{align}
A basic tool to estimate the terms $T_1,T_2,T_3$ is Lemma \ref{lemimp9}, which asserts that
 for all $P\in\DD_\mu$ and all $k\geq0$,
$$\#\big\{R\in\sL:\exists \,Q\in\Trc_k(R) \mbox{ such that } P\in\TT(e'(Q))\big\}\leq C_1\,\log\Lambda_*.$$

\vv

\subsection{Estimate of \texorpdfstring{$T_1$}{T1}}
Recall that 
\begin{align*}
\Delta_{\wt\TT(e'(Q))}\RR\mu(x) & = \sum_{P\in\wt\End(e'(Q))} \chi_P(x)\,\big(m_{\mu,P}(\RR\mu) - m_{\mu,2Q}(\RR\mu)\big)\\
&\quad +  \chi_{Z(Q)}(x) \big(\RR\mu(x) -  m_{\mu,2R}(\RR\mu)\big).
\end{align*}
For $Q\in\MDW$, we write $S\prec Q$ if $S\in\DD_\mu$ is a maximal cube that belongs to $\TT(e'(Q))$. Then we denote
$$\wh\Delta_Q\RR\mu = \sum_{S\prec Q} (m_{\mu,S}(\RR_\mu) - m_{\mu,2Q}(\RR\mu)\big)\,\chi_S$$
and 
$$\wt E(Q) = \bigcup_{P\in\wt\End(e'(Q))} P,\qquad\quad \wt G(Q) = \wt E(Q) \cup Z(Q)
.$$
Notice that it may happen that $\wt G(Q)\neq e'(Q)$ because of the presence of negligible cubes.
Then we have
\begin{align*}
\sum_{P\in\wt\End(e'(Q))} \!\!\chi_P\,\big(m_{\mu,P}(\RR\mu) - m_{\mu,2Q}(\RR\mu)\big)  &=\!
\sum_{P\in\wt\End(e'(Q))}\!\! \chi_P\,\bigg(\sum_{S\in \wt\TT(e'(Q)):P\subsetneq S}\!\! \Delta_{S}\RR\mu
+ \wh\Delta_Q\RR\mu\bigg)\\
& = \sum_{S\in \wt\TT(e'(Q))\setminus \wt \End(e'(Q))}\!\Delta_{S}\RR\mu \sum_{P\in\wt\End(e'(Q)):P\subset S} \chi_P \\
&\quad + \chi_{\wt E(Q)}\wh\Delta_Q\RR\mu\\
&= \chi_{\wt E(Q)}\bigg(\sum_{S\in \wt\TT(e'(Q))\setminus \wt \End(e'(Q))}\!\! \Delta_{S}\RR\mu
+ \wh\Delta_Q\RR\mu\bigg).
\end{align*}
A similar argument shows that
$$\chi_{Z(Q)}(x) \big(\RR\mu(x) -  m_{\mu,2Q}(\RR\mu)\big) = \chi_{Z(Q)}\bigg(\sum_{S\in \wt\TT(e'(Q))}\! \Delta_{S}\RR\mu
+ \wh\Delta_Q\RR\mu\bigg).$$
Hence,
$$\Delta_{\wt\TT(e'(Q))}\RR\mu = \chi_{\wt G(Q)}\bigg(\sum_{P\in\wt\TT(e'(Q))\setminus \wt\End(e'(Q))} \Delta_P \RR\mu + \wh\Delta_Q\RR\mu\bigg).$$ 
It is also immediate to check that, for a fixed $Q$ and $P\in\wt\TT(e'(Q))\setminus \wt\End(e'(Q))$, the functions 
$\wh\Delta_Q\RR\mu$ and $\Delta_P \RR\mu$ are mutually orthogonal in $L^2(\mu)$. Then, since all the cubes $P\in\wt\TT(e'(Q))$ satisfy 
$P\sim\TT(e'(Q))$, we get
\begin{align*}
\|\Delta_{\wt \TT(e'(Q))} \RR\mu\|_{L^2(\mu)}^2 & \leq
\sum_{P\in\wt\TT(e'(Q))\setminus \wt\End(e'(Q))}\| \Delta_P \RR\mu\|_{L^2(\mu)}^2  + \|\wh\Delta_Q\RR\mu\|_{L^2(\mu)}^2 \\
& \leq \sum_{P\sim\TT(e'(Q))} \|\Delta_P \RR\mu\|_{L^2(\mu)}^2  + \|\wh\Delta_Q\RR\mu\|_{L^2(\mu)}^2.
\end{align*}
Therefore,
\begin{align*}
T_1 & \lesssim_{\Lambda_*}\sum_{R\in\sL}\,
\sum_{k\geq0} B^{-k/2} \sum_{Q\in\Trc_k(R)}
\sum_{P\sim\TT(e'(Q))} \|\Delta_P \RR\mu\|_{L^2(\mu)}^2\\
&\quad + \sum_{R\in\sL}\,\sum_{k\geq0} B^{-k/2} \sum_{Q\in\Trc_k(R)}
\|\wh\Delta_Q\RR\mu\|_{L^2(\mu)}^2\\
& =: T_{1,1} + T_{1,2}.
\end{align*}
Regarding the term $T_{1,1}$, by Fubini we have
\begin{align*}
T_{1,1}
 &\leq
\sum_{P\in \DD_\mu} \|\Delta_P \RR\mu\|_{L^2(\mu)}^2
 \sum_{k\geq0} B^{-k/2}\,
\# A(P,k),
\end{align*}
where
\begin{equation}\label{eqApq2}
A(P,k)= 
\big\{R\in \sL:\exists \,Q\in\Trc_k(R) \text{ such that }P\sim\TT(e'(Q))\big\}.
\end{equation}
From the definition  \rf{defsim0} and Lemma \ref{lemimp9}, it follows that
\begin{align*}
\#A(P,k) &\leq \sum_{\substack{
P'\in\DD_\mu: 20P'\cap20P\neq\varnothing\\A_0^{-2}\ell(P)\leq \ell(P')\leq A_0^2\ell(P)
}} \!\!\#
\big\{R\in \sL:\exists \,Q\in\Trc_k(R) \text{ such that }P'\in\TT(e'(Q))\big\}\\
&\lesssim \sum_{\substack{
P'\in\DD_\mu: 20P'\cap20P\neq\varnothing\\A_0^{-2}\ell(P)\leq \ell(P')\leq A_0^2\ell(P)
}}\!\!\!\!\log\Lambda_*\lesssim \log\Lambda_*.
\end{align*}
Hence,
$$T_{1,1}
 \lesssim_{\Lambda_*}
\sum_{P\in \DD_\mu} \|\Delta_P \RR\mu\|_{L^2(\mu)}^2= \|\RR\mu\|_{L^2(\mu)}^2.$$
Concerning $T_{1,2}$, we argue analogously:
\begin{align*}
T_{1,2}
 &\le 
\sum_{Q\in \MDW} \|\wh\Delta_Q\RR\mu\|_{L^2(\mu)}^2
 \sum_{k\geq0} B^{-k/2}\,
\# \wt A(Q,k),
\end{align*}
where
$$\wt A(Q,k)= \big\{R\in \sL:
Q\in\Trc_k(R)\big\}.$$
Since 
$$\#\wt A(Q,k) \leq\#A(Q,k)\lesssim\log\Lambda_*,$$
we deduce that
$$T_{1,2}
 \lesssim_{\Lambda_*}
\sum_{Q\in \MDW} \|\wh \Delta_Q \RR\mu\|_{L^2(\mu)}^2.$$
As the next lemma shows, the right hand side above is also bounded by $C\|\RR\mu\|_{L^2(\mu)}^2$.
So we have
$$T_1
 \lesssim_{\Lambda_*} \|\RR\mu\|_{L^2(\mu)}^2.$$

\vv

\begin{lemma}\label{lemortog}
For any $f\in L^2(\mu)$, we have
$$\sum_{Q\in \MDW} \|\wh \Delta_Q f\|_{L^2(\mu)}^2\lesssim \|f\|_{L^2(\mu)}^2.$$
\end{lemma}

We defer the proof of this result to Section \ref{sec10.3}.

\vv

\subsection{Estimate of \texorpdfstring{$T_2$}{T2}}

By the same argument we used to deal with $T_{1,1}$ above, we get
$$T_2
 \leq
\sum_{P\in \HE} \EE(4P)
 \sum_{k\geq0} B^{-k/2}\,
\# A(P,k),
$$
where $A(P,K)$ is given by \rf{eqApq2}.
Since $\#A(P,k)\lesssim\log\Lambda_*$, we obtain
$$T_2
 \lesssim_{\Lambda_*}
\sum_{P\in \HE} \EE(4P).$$

\vv

\subsection{Estimate of \texorpdfstring{$T_3$}{T3}}

We have
\begin{align*}
T_3 & = \sum_{R\in\sL}\,
\sum_{k\geq0} B^{-k/2} \!\!\!\!\sum_{Q\in\Trc_k(R)}\!
\theta_0^2\,\ve_Z^{-1}\,\mu(Z(Q))\\
& =
\sum_{R\in\sL}\,
\sum_{k\geq0} B^{-k/2} \!\!\!\!\sum_{Q\in\Trc_k(R)}\!
\int_{Z(Q)}\theta_0^2\,\ve_Z^{-1}\,d\mu\\
&= \int\theta_0^2\,\ve_Z^{-1}\, \bigg(\sum_{R\in\sL}\,\sum_{k\geq0} B^{-k/2}\!\!\!
\sum_{Q\in\Trc_k(R)}\!\chi_{Z(Q)}\bigg)\,d\mu.
\end{align*}
By Fubini, we have
$$\sum_{R\in\sL}\,\sum_{k\geq0} B^{-k/2}\!\!\!
\sum_{Q\in\Trc_k(R)}\!\chi_{Z(Q)} \leq \sum_{k\geq0} B^{-k/2} \,\# D(x,k),$$
where
$$D(x,k) = 
\big\{R\in \sL:\exists \,Q\in\Trc_k(R) \text{ such that }x\in Z(Q)\big\}.
$$
Observe now that, given $j\ge1$, if we let
\begin{align*}D_j(x,k) = 
\big\{R\in \sL: \exists \,Q\in&\Trc_k(R) \text{ such that $\TT(e'(Q))$ contains} \\
&\text{
every $P\in\DD_\mu$ such that $x\in P$ and $\ell(P)\leq A_0^{-j}$}\big\},
\end{align*}
then we have
$$D(x,k) = \bigcup_{j\geq 1} D_j(x,k),$$
and moreover $D_j(x,k)\subset D_{j+1}(x,k)$ for all $j$.
From Lemma \ref{lemimp9} we deduce that
$$\#D_j(x,k)\leq C\,\log\Lambda_*\quad \mbox{ for all $j\geq 1$.}$$
Thus, $\#D(x,k)\leq C\,\log\Lambda_*$ too. Consequently,
$$\sum_{R\in\sL}\,\sum_{k\geq0} B^{-k/2}\!\!\!
\sum_{Q\in\Trc_k(R)}\!\chi_{Z(Q)} \lesssim_{\Lambda_*} 1,$$
and so
$$T_3\lesssim_{\Lambda_*} \ve_Z^{-1}
\theta_0^2\,\|\mu\|.$$
Together with the estimate we obtained for $T_1$ and $T_2$, this yields
$$\sigma(\ttt)\lesssim_{\Lambda_*} \|\RR\mu\|_{L^2(\mu)}^2 + \sum_{P\in \HE} \EE(4P)+ \ve_Z^{-1}
\theta_0^2\,\|\mu\|,$$
which concludes the proof of Main Lemma \ref{mainlemma}, modulo the proof of Lemma \ref{lemortog}.
\vv


\subsection{Proof of Lemma \ref{lemortog}}\label{sec10.3}
For $Q\in \MDW$, denote by $\cA(Q)$ the family of cubes $R\in\DD_\mu$ such that $R\cap2Q\neq \varnothing$
and $\ell(R)=A_0\,\ell(Q)$. 
Also, let $\cF(Q)$ be the family of cubes $P$ which are contained in some cube from $\cA(Q)$ and satisfy
$\ell(Q)\leq \ell(P)\leq A_0\,\ell(Q)$. Notice that, by Lemma \ref{lempois00}, the cubes from $\cA(Q)$ belong to $\DD_\mu^{db}$. So taking into account that $Q\subset 2B_R$ for any $R\in\cA(Q)$, that $R\subset CQ$, and that
$Q$ is $\PP$ doubling,
\begin{equation}\label{eqcompar39}
\mu(Q)\approx \mu(R)\quad\mbox{ for all $R\in\cA(Q)$.}
\end{equation}

Denote
$$\widecheck \Delta_Q f = \sum_{R\in\cA(Q)} \big(m_{\mu,R}(f)- m_{\mu,2Q}(f)\big)\,\chi_{R}.$$
It is immediate to check that
$$\|\wh\Delta_Q f\|_{L^2(\mu)}\lesssim \sum_{P\in\cF(Q)}\|\Delta_Pf\|_{L^2(\mu)} + \|\widecheck\Delta_Q f\|_{L^2(\mu)}.$$
Remark that the main advantage of the operator $\wk\Delta_Q$ over $\wh\Delta_Q$ is that the cubes
$R\in\cA(Q)$ involved in the definition of $\wk\Delta_Q$ are doubling, which may not be the case for the cubes $S\prec Q$ in the definition of $\wh\Delta_Q$.
From the last inequality, we get
$$\sum_{Q\in\MDW}\|\wh\Delta_Q f\|_{L^2(\mu)}^2 \lesssim \sum_{Q\in\MDW}\sum_{P\in\cF(Q)}\|\Delta_Pf\|_{L^2(\mu)}^2
+ \sum_{Q\in\MDW}\|\widecheck\Delta_Q f\|_{L^2(\mu)}^2 
$$
Since
\begin{align*}
\sum_{Q\in\MDW}\sum_{P\in\cF(Q)}\|\Delta_Pf\|_{L^2(\mu)}^2 &\leq \sum_{P\in\DD_\mu}\|\Delta_Pf\|_{L^2(\mu)}^2
\sum_{Q\in\MDW: P\in\cF(Q)}1\\
&\lesssim \sum_{P\in\DD_\mu}\|\Delta_Pf\|_{L^2(\mu)}^2\lesssim\|f\|_{L^2(\mu)}^2,
\end{align*}
the lemma follows from the next result.

\begin{lemma}\label{lemortog2}
For any $f\in L^2(\mu)$, we have
$$\sum_{Q\in \DD_\mu^\PP} \|\wk \Delta_Q f\|_{L^2(\mu)}^2\lesssim \|f\|_{L^2(\mu)}^2.$$
\end{lemma}

\begin{proof}
For $Q\in \DD_\mu^\PP$ and $P\in\DD_\mu$ such that $P\in\cA(Q)$, let
$$\vphi_{Q,P} = \left(\frac1{\mu(P)}\,\chi_P - \frac1{\mu(2Q)}\,\chi_{2Q}\right)\,\mu(P)^{1/2}.$$
Observe that
$$\|\wk\Delta_Q f\|_{L^2(\mu)}^2 = \sum_{P\in\cA(Q)} \big|m_{\mu,P}(f) - m_{\mu,2Q}(f)\big|^2\,\mu(P) 
=\sum_{P\in\cA(Q)}\langle f,\vphi_{Q,P}\rangle^2 .$$
So we have to show that
$$\sum_{Q\in\DD_\mu^\PP} \sum_{P\in\cA(Q)}\langle f,\,\vphi_{Q,P}\rangle^2 \lesssim \|f\|_{L^2(\mu)}^2.$$
To shorten notation, we denote by $\cI$ the set of all pairs $(Q,P)$ with $Q\in\DD_\mu^\PP$ and $P\in\cA(Q)$, so that 
the double sum above can be written as $\sum_{(Q,P)\in \cI}$.

Arguing by duality, we have
\begin{align*}
\bigg(\sum_{(Q,P)\in \cI}\langle f,\vphi_{Q,P}\rangle^2 \bigg)^{1/2}& =
\sup \bigg| \sum_{(Q,P)\in \cI}\langle f,\,\vphi_{Q,P}\rangle \,b_{Q,P}\bigg|=  \sup \bigg| \bigg\langle f,\,\sum_{(Q,P)\in \cI} b_{Q,P}\,\vphi_{Q,P}\bigg\rangle \,\bigg|
,
\end{align*}
where the supremum is taken over all the sequences $b:=\{b_{Q,P}\}_{(Q,P)\in \cI}$
such that $\|b\|_{\ell^2}\leq1$.
Since
$$\bigg(\sum_{(Q,P)\in \cI}\langle f,\vphi_{Q,P}\rangle^2 \bigg)^{1/2} \leq \big\|f\big\|_{L^2(\mu)}\,\sup\bigg\|\sum_{(Q,P)\in \cI} b_{Q,P}\,\vphi_{Q,P}\bigg\|_{L^2(\mu)},$$
to prove the lemma it suffices to show that
$$\bigg\|\sum_{(Q,P)\in \cI} b_{Q,P}\,\vphi_{Q,P}\bigg\|_{L^2(\mu)}\lesssim 1\quad\mbox{ for all
$b=\{b_{Q,P}\}_{(Q,P)\in \cI}$
such that $\|b\|_{\ell^2}\leq1$.}$$
To this end, we write
\begin{align}\label{eqdobb836}
\bigg\|\sum_{(Q,P)\in \cI} b_{Q,P}\,\vphi_{Q,P}\bigg\|_{L^2(\mu)}^2 & = 
\sum_{(Q,P),(R,S)}\big\langle b_{Q,P}\,\vphi_{Q,P},\, b_{R,S}\,\vphi_{R,S}\big\rangle \\
& \le 2
\sum_{\substack{(Q,P),(R,S)\\ \ell(Q)\leq \ell(R)}} |b_{Q,P}\, b_{R,S}| \,\, \big|\big\langle\vphi_{Q,P},\, \vphi_{R,S}\big\rangle \big|.\nonumber
\end{align}

Denote 
$$a(Q)=\bigcup_{P\in\cA(Q)} P.$$
Observe that, for some $C$ depending just on $A_0$,
$$\supp\vphi_{Q,P}\subset \overline{a(Q)}\subset CQ,\qquad \supp\vphi_{R,S}\subset \overline{a(R)}\subset CR,$$
and
$$\big\|\vphi_{Q,P}\big\|_{L^\infty(\mu)}\lesssim \frac1{\mu(Q)^{1/2}},\qquad \big\|\vphi_{R,S}\big\|_{L^\infty(\mu)}\lesssim \frac1{\mu(R)^{1/2}},$$
 taking into account \rf{eqcompar39}.
Thus, for $(Q,P)$, $(R,S)\in\mathcal{I}$ with $\ell(Q)\leq\ell(R)$, we have
$$\big|\big\langle\vphi_{Q,P},\, \vphi_{R,S}\big\rangle \big| = \left|\int_{a(Q)} \vphi_{Q,P}\, \vphi_{R,S}\,d\mu\right|
\lesssim \frac{\mu(a(Q))}{\mu(Q)^{1/2}\,\mu(R)^{1/2}}
 \approx \bigg(\frac{\mu(Q)}{\mu(R)}\bigg)^{1/2}.$$
Further, using that $\vphi_{Q,P}$ has zero mean and that $\vphi_{R,S}$ is constant in $2R\cap S$, in $2R\setminus S$, 
and in $S\setminus 2R$,
it follows that $\big\langle\vphi_{Q,P},\, \vphi_{R,S}\big\rangle =0$ in the following cases:
\begin{itemize}
\item[(i)] if $a(Q)\cap (2R\cup S) = \varnothing$,
\item[(ii)] if $a(Q)\subset 2R\cap S$,
\item[(iii)] if $a(Q)\subset 2R\setminus S$,
\item[(iv)] if $a(Q)\subset S\setminus 2R$.
\end{itemize}
For $d>0$, denote
$$\mathcal N_d(S) = \{x\in\supp\mu\setminus S:\dist(x,S)\leq d\}
\cup \{x\in S:\dist(x,\supp\mu\setminus S)\leq d\}$$
and, analogously,
$$\mathcal N_d(2R) = \{x\in\supp\mu\setminus 2R:\dist(x,2R)\leq d\}
\cup \{x\in 2R:\dist(x,\supp\mu\setminus 2R)\leq d\}.$$
Observe that if none of the conditions (i), (ii), (iii), (iv), holds, then
\begin{equation}\label{eqdd22}
a(Q)\subset \mathcal N_{\diam(a(Q))}(S) \cup \mathcal N_{\diam(a(Q))}(2R).
\end{equation}

From the thin boundary condition \rf{eqfk490} and the fact that $2R$ is a finite number of
 cubes of the same generation as $R$, using also that $R$ is $\PP$-doubling and $S$ is doubling, we deduce that
\begin{equation}\label{eqdd23}
\mu\big(\mathcal N_d(S)\big) + \mu\big(\mathcal N_d(2R)\big) \lesssim \left(\frac d{\ell(R)}\right)^{1/2}\,\mu(R)
\quad \mbox{ for all $d\in (0,C\ell(R)$),}
\end{equation}
with the implicit constant depending on $C$.
Consequently, denoting by ``$(Q,P) \dashv (R,S)$" the situation when $\ell(Q)\leq \ell(R)$ and \rf{eqdd22} holds, by \rf{eqdobb836} we get
\begin{align*}
\bigg\|&\sum_{(Q,P)\in \cI} b_{Q,P}\,\vphi_{Q,P}\bigg\|_{L^2(\mu)}^2  \le 2
\sum_{(Q,P)\dashv(R,S)} |b_{Q,P}\, b_{R,S}| \,\, \bigg(\frac{\mu(Q)}{\mu(R)}\bigg)^{1/2}\\
&\quad\leq 2\bigg(\sum_{(R,S)\in \cI} |b_{R,S}|^2\bigg)^{1/2}\, \Bigg(\sum_{(R,S)\in \cI} 
\Bigg(\sum_{\substack{(Q,P)\in \cI:\\ (Q,P)\dashv(R,S)}}\!\!\! |b_{Q,P}| \, \bigg(\frac{\mu(Q)}{\mu(R)}\bigg)^{1/2}\Bigg)^2\Bigg)^{1/2}\\
&\quad \lesssim \Bigg( \sum_{(R,S)\in\cI} \Bigg(
\sum_{\substack{(Q,P)\in \cI:\\ (Q,P)\dashv(R,S)}}\!\!\! |b_{Q,P}|^2 \,\bigg(\frac{\ell(Q)}{\ell(R)}\bigg)^{1/4}\Bigg)
 \,
\Bigg( \sum_{\substack{(Q,P)\in \cI:\\ (Q,P)\dashv(R,S)}} \bigg(\frac{\ell(R)}{\ell(Q)}\bigg)^{1/4}
\frac{\mu(Q)}{\mu(R)}\Bigg)\Bigg)^{1/2}.
\end{align*}
We consider now the last sum on the right hand side, which equals
\begin{align*}
\sum_{\substack{(Q,P)\in \cI:\\ (Q,P)\dashv(R,S)}} \bigg(\frac{\ell(R)}{\ell(Q)}\bigg)^{1/4}
\frac{\mu(Q)}{\mu(R)} = \sum_{k\geq 0}
\sum_{\substack{(Q,P)\in \cI:\\ (Q,P)\dashv(R,S)\\
\ell(Q)=A_0^{-k}\ell(R)}}\!\! A_0^{k/4}\,\,
\frac{\mu(Q)}{\mu(R)} 
\end{align*}
Notice that, by \rf{eqdd22} and \rf{eqdd23},  we have
$$\sum_{\substack{(Q,P)\in \cI:\\ (Q,P)\dashv(R,S)\\
\ell(Q)=A_0^{-k}\ell(R)}}\mu(Q) \lesssim 
\mu\big(\mathcal N_{CA_0^{-k}\ell(R)}(S)\big) + \mu\big(\mathcal N_{CA_0^{-k}\ell(R)}(2R)\big)\lesssim A_0^{-k/2}\,\mu(R).
$$
Therefore,
$$\sum_{\substack{(Q,P)\in \cI:\\ (Q,P)\dashv(R,S)}} \bigg(\frac{\ell(R)}{\ell(Q)}\bigg)^{1/4}
\frac{\mu(Q)}{\mu(R)}\lesssim
\sum_{k\geq 0} A_0^{k/4}\,\,
\frac{A_0^{-k/2}\,\mu(R)}{\mu(R)} \lesssim 1.$$

We deduce that
\begin{align*}
\bigg\|\sum_{(Q,P)\in \cI} b_{Q,P}\,\vphi_{Q,P}\bigg\|_{L^2(\mu)}^2 
& \lesssim
 \Bigg( \sum_{(R,S)\in\cI} \Bigg(
\sum_{\substack{(Q,P)\in \cI:\\ (Q,P)\dashv(R,S)}}\!\!\! |b_{Q,P}|^2 \,\bigg(\frac{\ell(Q)}{\ell(R)}\bigg)^{1/4}\Bigg)
\Bigg)^{1/2}\\
& =  \Bigg( \sum_{(Q,P)\in\cI} |b_{Q,P}|^2 
\sum_{\substack{(R,S)\in \cI:\\ (Q,P)\dashv(R,S)}}\!\!\bigg(\frac{\ell(Q)}{\ell(R)}\bigg)^{1/4}
\Bigg)^{1/2}.
\end{align*}
Since
$$\sum_{\substack{(R,S)\in \cI:\\ (Q,P)\dashv(R,S)}}\!\!\bigg(\frac{\ell(Q)}{\ell(R)}\bigg)^{1/4}
\lesssim 
\sum_{\substack{R\in\DD_\mu:\\ \ell(R)\geq \ell(Q)\\ a(Q)\cap a(R)\neq\varnothing}}\!\!\bigg(\frac{\ell(Q)}{\ell(R)}\bigg)^{1/4}\lesssim1,$$
we infer that
$$\bigg\|\sum_{(Q,P)\in \cI} b_{Q,P}\,\vphi_{Q,P}\bigg\|_{L^2(\mu)}^2 \lesssim1,$$
as wished.
\end{proof}

\vv

%% file: riesz-DX7-S9-2.tex

\section{The proof of the Main Theorem \ref{propomain}} \label{sec9}

Recall that, for a cube $R\in\ttt$, $\tree(R)$ denotes the subfamily of the cubes from $\DD_\mu(R)$ which are not strictly contained in any cube
from $\End(R)$.
In this section we will prove the following result. 

\begin{lemma}\label{lemtreebeta}
For each $R\in\ttt$, the following holds:
\begin{align*}
\sum_{Q\in\tree(R)} \!\beta_{\mu,2}(2B_Q)^2\,\Theta(Q)\,\mu(Q) & \lesssim_{\Lambda,\delta_0}\!\sum_{Q\in\tree(R)}\!\|\Delta_Q\RR\mu\|_{L^2(\mu)}^2\\&\quad \;+ \!\sum_{Q\in\tree(R)\cap\HE}\! \EE(4Q) +\Theta(R)^2\,\mu(R).
\end{align*}
\end{lemma}

It is clear that together with Main Lemma \ref{mainlemma}, this yields Main
Theorem \ref{propomain}.


\vv
\subsection{The approximating measure \texorpdfstring{$\eta$}{eta} on a subtree \texorpdfstring{$\wh\tree_0(R)$}{Tree\_0(R)}}\label{subsec:91}

To prove Lemma \ref{lemtreebeta} for a given cube $R_0\in\ttt$ (in place of $R$), we will consider a corona 
decomposition of $\tree(R_0)$ into subtrees by introducing appropriate new stopping conditions. 
In this section we will deal with the construction of each subtree and an associated AD-regular measure
which approximates $\mu$ in that subtree.
To this end we need some additional notation. 

First, for a cube 
$R\in\tree(R_0)\cap\DD_\mu^\PP$, we write $Q\in \BR(R)$ (which stands for ``big Riesz transform'') if $Q$ is a $\PP$-doubling maximal cube which does not belong to $\HD(R_0)\cup\LD(R_0)$ and satisfies
$$|\RR_\mu\chi_{2R\setminus 2Q}(x_Q)|\geq K\,\Theta(R),$$
where $K$ is some big constant to be fixed below, depending on $\Lambda$, $\delta_0$, and $M_0$. Also, for a cube $Q\in\tree(R_0)$,
we denote by $\wh \Ch(Q)$ the family of maximal cubes $P\in\DD_\mu(Q)\setminus\{Q\}$ that satisfy one of the following conditions:
\begin{itemize} 
\item $P\in\DD_\mu^\PP$, i.e.\ $P$ is $\PP$-doubling, or
\item $P\in\LD(R_0)$.
\end{itemize}
From Lemma \ref{lemdobpp}, it is immediate to check that if $Q$ is not contained in any cube from $\LD(R_0)$, then the cubes from $\wh\Ch(Q)$
cover $Q$, and also 
\begin{equation}\label{eqcompa492}
\ell(P)\approx_{\Lambda,\delta_0}\ell(Q)\quad \mbox{ for each $P\in\wh\Ch(P)$,} 
\end{equation}

Given $R\in\DD_\mu^\PP\in\tree(R_0)\setminus \End(R_0)$, we will construct a tree $\wh \tree_0(R)$ inductively, consisting just of $\PP$-doubling cubes and stopping cubes from $\LD(R_0)$. At the
same time we will construct an approximating AD-regular measure for this tree. We will do this by ``spreading''
the measure of the cubes from $\wh\tree_0(R)\cap \LD(R_0)$ among the other cubes from $\wh\tree_0(R)$. To this end, we will consider some coefficients $s(Q)$, $Q\in\DD_\mu^\PP\cap\wh\tree_0(R)$, that, in a sense, quantify the additional measure $\mu$ spreaded on $Q$
due to the presence of close cubes from $\LD(R_0)$. The algorithm is the following.

First we choose $R$ as the root of $\wh \tree_0(R)$, and we set $s(R)=0$. Next, suppose that $Q\in\wh \tree_0(R)$
(in particular, this implies that $Q\in\tree(R_0)$), and assume that we have not 
decided yet if the cubes from $\wh\Ch(Q)$ belong to $\wh \tree_0(R)$. 
First we decide that $Q\in\wh\sss(R)$ if one of the following conditions hold:
\begin{itemize}
\item[(i)] $Q\in\HD(R_0)\cup\LD(R_0)\cup\BR(R)$, or\vspace{1mm}

\item[(ii)] $s(Q)\geq \mu(Q)$ and (i) does not hold, or \vspace{1mm}

\item[(iii)] $\sum_{P\in\wh\Ch(Q)\cap \LD(R_0)}\mu(P) \geq \frac12\,\mu(Q)$ and neither (i) nor (ii) hold.
\end{itemize}
\vspace{1mm}

\noi If $Q\in\wh\sss(R)$, no descendants of $Q$ are allowed to belong to $\wh\tree_0(R)$.
Otherwise, all the cubes from $\wh\Ch(Q)$ are chosen to belong to $\wh\tree_0(R)$, and for each $P\in\wh\Ch(Q)$, we define
$$s(P)=-\mu(P) \quad \mbox{ if $P\in\LD(R_0)$,}$$
and, otherwise, we set 
\begin{equation}\label{eqdqaq14}
t(Q) = \sum_{S\in\wh\Ch(Q)\cap\LD(R_0)}\mu(S)\quad \mbox{ and }\quad
s(P) = \big(s(Q) + t(Q)\big)\, 
\frac{\mu(P)}{\mu(Q)-t(Q)}.
\end{equation}
Observe that
\begin{align*}
\sum_{P\in\wh\Ch(Q)}s(P) & = \sum_{P\in\wh\Ch(Q)\cap \LD(R_0)} s(P) + \sum_{P\in\wh\Ch(Q)\setminus \LD(R_0)}s(P) \\
& = -t(Q) + \big(s(Q) + t(Q)\big) = s(Q).
\end{align*}

By induction, the coefficients $s(\cdot)$ satisfy the following.
If $Q\in\wh \tree_0(R)$ and $\cI$ is some finite family of cubes from $\wh\tree_0(R)\cap\DD_\mu(Q)$
which cover $Q$ and are disjoint, then
\begin{equation}\label{eqaq89}
\sum_{P\in \cI} s(P) = s(Q).
\end{equation}
Further, $s(Q)\geq0$ for all $Q\in\wh\tree_0(R)\setminus \LD(R_0)$.

Now we are ready to define an approximating measure $\eta$ associated with $\wh\tree_0(R)$. 
First, we denote
$$\wh \sG(R) = R\,\setminus \bigcup_{Q\in\wh\sss(R)} Q,$$
and for each $Q\in\DD_\mu$ we let $D_Q$ be an $n$-dimensional disk passing through $x_Q$
with radius $\frac12 \,r(Q)$ (recall that $r(Q)$ is the radius of $B(Q)$).
In  case that $\mu(\wh\sG(R))=0$, we define
$$\eta = \sum_{Q\in\wh \sss(R)} \big(s(Q) + \mu(Q)\big)\,\frac{\HH^n\rest_{D_Q}}{\HH^n(D_Q)}.$$
Observe that $\eta(D_Q)=\mu(Q)+ s(Q)$ for all $Q\in\wh\sss(R)$ and, in particular $\eta(D_Q)=0$
if $Q\in\wh\sss(R)\cap\LD(R_0)$.

In case that $\mu(\wh\sG(R))\neq0$, we have to be a little more careful. For a given
$N\geq1$ we let $\wh\sss_N(R)$ be the family consisting of all the cubes from 
$\wh\sss(R)$ with side length larger that $A_0^{-N}\,\ell(R)$, and we let $\cI_N$ be the family of the cubes from $
\wh\tree_0(R)$ which have side length smaller than $A_0^{-N}\,\ell(R)$ and are maximal.
We denote 
\begin{equation}\label{eqdefetaN}
\eta_N = \sum_{Q\in\wh \sss_N(R)} \big(s(Q) + \mu(Q)\big)\,\frac{\HH^n\rest_{D_Q}}{\HH^n(D_Q)}+
\sum_{Q\in \cI_N(R)} \big(s(Q) + \mu(Q)\big)\,\frac{\mu\rest_Q}{\mu(Q)}
,
\end{equation}
and we let $\eta$ be a weak limit of $\eta_N$ as $N\to\infty$.

As in Section \ref{sec6.2*}, we use the following notation.
To each $Q\in\wh\tree_0(R)$ we associate another ``cube'' $Q^{(\eta)}$ defined as follows:
$$Q^{(\eta)}= (\sG(R)\cap Q)\cup \bigcup_{P\in\wh \sss(R):P\subset Q} D_P.$$
We let
$$\wh\tree_0^{(\eta)}(R) \equiv \wh\tree_0^{(\eta)}(R^{(\eta)}):= \{Q^{(\eta)}: Q\in\wh\tree_0(R)\}.$$
For $S=Q^{(\eta)}\in \wh\tree_0^{(\eta)}(R)$ with $Q\in\wh\tree_0(R)$, we denote $Q=S^{(\mu)}$ and we write
$\ell(S):=\ell(Q)$.

Observe now that, from \rf{eqaq89} and the definition of $\eta$, we have the key property
\begin{equation}\label{eqetq5}
\eta(Q^{(\eta)})= \mu(Q) + s(Q)\quad \mbox{ for all $Q\in\wh\tree_0(R)$}.
\end{equation}
So $s(Q)$ is the measure added to $\mu(Q)$ to obtain $\eta(Q^{(\eta)})$.

\vv

\begin{lemma}\label{lem9.2}
The measure $\Theta(R_0)^{-1}\eta$ is AD-regular (with a constant depending on $\Lambda$ and $\delta_0$), and 
$\eta(Q^{(\eta)})=0$ for all $Q\in\LD(R_0)\cap\wh\tree_0(R)$.
\end{lemma}

\begin{proof}
The fact that $\eta(Q^{(\eta)})=0$ for all $Q\in\LD(R_0)\cap\wh\tree_0(R)$ follows by construction and has already been mentioned above.
To prove the AD-regularity of $\eta$, by standard arguments, it is enough to show that
$$\eta(Q^{(\eta)})\approx_{\Lambda,\delta_0} \Theta(R_0)\,\ell(Q)^n\quad\mbox{ for all $Q\in\wh\tree_0(R)\setminus
\LD(R_0)$,}$$
taking into account \rf{eqcompa492}.
Given such a cube $Q$, the fact that $Q\not\in\LD(R_0)$ ensures that
$$\eta(Q^{(\eta)})\geq \mu(Q)\gtrsim\delta_0\,\Theta(R_0)\,\ell(Q)^n.
$$
To show the converse estimate we can assume $Q\neq R$. By
the condition (ii) in the definition of $\wh \sss(\wh Q)$, where $\wh Q$ is the first ancestor of $Q$ in $\wh \tree_0(R)$
(i.e., $\wh Q$ is the smallest cube from $\wh\tree_0(R)$ that strictly contains $Q$), we have
$$s(\wh Q)\leq \mu(\wh Q).$$
Also, by (iii) (which does not hold for $\wh Q$), the coefficient $t(\wh Q)$ in \rf{eqdqaq14} satisfies
$$t(\wh Q) = \sum_{P\in\wh\Ch(\wh Q)\cap \LD(R_0)}\mu(P) < \frac12\,\mu(\wh Q).$$
Therefore,
\begin{equation}\label{eqaq745}
s(Q) = \big(s(\wh Q) + t(\wh Q)\big)\, 
\frac{\mu(Q)}{\mu(\wh Q)-t(\wh Q)} \leq 2 \big(\mu(\wh Q) + \frac12\,\mu(\wh Q)\big)\, 
\frac{\mu(Q)}{\mu(\wh Q)} = 3\,\mu(Q),
\end{equation}
and so
$$\eta(Q^{(\eta)}) = s(Q) + \mu(Q)\leq 4\,\mu(Q)\leq 4\,\mu(\wh Q)\lesssim \Lambda\, \Theta(R_0)\,\ell(\wh Q)^n
\approx_{\Lambda,\delta_0} \Theta(R_0)\,\ell(Q)^n,
$$
taking into account that $\wh{Q}\not \in\HD(R_0)$, by (i).
\end{proof}
\vv

\begin{rem}
For the record, notice that from \rf{eqaq745} it follows that, for all $Q\in\wh\tree_0(R)$, either
\begin{equation}\label{eq:aQmuQ}
0\leq s(Q)\leq 3\,\mu(Q),
\end{equation}
or 
$$s(Q)=-\mu(Q).$$
The latter case happens if and only if $Q\in\LD(R_0)$.
\end{rem}
\vv

\begin{rem}
Consider the measure defined by
$$\eta' = \sum_{Q\in\wh \sss(R)\setminus\LD(R_0)} \mu(Q)\,\frac{\HH^n\rest_{D_Q}}{\HH^n(D_Q)} + \mu\rest_{\wh\sG(R)}.$$
This measure is mutually absolutely continuous with $\eta$. Further, since the
coefficients $s(Q)$, with $Q\in\wh\tree_0(R)$, are uniformly bounded (by the previous remark), it turns out that
$$\eta'= \rho\,\eta,$$
for some function $\rho\in L^\infty(\eta)$ satisfying $\rho\approx1$.
Consequently, by Lemma \ref{lem9.2}, $\eta'$ is also AD-regular.
\end{rem}

\vv

For a family of cubes $\cI\subset \tree(R_0)$, we denote 
$$\wh\Ch(\cI)=\bigcup_{Q\in \cI} \wh\Ch(Q).$$
For $Q\in\wh\tree_0(R)$, we write $Q\in (i)_R$, if $Q\in\wh\sss(R)$ and the condition (i) in the definition of 
$\wh\sss(R)$ holds for $Q$, and analogously regarding the notations $Q\in (ii)_R$ and $Q\in (iii)_R$.

\begin{lemma}\label{lem9.5*}
The following holds:
$$\sum_{Q\in\wh\sss(R)\cap(\LD(R_0)\cup\HD(R_0)\cup\BR(R))}\mu(Q) 
+\sum_{Q\in\wh\Ch((iii)_R)\cap \LD(R_0)}\mu(Q) + \mu(\wh\sG(R))
\approx \mu(R).$$
\end{lemma}

\begin{proof}
It is clear that the left hand side above is bounded by $\mu(R)$.
For the converse estimate,  we write
\begin{equation}\label{eqsplit63}
\mu(R) = \sum_{Q\in\wh\sss(R)}\mu(Q) + \mu(\wh\sG(R)) = 
\sum_{Q\in (i)_R} \mu(Q) + \sum_{Q\in (ii)_R} \mu(Q) +\sum_{Q\in (iii)_R} \mu(Q)  
 + \mu(\wh\sG(R)).
 \end{equation}
 By construction,
\begin{equation}\label{eqsplit64}
\sum_{Q\in (i)_R} \mu(Q) =\sum_{Q\in \wh\sss(R)\cap(
\HD(R_0)\cup\LD(R_0)\cup\BR(R))}\mu(Q).
\end{equation}
Also, if $Q\in (iii)_R$, then
$$\mu(Q)\leq 2\sum_{P\in\wh\Ch(Q)\cap \LD(R_0)}\mu(P),$$
and thus
\begin{equation}\label{eqsplit65}
\sum_{Q\in (iii)_R} \mu(Q)\leq 2 \sum_{P\in\wh\Ch(\wh\sss(R))\cap \LD(R_0)}\mu(P).
\end{equation}

On the other hand, if $Q\in (ii)_R$, then $0\leq \mu(Q)\leq s(Q)$, and so
$$\sum_{Q\in (ii)_R} \mu(Q)\leq \sum_{Q\in \wh\sss(Q):s(Q)\geq 0} s(Q) = \sum_{Q\in \wh\sss(R)\setminus \LD(R_0)} s(Q).$$
For a given $N\geq 1$, consider the families $\wh\sss_N(R)$ and $\cI_N$ defined just above \rf{eqdefetaN}.
Notice that $\cJ_N:= \wh\sss_N(R)\cup\cI_N$ is a finite family of cubes which cover $R$, and thus, 
from the property \rf{eqaq89},
it follows that
$$0 = s(R) = \sum_{Q\in \cJ_N}s(Q) = \sum_{Q\in \cJ_N\cap \LD(R_0)} s(Q) + 
\sum_{Q\in \cJ_N\setminus \LD(R_0)} s(Q). $$
Since $s(Q) = -\mu(Q)$ for all $Q\in \cJ_N\cap \LD(R_0)$, we deduce
$$\sum_{Q\in \cJ_N\setminus \LD(R_0)} s(Q) = \sum_{Q\in \cJ_N\cap \LD(R_0)} \mu(Q).$$ 
Letting $N\to\infty$ and taking into account that $s(Q)\geq 0$ for all
$Q\in \cJ_N\setminus \LD(R_0)$, we get
$$\sum_{Q\in \wh\sss(R)\setminus \LD(R_0)} s(Q)\leq \sum_{Q\in \wh\sss(R)\cap \LD(R_0)} \mu(Q),$$
and thus
\begin{equation}\label{eqsplit66}
\sum_{Q\in (ii)_R} \mu(Q)\leq \sum_{Q\in \wh\sss(R)\cap \LD(R_0)} \mu(Q).
\end{equation}

The lemma follows from the splitting \rf{eqsplit63} and the inequalities \rf{eqsplit64}, \rf{eqsplit65}, \rf{eqsplit66}.
\end{proof}

\vv

\begin{lemma}\label{lemriesz*eta}
The operator $\RR_\eta$ is bounded in $L^2(\eta)$, with
$$\|\RR_\eta\|_{L^2(\eta)\to L^2(\eta)}\lesssim_{\Lambda,\delta_0,K}\Theta(R).$$
\end{lemma}

\begin{proof}
To prove this lemma we will use the suppressed kernel $K_\Phi$ introduced in Section \ref{sec6.1},
with the following $1$-Lipchitz function:
$$\Phi(x) =\inf_{Q\in\wh\tree_0(R)} \big(\ell(Q) + \dist(x,Q)\big).$$
We will prove first that $\RR_{\Phi,\mu\rest_{2R}}$ is bounded in $L^2(\mu\rest_{2R})$ by applying Theorem
\ref{teontv}, and later on we will 
show that $\RR_\eta$ is bounded in $L^2(\eta)$ by approximation. 

In order to apply Theorem \ref{teontv}, we will show that
\begin{itemize}
\item[(a)] $\mu(B(x,r)\cap 2R)\leq C\,\Theta(R)\,r^n$ for all $r\geq \Phi(x)$, and
\item[(b)] $\sup_{r>\Phi(x)}\big|\RR_r(\chi_{2R}\mu)(x)\big|\leq C\Theta(R)$,
\end{itemize}
with $C$ possibly depending on $\Lambda$, $\delta_0$, and $K$. Once these conditions are proven, then Theorem \ref{teontv} applied to the measure $\Theta(R)^{-1}\,\mu\rest_{2R}$ ensures that 
\begin{equation}\label{eqphi894}
\|\RR_{\Phi,\mu\rest_{2R}}\|_{L^2(\mu\rest_{2R})\to L^2(\mu\rest_{2R})}\lesssim_{\Lambda,\delta_0,K}\Theta(R).
\end{equation}

The proof of (a) is quite similar to the proof of Lemma \ref{lem6.77}. However, we repeat
here the arguments for the reader's convenience.
In the case $r>\ell(R)/10$ we just use that
$$\mu(B(x,r)\cap 2R)\leq \mu(2R)\lesssim \Theta(R)\,\ell(R)^n\lesssim \Theta(R)\,r^n.$$
So we may assume that $\Phi(x)<r\leq \ell(R)/10$.
By the definition of $\Phi(x)$, there exists $Q\in\wh\tree_0(R)$ such that
$$\ell(Q) + \dist(x,Q)\leq r.$$
Therefore, $B_Q\subset B(x,4r)$ and so there exists an ancestor $Q'\supset Q$ which belongs to $\wh\tree_0(R)$ such that $B(x,r)\subset  2B_{Q'}$, with $\ell(Q')\approx r$.  
Then,
$$\mu(B(x,r)\cap 2R)\leq \mu(2B_{Q'}) \lesssim \Lambda\,\Theta(R_0)\,\ell(Q')^n\approx_{\Lambda,\delta_0} \Theta(R)\,r^n,$$
as wished.

Let us turn our attention to the property (b).
In the case $r>\ell(R)/10$ we have
$$\big|\RR_r(\chi_{2R}\mu)(x)\big|\leq \frac{\mu(2R)}{r^n}\lesssim \Theta(R).$$
In the case $\Phi(x)<r\leq \ell(R)/10$ we consider the same cube $Q'\in\wh\tree_0(R)$ as above, 
which satisfies $B(x,r)\subset  2B_{Q'}$ and $\ell(Q')\approx r$. Further, by replacing
$Q'$ by the first ancestor in $\wh\tree_0(R)$ if necessary, we may assume that
$$|\RR_\mu\chi_{2R\setminus 2Q'}(x_{Q'})|\lesssim K\,\Theta(R).$$
Since $|x-x_{Q'}|\lesssim\ell(Q')$, by standard arguments which use the fact that $K_\Phi$ is a Calder\'on-Zygmund kernel (see \rf{eqkafi1}
and \rf{eqkafi2}),
it follows that
$$\big|\RR_\mu\chi_{2R\setminus 2Q'}(x_{Q'}) - \RR_r(\chi_{2R}\mu)(x)\big|\lesssim \PP(Q')\lesssim
\Lambda\,\Theta(R_0)\approx_{\Lambda,\delta_0}\Theta(R).$$
Thus,
$$\big|\RR_r(\chi_{2R}\mu)(x)\big|\leq 
\big|\RR_\mu\chi_{2R\setminus 2Q'}(x_{Q'})\big| + 
\big|\RR_\mu\chi_{2R\setminus 2Q'}(x_{Q'}) - \RR_r(\chi_{2R}\mu)(x)\big|\lesssim_{\Lambda,\delta_0,K}\Theta(R).$$
So both (a) and (b) hold, and then \rf{eqphi894} follows.
\vv

Next we deal with the $L^2(\eta)$ boundedness of $\RR_\eta$. 
First notice that  $\RR_{\mu\rest_{\wh\sG(R)}}$ is bounded in $L^2(\mu\rest_{\wh\sG(R)})$ with norm
at most $C\Theta(R)$ because $\Phi(x)=0$ on $\wh\sG(R)$. Since $\eta\rest_{\wh\sG(R)} = \rho\,\mu\rest_{\wh\sG(R)}$ for some function $\rho\approx1$, $\RR_{\eta\rest_{\wh\sG(R)}}$ is also bounded in $L^2(\eta\rest_{\wh\sG(R)})$ with norm bounded by $C\Theta(R)$.
So it suffices to show that $\RR_{\eta\rest_{\wh\sG(R)^c}}$ is bounded in $L^2(\eta\rest_{\wh\sG(R)^c})$.
This follows from the fact that if $\alpha$ and $\beta$ are Radon measures with polynomial growth of degree $n$ such that $\RR_\alpha$ is bounded in $L^2(\alpha)$ and 
$\RR_\beta$ is bounded in $L^2(\beta)$, then $\RR_{\alpha+\beta}$ is bounded in $L^2(\alpha+\beta)$,
and then choosing $\alpha= \Theta(R)^{-1}\mu\rest_{\wh\sG(R)}$ and $\beta= \Theta(R)^{-1}\mu\rest_{\wh\sG(R)^c}$.
See for example Proposition 3.1 from \cite{NToV2}.
 
It remains to show that $\RR_{\eta\rest_{\wh\sG(R)^c}}$ is bounded in $L^2(\eta\rest_{\wh\sG(R)^c})$
with norm bounded above by $C\,\Theta(R)$.
Notice first that, by \rf{eqfk490}, there exists some constant $b>0$ depending at most on 
$C_0,A_0,n$ such that
$$\mu\bigl(\{x\in Q:\dist(x,\supp\mu\setminus Q)\leq b\,\ell(Q)\}\bigr) \leq\frac12\,\mu(Q).$$
We denote
\begin{equation}\label{eqsepbb*}
Q^{(0)}= \{x\in Q:\dist(x,\supp\mu\setminus Q)> b\,\ell(Q)\},
\end{equation}
so that $\mu(Q^{(0)})\geq \frac12\,\mu(Q)$.

We have to show that
\begin{equation}\label{eqbeta8430}
\|\RR (g\,\eta\rest_{\wh\sG(R)^c})\|_{L^2(\eta\rest_{\wh\sG(R)^c})} \lesssim_{\Lambda,\delta_0,K}\Theta(R)\, \|g\|_{L^2(\eta)}
\end{equation}
for any given $g\in L^2(\eta\rest_{\wh\sG(R)^c})$, with $\RR (g\,\eta\rest_{\wh\sG(R)^c})$ 
understood in the principal value sense. To this end, we take
the function $f\in L^2(\mu\rest_{2R})$ defined as follows:
\begin{equation}\label{deff99}
f = \sum_{Q\in\wh\sss(R)} \int_{D_Q}\! g\,d\eta \,\,\frac{\chi_{Q^{(0)}}}{\mu(Q^{(0)})}.
\end{equation}
We also consider the signed measures
$$\alpha = f\,\mu,\qquad \beta = g\,\eta,$$
so that $\alpha(Q) = \alpha(Q^{(0)}) = \beta(D_Q)$ for all $Q\in\wh\sss(R)$.

As a preliminary step to obtain \rf{eqbeta8430}, we will show first 
\begin{equation}\label{eqbeta843}
\|\RR_\Phi \beta\|_{L^2(\eta\rest_{\wh\sG(R)^c})} \lesssim_{\Lambda,\delta_0,K}\Theta(R)\, \|g\|_{L^2(\eta)}.
\end{equation}
For that purpose, first we will estimate the term  $|\RR_\Phi \alpha(x) -\RR_\Phi\beta(y)|$,
with $x\in Q^{(0)}$, $y\in D_Q$, for $Q\in\wh\sss(R)$, in terms of the coefficients
$$\PP_\alpha(Q) := \sum_{P\in\DD_\mu:P\supset Q} \frac{\ell(Q)}{\ell(P)^{n+1}} \,|\alpha|(2B_P)
\quad \mbox{ and }\quad
\QQ_\alpha(Q) := \sum_{P\in\wh\sss(R)} \frac{\ell(P)}{D(P,Q)^{n+1}}\,|\alpha|(P),
$$
and
$$\PP_\beta(Q) := \sum_{P\in\DD_\mu:P\supset Q} \frac{\ell(Q)}{\ell(P)^{n+1}} \,|\beta|(2B_P)
\quad \mbox{ and }\quad
\QQ_\beta(Q) := \sum_{P\in\wh\sss(R)} \frac{\ell(P)}{D(P,Q)^{n+1}}\,|\beta|(D_P).
$$
We claim that
\begin{equation}\label{eqclaim7284}
|\RR_\Phi \alpha(x) -\RR_\Phi\beta(y)| \lesssim \PP_\alpha(Q) + \QQ_\alpha(Q)+\PP_\beta(Q) + \QQ_\beta(Q)
\end{equation}
for all $x\in Q^{(0)}$, $y\in D_Q$, with $Q\in\wh\sss(R)$.
The arguments to prove this are quite similar to the ones in Lemma \ref{lemaprox1}, but we 
will show the details for completeness.
By the triangle inequality, for $x$, $y$ and $Q$ as above, we have
\begin{align*}
\big|\RR_\Phi \alpha(x) - \RR_\Phi \beta(y)\big| & \leq 
\big|\RR_\Phi \alpha(x) - \RR_\Phi \alpha(x_Q)\big|\\
&\quad + 
\big|\RR_\Phi \alpha(x_Q) - \RR_\Phi \beta(x_Q)\big| 
+ \big|\RR_\Phi \beta(x_Q) - \RR_\Phi \beta(y)\big| \\
& =: I_1 + I_2 + I_3.
\end{align*}

First we estimate $I_1$ using the properties of the kernel $K_\Phi$ in \rf{eqkafi1}
and \rf{eqkafi2}, and taking into account that for $x\in Q^{(0)}$ (and thus for $x_Q$)
$\Phi(x)\approx_b \ell(Q)$, because of the separation condition in the definition of
$Q^{(0)}$ in \rf{eqsepbb*}. Then we get
\begin{align*}
|I_1| &\leq \int |K_\Phi(x,z) -  K_\Phi(x_Q,z)|\,d|\alpha|(z)\\
& = \bigg(\int_{2B_Q} +
\sum_{P\in\DD_\mu:P\supset Q} \int_{2B_{\wh P}\setminus 2B_P}\bigg) |K_\Phi(x,z) -  K_\Phi(x_Q,z)|\,d|\alpha|(z)\\
&\lesssim \sum_{P\in\DD_\mu:P\supset Q} \frac{\ell(Q)}{\ell(P)^{n+1}} \,|\alpha|(2B_{\wh P}) \lesssim \PP_\alpha(Q),
\end{align*}
where above we denoted by $\wh P$ the parent of $P$.
The same estimate holds for the term $I_3$ (with $\alpha$ replaced by $\beta$), using that $\Phi(x)\approx \ell(Q)$ for all $x\in D_Q$,
since $D_Q\subset\frac12B(Q)$ and $B(Q)\cap\supp\mu\subset Q$. So
$$|I_3| \lesssim \PP_\beta(Q).$$
Finally we deal with the term $I_2$.
Since $\alpha(P^{(0)})=\beta(D_P)$ for all $P\in\wh\sss(R)$, we have
\begin{align*}
I_2 & \leq \sum_{P\in\wh\sss(R)} \left| \int K_\Phi(x_Q-z)\,d\big(\alpha\rest_{P^{(0)}} - 
\beta\rest_{D_P} \big)(z)\right|\\
& \leq \sum_{P\in\wh\sss(R)}  \int |K_\Phi(x_Q-z)-K_\Phi(x_Q-x_P)|\,d\big(|\alpha|\rest_{P^{(0)}} + 
|\beta|\rest_{D_P} \big)(z)
\end{align*}
From the separation condition in \rf{eqsepbb*} and the fact that $D_P\subset\frac12B(P)$, we infer that, for $P,Q\in\wh\sss(R)$ with $P\neq Q$ and $z\in P^{(0)}\cup D_P$,
$$|x_Q-z|\approx |x_Q - x_P|\gtrsim \ell(Q) + \ell(P).$$
Hence, in the case $P\neq Q$,
$$\int |K_\Phi(x_Q-z)-K_\Phi(x_Q-x_P)|\,d\big(|\alpha|\rest_{P^{(0)}} + 
|\beta|\rest_{D_P} \big)(z)\lesssim \frac{\ell(P)}{D(P,Q)^{n+1}}\,\big(|\alpha|(P^{(0)}) + |\beta|(D_P)\big).$$
The same inequality holds in the case $P=Q$ using \rf{eqkafi1} and the fact that $\Phi(x_Q) 
\approx\ell(Q)$.
So we deduce that
$$I_2\lesssim \QQ_\alpha(Q) + \QQ_\beta(Q).$$
Gathering the estimates obtained for $I_1$, $I_2$, $I_3$, the claim \rf{eqclaim7284} follows.

Now we are ready to show \rf{eqbeta843}. By the claim just proven and using that
$\eta(B_Q)\lesssim\mu(Q)$, we obtain
\begin{align}\label{eqdhvn43}
\|\RR_\Phi \beta\|_{L^2(\eta\rest_{\wh\sG(R)^c})}^2 & =
\sum_{Q\in\wh\sss(R)} \int_{B_Q} |\RR_\Phi \beta|^2\,d\eta\\
& \lesssim \sum_{Q\in\wh\sss(R)} \inf_{x\in Q^{(0)}}|\RR_\Phi \alpha(x)|^2\,\mu(Q) + 
\sum_{Q\in\wh\sss(R)}\big(\PP_\alpha(Q)^2 + \QQ_\alpha(Q)\big)^2\,\mu(Q) \nonumber\\
 &\quad+ 
\sum_{Q\in\wh\sss(R)}\big(\PP_\beta(Q)^2 + \QQ_\beta(Q)\big)^2\,\eta(D_Q).\nonumber
\end{align}
Since $\RR_{\Phi,\mu\rest_{2R}}$ is bounded in $L^2(\mu\rest_{2R})$ with norm bounded by 
$C(\Lambda,\delta_0,K)\,\Theta(R)$, we infer that
\begin{align*}
\sum_{Q\in\wh\sss(R)} \inf_{x\in Q^{(0)}}|\RR_\Phi \alpha(x)|^2\,\mu(Q)
\leq \int_R |\RR_\Phi \alpha|^2d\mu \leq \|\RR_\Phi (f\mu)\|^2_{L^2(\mu\rest_{2R})}
\lesssim_{\Lambda,\delta_0,K} \Theta(R)^2\,\|f\|_{L^2(\mu)}^2.
\end{align*}

To estimate the second sum on the right hand side of \rf{eqdhvn43} we use the fact
that
$$\sum_{Q\in\wh\sss(R)}\QQ_\alpha(Q)^2\,\mu(Q) \lesssim 
\sum_{Q\in\wh\sss(R)}\PP_\alpha(Q)^2\,\mu(Q).$$
This follows by the same argument as in Lemma \ref{lemregpq}, with $p=2$. Indeed, in that lemma
one does not use any specific property of the measure $\mu$ or the family $\Reg$, apart from the
fact the cubes from $\Reg$ are pairwise disjoint. So the lemma applies to $\wh\sss(R)$ too.
Observe also that, for every $x\in Q$, $2B_P\subset B(x,2\ell(P))\subset CB_P$, for some $C>1$, and thus
\begin{align*}
\PP_\alpha(Q) &\leq \sum_{P\in\DD_\mu:P\supset Q} \frac{\ell(Q)}{\ell(P)^{n+1}} \,
\frac{|\alpha|(B(x,2\ell(P)))}{\mu(B(x,2\ell(P)))}\,\mu(CB_P) \\
&\lesssim \PP(Q)\,\cM_\mu f(x)\lesssim_{\Lambda,\delta_0}\Theta(R)\,\cM_\eta f(x).
\end{align*}
Consequently,
\begin{align*}
\sum_{Q\in\wh\sss(R)}\big(\PP_\alpha(Q)^2 + \QQ_\alpha(Q)\big)^2\,\mu(Q) &
\lesssim \sum_{Q\in\wh\sss(R)}\PP_\alpha(Q)^2\,\mu(Q) \\
&\lesssim_{\Lambda,\delta_0}\Theta(R)^2
\int |\cM_\mu f|^2\,d\mu\lesssim_{\Lambda,\delta_0}\Theta(R)^2\,\|f\|_{L^2(\mu)}^2.
\end{align*}

The last sum on the right hand side of \rf{eqdhvn43} is estimated similarly. Indeed, 
by Lemma \ref{lemregpq} we also have
$$\sum_{Q\in\wh\sss(R)}\QQ_\beta(Q)^2\,\eta(D_Q) \lesssim 
\sum_{Q\in\wh\sss(R)}\PP_\beta(Q)^2\,\eta(D_Q),$$
and as above,
$$\PP_\beta(Q) \lesssim_{\Lambda,\delta_0}\Theta(R)\,\cM_\eta g(x).
$$
Then it follows that
$$\sum_{Q\in\wh\sss(R)}\big(\PP_\beta(Q)^2 + \QQ_\beta(Q)\big)^2\,\mu(Q) \lesssim_{\Lambda,\delta_0}\Theta(R)^2\,\|g\|_{L^2(\eta)}^2.$$
By \rf{eqdhvn43} and the preceding esitmates, to complete the proof of \rf{eqbeta843} it just remains
to notice that, by the definition of $f$ in \rf{deff99} and Cauchy-Schwarz, $\|f\|_{L^2(\mu)}\lesssim
\|g\|_{L^2(\eta)}$.

In order to prove \rf{eqbeta8430}, we denote
$$\RR^{r(Q)/4}\beta(x) = \RR\beta(x) - \RR_{r(Q)/4}\beta(x),$$
and we split
\begin{align}\label{eqalgud77}
\|\RR (g\,\eta\rest_{\wh\sG(R)^c})\|_{L^2(\eta\rest_{\wh\sG(R)^c})}^2 
 & = \sum_{Q\in\wh\sss(R)} \int_{D_Q} |\RR \beta|^2\,d\eta\\
& \lesssim \!\sum_{Q\in\wh\sss(R)}\! \bigg(\int_{D_Q} |\RR^{r(Q)/4}\beta|^2\,d\eta +\int_{D_Q} 
|\RR_{r(Q)/4}\beta - \RR_\Phi \beta|^2\,d\eta\bigg)\nonumber\\
& \quad + \|\RR_\Phi \beta\|_{L^2(\eta\rest_{\wh\sG(R)^c})}^2.\nonumber
\end{align}
Using the fact that $\RR_{\eta\rest_{D_Q}}$ is bounded in $L^2(\eta\rest_{D_Q})$ with norm 
comparable to $\eta(B_Q)/r(Q^n)$ (because $D_Q$
is an $n$-dimensional disk), we deduce that
$$\int_{D_Q} |\RR^{r(Q)/4}\beta|^2\,d\eta = \int_{D_Q}
|\RR^{r(Q)/4}(\chi_{D_Q}g\,\eta)|^2\,d\eta \lesssim_{\Lambda,\delta_0} \Theta(R)^2\,\|\chi_{D_Q}g\|_{
L^2(\eta)}^2.$$
Regarding the second integral in \rf{eqalgud77}, 
observe that, by \rf{e.compsup''}, for all $x\in D_Q$ with
$Q\in \wh\sss(R)$,
$$\bigl|\RR_{r(Q)/4}\beta(x) - \RR_{\Phi}\beta(x)\bigr|\lesssim  \sup_{r> \Phi(x)}\frac{|\beta|(B(x,r))}{r^n} \lesssim_{\Lambda,\delta_0} \Theta(R)\,\cM_\mu g(x).$$
Then, by the last estimates and \rf{eqbeta843}, we get
\begin{align*}
\|\RR_\Phi (g\,\eta\rest_{\wh\sG(R)^c})\|_{L^2(\eta\rest_{\wh\sG(R)^c})}^2 
& \lesssim_{\Lambda,\delta_0}\Theta(R)^2 \sum_{Q\in\wh\sss(R)} \int_{D_Q} \big(|g|^2 + |\cM_\eta g|^2\big)\,d\eta + \|\RR_\Phi \beta\|_{L^2(\eta\rest_{\wh\sG(R)^c})}^2\\
&\lesssim_{\Lambda,\delta_0,K}\Theta(R)^2\,
\|g\|_{L^2(\eta)}^2,
\end{align*}
which concludes the proof of the lemma.
\end{proof}

\vv

Next, for each $R\in\DD_\mu^\PP\setminus\End(R_0)$, we define the family $\wh\End(R)$ and $\wh\tree(R)$, which can be considered as an enlarged version of $\wh\tree_0(R)$.
First we define 
\begin{equation*}
\wh{\Stop_*}(R) = (i)_R\cup (ii)_R\cup \wh\Ch((iii)_R).
\end{equation*}
Let $\wh\End(R)$ be the family of maximal $\PP$-doubling cubes which are contained in the cubes from $\wh\Stop_*(R)$.
Notice that $R\not\in\wh\End(R)$. Further, the cubes from $\wh\sss(R)\cap(\HD(R_0)\cup\BR(R)\cup(ii)_R)$
belong to $\wh\End(R)$ because they are $\PP$-doubling.
Then we let $\wh\tree(R)$ be the family of cubes from $\DD_\mu$ which are contained in $R$ and are not strictly contained in any cube from $\wh\End(R)$. Observe that we do not ask the cubes from 
$\wh\tree(R)$ to be $\PP$-doubling. Similarly, we define $\wh\tree_*(R)$ as the family of cubes from $\DD_\mu$ which are contained in $R$ and are not strictly contained in any cube from $\wh\sss_*(R)$.

\vv
\subsection{Estimating the \texorpdfstring{$\beta$}{beta} numbers on \texorpdfstring{$\wh\tree(R)$}{Tree(R)}}

Our goal is now to prove the following estimate.
\begin{lemma}\label{lembetas99}
	For each $R\in\DD_\mu^\PP\setminus\End(R_0)$, we have
	$$\sum_{Q\in\wh \tree(R)}\beta_{\mu,2}(2B_Q)^2\,\mu(Q)\lesssim_{\Lambda,\delta_0,K}\Theta(R_0)\,\mu(R).$$
\end{lemma}
We split the proof into several steps. Fix $R\in\DD_\mu^\PP\setminus\End(R_0)$. First we deal with cubes in $\wh\tree(R)\setminus\wh\tree_*(R)$.
\begin{lemma}
	We have
	\begin{equation*}
	\sum_{Q\in\wh\tree(R)\setminus\wh\tree_*(R)}\beta_{\mu,2}(2B_Q)^2\mu(Q)\lesssim_{\Lambda}\Theta(R_0)\mu(R).
	\end{equation*}
\end{lemma}
\begin{proof}
	We use the trivial bound $\beta_{\mu,2}(2B_Q)^2\lesssim\theta_\mu(2B_Q)$ and Lemma \ref{lemdobpp} to get
	\begin{align*}
	\sum_{Q\in\wh\tree(R)\setminus\wh\tree_*(R)}\beta_{\mu,2}(2B_Q)^2&\mu(Q)
	\lesssim\sum_{Q\in\wh\tree(R)\setminus\wh\tree_*(R)}\theta_{\mu}(2B_Q)\mu(Q)\\
	&=\sum_{P\in\wh\Stop_*(R)}\sum_{Q\in\wh\tree(R):\, Q\subset P }\theta_{\mu}(2B_Q)\mu(Q)\\
	&= \sum_{m\ge 0} \sum_{P\in\wh\Stop_*(R)}\sum_{\substack{Q\in\wh\tree(R)\\ Q\subset P,\, \ell(Q)=A_0^{-m}\ell(P) }}\theta_{\mu}(2B_Q)\mu(Q)\\
	&\le \sum_{m\ge 0} \sum_{P\in\wh\Stop_*(R)}\sum_{\substack{Q\in\wh\tree(R)\\ Q\subset P,\, \ell(Q)=A_0^{-m}\ell(P) }}A_0^{-m/2}\PP(P)\mu(Q)\\
	&\le \sum_{m\ge 0} \sum_{P\in\wh\Stop_*(R)}A_0^{-m/2}\PP(P)\mu(P)\approx  \sum_{P\in\wh\Stop_*(R)}\PP(P)\mu(P).
	\end{align*}
	Recall that for $Q\in\tree(R_0)$ we have $\PP(Q)\lesssim_{\Lambda}\Theta(R_0)$, and so
	\begin{equation*}
	\sum_{P\in\wh\Stop_*(R)}\PP(P)\mu(P)\lesssim_{\Lambda}\Theta(R)\sum_{P\in\wh\Stop_*(R)}\mu(P) \le \Theta(R_0)\mu(R).
	\end{equation*}
\end{proof}
It remains to prove
\begin{equation*}
\sum_{Q\in\wh \tree_*(R)}\beta_{\mu,2}(2B_Q)^2\,\mu(Q)\lesssim_{\Lambda,\delta_0,K}\Theta(R_0)\,\mu(R).
\end{equation*}
Consider the set
\begin{equation*}
\Gamma = \supp\eta = \wh\GG(R) \cup \bigcup_{Q\in\wh\sss(R)\setminus\LD(R_0)} D_Q .
\end{equation*}
Denote $\nu = \HH^n|_{\Gamma}$. We showed in Lemma \ref{lem9.2} that $\Theta(R_0)^{-1}\eta$ is an AD-regular measure, and so it follows by standard arguments (using e.g. \cite[Theorem 6.9]{Mattila-llibre}) that $\Gamma$ is an AD-regular set, and that $\eta = \rho \nu$ for some density $\rho$ satisfying $\rho\approx_{\Lambda,\delta_0}\Theta(R_0)$. It is also immediate to check that Lemma \ref{lemriesz*eta} implies that $\RR_\nu$ is bounded in $L^2(\nu)$, with
\begin{equation*}
\|\RR_\nu\|_{L^2(\nu)\to L^2(\nu)}\lesssim_{\Lambda,\delta_0,K}1.
\end{equation*}
Hence, by the main result of \cite{NToV1} we know that $\Gamma$ is uniformly $n$-rectifiable. This allows us to use the $\beta$ numbers characterization of uniform rectifiability \cite{DS1} to get
\begin{equation}\label{eq:betasUR}
\int_{B(z,r_0)}\int_0^{r_0}\beta_{\nu,2}(x,r)^2\, \frac{dr}{r}d\nu(x)\lesssim_{\Lambda,\delta_0,K}r_0^n,\quad \text{for }z\in\supp\nu, r_0\in(0,\diam(\supp\nu)).
\end{equation}

To transfer these estimates back to the measure $\mu$ we will argue similarly as in Section 7 of \cite{Azzam-Tolsa}. It will be convenient to work with regularized cubes, as we did in Section \ref{sec6.2*}. Consider a function
\begin{equation*}
d_{R,*}(x) = \inf_{Q\in\wh\tree_*(R)} (\dist(x,Q)+\ell(Q)).
\end{equation*}
Note that $d_{R,*}(x)=0$ for $x$ in the closure of  $\wh\GG(R),$ and $d_{R,*}(x)>0$ everywhere else. Moreover, $d_{R,*}$ is 1-Lipschitz. For each $x\in  R\setminus\overline{\wh\GG(R)}$ we define $Q_x$ to be the maximal cube from $\DD_\mu$ that contains $x$ and satisfies
\begin{equation*}
\ell(Q_x)\le \frac{1}{60}\inf_{y\in Q_x} d_{R,*}(y).
\end{equation*}
The family of all the cubes $Q_x,\ x\in R\setminus\overline{\wh\GG(R)},$ will be denoted by $\Reg_*(R)$. We define also the regularized tree $\treg_*(R)$ consisting of the cubes from $\DD_\mu$ that are contained in $R$ and are not strictly contained in any of the $\Reg_*(R)$ cubes. 

It follows easily from the definition of $\Reg_*(R)$ that $\wh\tree_*(R)\subset\treg_*(R)$.
Observe also that $\Reg_*(R)$ consists of pairwise disjoint cubes and satisfies
\begin{equation*}
\mu\bigg(R\setminus \bigg(\bigcup_{P\in\Reg_*(R)}P\cup \wh\GG(R)\bigg)\bigg)=0.
\end{equation*}

The following is an analogue of Lemma \ref{lem74}.
\begin{lemma}\label{lem:reg prop}
	The cubes from $\Reg_*(R)$ satisfy the following properties:
	\begin{itemize}
		\item[(a)] If $P\in\Reg_*(R)$ and $x\in B(x_{P},50\ell(P))$, then $10\,\ell(P)\leq d_{R,*}(x) \leq c\,\ell(P)$,
		where $c$ is some constant depending only on $n$. 
		\item[(b)] There exists some absolute constant $c>0$ such that if $P,\,P'\in\Reg_*(R)$ satisfy $B(x_{P},50\ell(P))\cap B(x_{P'},50\ell(P'))
		\neq\varnothing$, then
		$$c^{-1}\ell(P)\leq \ell(P')\leq c\,\ell(P).$$
		\item[(c)] For each $P\in \Reg_*(R)$, there are at most $N$ cubes $P'\in\Reg_*(R)$ such that
		$$B(x_{P},50\ell(P))\cap B(x_{P'},50\ell(P'))
		\neq\varnothing,$$
		where $N$ is some absolute constant.		
	\end{itemize}
\end{lemma}
As before, we omit the proof.

\begin{lemma}\label{lem:QtregPtree}
	For all $Q\in\treg_*(R)$ there exists some $P\in\wh\tree_*(R)$ such that $\ell(Q)\approx\ell(P)$ and $2B_Q\subset C B_{P}\subset C' B_{Q}$, where $C$ and $C'$ are some absolute constants. In consequence,
	\begin{equation}\label{eq:treg dens bdd}
	\PP(Q)\lesssim_{\Lambda}\Theta(R_0).
	\end{equation}
\end{lemma}
\begin{proof}
	Let $Q\in\treg_*(R)$. If $Q\cap\wh\GG(R)\neq\varnothing,$ then $Q\in\wh\tree_*(R)$ and we can take $P=Q$. If $Q\cap\wh\GG(R)=\varnothing,$ then there exists some $Q_0\in\Reg_*(R)$ such that $Q_0\subset Q$. By the definition of $d_{R,*}$ and Lemma \ref{lem:reg prop} (a), there exists $P_0\in\wh\tree_*(R)$ such that
	\begin{equation*}
	\dist(x_{Q_0},P_0)+\ell(P_0)\le 2 d_{R,*}(x_{Q_0})\approx \ell(Q_0).
	\end{equation*}
	In particular, $\ell(P_0)\lesssim \ell(Q_0)\le\ell(Q)$. If $\ell(P_0)\ge\ell(Q),$ set $P=P_0$, otherwise let $P$ be the ancestor of $P_0$ with $\ell(P)= \ell(Q)$. Clearly, $\ell(P)\approx\ell(Q)$, and moreover
	\begin{equation*}
	\dist(x_Q,x_P)\le \dist(x_{Q_0},P_0)+\ell(Q)+\ell(P)\lesssim\ell(Q_0)+\ell(Q)+\ell(P)\approx \ell(Q)\approx\ell(P),
	\end{equation*}
	which implies $2B_Q\subset C B_{P}\subset C' B_{Q}$ for some absolute $C$ and $C'$.
	
	Finally, to see $\PP(Q)\lesssim_{\Lambda}\Theta(R_0)$ recall that $\PP(P)\lesssim \Lambda\Theta(R_0)$ for all $P\in\wh\tree_*(R)\subset\tree(R_0)$, and we have $\PP(Q)\lesssim\PP(P)$ because $B_Q\subset C B_P$ and $\ell(Q)\approx \ell(P)$.
\end{proof}

The following lemma states that the uniformly rectifiable set $\Gamma$ lies relatively close to all the cubes from $\Reg_*(R)$. This property will be crucial in our subsequent estimates.
\begin{lemma}\label{lem:RegGamma}
	There exists $C_*=C_*(\Lambda,\delta_0)$ such that for all $Q\in\Reg_*(R)$ we have
	\begin{equation}\label{eq:RegGamma}
	\frac{C_*}{2}B_Q\cap\Gamma\neq\varnothing.
	\end{equation}
\end{lemma}
\begin{proof}
	Let $Q\in\Reg_*(R)$ and let $P\in\wh\tree_*(R)$ be the cube from Lemma \ref{lem:QtregPtree}.
	In particular, we have
	\begin{equation}\label{eq:BPCBQ}
	2B_P\subset C B_Q
	\end{equation}
	for some absolute constant $C$.
	
	If $P$ contains some cube $P_1\in\wh\Stop(R)\setminus\LD(R_0)$, then we are done, because in that case $D_{P_1}\subset 2B_P\subset C B_Q$, and $D_{P_1}\subset\Gamma$. Similarly, if $\wh\GG(R)\cap P\neq\varnothing$, then there is nothing to prove.
	
	Now suppose that $P\cap\wh\GG(R)=\varnothing$ and $P$ does not contain any cube from $\wh\Stop(R)\setminus\LD(R_0)$. Since $P\in\wh\tree_*(R)$ and $P\cap\wh\GG(R)=\varnothing$, there exists some $P_1\in\wh\sss_*(R)$ such that $P_1\subset P$. By our assumption $P_1\notin\wh\sss(R)\setminus\LD(R_0)$, and so
	\begin{equation*}
	P_1\in \wh\sss_*(R)\setminus \left(\wh\sss(R)\setminus\LD(R_0)\right) = \wh\Ch((iii)_R) \cup \left(\wh\sss(R)\cap\LD(R_0)\right).
	\end{equation*}
	
	There are two cases to consider. Suppose that $P_1\in\wh\Ch((iii)_R)$. Let $S\in (iii)_R$ be such that $P_1\in\wh\Ch(S)$. Since $S\notin\LD(R_0)$ (otherwise we'd have $S\in (i)_R$), \eqref{eqcompa492} gives $\ell(P_1)\approx_{\Lambda,\delta_0}\ell(S)$. Thus, there exists some constant  $C(\Lambda,\delta_0)$ such that
	\begin{equation*}
	D_S\subset C(\Lambda,\delta_0)B_{P_1}\subset C(\Lambda,\delta_0)B_{P}\overset{\eqref{eq:BPCBQ}}{\subset}C\, C(\Lambda,\delta_0) B_Q.
	\end{equation*}
	Since $D_S\subset\Gamma$, we get \eqref{eq:RegGamma} as soon as $C_*\ge 2C\, C(\Lambda,\delta_0)$.
	
	Finally, suppose that $P_1\in\wh\sss(R)\cap\LD(R_0)$. Let $P_0\in\wh\tree_0(R)\setminus\wh\sss(R)$ be the unique cube such that $P_1\in\wh\Ch(P_0)$. By \eqref{eqcompa492} we have $\ell(P_0)\approx_{\Lambda,\delta_0}\ell(P_1)$, and so
	\begin{equation*}
	2B_{P_0}\subset C(\Lambda,\delta_0) B_{P_1}\subset C(\Lambda,\delta_0) B_{P}\subset C\, C(\Lambda,\delta_0) B_{Q}.
	\end{equation*}
	We claim that $2B_{P_0}\cap\Gamma\neq\varnothing$, and so \eqref{eq:RegGamma} is satisfied if we assume $C_*\ge 2C\, C(\Lambda,\delta_0)$. First, if $2B_{P_0}\cap\wh\GG(R)\neq\varnothing$, then there is nothing to prove. Assume the contrary. In that case $P_0$ is covered by cubes from $\wh\Stop(R)$. We claim that there exists some $S\in\wh\sss(R)\setminus\LD(R_0)$ such that $S\subset P_0$. Indeed, otherwise $P_0$ would be covered by cubes from $\wh\sss(R)\cap\LD(R_0)$, but then
	\begin{equation*}
	-\mu(P_0)=\sum_{P'\in\wh\sss(R)\cap\LD(R_0)}-\mu(P')=\sum_{P'\in\wh\sss(R)\cap\LD(R_0)}s(P') \overset{\eqref{eqaq89}}{=}s(P_0)\overset{\eqref{eq:aQmuQ}}{\ge}0,
	\end{equation*}
	which is a contradiction. Thus, there exists $S\in\wh\sss(R)\setminus\LD(R_0)$ such that $S\subset P_0$, which implies $D_{S}\subset 2B_{P_0}$. Since $D_{S}\subset\Gamma$, we are done.
\end{proof}

In the following lemma we define functions supported on $\Gamma$ that approximate $\mu$ at the level of $\Reg_*(R)$.

\begin{lemma}\label{lem:gP}
	There exist functions $g_P:\Gamma\to\R,\ P\in\Reg_*(R),$ such that each $g_P$ is supported in $\Gamma\cap \overline{C_* B_P}$,
	\begin{equation}\label{eq:intgP}
	\int_{\Gamma} g_P\ d\nu = \mu(P),
	\end{equation}
	and
	\begin{equation}\label{eq:sumgP}
	\sum_{P\in\Reg_*(R)} g_P \lesssim_{\Lambda,\delta_0} \Theta(R).
	\end{equation}
\end{lemma}
\begin{proof}
	Assume first that the family $\Reg_*(R)$ is finite.
	 We label the cubes from $\Reg_*(R)$ in the order of increasing sidelength, that is we let $P_1$ be a cube with the minimal sidelength, and then we label all the remaining cubes so that $\ell(P_i)\le\ell(P_{i+1})$.
	
	The functions $g_{i}:=g_{P_i}$ will be of the form $g_i = \alpha_i\chi_{A_i}$ where $\alpha_i\ge 0 $ and $A_i\subset \Gamma\cap C_* B_{P_i}$. We begin by setting $\alpha_1 = \mu(P_1)/\nu(C_* B_{P_1})$ and $A_1 = C_* B_{P_1}\cap\Gamma$. Clearly, \eqref{eq:intgP} holds for $P_1$. Moreover, using the fact that $\nu$ is AD-regular and $\frac{C_*}{2} B_{P_1}\cap\Gamma\neq\varnothing$ we get
	\begin{equation*}
	\|g_1\|_\infty = \alpha_1 = \frac{\mu(P_1)}{\nu(C_* B_{P_1})}\approx_{\Lambda,\delta_0}\frac{\mu(P_1)}{\ell(P_1)^n}\overset{\eqref{eq:treg dens bdd}}{\lesssim_{\Lambda}}\Theta(R_0).
	\end{equation*}
	
	We define the remaining $g_k,\, k\ge 2$ inductively. Suppose that $g_1,\dots,g_{k-1}$ have already been constructed, and they satisfy
	\begin{equation}\label{eq:sumgi}
	\sum_{i=1}^{k-1} g_i \le C' \Theta(R)
	\end{equation}
	for some constant $C'=C'(\Lambda,\delta_0)$ to be fixed below. Let $P_{i_1},\dots,P_{i_m}$ be the subfamily of $P_1,\dots,P_{k-1}$ consisting of cubes such that $C_* B_{P_k}\cap C_* B_{P_{i_j}}\neq\varnothing$. Due to the non-decreasing sizes of $P_i$'s we have $P_{i_j}\subset C_* B_{P_{i_j}}\subset 3C_* B_{P_k}$. Hence, applying \eqref{eq:intgP} to $g_{i_j}$ we get
	\begin{equation*}
	\sum_j \int_{\Gamma} g_{i_j}\ d\nu = \sum_j\mu(P_{i_j})\le\mu(3C_* B_{P_k})\overset{\eqref{eq:treg dens bdd}}{\le}C({\Lambda,\delta_0})\Theta(R_0)\ell(P_k)^n\le C''\Theta(R_0)\nu(\Gamma\cap C_* B_{P_k}),
	\end{equation*}
	for some $C''$ depending on $\Lambda,\delta_0$. By the Chebyshev's inequality
	\begin{equation*}
	\nu\left(\Gamma\cap\big\{\textstyle{\sum_j}\, g_{i_j}\ge 2C''\Theta(R_0)\big\}\right)\le \frac{1}{2}\nu(\Gamma\cap C_* B_{P_k}).
	\end{equation*}
	Set
	\begin{equation*}
	A_k = \Gamma\cap C_* B_{P_k}\cap \big\{\textstyle{\sum_j}\, g_{i_j}\le 2C''\Theta(R_0)\big\},
	\end{equation*}
	and then by the preceding estimate $\nu(A_k)\ge \nu(\Gamma\cap C_* B_{P_k})/2.$ We define
	\begin{equation*}
	\alpha_k = \frac{\mu(P_k)}{\nu(A_k)},
	\end{equation*}
	so that for $g_k=\alpha_k\chi_{A_k}$ we have $\int g_k\, d\nu = \mu(P_k)$. Moreover, using AD-regularity of $\nu$ and the fact that $\frac{C_*}{2} B_{P_k}\cap\Gamma\neq\varnothing$
	\begin{equation*}
	\alpha_k \le 2 \frac{\mu(P_k)}{\nu(C_* B_{P_k})} \le C({\Lambda,\delta_0})\frac{\mu(P_k)}{\ell(P_k)^n}\overset{\eqref{eq:treg dens bdd}}{\le} C'''\Theta(R_0)
	\end{equation*}
	for some $C'''$ depending on $\Lambda,\delta_0$. Hence, by the definition of $A_k$
	\begin{equation*}
	g_k(x) + \sum_j g_{i_j}(x) \le C'''\Theta(R_0) + 2C''\Theta(R_0), \quad\text{for $x\in A_k$.}
	\end{equation*}
	For $x\not\in A_k$ we have $g_k=0$, and so it follows from the above and the inductive assumption \eqref{eq:sumgi} that for $C' = C''' + 2C''$ we have
	\begin{equation*}
	\sum_{i=1}^{k} g_i \le C' \Theta(R),
	\end{equation*}
	which closes the induction.
	
	Suppose now that the family $\Reg_*(R)$ is infinite. We can relabel it so that $\Reg_*(R)=\{P^i\}_{i\in\N}$. For each $N$ we consider the family $\{P^i\}_{1\le i\le N}$. We construct functions $g_{P^1}^N,\dots, g_{P^N}^N$ as above, so that they satisfy
	\begin{equation*}
	\int g_{P^1}^N\ d\nu = \mu(P^1),\qquad \sum_{i=1}^{N} g_{P^i}^N \le C' \Theta(R).
	\end{equation*}
	There exists a subsequence $I_1\subset\N$ such that $\{g_{P^1}^k\}_{k\in I_1}$ is convergent in the weak-$*$ topology of $L^\infty(\nu)$ to some function $g_{P^1}\in L^{\infty}(\nu)$. We take another subsequence $I_2\subset I_1$ such that $\{g_{P^2}^k\}_{k\in I_2}$ is convergent in the weak-$*$ topology of $L^\infty(\nu)$ to some $g_{P^2}\in L^{\infty}(\nu)$. Proceeding in this fashion we obtain a family $\{g_{P^i}\}_{i\in\N}$ such that $\supp g_{P^i}\subset \overline{C_* B_{P^i}}$, and the properties \eqref{eq:intgP}, \eqref{eq:sumgP} are preserved (because of the weak-$*$ convergence).
\end{proof}

Recall that by the uniform rectifiability of $\nu$ we have a good estimate on the $\beta_{\nu,2}$ numbers \eqref{eq:betasUR}. We will now use Lemmas \ref{lem:RegGamma} and \ref{lem:gP} to transfer these estimates to the measure $\mu$ and obtain
\begin{equation*}
	\sum_{Q\in\wh \tree_*(R)}\beta_{\mu,2}(2B_Q)^2\,\mu(Q)\lesssim_{\Lambda,\delta_0,K}\Theta(R_0)\,\mu(R).
\end{equation*}
In fact, we will show that
\begin{equation*}
	\sum_{Q\in\treg_*(R)}\beta_{\mu,2}(2B_Q)^2\,\mu(Q)\lesssim_{\Lambda,\delta_0,K}\Theta(R_0)\,\mu(R),
\end{equation*}
and the former estimate will follow, since $\wh \tree_*(R)\subset\treg_*(R)$.

Let $Q\in\treg_*(R)$, and let $L_Q$ be an $n$-plane minimizing $\beta_{\nu,2}(C_*'B_Q)$, where $C_*'>2$ is some constant depending on $C_*$, to be chosen in Lemma \ref{lem:PsizeQ}. We estimate
\begin{multline}\label{eq:bet1}
\beta_{\mu,2}(2B_Q)^2\mu(Q)\lesssim \frac{\mu(Q)}{\ell(Q)^n}\int_{2B_Q} \left(\frac{\dist(x,L_Q)}{\ell(Q)}\right)^2\,d\mu(x)\\
\overset{\eqref{eq:treg dens bdd}}{\lesssim_\Lambda}\Theta(R_0)\int_{2B_Q} \left(\frac{\dist(x,L_Q)}{\ell(Q)}\right)^2\,d\mu(x)\\
= \Theta(R_0)\left(\int_{2B_Q\cap R\setminus \wh\GG(R)}\dots\,d\mu(x) + \int_{2B_Q\cap\wh\GG(R)} \dots\,d\mu(x) + \int_{2B_Q\setminus R} \dots\,d\mu(x)\right)\\
=:\Theta(R_0)(I_1+I_2+I_3).
\end{multline}
Concerning $I_3$, we use the trivial estimate
\begin{equation}\label{eq:bet2}
I_3\lesssim\mu(2B_Q\setminus R).
\end{equation}
Estimating $I_2$ is simple because on $\wh\GG(R)$ we have $\mu=\rho'\nu$ for some $\rho'\lesssim_{\Lambda} \Theta(R_0)$, and so
\begin{multline}\label{eq:bet3}
I_2 \lesssim_{\Lambda} \Theta(R_0)\int_{2B_Q\cap\wh\GG(R)} \left(\frac{\dist(x,L_Q)}{\ell(Q)}\right)^2\,d\nu(x)\le \Theta(R_0)\int_{C_*'B_Q} \left(\frac{\dist(x,L_Q)}{\ell(Q)}\right)^2\,d\nu(x)\\
\approx_{C_*'} \Theta(R_0)\beta_{\nu,2}(C_*'B_Q)^2\ell(Q)^n.
\end{multline}

Bounding $I_1$ requires more work. First, we use the fact that $R\setminus\wh\GG(R)$ is covered $\mu$-a.e. by $\Reg_*(R)$:
\begin{align*}
I_1 &\le \sum_{P\in\Reg_*(R) :\,  P\cap 2B_Q\neq\varnothing} \int_{P} \left(\frac{\dist(x,L_Q)}{\ell(Q)}\right)^2\,d\mu(x) \\
&= \sum_{P\in\Reg_*(R) :\,  P\cap 2B_Q\neq\varnothing} 
\bigg(\int_{\Gamma} \left(\frac{\dist(x,L_Q)}{\ell(Q)}\right)^2 g_P(x)\,d\nu(x)\\
&\quad+ \int \left(\frac{\dist(x,L_Q)}{\ell(Q)}\right)^2 (\chi_{P}(x)d\mu(x) - g_P(x)d\nu(x))\bigg)
\\
&=: \sum_{P\in\Reg_*(R) :\,  P\cap 2B_Q\neq\varnothing} \left(I_{11}(P) + I_{12}(P)\right).
\end{align*}
We need the following auxiliary result.

\begin{lemma}\label{lem:PsizeQ}
	If $P\in\Reg_*(R)$ is such that $P\cap 2B_Q\neq\varnothing$, then $\ell(P)\lesssim\ell(Q)$ and in consequence $\overline{C_*B_P}\subset C_*'B_Q$ for some $C_*'=C_*'(C_*)\ge C_*$.
\end{lemma}
\begin{proof}
	If $\ell(P)\le\ell(Q)$ then there is nothing to prove, so suppose $\ell(P)>\ell(Q)$ (in particular $\ell(P)\ge A_0\ell(Q))$. In that case we have $2B_Q\subset 2B_P$. 
	
	Note that if we had $Q\setminus\wh\GG(R)=\varnothing$, then $d_{R,*}(x_Q)=0$, but by Lemma \ref{lem:reg prop} (a) we know that $d_{R,*}(x_Q)\ge 10\ell(P)$.
	Hence, there exists some $S\in\Reg_*(R)$ such that $S\subset Q$. Together with the fact that $2B_Q\subset 2B_P$ this implies $B_S\cap 2B_P\neq\varnothing$. By Lemma \ref{lem:reg prop} (b) this gives
	\begin{equation*}
	\ell(P)\approx \ell(S)\le\ell(Q).
	\end{equation*}
\end{proof}
By the lemma above and the preceding estimate we get
\begin{equation}\label{eq:bet4}
I_1 \le \sum_{P\in\Reg_*(R) :\,  C_*B_P\subset C_*'B_Q} I_{11}(P) + \sum_{P\in\Reg_*(R) :\,  C_*B_P\subset C_*'B_Q} I_{12}(P).
\end{equation}
We estimate the first sum as follows:
\begin{multline}\label{eq:bet5}
\sum_{P\in\Reg_*(R) :\,  C_*B_P\subset C_*'B_Q} \int_{\Gamma} \left(\frac{\dist(x,L_Q)}{\ell(Q)}\right)^2 g_P(x)\,d\nu(x)\\
=  \int_{\Gamma} \left(\frac{\dist(x,L_Q)}{\ell(Q)}\right)^2 \sum_{P\in\Reg_*(R) :\,  C_*B_P\subset C_*'B_Q}  g_P(x)\,d\nu(x)\\
\overset{{\supp g_P\subset C_* B_P}}{\le} \int_{\Gamma\cap C_*'B_Q} \left(\frac{\dist(x,L_Q)}{\ell(Q)}\right)^2 \sum_{P\in\Reg_*(R)}  g_P(x)\,d\nu(x)\\
\overset{\eqref{eq:sumgP}}{\lesssim}_{\Lambda,\delta_0}\Theta(R_0)\int_{\Gamma\cap C_*'B_Q} \left(\frac{\dist(x,L_Q)}{\ell(Q)}\right)^2\,d\nu(x)\approx_{C_*'} \Theta(R_0)\beta_{\nu,2}(C_*'B_Q)^2\ell(Q)^n.
\end{multline}
Concerning $I_{12}(P)$, observe that since $\int g_P\, d\nu=\mu(P)$ by \eqref{eq:intgP}, we have
\begin{equation*}
I_{12}(P) = \int \left(\left(\frac{\dist(x,L_Q)}{\ell(Q)}\right)^2 - \left(\frac{\dist(x_P,L_Q)}{\ell(Q)}\right)^2\right) \left(\chi_{P}(x)d\mu(x) - g_P(x)d\nu(x)\right).
\end{equation*}
For $x\in\supp(\chi_{P}(x)d\mu(x) - g_P(x)d\nu(x))\subset C_*B_P\subset C_*'B_Q$ we have 
\begin{multline*}
\left|\left(\frac{\dist(x,L_Q)}{\ell(Q)}\right)^2 - \left(\frac{\dist(x_P,L_Q)}{\ell(Q)}\right)^2\right| \le \frac{|x-x_P|}{\ell(Q)}\cdot\frac{\dist(x,L_Q)+\dist(x_P,L_Q)}{\ell(Q)}\\
\lesssim\frac{C_*\ell(P)}{\ell(Q)}\cdot\frac{C_*'\ell(Q)}{\ell(Q)} \approx_{C_*,C_*'} \frac{\ell(P)}{\ell(Q)}.
\end{multline*}
Hence,
\begin{equation}\label{eq:bet6}
I_{12}(P) \lesssim_{C_*,C_*'}\frac{\ell(P)}{\ell(Q)}\mu(P).
\end{equation}
Recall that $C_*$ depends on $\Lambda,\delta_0$, and $C_*'$ depends on $C_*$. Thus, putting together the estimates \eqref{eq:bet4}, \eqref{eq:bet5}, and \eqref{eq:bet6} yields
\begin{equation*}
I_1\lesssim_{\Lambda,\delta_0} \Theta(R_0)\beta_{\nu,2}(C_*'B_Q)^2\ell(Q)^n + \sum_{\substack{P\in\Reg_*(R):\\  C_*B_P\subset C_*'B_Q}}\frac{\ell(P)}{\ell(Q)}\mu(P).
\end{equation*}
Together with \eqref{eq:bet1}, \eqref{eq:bet2}, and \eqref{eq:bet3} this gives
\begin{multline*}	
\beta_{\mu,2}(2B_Q)^2\mu(Q)\lesssim_{\Lambda,\delta_0} \Theta(R_0)^2\beta_{\nu,2}(C'_* B_Q)^2\ell(Q)^n \\ +
\Theta(R_0)\sum_{\substack{P\in\Reg_*(R):\\  C_*B_P\subset C_*'B_Q}}\frac{\ell(P)}{\ell(Q)}\mu(P) + \Theta(R_0)\mu(2B_Q\setminus R).
\end{multline*}
Summing over $Q\in\treg_*(R)$ we get
\begin{multline*}	
\sum_{Q\in\treg_*(R)}\beta_{\mu,2}(2B_Q)^2\mu(Q)\lesssim_{\Lambda,\delta_0} \Theta(R_0)^2\sum_{Q\in\treg_*(R)}\beta_{\nu,2}(C'_* B_Q)^2\ell(Q)^n \\+
\Theta(R_0)\sum_{Q\in\treg_*(R)}\sum_{\substack{P\in\Reg_*(R):\\  C_*B_P\subset C_*'B_Q}}\frac{\ell(P)}{\ell(Q)}\mu(P) + \Theta(R_0)\sum_{Q\in\treg_*(R)}\mu(2B_Q\setminus R)\\
=:\Theta(R_0)^2 S_1+\Theta(R_0)S_2+\Theta(R_0)S_3.
\end{multline*}

Concerning $S_1$, note that by \eqref{eq:RegGamma} we know that if $Q\in\treg_*(R)$, then $\nu(C_* B_Q\cap\Gamma)\approx_{\Lambda,\delta_0}\ell(Q)^n$ and for all $x\in C_* B_Q\cap\Gamma$ we have $C'_* B_Q\subset B(x,2C_*' \ell(Q))$. Thus, $\beta_{\nu,2}(C'_* B_Q)\lesssim \beta_{\nu,2}(x,r)$ for $2C_*' \ell(Q)<r<3C_*' \ell(Q)$. Observe also that the sets $C_* B_Q\cap\Gamma$ corresponding to cubes of the same generation have bounded intersection. It follows easily that
\begin{multline*}
S_1=\sum_{Q\in\treg_*(R)}\beta_{\nu,2}(C'_* B_Q)^2\ell(Q)^n \lesssim_{\Lambda,\delta_0} \int_{5C_*' B_{R}}\int_0^{5C_*'\ell(R)}\beta_{\nu,2}(x,r)^2\, \frac{dr}{r}d\nu(x)\\
\overset{\eqref{eq:betasUR}}{\lesssim}_{\Lambda,\delta_0,K}\ell(R)^n.
\end{multline*}

To estimate $S_2$ we change the order of summation:
\begin{equation*}
S_2 = \sum_{P\in\Reg_*(R)}\mu(P)\sum_{\substack{Q\in\treg_*(R):\\  C_*B_P\subset C_*'B_Q}}\frac{\ell(P)}{\ell(Q)}.
\end{equation*}
Note that the inner sum is essentially a geometric series, and so
\begin{equation*}
S_2 \lesssim_{C_*,C_*'} \sum_{P\in\Reg_*(R)}\mu(P)\le \mu(R).
\end{equation*}

Finally, we can bound $S_3$ using the small boundaries property of the David-Mattila lattice \eqref{eqsmb2}. To be more precise, note that for $Q\in\treg_*(R)$ if $2B_Q\setminus R\neq\varnothing$ and $\ell(Q)=A_0^{-k}\ell(R)$, then necessarily $Q\subset N_{k-1}(R)$, and even $2B_Q\subset N_{k-1}(R)$. Furthermore, the balls $2B_Q$ for cubes of the same generation have only bounded intersection. Thus,
\begin{equation*}
S_3\le \sum_{k\ge 1}\sum_{\substack{Q\subset N_{k-1}(R),\\ \ell(Q)=A_0^{-k}\ell(R)}}\mu(2B_Q)\lesssim \sum_{k\ge 1}\mu(N_{k-1}(R))\overset{\eqref{eqsmb2}}{\lesssim}\mu(90B(R))\approx\mu(R).
\end{equation*}

Putting the estimates for $S_1,\ S_2$ and $S_3$ together we arrive at
\begin{equation*}
\sum_{Q\in\treg_*(R)}\beta_{\mu,2}(2B_Q)^2\mu(Q)\lesssim_{\Lambda,\delta_0,K} \Theta(R_0)^2\ell(R)^n + \Theta(R_0)\mu(R)\lesssim_{\delta_0}\Theta(R_0)\mu(R),
\end{equation*}
where in the last estimate we used the fact that $\Theta(R_0)\lesssim\delta_0^{-1}\Theta(R)$ (note that $R\not\in\LD(R_0)$ because $\DD_\mu^\PP\cap\LD(R_0)\subset\End(R_0)$ and we assume $R\in\DD_\mu^\PP\setminus\End(R_0)$). This finishes the proof of Lemma \ref{lembetas99}.


\vv

\subsection{The corona decompostion and the proof of Lemma \ref{lemtreebeta}}

Now we 
 define $\wh \ttt= \wh\ttt(R_0)$ inductively.
We set $\wh \ttt_0=\{R_0\}$ and, assuming $\wh\ttt_k$ to be defined, we let
$$\wh\ttt_{k+1} = \bigcup_{R\in\wh\ttt_k} (\wh \End(R)\setminus\End(R_0)).$$
Then we let
$$\wh\ttt =\bigcup_{k\geq0} \wh\ttt_k.$$
In this way, we have
$$\tree(R_0)=\bigcup_{R\in\wh\ttt} \wh\tree(R).$$

\vv

\begin{lemma}\label{lemtop8}
We have
$$\sigma(\wh \ttt)  \lesssim_{\Lambda,\delta_0}\sigma(R_0) +\!\sum_{Q\in\tree(R_0)}\!\|\Delta_Q\RR\mu\|_{L^2(\mu)}^2+ \sum_{Q\in\tree(R_0)\cap\HE}\! \EE(4Q).$$
\end{lemma}

\begin{proof} 
By Lemma \ref{lem9.5*}, we have
$$\sum_{Q\in\wh\sss(R)\cap(\LD(R_0)\cup\HD(R_0)\cup\BR(R))}\mu(Q) 
+\sum_{Q\in\wh\Ch((iii)_R)\cap \LD(R_0)}\mu(Q) + \mu(\wh\sG(R))
\approx \mu(R).$$
By the construction of $\wh\tree(R)$, the cubes from 
$\wh\sss(R)\cup \wh\Ch((iii)_R)$ belong to $\wh\tree(R)$ and the ones from $\wh\sss(R)\cap\BR(R)$ belong to $\wh \End(R)$, and so
$$\mu(R)\approx \sum_{Q\in\wh\tree(R)\cap(\LD(R_0)\cup\HD(R_0))}\mu(Q) 
+ \sum_{Q\in\wh\End(R)\cap\BR(R)}\!\!\mu(Q)
+ \mu(\wh\sG(R)).$$
Notice that the families $\wh\tree(R)$, with $R\in\wh\ttt$, are disjoint, with the possible exception of
the roots and ending cubes of the trees $\wh\tree(\cdot)$, which may belong to two different trees. 
Then we deduce that
\begin{align*}
\sum_{R\in\wh\ttt} \mu(R) & \approx \sum_{R\in\wh\ttt} \bigg(
\sum_{Q\in\wh\tree(R)\cap(\LD(R_0)\cup\HD(R_0))}\!\!\!\!\mu(Q) + \sum_{Q\in\wh\End(R)\cap\BR(R)}\!\!\!\mu(Q)\bigg)
+\sum_{R\in\wh\ttt}\! \mu(\wh\sG(R))\\
& \lesssim \mu(R_0) + \sum_{R\in\wh\ttt}\,\sum_{Q\in\wh\End(R)\cap\BR(R)}\!\!\mu(Q).
\end{align*}
Since the cubes $Q\in\BR(R)$ do not belong to $\LD(R_0)$, we have $\Theta(Q)\approx_{\Lambda,\delta_0} \Theta(R_0)$ for such cubes. The same happens for $R\in\wh\ttt$, and thus
\begin{equation}\label{eqthus883}
\sigma(\wh \ttt)  \lesssim_{\Lambda,\delta_0} \sigma(R_0) + \sum_{R\in\wh\ttt} \Theta(R)^2\sum_{Q\in\wh\End(R)\cap\BR(R)}\mu(Q).
\end{equation}

To estimate the last sum above, we claim that for a given $Q\in\BR(R)\cap\wh\End(R)$ we have
$$\big|\RR(\chi_{2R\setminus 2Q}\mu)(x_Q) - 
\big(m_{\mu,Q}(\RR\mu) - m_{\mu,R}(\RR\mu)\big)\big|
\lesssim \PP(R) + \left(\frac{\EE(4R)}{\mu(R)}\right)^{1/2} + \PP(Q) +  \left(\frac{\EE(2Q)}{\mu(Q)}\right)^{1/2}.$$
This is proved exactly in the same way as Lemma \ref{lemaprox2} (see also \rf{eqal842}) and so we omit the arguments. 
In case that both $R,Q\not\in \HE$, then 
$$\left(\frac{\EE(4R)}{\mu(R)}\right)^{1/2}\leq M_0\,\Theta(R)\quad \text{ and }\quad   \left(\frac{\EE(2Q)}{\mu(Q)}\right)^{1/2}\leq M_0\,\Theta(Q),$$
and so
we get
$$\big|\RR(\chi_{2R\setminus 2Q}\mu)(x_Q) - 
\big(m_{\mu,Q}(\RR\mu) - m_{\mu,R}(\RR\mu)\big)\big|
\leq C(\Lambda,\delta_0,M_0) \Theta(R).$$
Thus, by the $\BR(R)$ condition,
$$K\,\Theta(R) \leq |\RR(\chi_{2R\setminus 2Q}\mu)(x_Q)| \leq 
\big|m_{\mu,Q}(\RR\mu) - m_{\mu,R}(\RR\mu)\big| + C(\Lambda,\delta_0,M_0) \Theta(R).$$
Hence, for $K\geq 2\,C(\Lambda,\delta_0,M_0)$, we obtain
$$\frac12K\,\Theta(R) \leq 
\big|m_{\mu,Q}(\RR\mu) - m_{\mu,R}(\RR\mu)\big|.$$

In the general case where $Q$ and $R$ may belong to $\HE$, by analogous arguments, we get
$$\frac12K\,\Theta(R) \leq 
\big|m_{\mu,Q}(\RR\mu) - m_{\mu,R}(\RR\mu)\big| + \chi_\HE(R) \left(\frac{\EE(4R)}{\mu(R)}\right)^{1/2}
+ \chi_\HE(Q) \left(\frac{\EE(2Q)}{\mu(Q)}\right)^{1/2},$$
where $\chi_\HE(P)=1$ if $P\in\HE$ and $\chi_\HE(P)=0$ otherwise. Since
$$m_{\mu,Q}(\RR\mu) - m_{\mu,R}(\RR\mu) = \chi_Q \sum_{P\in\wh\tree(R)\setminus\wh\End(R)}\Delta_P(\RR\mu),$$
assuming $K\geq1$, we get
\begin{align}\label{eqfje42}
\Theta(R)^2\sum_{Q\in\wh\End(R)\cap\BR(R)}\mu(Q) &\lesssim 
\sum_{Q\in\wh\End(R)\cap\BR(R)} \int_Q \Big|\sum_{P\in\wh\tree(R)\setminus\wh\End(R)}\Delta_P(\RR\mu)
\Big|^2\,d\mu \\
&\quad + \chi_\HE(R)\!\!\sum_{Q\in\wh\End(R)\cap\BR(R)}  \frac{\EE(4R)}{\mu(R)}\,\mu(Q) +\!
\sum_{Q\in\wh\End(R)\cap\HE} \! \EE(2Q).\nonumber
\end{align}
By orthogonality, the first sum on the right hand is bounded by
$$\int \Big|\sum_{P\in\wh\tree(R)\setminus\wh\End(R)}\Delta_P(\RR\mu)
\Big|^2\,d\mu = \sum_{P\in\wh\tree(R)\setminus\wh\End(R)}\|\Delta_P(\RR\mu)\|_{L^2(\mu)}^2.$$
Also, it is clear that the second sum on the right had side of \rf{eqfje42} does not exceed 
$\chi_\HE(R)\,\EE(4R)$. Therefore,
\begin{multline*}
\Theta(R)^2\!\!\sum_{Q\in\wh\End(R)\cap\BR(R)}\mu(Q) \\
\lesssim
\sum_{P\in\wh\tree(R)\setminus\wh\End(R)}\|\Delta_P(\RR\mu)\|_{L^2(\mu)}^2
+ \chi_\HE(R)\,\EE(4R) + \sum_{Q\in\wh\End(R)\cap\HE} \! \EE(2Q).
\end{multline*}

Plugging the previous estimate into \rf{eqthus883}, we obtain
\begin{align*}
\sigma(\wh \ttt)  & \lesssim_{\Lambda,\delta_0} \sigma(R_0) + \sum_{R\in\wh\ttt} 
\sum_{P\in\wh\tree(R)\setminus\wh\End(R)}\|\Delta_P(\RR\mu)\|_{L^2(\mu)}^2\\
&\quad
+ \sum_{R\in\wh\ttt\cap\HE}\EE(4R) +\sum_{R\in\wh\ttt}\,\sum_{Q\in\wh\End(R)\cap\HE} \! \EE(2Q)\\
& \lesssim_{\Lambda,\delta_0} \sigma(R_0) + \sum_{P\in\tree(R_0)} 
\|\Delta_P(\RR\mu)\|_{L^2(\mu)}^2 
+ \sum_{Q\in\tree(R_0)\cap\HE}\EE(4Q),
\end{align*}
as wished.
\end{proof}
\vv

\begin{proof}[\bf Proof of Lemma \ref{lemtreebeta}]
Given $R_0\in\ttt$, combining Lemmas \ref{lembetas99} and \ref{lemtop8}, we obtain
\begin{align*}
\sum_{Q\in\tree(R_0)} \!\!\!\beta_{\mu,2}(2B_Q)^2\,\Theta(Q)\,\mu(Q)  & 
\lesssim_{\Lambda,\delta_0}  \sum_{R\in \wh\ttt} \Theta(R)\sum_{Q\in\wh\tree(R)} \beta_{\mu,2}(2B_Q)^2\,\mu(Q)\\
& \lesssim_{\Lambda,\delta_0}  \sum_{R\in \wh\ttt} \Theta(R)^2 \mu(R)\\
& 
\lesssim_{\Lambda,\delta_0}\sigma(R_0) +\!\!\sum_{Q\in\tree(R_0)}\!\|\Delta_Q\RR\mu\|_{L^2(\mu)}^2+ \!\!\sum_{Q\in\tree(R_0)\cap\HE}\! \!\EE(4Q).
\end{align*}
\end{proof}